\documentclass[a4paper,leqno,11pt]{amsart}
\usepackage{etex}
\usepackage[top=1.5in,bottom=1.3in,left=1.3in,right=1.3in,marginparwidth=1.5cm]{geometry}

\usepackage{times}
\usepackage{tensor}

\usepackage{amssymb}
\usepackage{amsmath}
\usepackage{amsthm}
\usepackage{amsxtra}
\usepackage{latexsym}
\usepackage{tikz}
\usepackage{mathdots}
\usepackage{mathtools}
\usepackage[new]{old-arrows} 
\def\checkmark{\tikz\fill[scale=0.4](0,.35) -- (.25,0) -- (1,.7) -- (.25,.15) -- cycle;}

\usepackage{pict2e}
\usepackage{caption}

\usepackage[mathscr]{eucal}
\usepackage{mathrsfs}
\usepackage{upgreek}
\usepackage{wasysym}
\usepackage[all,cmtip,dvipdfmx]{xy}
\usepackage{xcolor}
\definecolor{rouge}{rgb}{0.85,0.1,.4}
\definecolor{bleu}{rgb}{0.1,0.2,0.9}
\definecolor{violet}{rgb}{0.7,0,0.8}

\usepackage{multirow}
\usepackage{diagbox}
\usepackage{arydshln}
\usepackage{bbold}
\usepackage{fix-cm} 
\usepackage{calrsfs}
\usepackage{calligra}

\usepackage[english]{babel}

\usepackage{enumerate}
\usepackage[driverfallback=dvipdfmx,
colorlinks=true,linkcolor=bleu,urlcolor=violet,citecolor=rouge]{hyperref}

\newcommand{\ov}{\overline}
\newcommand{\asym}{an asymptotic datum} 
\newcommand{\W}{\mathscr{W}}

\newcommand{\bra}{{\langle}}
\newcommand{\ket}{{\rangle}}
\newcommand{\ch}{\on{ch}}

\newcommand{\Lam}{\Lambda}

\newcommand{\cprime}{$'$}
\newcommand{\on}{\operatorname}
\newcommand{\+}{\mathop{\oplus}}
\renewcommand{\*}{\otimes}
\newcommand{\mc}{\mathcal}
\newcommand{\mf}{\mathfrak}

\newcommand{\fing}{\mf{g}}

\newcommand{\affg}{\widehat{\mf{g}}}

\newcommand{\affh}{\widehat{\mf{h}}}

\newcommand{\Z}{\mathbb{Z}}
\newcommand{\C}{\mathbb{C}}

\newcommand{\ra}{\rightarrow}
\newcommand{\lam}{\lambda}

\newcommand{\vac}{{|0\rangle}}

\newcommand{\imag}{\mathrm{i}} 

\newcommand{\KL}{\mathbf{KL}}
\newcommand{\ooh}{\circ}

\newcommand{\bs}{\boldsymbol}

\def\A{\mathbf{A}}
\def\Weight{\mathbf{w}}
\def\G{\mathbf{g}}
\def\g{\mathfrak{g}}
\def\l{\mathfrak{l}}

\def\h{\mathfrak{h}}

\def\O{\mathbb{O}}

\def\Slo{\mathcal{S}}

\def\P{\mathscr{P}}
\def\sl{\mathfrak{sl}}
\def\so{\mathfrak{so}}
\def\sp{\mathfrak{sp}}
\def\eps{\varepsilon}

\def\bs{\boldsymbol}
\def\P{\mathscr{P}}

\def\le{\leqslant}
\def\leq{\leqslant}
\def\ge{\geqslant}
\def\geq{\geqslant}

\DeclareMathOperator{\ad}{ad}
\DeclareMathOperator{\Hom}{Hom}

\theoremstyle{plain}
\newtheorem{Th}{Theorem}[section]

\newtheorem{Pro}[Th]{Proposition}

\newtheorem{Lem}[Th]{Lemma}

\newtheorem{Co}[Th]{Corollary}

\newtheorem{Conj}[Th]{Conjecture}

\theoremstyle{remark}
\newtheorem{Def}[Th]{Definition}

\newtheorem{Rem}[Th]{Remark}

\newlength{\larg}
\title[Collapsing levels of $W$-algebras]{Singularities of
nilpotent Slodowy slices and
collapsing levels of $W$-algebras}

\subjclass[2010]{}
\keywords{}
\author[Tomoyuki Arakawa]{Tomoyuki Arakawa} 
\address{\textsuperscript{1}Research Institute for Mathematical Sciences
\\ Kyoto University\\ 
 Kyoto 606-8502 JAPAN}
\email{arakawa@kurims.kyoto-u.ac.jp}

\author[Jethro van Ekeren]{Jethro van Ekeren} 
\address{\textsuperscript{2}
Instituto de Matem\'{a}tica Pura e Aplicada\\
Jardim Bot\^{a}nico, RJ 22.460-320 BRAZIL}
\address{\textsuperscript{3}
Departamento de Matem\'{a}tica Aplicada\\
Universidade Federal Fluminense \\
Niter\'{o}i, RJ 24.210-201 BRAZIL}
\email{jethro@impa.br}

\author[Anne Moreau]{Anne Moreau} 
\address{\textsuperscript{4}Universit\'{e} Paris-Saclay, CNRS\\ 
Laboratoire de Math\'{e}matiques d'Orsay\\
Rue Michel Magat, B\^{a}t. 307\\
91405 Orsay, FRANCE}
\email{anne.moreau@universite-paris-saclay.fr}

\begin{document}

\begin{abstract}
We develop techniques to construct isomorphisms between simple 
affine $W$-algebras and affine vertex algebras at admissible levels. 
We then apply these techniques to obtain many new, and conjecturally all, admissible collapsing levels for affine 
$W$-algebras. 
In short, if a simple affine $W$-algebra at a given level is equal to its affine vertex subalgebra generated by the centraliser of an $\sl_2$-triple associated with the underlying nilpotent orbit, then that level is said to be {\em collapsing}. 
Collapsing levels are important both in representation theory and in theoretical physics. 
Our approach relies on two fundamental invariants of vertex algebras. 
The first one is the {\em associated variety} which, in the context of admissible level simple affine 
$W$-algebras, leads to the Poisson varieties known as nilpotent Slodowy slices. 
We exploit the singularities of these varieties to detect possible collapsing levels. 
The second invariant is the {\em asymptotic datum}. We prove a general result asserting that, under appropriate hypotheses, equality of asymptotic data implies isomorphism at the level of vertex algebras. 
Then we use this 
to give a sufficient criterion, of combinatorial nature, for an admissible level to be collapsing. 
Our methods also allow us to study isomorphisms between quotients  
of $W$-algebras and extensions of simple affine vertex algebras at admissible levels. Based on such examples, we are led to formulate a general conjecture: for any finite extension of vertex algebras, the induced morphism between associated Poisson
varieties is dominant. 
\end{abstract}

\maketitle

\section{Introduction}

Vertex algebras have emerged as fundamental objects mediating interactions between representation theory and two-dimensional conformal field theory, three-dimensional topology and, more recently, four-dimensional physics in the guise of the 4D/2D duality. 
This article concerns two important classes of vertex algebras: the affine Kac-Moody vertex algebras and the affine $W$-algebras. 
Examples of $W$-algebras were first uncovered as vertex algebra extensions of the Virasoro algebra \cite{Fat.Zam.1987}, and these algebras were later perceived to be natural quantisations of Poisson structures of certain completely integrable models (the Gelfand-Dickii brackets in the KP hierarchy for instance, see the review \cite{Bou.Sch.review} and references therein), as well as natural chiralisations of quantised transverse slices. By $W$-algebra in general we shall mean the quantised Hamiltonian (Drinfeld-Sokolov) reduction of an affine Kac-Moody vertex algebra \cite{FeiFre90,KacRoaWak03}, or a quotient thereof. 

\subsection{Affine vertex algebras, $W$-algebras and collapsing levels}
Let $G$ be a complex connected, simple algebraic group of adjoint type with Lie algebra $\g$, and let $k \in \C$ be a complex number, referred to as the level. 
The universal affine Kac-Moody vertex algebra $V^k(\g)$ is constructed from these data, and its simple quotient is denoted $L_k(\g)$; see Section~\ref{sec:admissible}. 
The Hamiltonian reduction $H^0_{DS, f}(-)$ takes an additional choice of nilpotent element $f \in \g$ as input, the vertex algebra $\W^k(\g, f) = H^0_{DS, f}(V^k(\g))$ depending on $f$ only through its adjoint orbit $\O = G \cdot f$. We denote by $\W_k(\g, f)$ the simple quotient of the universal $W$-algebra $\W^k(\g, f)$. 
See Section \ref{sec:Asymptotics_of_W_algebras} for more details about $W$-algebras.

The theory of vertex algebras is connected to Poisson geometry through the associated variety construction 
\cite{Arakawa15a,AraMor16b,AraMor17}. In particular, the associated varieties of the vertex algebras $L_k(\g)$ and $\W_k(\g, f)$ include extensively studied families of Poisson varieties such as nilpotent orbit closures and nilpotent Slodowy slices. 
The theory of vertex algebras, and in particular the structure of the vertex algebras $L_k(\g)$ and $\W_k(\g, f)$ at special levels $k$ called {\em admissible levels} (see Definition \ref{Def:admissible}), is connected as well 
to the theory of modular functions \cite{KacWak88,Zhu96,AEkeren19admissible,AEkeren19}. 
Both of these themes will be put to use in our investigation of the structure of $W$-algebras in this work.

The universal $W$-algebras are easy to describe as graded vector spaces. Via spectral sequence arguments the algebra  $\W^k(\g, f)$ is seen to have a PBW basis corresponding to a set of strong generators indexed by a basis of the centraliser $\g^f$ \cite{dBT1993,FBZ.book,KacWak03}. The relations between these generators are, however, extremely complicated and in general are not known explicitly. This hampers direct analysis of the simple quotient $\W_k(\g, f)$ via generators and relations. Though, importantly, it is known that $\W^k(\g, f)$ contains an embedded copy of the affine vertex algebra $V^{k^\natural}(\g^\natural)$, where $\g^\natural \subset \g$ denotes the centraliser subalgebra of an $\mathfrak{sl}_2$-triple $(e, h, f)$ containing $f$, and $k^\natural$ is some level determined by the other data. It follows that the simple quotient $\W_k(\g, f)$ contains an embedded homomorphic image of the affine vertex algebra, and the level $k$ is then said to be a \emph{collapsing level} (see Definition \ref{Def:collapsing}) 
if the embedded subalgebra coincides with $\W_k(\g, f)$ itself, that is to say if we have an isomorphism 
of vertex algebras: 
\begin{align}\label{eq:coll.intro.isom} 
\W_k(\g, f) \cong L_{k^\natural}(\g^\natural).  
\end{align}

In \cite{Adamovic-et-al_collapsing}
it was shown 
that 
collapsing levels have remarkable applications
to the representation theory of affine vertex algebras. They are also useful in elucidating the structure of modular tensor categories of representations of simple $W$-algebras at admissible, not necessarily collapsing, levels \cite{AEkeren19}. 
Furthermore it has recently been observed \cite{XieYan2} that many collapsing levels for quasi-lisse $W$-algebras
should come from non-trivial isomorphisms of 4D $N=2$ SUSY field theories, via the 4D/2D duality \cite{BeeLemLie15}.

There is a full classification of collapsing levels for the case that $f$ is a minimal nilpotent element $f_{\text{min}}$, including the case in which $\g$ is a simple Lie superalgebra (\cite{AdaKacMos17,AdaKacMos18}). For more general nilpotent elements $f$, since the commutation relations in $\W^k(\g,f)$ are unknown, almost nothing is known 
about collapsing levels, and to discover them we must appeal to more indirect methods. In this work we exploit two important invariants of vertex algebras: associated varieties and asymptotic data. 
{The general strategy we use to detect (admissible) collapsing levels is given in Section \ref{sec:strategy}.}

\subsection{Associated variety, asymptotic data and main results}
To a vertex algebra $V$ one attaches, in a functorial manner, a certain affine Poisson variety $X_V$ referred to as the {\em associated variety} \cite{Ara12}. The associated variety of $V^k(\g)$ is $\g^*$ and the associated varieties of simple affine vertex algebras at admissible levels are nilpotent orbit closures with the induced Kirillov-Kostant-Souriau Poisson structures \cite{Arakawa15a}. At the level of associated varieties, the Drinfeld-Sokolov reduction $H^0_{DS, f}(-)$ corresponds to intersection with the Slodowy slice $\Slo_f\cong f+\g^{e}$ \cite{De-Kac06} \cite{Arakawa15a}. In general the intersection 
\begin{align}\label{eq:nilp.Slo.intro}
\Slo_{\O,f} = \overline{\O}\cap \Slo_f
\end{align}
of a nilpotent orbit closure $\overline{\O}$ with a Slodowy slice $\Slo_f$, is referred to as a \emph{nilpotent Slodowy slice}. Normalisations of nilpotent Slodowy slices are symplectic singularities in the sense of 
Beauville \cite{Bea00} and, like nilpotent orbit closures, these varieties are studied for their role in representation theory and in the theory of symplectic singularities.

The nilpotent Slodowy slices are best understood in the case of {\em minimal degeneration} in which $G \cdot f$ is an open subvariety of the boundary\footnote{The boundary of $\O$ in $\overline{\O}$  is precisely the  
singular locus of $\overline{\O}$ 
as was shown by Namikawa \cite{Namikawa04} using results 
of Kaledin and Panyushev \cite{Kaledin06,Panyushev}; 
this can also be deduced from \cite{KraftProcesi81,KraftProcesi82,FuJutLev17}.} of $\O$ in $\overline{\O}$. 
In the context of this class of examples, one has the celebrated result of Brieskorn and Slodowy (\cite{Brieskorn,Slo80}) 
 confirming a conjecture of Grothendieck, that the nilpotent Slodowy slice associated with the principal nilpotent orbit $\O = \O_{\text{prin}}$ and a subregular nilpotent element $f_{\text{subreg}}$ has a simple singularity of the same type as $G$, for $G$ of type $A, D, E$.

The second invariant of vertex algebras which we make use of, the {\em asymptotic datum} 
(see Definition \ref{Def:asymp.datum}), 
originates in the phenomenon of modular invariance of characters. Consider the character, i.e., normalised graded dimension
\[
\chi(\tau) = q^{-c/24} \sum \dim(V_n) q^{n},
\]
of an affine vertex algebra or $W$-algebra $V$ of central charge $c$. Under favourable circumstances the character converges to a function of $\tau$ a variable in the complex upper half plane (having set $q = e^{2\pi i \tau}$). Explicit formulas for these characters, ultimately coming from the Weyl-Kac character formula, yield modular transformation rules, 
\cite{Kac90,KacWak88}, and asymptotic behaviour of $\chi(\tau)$ of the form
$$\chi_V(\tau) \sim \A_V e^{\frac{\pi {\imag} }{12 \tau}\G_V}\quad
\text{as }\tau\downarrow 0.$$
Here $\tau\downarrow 0$ means $\tau$ tends to $0$ along the positive imaginary axis. The invariants $\A_V$ and $\G_V$ are called the {\em asymptotic dimension} and {\em asymptotic growth} of $V$. Explicit formulas have been given for principal admissible level $k$ by Kac and Wakimoto \cite{KacWak08}, and for coprincipal admissible levels in the present work. See Sections \ref{sec:admissible} and \ref{sec:Asymptotics_of_W_algebras} and in particular Propositions \ref{prop:general.affine.asymp.coprincipal} and \ref{Pro:asymptotic_data_H_DS}.

Plainly the isomorphism \eqref{eq:coll.intro.isom} entailed by a collapsing level induces an equality of asymptotic data. We now come to our first main theorem, which is a sort of converse 
(see also Theorem~\ref{Th:main} and Proposition \ref{Pro:asymptotics-and-collapsing}).
\begin{Th}
\label{Th:main-intro}
Assume that
$k$ and $k^\natural$ are admissible levels for $\g$ and $\g^\natural$, 
respectively, 
that $f \in X_{L_k(\g)}$ and that  
$\chi_{H_{DS,f}^0(L_k(\g))}(\tau) \sim \chi_{L_{k^\natural}(\g^\natural)}(\tau),$ 
as $\tau\downarrow 0$. 
Then 
$k$ is a collapsing level, that is, $\W_k(\g,f) \cong L_{k^\natural}(\g^\natural)$. 
\end{Th}
Using this theorem we establish many new infinite families of collapsing levels for $\g$ of classical type, and approximately one hundred collapsing levels for $\g$ of exceptional type. 
To give a flavour of our results, we quote 
the following which is part of Theorem \ref{Th:main_sl_n-2}, and 
is representative of the results in classical types.  
See also Theorems \ref{Th:main_sl_n-1}, \ref{Th:main_sp_n-1}, \ref{Th:main_sp_n-2}, \ref{Th:main_so_n-1} 
and~\ref{Th:main_so_n-2}.
\begin{Th}
\label{Th:main_sl_n-2-intro-version} 
Let $\g=\sl_n$ with $n \geq 2$, 
and let $k = -n+p/q$ where $q \geq 1$ and $p \geq n$ is coprime to $q$ 
(that is, $n$ is admissible for $\g=\sl_n$). 
Pick a nilpotent element 
$f \in \overline{\O}_k$ corresponding to the partition $(q^m,1^s)$ of $n$, 
with $m\ge 0$ and $s >0$. 
Then $k$ is collapsing if and only if $p=n=h^{\vee}_{\sl_n}$. 
Moreover, 
$$\W_{-n + n/q}(\mf{sl}_n,f)  \cong L_{-s+s/q}(\mf{sl}_s).$$
\end{Th}
Similarly, we exhibit here, as an illustration, 
one of the isomorphisms which we uncover in exceptional types with the help of Theorem \ref{Th:main-intro}: 
\[
\W_{-12+13/3}(E_6,2A_2)\cong L_{-4+7/3}(G_2). 
\]
The remaining cases in the exceptional types are covered in Theorems \ref{Th:main_E6}, \ref{Th:main_E7}, \ref{Th:main_E8}, \ref{Th:main_F4}, and \ref{Th:main_G2}.
We conjecture (see Conjectures~\ref{Conj:exhaustive} and \ref{Conj:exhaustive_exceptional}) 
that our list of admissible collapsing levels is exhaustive. 
We also obtain a number of cases where the simple $W$-algebra $\W_k(\g,f)$ 
is merely a finite extension of its simple affine vertex algebra 
(see all the above cited theorems). 

\subsection{Nilpotent Slodowy slices}
The nature of the formulas for asymptotic data are such that it is not feasible to find collapsing levels by a naive search for coincidences between respective asymptotic data. For this reason we turn to the more refined invariant given by the associated variety. Although the problem of determining the associated variety of $L_k(\g)$ is wide open in general, many of the Poisson varieties arising as associated varieties of simple $W$-algebras are nilpotent Slodowy slices. More precisely, whenever $k$ is an admissible level for $\g$, the associated variety $X_{H^0_{DS, f}(L_k(\g))}$ is the nilpotent Slodowy slice \eqref{eq:nilp.Slo.intro} above, with $\O = \O_k \subset \g$ a certain nilpotent orbit determined by $k$ \cite{Arakawa15a} \cite{AEM}. It is conjectured in general (and confirmed in many cases) that $H^0_{DS, f}(L_k(\g))$ is simple, so that $\W_k(\g, f) = H^0_{DS, f}(L_k(\g))$ in fact (see Conjecture~\ref{Conj:isom}). 
A collapsing level thus induces  \cite{Arakawa15a}, in these cases, an isomorphism
\begin{align}\label{eq:nilp.Slo.iso.orbit}
\overline{\O}_k \cap \Slo_f \cong \overline{\O}_{k^\natural}.
\end{align}
In particular, the singularity of the nilpotent Slodowy slice in $\g$ on the left-hand-side should be of the same type as that
of the nilpotent orbit closure in $\g^\natural$ on the right-hand-side. We may therefore apply known results on the geometry of nilpotent Slodowy slices to find candidates for collapsing levels.

Kraft and Procesi studied nilpotent Slodowy slices for minimal degenerations in the classical types \cite{KraftProcesi81,KraftProcesi82}, motivated by the normality problem for nilpotent orbit closures. They introduced the row/column removal process, reviewed in Section \ref{sec:strategy} below, and which we now briefly describe. Nilpotent orbits in simple Lie algebras of classical type are parametrised by certain classes of integer partitions and, roughly speaking, the row/column removal rule is a set of combinatorial operations on a pair of partitions $(\bs{\lambda}, \bs{\mu})$ under which the nilpotent Slodowy slice $\Slo_{\O_{\bs{\lambda}}, f_{\bs{\mu}}}$ remains unchanged up to isomorphism. Using the row/column removal rule, it is possible to identify classes of nilpotent orbits for which an isomorphism of the type \eqref{eq:nilp.Slo.iso.orbit} holds. Recently, Fu, Juteau, Levy and Sommers \cite{FuJutLev17} have complemented the work of Kraft and Procesi by determining the generic singularities of nilpotent orbit closures $\overline{\O}$ in exceptional types, which they did through a study of the nilpotent Slodowy slices $\Slo_{\O,f}$ at minimal degenerations~$G \cdot f$. We make use of these results also in the sequel.

Upgrading the isomorphisms of nilpotent Slodowy slices to collapsing levels  
is then a matter of computing the level $k^\natural$ in terms of $k$, and comparing asymptotic data. 
This is straightforward in principle, but in practice 
it is extremely complicated, and exact (as opposed to numerical) computations require extensive use of classical cyclotomic product identities such as those presented in Section \ref{sec:identities}.

We observe that many of our examples of 
collapsing levels 
are 
of the form
$$-h_\g^\vee+\frac{h_\g^\vee}{q},$$ 
with $(h_\g^\vee,q)=(r^{\vee},q)=1$,
or of the form 
$$-h_\g^\vee+\frac{h_\g +1}{q},$$ with $(h_\g+1 ,q)=1$, $(r^{\vee},q)=r^\vee$, 
where  $h_\g$ is the Coxeter number,
$h_\g^\vee$ is the dual Coxeter number,
$r^{\vee}$ is the lacing number of $\g$,
respectively.
A level of the first form
is called a {\em boundary principal admissible level} \cite{KacRoaWak03} (see also \cite{KacWak17}). 
The vertex algebras 
$L_{k}(\g)$ and $\W_{k}(\g,f)$ 
at boundary principal admissible level $k$
appear as vertex algebras associated with Argyres-Douglas theories 
via the 4D/2D duality 
(\cite{SonXieYan17,Dan,WanXie}). 
Collapsing levels which are boundary principal admissible 
have been studied by Xie and Yan \cite{XieYan2} in this connection, 
and our results confirm their conjectures \cite[Section 3.4]{XieYan2}. 

\subsection{Related problems}
Aside from determination of collapsing levels, the methods we develop in this article can be used to prove other results of a similar flavour, some of which are remarked upon in the body of the text. 
For instance (see Remark \ref{rem:min.mod.cases.sp} for details) if $\g = \mf{sp}_n$ and $k = -h_\g^\vee + p/q$ where $q$ is twice an odd integer and $p=h+1$, then for $f \in \O_k = \O_{\bs\lam}$ where $\bs\lam =(\frac{q}{2}+1, (\frac{q}{2})^{m}, 2)$,
\begin{align*}
\W_k(\g,f)\cong H_{DS, f}^0(L_{k}(\g)) \cong \on{Vir}_{2, q/2}.
\end{align*}
In addition to such results, we pose a number of conjectures, mostly related to presentation of $W$-algebras as finite extensions of simple affine vertex algebras. By {\em finite extension} we mean here, for definiteness, a vertex algebra $W$ which decomposes as a finite direct sum of irreducible modules over its conformal vertex subalgebra $V$. In particular we make the following:
\begin{Conj}\label{Conj:birational}
If $W$ is a finite extension of the vertex algebra $V$ then the corresponding morphism of Poisson algebraic varieties $\pi\colon X_W \to X_V$, is a dominant morphism.
\end{Conj}
The validity of Conjecture \ref{Conj:birational} would imply the widely-believed fact that a finite extension of a lisse vertex algebra is also lisse. We plan to return to these matters in future work.

\subsection*{Plan of the article} 
The rest of the article is organised as follows. 
In Section \ref{Section:Asymptotic data} we collect results on asymptotic data for vertex algebras and 
their modules. 
We exhibit a large class of vertex algebras admitting an asymptotic datum 
(see Proposition \ref{pro:quasi_liss_asymptotic}). 
Section \ref {sec:admissible} is about affine vertex algebras. 
The main result of the section describes the asymptotic data of simple affine vertex algebras 
at admissible levels, and of their simple ordinary representations  
(see Corollary~\ref{Co:asymptotic_data_L}). 
Section \ref{sec:Asymptotics_of_W_algebras} gathers together 
several properties of $W$-algebras. 
One of the main results of this section is 
Proposition~\ref{Pro:asymptotic_data_H_DS} 
which gives the asymptotic datum 
of the Drinfeld-Sokolov reduction 
$H_{DS,f}^0(L(\lam))$ 
for all simple ordinary $L_k(\g)$-module $L(\lam)$ 
for admissible~$k$. 
The notion of collapsing level is introduced in Section~\ref{sec:collapsing}. 
In Section \ref{sec:strategy} we then explain in detail our strategy to find collapsing admissible levels. 
The identities obtained in Section \ref{sec:identities} are useful to compute the asymptotic 
data in the sections which follow. 
Our results for the classical types are presented in Sections~\ref{sec:sl_n} and~\ref{sec:classical}. 
See Theorems \ref{Th:main_sl_n-1} and \ref{Th:main_sl_n-2} 
for $\g=\sl_n$, Theorems~\ref{Th:main_sp_n-1} and~\ref{Th:main_sp_n-2} for 
$\g=\sp_n$, Theorems~\ref{Th:main_so_n-1} and~\ref{Th:main_so_n-2} for $\g=\so_n$.
Our results for the exceptional types are presented in Section~\ref{sec:exceptional}. 
See Theorem~\ref{Th:main_E6} for $\g=E_6$, 
Theorem~\ref{Th:main_E7} for $\g=E_7$, 
Theorem~\ref{Th:main_E8} for $\g=E_8$, Theorem~\ref{Th:main_F4}
 for $\g=F_4$, and Theorem~\ref{Th:main_G2}  for $\g=G_2$. 
Finally, useful data related to nilpotent orbits in the exceptional types are collected 
in Appendix \ref{App:Centralisers_exceptional}.

\subsection*{Acknowledgements} 
The authors wish to express their gratitude to Daniel Juteau for 
the numerous fruitful discussions on nilpotent Slodowy slices. 
They are also grateful to Wenbin Yan for several helpful comments concerning collapsing levels in 4D superconformal theory, 
and they warmly thank 
Thomas Creutzig and Kazuya Kawasetsu 
whose question about the isomorphism $\W_{-14/3}(\sl_7) 
\cong L_{-8/3}(\sl_4)$ and corresponding associated varieties 
was the starting point of this work. 
TA and JvE wish to thank the University of Lille 
and the Laboratoire Painlev\'{e} for the organisation and the financial support of  
the conference {\em Geometric and automorphic aspects of W-algebras} 
in 2019 where the collaboration between the three authors started. 
AM gratefully acknowledges the RIMS of Kyoto, 
where part of this paper was
written, for financial support and its hospitality 
during her stay 
in autumn 2018. 

TA is partially supported by JSPS KAKENHI Grant Number J21H04993.
JvE is supported by the Serrapilheira Institute
(grant number Serra -- 1912-31433), by CNPq grants 303806/2017-6 and 402449/2021-5, and by FAPERJ grants E-26/010.002607/2019 and E-26/201.445/2021. 
AM is partially supported by ANR Project GeoLie Grant number ANR-15-CE40-0012.

\tableofcontents 

\section{Asymptotic data of vertex algebras}
\label{Section:Asymptotic data}
A {\em vertex algebra} is a complex vector space $V$ equipped with a distinguished 
vector $\vac \in V$, an endomorphism $T \in {\rm End}\,V$ and a linear map $$V \to ({\rm End}\,V)[[z,z^{-1}]], \quad a \mapsto a(z)=\sum_{n\in \Z} a_{(n)} z^{-n-1}$$ 
satisfying the following axioms: 
\begin{itemize}
	\item $a(z)b \in V((z))$ for all $a,b \in V$, 
	\item (vacuum axiom) $\vac (z) = {\rm Id}_V$ and $a(z) \vac \in a+ zV[[z]]$ for all $a \in V$, 
	\item (translation invariance axiom) $[T, a(z)] = \tfrac{\partial}{\partial z} a(z)$, 
	\item (locality axiom) $(z-w)^{N_{a,b}} [a(z),b(w)]=0$ for a sufficiently large integer $N_{a,b}$, for all $a,b \in V$. 
\end{itemize}

A vertex algebra $V$ is called {\em conformal} if 
there exists a vector $\omega$ called the
{\em conformal vector} such that
 $L(z) = \omega(z) =\sum_{n\in \Z}L_nz^{-n-2}$ satisfies
 \begin{itemize}
\item[(a)] $[L_m,L_n]=(m-n)L_{m+n}+\dfrac{m^3-m}{12}\delta_{m+n,0}c_V$,
 where $c_V$ is a constant called the {\em central charge} of $V$,
\item[(b)] $L_0$ acts semisimply on $V$, 
\item[(c)] $L_{-1}=T$ on $V$.
\end{itemize}
  For a conformal vertex algebra $V$
 and a $V$-module $M$, we set 
 $M_d=\{m\in M\colon L_0 m=dm\}$. The $L_0$-eigenvalue of a nonzero $L_0$-eigenvector $m \in M$ is called its
 {\em conformal weight}. For an element $a \in V$ of conformal weight $\Delta$ we write $a(z) = \sum_{n \in \Z} a_n z^{-n-\Delta}$. Note that in general $a_0 : M_d \rightarrow M_d$.

 A finitely generated $V$-module $M$ is called {\em ordinary} if 
 $L_0$ acts semisimply,
 $\dim M_d<\infty$ for all $d$,
 and the conformal weights of $M$ are bounded from below.
  The minimum conformal weight of
 a simple ordinary $V$-module $M$ is called the {\em conformal dimension} of $M$. More generally, a $V$-module $M$ will be said to be of {\em positive energy} (also called admissible) if it possesses a $\Z_+$-grading $M = \bigoplus_{k \in \Z_+} M_k$ such that $a_n M_k \subset M_{k-n}$ for all $a \in V$. The {\em normalised character} of 
 an ordinary representation $M$ 
 is defined by
\begin{align*}
\chi_M(\tau)={\rm tr}_M q^{L_0-c_V/24}=q^{-c_V/24}\sum_{d\in \C}(\dim M_d)q^d,\quad q=e^{2\pi 
{\imag} \tau} {\; \text{ with }\; \tau \in \C}.
\end{align*}
A conformal vertex  algebra is called {\em conical} if 
$V=\bigoplus_{\Delta\in \frac{1}{r}\Z_{\ge 0}}V_{\Delta}$
for some $r\in \Z_{\ge 0}$ and $V_0=\C$.
A $\Z$-graded conical  vertex  algebra is said to be of  {\em CFT-type}. 
Let $V$ be a vertex algebra of CFT-type. 
Then $V$ is called {\em self-dual} if
$V \cong V'$ as $V$-modules, where $M'$ denotes the contragredient 
dual \cite{Frenkel:1993aa} of the $V$-module $M$. 
Equivalently $V$ 
is self-dual if and only if it admits a non-degenerate symmetric invariant bilinear form.

The following  definition goes back to \cite[Conjecture 1]{KacWak88}\footnote{In \cite{KacWak88} 
the triple $(\A_V,\Weight_V,\G_V)$ was called the asymptotic dimension.}.
\begin{Def}\label{Def:asymp.datum}
A conformal vertex algebra $V$ is said to admit 
{\em \asym}\ if there exist 
{$\A_V\in \mathbb{R}$,
$\Weight_V \in \mathbb{R}$, $\G_V\in \mathbb{R}$} 
such that
$$\chi_V(\tau) \sim \A_V (-{\imag}  \tau)^{{\frac{\Weight_V}{2}}}e^{\frac{\pi {\imag} }{12 \tau}\G_V}\quad
\text{as }\tau\downarrow 0.$$
The numbers
$\A_V$, $\Weight_V$ and $\G_V$ are called the {\em asymptotic dimension} of $V$,
the {\em asymptotic weight},
and the {\em asymptotic growth}, respectively.
Similarly, an ordinary $V$-module $M$ is
said to admit 
\asym\  
if there exist 
$\A_M\in \C$,
$\Weight_M,\G_M\in \mathbb{R}$
such that
$$\chi_M(\tau) \sim \A_M (-{\imag}  \tau)^{\frac{\Weight_M}{2}}e^{\frac{\pi {\imag} }{12 \tau}\G_M}\quad
\text{as }\tau\downarrow 0.$$
\end{Def}

For a conformal vertex algebra $V$ and an ordinary $V$-module $M$,
\begin{align*}
\on{qdim}_V M:=\lim_{\tau\downarrow 0}\frac{\chi_M( \tau)}{\chi_V( \tau)}
\end{align*}
is called the {\em quantum dimension} of $M$ if it exists (\cite{DonJiaXu13}).
If 
both $V$ and $M$ admit asymptotic data,
$\G_V=\G_M$ and $\Weight_V=\Weight_M$,
then the quantum dimension of $M$ exists and is equal to the ratio of the asymptotic dimensions:
\begin{align}
\on{qdim}_V M=\frac{\A_M}{\A_V}.
\end{align}

\smallskip
Given a vertex algebra $V$ one naturally defines a Poisson algebra $R_V$, 
called the {\em Zhu $C_2$-algebra}, as follows (\cite{Zhu96}). 
Let $C_2(V)$ be the subspace of $V$ spanned by the 
elements $a_{(-2)}b$, where $a,b \in V$, and set $R_V =V/C_2(V)$. 
Then $R_V$ acquires a Poisson algebra structure via
$$1= \overline{\vac}, \qquad \bar{a} \cdot \bar{b} = \overline{a_{(-1)}b} \quad 
\text{ and }\quad \{\bar{a},\bar{b}\} =\overline{a_{(0)}b},$$ 
where $\bar{a}$ denotes the image of $a\in V$ in the quotient $R_V$.

The {\em associated variety} \cite{Ara12} $X_V$ of a vertex algebra $V$ is the affine Poisson 
variety defined by 
$$X_V = {\rm Specm}\, R_V.$$ 
A vertex algebra $V$ called 
{\em lisse} if  $\dim X_V=0$.
It is called 
{\em quasi-lisse} \cite{AraKaw18} if $X_V$ has finitely many symplectic leaves.

A vertex algebra $V$ is called {\em rational}
if any finitely generated positively graded $V$-module is
completely reducible. 
For a lisse conformal vertex algebra $V$,
any finitely generated $V$-module is ordinary,
and there exist finitely many simple $V$-modules (\cite{AbeBuhDon04}).

The following fact is well-known.
\begin{Pro}
\label{Pro:asym-datum-rational}
Let $V$ be a  finitely strongly generated, rational, lisse
self-dual
simple
vertex operator algebra of CFT-type.
Then any simple $V$-module $M$
 admits \asym\  
with $\Weight_M=0$.

\end{Pro}
\begin{proof}
We include a proof for completeness. 
By \cite{AbeBuhDon04},
any simple $V$-module is ordinary,
and there exist finitely many simple $V$-modules,
say, $\{L_i\colon i=0,\ldots,r\}$.
Let $h_i$ be 
the conformal dimension of $L_i$.
Then
$$\chi_{L_i}(\tau)=q^{h_i-c/24}\sum_{d\ge 0}(\dim (L_i)_{h_i+d})q^d.$$
By Zhu's theorem \cite{Zhu96},
the vector space spanned by $\chi_{L_i}(\tau)$,
$i=0,\dots, r$,
is invariant under the natural action of the modular group $\on{SL}_2(\Z)$.
Hence, 
\begin{align*}
\chi_{L_i}(\tau)=\sum_{j=1}^r S_{i,j}\chi_{L_j}(-1/\tau)
\end{align*}
for some $S_{i,j}\in \C$, $j=0,\dots, r$.
The assertion follows.
\end{proof}

The following assertion is widely believed.
\begin{Conj}\label{Conj:uniqueness-of-minimal-conformal-dim}
Let $V$ be
 a rational, lisse, simple, self-dual conformal vertex algebra $V$,
 $\{L_0,\dots, L_r\}$ the complete set of simple $V$-modules.
There exists a unique  simple module $L_{i_{\ooh}}$  with conformal dimension $h_{\text{min}}:=h_{i_{\ooh}}$
such that $h_i> h_{\text{min}}$ for all   $i\ne i_{\ooh}$,
and that 
$S_{i i_{\ooh}}\ne 0$ for all $i$.
\end{Conj}
In Theorem \ref{Th:min-conf-dim},
we confirm the uniqueness
of the simple module with minimal conformal dimension
for exceptional $W$-algebras (\cite{KacWak03,AraFutRam17}). See also Conjecture \ref{Conj:rationality} and the subsequent remarks.

\begin{Pro}\label{Pro:effective.c.c=growth}
Let $V$ be as in Conjecture \ref{Conj:uniqueness-of-minimal-conformal-dim}, with simple modules $\{L_0,\dots, L_r\}$ ordered so that $L_0 = V$, and assume  
that there exists a unique  simple module $L_{i_{\ooh}}$  
satisfying the assertion of 
  Conjecture \ref{Conj:uniqueness-of-minimal-conformal-dim}.
  Then
\begin{align*}
\G_{L_i}=c_V-24 h_{\text{min}},\quad \A_{L_i}=\dim(L_{i_\ooh})_{h_{\text{min}}} S_{i,i_{\ooh}}
\end{align*}
for all $i$.
Moreover, 
\begin{align*}
\on{qdim}_V(L_i)=\frac{S_{i, i_{\ooh}}}{S_{0, i_{\ooh}}},
\end{align*}
and
\begin{align}
\on{qdim}_V(L_i\boxtimes L_j)=\on{qdim}_V(L_i)\on{qdim}_V(L_j),
\label{eq:tensor-qdim}
\end{align}
where $\boxtimes$ is the fusion product
(\cite{HuaLep92,HuaLep94,HuaLep95,HuaLep95III}).
In particular,
 the quantum dimension is well-defined for all simple $V$-modules.
\end{Pro}
\begin{proof}
The  assertions except for the last follows from the proof of Proposition \ref{Pro:asym-datum-rational}.
For the  assertion \eqref{eq:tensor-qdim}, see \cite[Remark 4.10]{DonJiaXu13}.
\end{proof}
The number $c_V-24 h_{\text{min}}$
 is called the {\em effective central charge} of $V$  in the literature (\cite{DonMas04}).

\begin{Rem}
Let $V$ be as in Proposition \ref{Pro:effective.c.c=growth}.
By a result of Huang \cite{Hua08rigidity},
the category $V\on{-mod}$ of finitely generated $V$-modules forms a modular tensor category.
In this context,  the quantum dimension of a simple $V$-module is the same as the 
 {\em Frobenius-Perron dimension} (\cite{EtiGelNik15}) of $V$ in $V\on{-mod}$.
\end{Rem}

\begin{Pro}
\label{pro:quasi_liss_asymptotic}
Let $V$ be a  finitely strongly generated, quasi-lisse
vertex operator algebra of CFT-type.
Then any simple ordinary $V$-module $L$ admits \asym.
\end{Pro}
\begin{proof}
By \cite{AraKaw18},
the set $\{L_i\}$ of simple ordinary $V$-modules is finite, and the characters
$\chi_{L_i}(\tau)$ are solutions of a modular linear differential equation.
Since the space spanned by the solutions of a modular linear differential equation
is invariant under the natural action of $\on{SL}_2(\Z)$,
the assertion follows in a similar manner as Proposition \ref{Pro:asym-datum-rational},
except that a solution of a modular linear differential equation
may have  logarithmic terms, 
that is, it has the form
\begin{align*}
q^{\beta_i}\sum_{i=1}^e f_i(q)(\log q)^{e-i},\quad f_i(q)\in \C[[q]].
\end{align*}
\end{proof}

Let $\on{Vir}^c$ be the 
universal Virasoro vertex algebra at central charge $c\in \C$,
$\on{Vir}_c$  the unique simple quotient of $\on{Vir}^c$.
\begin{Lem}[{\cite{FeuFuc84, KacWak88}}]\label{Lem:assymptotics-of-Viraosro}
A quotient of a universal Virasoro vertex algebra
admits  an asymptotic datum.
\end{Lem}

\begin{proof}
It is well-known that $\on{Vir}^c$ has length two if $c=1-6(p-q)^2/pq$ for some $p,q\in \Z_{\geq 2}$, $(p,q)=1$, and otherwise $\on{Vir}^c=\on{Vir}_c$. Hence,
a quotient $V$ of $\on{Vir}^c$ is either  $\on{Vir}^c$ or $\on{Vir}_c$.
If $V=\on{Vir}^c$,
then 
\begin{align*}
\chi_V(\tau)=\frac{(1-q)q^{(1-c_V)/24}}{\eta(q)},
\end{align*}
where $\eta(q)=q^{1/24}\prod_{j\geq 1}(1-q^j)$.
Hence (indicating by $+ \cdots$ terms of lower growth)
\begin{align*}
\chi_V(e^{2\pi \imag \tau})
&= (1 - e^{2\pi \imag \tau}) e^{2\pi \imag \tau (1-c_V)/24} (-\imag\tau)^{\frac{1}{2}} \left( e^{2\pi \imag \left(-\frac{1}{\tau}\right)\left(-\frac{1}{24}\right)} + \cdots \right), \\
&\sim (-2\pi \imag \tau) (-\imag\tau)^{\frac{1}{2}} e^{\frac{\pi \imag}{12\tau}},
\end{align*}
where we have used l'Hopital's rule. So $V$ admits an asymptotic datum with $\A_V=2\pi$, $\Weight_V=3$, $\G_V=1$.

If $V=\on{Vir}_c$ with  $c=1-6(p-q)^2/pq$, $p,q\in \Z_{\geq 2}$, $(p,q)=1$,
then as it is well-known \cite{FeuFuc84, KacWak88},  
$V$ admits \asym\ with $\Weight_V=0$,
\begin{align}
\A_V= \left( \frac{8}{pq} \right)^{1/2} \sin\left(\frac{\pi a(p-q)}{q}\right) \sin\left(\frac{\pi b (p-q)}{p} \right), 
\label{eq:minmod.as.dim}
\end{align}
where $(a, b)$ is the unique solution of $pa-qb=1$ in integers $1 \leq a \leq q$ and $1 \leq b \leq p$, 
and
\begin{align}
\G_V=1-\frac{6}{pq}.
\label{eq:growth-of-Virasoro-minimal}
\end{align}
\end{proof}

The simple Virasoro vertex algebra 
$\on{Vir}_c$ with  $c=1-6(p-q)^2/pq$, $p,q\in \Z_{\geq 2}$, $(p,q)=1$,
is known to be rational and lisse (\cite{Wan93}).
The simple $\on{Vir}_c$-modules are
the $(p,q)$-minimal series representations of the Virasoro algebra,
and for each simple $\on{Vir}_c$-module $L$ we have
$\Weight_L=0$,
$\G_L=\G_{\on{Vir}_c}$ and $\A_L>0$ (\cite{FeuFuc84, KacWak88}).

\begin{Lem}
Let $V$ be a conformal vertex algebra
with central charge $c=1-6(p-q)^2/pq$, $p,q\in \Z_{\geq 2}$, $(p,q)=1$,
and suppose that $V$ admits an asymptotic datum with $\G_V<1$.
Then $V$ is a direct sum of simple $(p,q)$-minimal series representations of the Virasoro algebra.
If further $\A_V=\A_{\on{Vir}_c}$, 
then
$V\cong \on{Vir}_c$.
\end{Lem}
\begin{proof}
The vertex algebra homomorphism
$\on{Vir}^c\ra V$, $\omega_{\on{Vir}^c}\ra \omega_V$,
factors through the embedding $\on{Vir}_c\hookrightarrow V$
because otherwise 
$\G_V\geq \G_{\on{Vir}^c}=1$.
Thus the rationality of $\on{Vir}_c$ proves the  first statement,
and so we have $V=\bigoplus_{i}L_i^{\oplus_{m_i}}$,
where $\{L_i\}$ is the set of  simple $(p,q)$-minimal series representations of the Virasoro algebra
and $m_i\in \Z_{\geq 0}\cup \{\infty\}$.
It follows that $\A_V=\sum_i m_i \A_{L_i}$,
and we get the second assertion.
\end{proof}
Recall that a homomorphism
 $f\colon V\ra W$ 
of conformal  vertex algebras is called {\em conformal}
if $\omega_W=f(\omega_V)$.
\begin{Pro}\label{Pro:conformal}
Let $f:V\ra W$ be a homomorphism of 
 conformal vertex algebras.
Suppose that 
\begin{itemize}
\item $f(\omega_V)\in W_2$ and $(\omega_W)_{(2)}f(\omega_V)=0$,
\item the simple quotient $L$ of $V$ admits an asymptotic datum,
\item $W$ is a quotient of a conformal vertex algebra $\tilde{W}$ that admits an asymptotic datum.
\end{itemize}
 If $\G_L=\G_{\tilde{W}}$, then $f$  is conformal.
\end{Pro}

\begin{proof}
 Our aim is to show that $\omega=\omega_W-f(\omega_V)$ vanishes. Now, if $\omega\ne 0$ then \cite[Theorem 3.11.12]{LepLi04} asserts that the commutant subalgebra
\[
\on{Com}(W, f(V)) = \{w \in W \colon \text{$f(v)_{(j)}w$ for all $v \in V$ and $j \in \Z_+$} \}
\]
is conformal with Virasoro vector $\omega$ of some central charge $c$. We note that $\omega \notin f(V)$, for otherwise we would have $\omega_{(0)}\omega = 0$, which contradicts that $\omega$ be a conformal vector. We thus proceed to show that $\omega \in f(V)$ and deduce that $\omega = 0$ by contradiction.

Consider the homomorphism $\tilde{f} :\on{Vir}^c \otimes V \ra W$ of conformal vertex algebras 
that sends the conformal vector $\omega^{(0)}$ of $\on{Vir}^c $ to $\omega$, and such that $\tilde{f}(\vac \otimes v) = f(v)$ for all $v \in V$. We now have the vertex subalgebra $U=\tilde{f}(\on{Vir}^c  \otimes  V)\subset W$ whose simple quotient is isomorphic to
$\on{Vir}_c\* L$. We thus have
$\G_{\tilde{W}}\ge \G_L+\G_{\on{Vir}_c}$ and hence $\G_{\on{Vir}_c}=0$ by our hypotheses. From Lemma \ref{Lem:assymptotics-of-Viraosro} it follows that $c=0$.

We recall that $\on{Vir}_0 = \C$, that the maximal proper ideal of $\on{Vir}^{0}$ is generated by $\omega^{(0)}$ and that this ideal is simple as a $\on{Vir}^0$-module.

Let $N$ be the maximal ideal of $V$, so that $L = V / N$. We now consider the surjection
\[
\tilde{f}_1 : \on{Vir}^0 \otimes (V / N) \rightarrow U / \tilde{f}(\on{Vir}^c \otimes N)
\]
and we define
\[
K=\{v\in L \colon \tilde{f}_1( \omega^{(0)} \* v) \in \tilde{f}(\on{Vir}^c \otimes N)\}.
\]
Since $K \subset L$ is a $V$-submodule, we have either $K=L$ or $K=0$.

If $K = 0$ then $\tilde{f}_1$ is an isomorphism, because of simplicity of the maximal proper ideal in $\on{Vir}^0$. But then we would have $\G_{\tilde{W}} \ge \G_L+\G_{\on{Vir}^0} > \G_L$ since $\G_{\on{Vir}^0}>0$, and this contradicts our hypotheses. 

Therefore $K = L$. Since $(\on{Vir}^0 \otimes V)_2 = \C \omega^{(0)} \otimes \vac \oplus \vac \otimes V_2$, it follows that
\[
\tilde{f}(\omega^{(0)} \otimes \vac) \subset \tilde{f}(\vac \otimes N) = f(N) \subset f(V)
\]
and in particular that $\omega \in f(V)$. Thus $\omega = 0$ and we are done.
\end{proof}

\section{Admissible affine vertex algebras} 
\label{sec:admissible}
Let $\g$ be a complex simple Lie algebra of rank $\ell$.
Let $\g=\mf{n}_- \+ \h\+ \mf{n}_+$ be a triangular decomposition with 
 a Cartan subalgebra $\mf{h}$, 
$\Delta$ the root system of $(\g,\mf{h})$ 
and $\Delta_{+}$ a set of positive roots for $\Delta$,
$\Pi=\{\alpha_1,\dots,\alpha_{\ell}\}$ the set 
of simple roots.
Let $\theta$ be the highest root,
 $\theta_s$ the highest short root.
We also have $\Delta^\vee$ the set of coroots. Let $P$ be the weight lattice, $Q$ the root lattice 
and $Q^\vee$ the coroot lattice. 
The lattice $P$ is dual to $Q^\vee$ and we write $P^\vee$ for the dual of $Q$. 
Recall that the Coxeter number and the dual Coxeter number of $\g$ are 
denoted by $h_\g$ and $h_\g^\vee$, respectively.  
Identifying $\h$ with $\h^*$ using the inner product  
$$(~|~)_\g=\displaystyle{\frac{1}{2 h_\g^\vee}}\times\text{Killing form of }\g,$$ 
we view $Q^\vee$ as a sub-lattice 
of both $P$ and $Q$. 
We denote by $\rho$ the Weyl vector, i.e., the half-sum of positive roots.

For $\lam\in \h^{*}$,
let $L_{\g}(\lam)$
be the irreducible highest weight representation
of $\g$ with highest weight $\lam$,
and
let 
\begin{align}
J_{\lam}=\on{Ann}_{U(\g)}L_{\g}(\lam).
\end{align}

Let $\widetilde{\g} = \g[t,t^{-1}] \oplus \C K \+ \C D$  be the affine Kac-Moody algebra, 
with the commutation relations:  
$$[x t^m,y t^n] = [x,y] t^{m+n} + m \delta_{m+n,0} (x|y)_\g K, \quad [D,xt^n]=nxt^n,\quad
[K,\affg]=0,$$
for all $x,y \in\g$ and all $m,n\in\Z$
Let $\widetilde{\g}=\widehat{\mf{n}}_-\+ \widetilde{\h}\+\widehat{\mf{n}}_+$
be the 
standard triangular decomposition,
that is, 
$\widetilde{\h} = \h \oplus \C K\+ \C D $ is  the  Cartan subalgebra of $\widetilde{\g}$,
$\widehat{\mf{n}}_+=\mf{n}_+ +t\g[t]$,
$\widehat{\mf{n}}_-=\mf{n}_-+t^{-1}\g[t^{-1}]$.

Let $\affg =[\widetilde{\g},\widetilde{\g}]= \g[t,t^{-1}] \oplus \C K$,
and let $\affh=\h\+ \C K\subset \affg$,
so that $\affg=\widehat{\mf{n}}_-\+ \affh\+ \widehat{\mf{n}}_+$. 
The Cartan subalgebra
 $\widetilde{\h}$ is equipped  
with a bilinear form extending that on $\h$ by
\begin{align*}
(K | D) = 1, \quad (\h | \C K \oplus \C D) = (K | K) = (D | D) = 0.
\end{align*}
We write $\delta$ and $\Lambda_0$ for the elements of $\widetilde{\h}^*$ orthogonal to $\h^*$ and dual to $K$ and $D$, respectively. We have the (real) root system
\begin{align*}
 \widehat{\Delta}^{\text{re}} &= \{\alpha+n\delta  \colon n \in \Z, \alpha \in \Delta\}= \widehat{\Delta}^{\text{re}}_+\sqcup  (-\widehat{\Delta}^{\text{re}}_+),\\
 &\widehat{\Delta}^{\text{re}}_+=\{\alpha+n\delta\colon \alpha\in \Delta_+, \ n\ge 0\}\sqcup \{\alpha+n\delta\colon \alpha\in \Delta, \ n> 0\},
\end{align*}
and the affine Weyl group $\widehat{W}$, generated by reflections
$r_\alpha$ for $\alpha\in  \widehat{\Delta}^{\text{re}}$. 
For $\alpha \in \h^*$ the translation $t_\alpha : \widetilde{\h}^* \rightarrow \widetilde{\h}^*$ is defined by
\begin{align*}
t_\alpha(\lam) = \lam + \lam(K)\alpha - \left[(\alpha | \lam) + \frac{|\alpha|^2}{2} \lam(K) \right] \delta.
\end{align*}
For $\alpha \in Q^\vee$ we have $t_\alpha \in \widehat{W}$ and  in fact 
$\widehat{W} \cong W \ltimes t_{Q^\vee}$. 
The extended affine Weyl group, which is the group of isometries of $\widehat{\Delta}$, is $\widetilde{W} = W \ltimes t_{P}$. 
Here, for $R$ a subset of $P$, $t_R$ stands 
for the set $\{t_\alpha \colon \alpha \in R\}$. 

Let $\tilde{\mc{O}}_k$ be the category $\mc{O}$ of $\widetilde{\g}$ at level $k$ (\cite{Kac90}),
and let $\tilde{\KL}_k$ be the full subcategory of $\tilde{\mc{O}}_k$ consisting of objects on which the action of $\g$ is integrable.
The simple objects of $\tilde{\mc{O}}_k$ are the irreducible highest weight representations $L(\lam)$
with $\lam\in \widetilde{\h}^*$ such that $\lam(K)=k$,
while the  simple objects of $\tilde{\on{KL}}_k$ are those $L(\lam)$ with $\lam\in P_+ + k\Lam_0+\C\delta$,
where $P_+$ is the set of dominant integral weights of $\g$.

For a weight $\lam \in \widehat\h^*$ the corresponding {\em integral root system} is
\begin{align*}
\widehat{\Delta}(\lam) = \{\alpha\in \widehat{\Delta}^{\text{re}} \colon \left<\lam, \alpha^\vee\right> \in \Z\},
\end{align*}
where $\alpha^{\vee}=2\alpha/(\alpha|\alpha)$ as usual,
and the subgroup
$\widehat{W}(\lam)$ of $\widehat{W}$ generated by $r_\alpha$ with $\alpha\in \widehat{\Delta}(\lam)$
is called the
{\em integral Weyl group} of $\lam$.

\begin{Def}
\label{Def:admissible}
 A weight $\lam \in \widetilde\h^*$ 
is said to be \emph{admissible} if 
\begin{enumerate}
\item $\lam$ is regular dominant, that is,
$\bra \lam+\hat{\rho},\alpha^{\vee}\ket >0$ for all $\alpha\in \widehat{\Delta}_+(\lam):=\widehat{\Delta}(\lam)\cap 
\widehat{\Delta}^{\text{re}}_+$,
\item $\mathbb{Q}\widehat{\Delta}^{\text{re}} =\mathbb{Q}\widehat{\Delta}(\lam)$.
\end{enumerate}
Here $\hat{\rho}=\rho+h_\g^{\vee}\Lam_0$ is the affine Weyl vector.
An {\em admissible} $\widetilde{\g}$-module is one of the form $L(\lam)$ for $\lam$ 
admissible. 
\end{Def}

Given any $k\in\C$, 
let
\begin{align*} 
V^k(\fing) =
 U(\affg)\otimes_{U(\mf{g}[t]\oplus \C K)}\C_k, 
\end{align*}
where $\C_k$ is  the one-dimensional representation 
 of 
$\g[t] \oplus \C K$ on which $\g[t]$ 
acts by 0 and $K$ acts as a multiplication by the scalar $k$.
There is a unique vertex algebra structure 
on $V^k(\g)$ such that $\vac$ is the image of $1\otimes 1$ 
in $V^k(\g)$ and 
$$x(z) :=(x_{(-1)}\vac)(z)= \sum_{n\in\Z} (xt^n) z^{-n-1}$$
for all $x\in \g$, where we regard $\g$ as a subspace of $V$ 
through the embedding $x \in \g \hookrightarrow x_{(-1)}\vac \in V^k(\g)$. 
The vertex algebra $V^k(\g)$ is called the {\em universal affine vertex algebra} associated with $\g$ at level $k$.

The vertex algebra $V^k(\g)$ has a conformal structure given by the Sugawara 
construction provided that $k$ is non-critical,
that is, $k\ne -h^{\vee}_\g$,
with central charge
$$c_{V^{k}(\g)}=  \dfrac{k  \dim \g }{k + h_\g^\vee}.$$  

A $V^k(\g)$-module is the same as a smooth $\affg$-module of level $k$. Recall that a $\affg$-module $M$ is \emph{smooth} if, for all $a \in \g$ and $m \in M$, one has $at^n m = 0$ for all sufficiently large positive $n \in \Z$. The highest weight $\widetilde{\g}$-module $L_k(\lam+k\Lam_0)$ (regarded as a $\affg$-module and thus $V^k(\g)$-module) acquires conformal dimension
\[
h_{L(\lam)} = \frac{(\lam|\lam+2\rho)}{2(k+h_\g^{\vee})}.
\]
For a non-critical level $k$, we consider a $V^k(\g)$-module $M$ as a $\widetilde{\g}$-module by letting $D$ act as the semisimplification of $-L_0$. 
In particular, we identify the categories $\mc{O}_k$ and $\KL_k$ of $\affg$-modules of level $k$  with the respective full subcategories of $\tilde{\mc{O}}_k$ and $\tilde\KL_k$ consisting of modules on which the universal Casimir element $C$ \cite[Section 2.5]{Kac90}
 of $\widetilde{\g}$ acts  nilpotently. 
Accordingly, the affine space $\affh^*_k:=\h^* + k \Lam_0$ is identified with an affine subvariety of $\widetilde{\h}^*$ by the correspondence $$\lam+k\Lam_0\mapsto \lam+k\Lam_0- h_{L(\lam)} \delta$$
preserving the corresponding linkage relations in $\mc{O}_k$.

 Let $L_k(\g)$ be the unique simple graded quotient of $V^k(\g)$.
 For any graded quotient $V$ of 
$V^k(\g)$, we have $R_V=V/t^{-2}\g[t^{-1}] V$. 
In particular, $R_{V^k(\g)} \cong \C[\g^*]$ and, hence, $X_{V^k(\g)}=\g^*$. 
Furthermore, 
$X_{L_k(\g)}$ is a subvariety of $\g^* \cong \g$, which is $G$-invariant and 
conic. 

More generally, let $\mathfrak{a}$ be 
a Lie algebra endowed with a symmetric invariant bilinear form $\kappa$, 
and 
$$\widehat{\mathfrak{a}}_\kappa= \mathfrak{a}[t,t^{-1}]\oplus \C \mathbf{1}$$ 
be the {\em Kac-Moody affinisation} of $\mathfrak{a}$. 
It is a Lie algebra with commutation relations
$$[x t^m, y t^n] = [x,y] t^{m+n} + m \delta_{m+n,0} \kappa(x,y) \mathbf{1}, 
\quad [\mathbf{1},\widehat{\mathfrak{a}}_\kappa]=0,$$ 
for all $x,y \in \mathfrak{a}$ and all $m,n \in \Z$. 
Then the $\widehat{\mathfrak{a}}_\kappa$-module 
\begin{align*} 
V^\kappa(\mathfrak{a}) =
 U(\mathfrak{a})\otimes_{U(\mathfrak{a}[t]\oplus \C {\mathbf 1})}\C, 
\end{align*}
where $\C$ is  the one-dimensional representation 
 of 
$\mathfrak{a}[t] \oplus \C {\mathbf 1}$ on which $\mathfrak{a}[t]$ 
acts by 0 and ${\mathbf 1}$ acts as the identity, 
has a unique vertex algebra structure 
such that $\vac$ is the image of $1\otimes 1$ 
in $V^\kappa(\mathfrak{a})$ and 
$$x(z) :=(x_{(-1)}\vac)(z)= \sum_{n\in\Z} (xt^n) z^{-n-1}$$
for all $x\in \mathfrak{a}$.  
We have $X_{V^\kappa(\mathfrak{a})} \cong \mathfrak{a}^*$ 
and, letting $L_\kappa(\mathfrak{a})$ be the 
unique simple graded quotient of $V^\kappa(\mathfrak{a})$, 
$X_{L_\kappa(\mathfrak{a})}$ is a subvariety of $\mathfrak{a}^*$, 
which is Poisson and conic.

For $M\in \mc{O}_k$ on which $L_0$ acts semisimply,
we consider the multivariable character $\chi_M$ of $M$, defined by
\begin{align*}
\chi_{M}(\tau, z, t) = e^{2\pi \imag k t} {\rm tr}_{M} (e^{2\pi \imag z} e^{2\pi \imag \tau (L_0 - c/24)}), \qquad (\tau, z, t) \in \mathbb{H} \times \h \times \C.
\end{align*}
{We also write, in particular, $\chi_{L(\lam)}(\tau) = \chi_{L(\lam)}(\tau, 0, 0)$}.

For admissible weight $\lam$ a closed form for $\chi_{L(\lam)}(\tau, z, t)$ was given in \cite{KacWak89}. It is convenient to write $v = 2 \pi \imag (-\tau D+z+t K) \in \widetilde{\h}^*$ as in \cite{Kac90}, then the character formula is
\[
\chi_{L(\lam)}(v) = \frac{A_{\lam+\widehat\rho}(v)}{A_{\widehat\rho}(v)},
\]
where
\[
A_{\lam}(v) = e^{-\frac{|\lam|^2}{2\lam(K)} (\delta, v)} \sum_{w \in \widehat W(\lam)} \varepsilon(w) e^{\left< w(\lam), v \right>}.
\]

The complex number $k$ is said to be {\em admissible}
for $\affg$ if 
$k \Lambda_0$ is  admissible.
If this is the case, $L_k(\g)$ is called a simple {\em admissible affine vertex algebra}.
By \cite[Proposition 1.2]{KacWak08},
$k$ is admissible if and only if 
 \begin{align}
k +h_\g^\vee = \dfrac{p}{q} \text{ with } p,q \in \Z_{\ge 1}, \; 
(p,q)=1, \; p \ge \begin{cases} 
h_\g^\vee & \text{if } (r^\vee,q)=1 \\
h_\g & \text{if } (r^\vee,q)=r^\vee,
\end{cases}
\label{eq:admissible-n}
\end{align}
where $r^\vee$ is the lacing 
number\footnote{i.e., $r^\vee=1$ for the types $A,D,E$, $r^\vee=2$ for the types $B,C,F_4$, 
and $r^\vee=3$ for the type $G_2$.} of $\g$.

If $k$ is admissible with $(r^\vee,q)=1$, we say that $k$ is {\em principal}. 
If $k$ is admissible with $(r^\vee,q)=r^\vee$, we say that $k$ is {\em coprincipal}.

\begin{Th}[{\cite{Arakawa15a}}] 
\label{Th:admissible-orbits}
Assume that the level $k= -h_\g^\vee +p/q$ is admissible. 
Then 
$X_{L_k(\fing)} =\overline {\O}_k$, 
where $\O_k$ is a certain nilpotent orbit of $\g$ 
which only depends on the denominator $q$. 
\end{Th}

In particular,
the associated variety of an admissible affine vertex algebra
is contained in the nilpotent cone $\mc{N}_\g$ of $\g$.
We note that the converse is not true and
 there are
affine vertex algebras at  non-admissible levels  
whose associated variety is contained in $\mc{N}_\g$ \cite{AM15}.

\begin{Th}[\cite{Arakawa2016}]
\label{Th:rationality-in-category-O}
Let $k$ be admissible,
$\lam\in \affh^*_k$.
Then
$L(\lam)$ is an $L_k(\g)$-module 
if and only if 
$\lam$ is  an admissible weight
such that $\widehat{\Delta}(\lam)=y( \widehat{\Delta}(k\Lam_0))$
for some $y\in \widetilde{W}$.
Moreover, any 
 $L_k(\g)$-module 
that lies in $\mc{O}_k$ 
is a direct sum of admissible representations $L(\lam)$
of $\affg$ of level $k$
for $\lam$ satisfying 
$\widehat{\Delta}(\lam)=y(\widehat{\Delta}(k\Lam_0))$
for some $y\in \widetilde{W}$.
\end{Th}

For a principal admissible number,
let
$\on{Pr}^k$ be the set of admissible weights $\lam\in \affh^*_k$ 
such that 
$\widehat{\Delta}(\lam)=y( \widehat{\Delta}(k\Lam_0))$
for some $y\in \widetilde{W}$.
Similarly,
for a coprincipal admissible number,
let
$\on{CoPr}^k$ be the set of admissible weights $\lam\in \affh^*_k$ such that 
such that 
$\widehat{\Delta}(\lam)=y( \widehat{\Delta}(k\Lam_0))$
for some $y\in \widetilde{W}$.
An element
of $\on{Pr}^k$ (resp.\ $\on{CoPr}^k$) is called a {\em principal admissible weight}
(resp.\ {\em coprincipal admissible weight}) of level $k$. Occasionally we shall use $\on{Adm}^k$ to refer to the set $\on{Pr}^k$ or $\on{CoPr}^k$, according as $k$ is a principal or coprincipal admissible number.

For $\lam \in \widehat{\h}^*$ let us denote by $\bar \lam \in \h^{*}$ the restriction of $\lam$ to $\h$. For $\lam\in \on{Pr}^{k}$ (resp.\ $\lam\in \on{CoPr}^{k}$),
the primitive ideal $J_{\bar \lam}$ is an maximal ideal of $U(\g)$,
and $J_{\bar \lam}=J_{\bar \mu}$ if and only if $\bar\mu\in W\circ \bar\lam$
for $\lam,\mu\in \on{Pr}^{k}$
(resp.\ for $\lam,\mu\in \on{CoPr}^{k}$) (\cite[Proposition 2.4]{AraFutRam17}). Here and throughout $\circ$ denotes the `dot' action $w \circ \lam = w (\lam+\rho)-\rho$. Set
\begin{align}
[\on{Pr}^{k}]=\on{Pr}^{k}/\sim,\quad
[\on{CoPr}^{k}]=\on{CoPr}^{k}/\sim,
\end{align}
where $\lam\sim \mu \iff \bar\mu\in W\circ \bar\lam$.

If $V$ is a $\frac{1}{2}\Z$-graded vertex algebra, we denote by $\on{Zhu}(V)$ the \emph{Ramond twisted} Zhu algebra of $V$, briefly recalling its construction from (\cite{Zhu96,De-Kac06}). Bilinear products $*_n : V \otimes V \rightarrow V$ are defined by
\[
a *_n b = \sum_{j \in \Z_{\ge 0}} \binom{\Delta(a)}{j} a_{(n+j)}b,
\]
for $a$ of conformal weight $\Delta(a)$, $b \in V$ arbitrary. We write in particular $a * b = a *_{-1} b$ and $a \circ b = a *_{-2} b$. Then the Zhu algebra is a quotient of $V$ by the subspace $V \circ V$ spanned by all elements of the form $a \circ b$ for $a, b \in V$. The operation $a \otimes b \mapsto a * b$ is well-defined in the quotient and turns it into an associative unital algebra with unit $[\vac]$.

The fundamental property satisfied by $\on{Zhu}(V)$ is the existence of a bijection between the set of irreducible $\on{Zhu}(V)$-modules and the set of irreducible positive energy Ramond twisted $V$-modules, sending a $V$-module $M$ to its lowest graded piece $M_{\text{low}}$ equipped with the $\on{Zhu}(V)$-action 
$[a] \cdot m = a_0 m$ for $a \in V$ and $m \in M_{\text{low}}$.

For the simple affine vertex algebra $L_k(\g)$
we have 
\begin{align}
\on{Zhu}(L_{k}(\g))\cong U(\g)/I_k
\label{eq:Zhu-of-affine}
\end{align}
for some two-sided ideal $I_k$ of $U(\g)$.

The statement of Theorem \ref{Th:rationality-in-category-O}
is strengthened by the following assertion.

\begin{Th}[\cite{AEkeren19}]\label{The:Zhu-of-admissible}
Let $k$ be admissible.
We have
\begin{align*}
\on{Zhu}(L_{k}(\g))\cong \prod U(\g)/J_{\bar \lam},
\end{align*}
where the product is taken over 
$[\lam]\in [\on{Adm}^k]$.
\end{Th}

\begin{Th}
\label{Th:finite_extensions}
Let $k$ be admissible,
$V$ a conical self-dual conformal vertex algebra,
and
$\varphi:V^k(\g)\ra V$ a  conformal vertex algebra homomorphism.
Then $\varphi$ factors through an embedding
$L_k(\g)\hookrightarrow V$.
In particular,
$V$ is a direct sum of admissible 
$\affg$-modules.
\end{Th}
\begin{proof}  
It is known \cite{KacWak88} that
the maximal proper submodule $\widetilde{N}$ of $V^k(\g)$ is generated by a singular vector $v_{\lam}$ of weight, say, $\lam$. Therefore $N = \varphi(\widetilde{N})$ is generated as a $\widehat{\g}$-module by $\varphi(v_{\lam})$, and we shall show that $\varphi(v_{\lam})=0$.

Let $M$ be an ordinary $V$-module. Since $\varphi$ is conformal, the  restricted dual  $M^*=\bigoplus_{d\in \C}\Hom_{\C}(M_d,\C)$ 
(\cite{Frenkel:1993aa}) as a $V$-module is the same as the restricted dual of $M$ as a $V^k(\g)$-module. In particular,
$V$ is self-dual as a $V^k(\g)$-module.

 Assume that $\varphi(v_{\lam})\ne 0$.
 Then we have a non-splitting exact sequence
 $0\ra N\ra V\ra V/N\ra 0$ of $\affg$-modules.
 By the self-duality of $V$,
 this gives a non-splitting exact sequence
 \begin{align}
0\ra (V/N)^*\ra V\ra N^*\ra 0,
\label{eq:non-spliting}
\end{align}
of 
 $V^k(\g)$-modules.

Let $w_{\lam}$ be a weight vector of $V$ of weight $\lam$ that is mapped to the highest weight vector of $N^*$. Since $V_0=\C|0\ket$ and \eqref{eq:non-spliting} is non split,
we have $|0\ket \in U(\affg)w_{\lam}$. It follows that $w_{\lam}$ is not a singular vector and also that $\varphi(V^k(\g))\subset U(\affg)w_{\lam}$. Thus $\varphi(v_{\lam})$ and $w_{\lam}$ are linearly independent vectors
of $U(\affg)w_{\lam}$, which are both primitive in the sense of \cite[2.6]{MooPia95}.
This implies that 
 $[U(\affg)w_{\lam}:L(\lam)]\ge 2$.

Let $P(\lam)$ be the projective cover of $L(\lam)$ in $\mc{O}_k$,
and consider the 
 homomorphism
$g:P(\lam)\twoheadrightarrow U(\affg)w_{\lam}\hookrightarrow V$ that sends the generator of $P(\lam)$ of weight $\lam$ to $w_{\lam}$.
 Since $\varphi$ is conformal,
the universal Casimir element $C$ of $\widetilde{\g}$ acts as  zero on $V$.
It follows that 
$g$ factors through a homomorphism
$P(\lam)/CP(\lam)\ra V$.
But then Lemma \ref{Lem:Arakaw-Fiebig} below says that the multiplicity 
of $L(\lam)$ in $g(P(\lam))\cong U(\affg)w_{\lam}$ is at most one.
Since this is a contradiction, we obtain $\varphi(v_{\lam})=0$ as desired.
The assertion on complete reducibility now follows from Theorem \ref{Th:rationality-in-category-O}.
\end{proof}

The following assertion is a direct consequence of \cite[Lemma 6.9]{AraFie08}.
\begin{Lem}\label{Lem:Arakaw-Fiebig}
Let $\mu\in \widehat{\h}^*_k$ be dominant,
suppose $\lam< \mu$ and there is no $\nu$ such that $\lam<\nu<\mu$.
Then $[P(\lam)/CP(\lam):L(\lam)]=1$.
\end{Lem}
\begin{proof}
We have \cite{KacKaz79}  $[M(\mu):L(\lam)]=1$ and
an exact sequence
 $0\ra M(\mu)\ra P(\lam)\ra M(\lam)\ra 0$,
 see the discussion just before \cite[Lemma 6.9]{AraFie08} for details.
Hence $[P(\lam):L(\lam)]=2$.
Let $\tilde{w}_{\lam}$ be a generator of $P(\lam)$ of weight $\lam$.
By \cite[Lemma 6.9]{AraFie08},
$CP(\lam)\ne 0$,
and hence $C\tilde{w}_{\lam}\ne 0$
because $\tilde{w}_{\lam}$ is a generator.
It follows that 
$[CP(\lam):L(\lam)]= 1$,
and hence,
$[P(\lam)/CP(\lam):L(\lam)]=1$.
\end{proof}

If 
$L(\lam)$ is an ordinary 
module  over the simple affine vertex algebra $L_k(\g)$ of principal (resp. coprincipal) admissible level $k$
then  $\lam\in \on{Pr}^k_{\Z}:=\{\lam\in  \on{Pr}^k  \colon 
\bra \lam,\alpha^{\vee}\ket \in \Z \text{ for all } \alpha \in \Delta\}$  (resp.\ $\lam\in \on{CoPr}^k_{\Z}
:=\{\lam\in  \on{CoPr}^k \colon
\bra \lam,\alpha^{\vee}\ket \in \Z \text{ for all } \alpha \in \Delta\}$). 
We have
$\widehat{\Delta}(\lam)=\{\alpha+nq\delta \colon 
\alpha\in \Delta,\ n\in \Z\}$ 
for $\lam\in \on{Pr}^k_{\Z}$,
while 
$\widehat{\Delta}(\lam)=\{\alpha+nq\delta \colon 
\alpha\in \Delta^{\rm long},\ n\in \Z\}
\sqcup \{\alpha+nq\delta / r^\vee\colon \alpha\in \Delta^{\rm short},\ n\in \Z\}$
for $\lam\in \on{CoPr}^k_{\Z}$,
where $\Delta^{\rm long}$ is the set of long roots of $\Delta$
and $\Delta^{\rm short}$ is the set of short roots of $\Delta$.
The set of simple roots
of $\widehat{\Delta}(\lam)$, both for $\lam\in \on{Pr}^k_{\Z}$ and for $\lam\in \on{CoPr}^k_{\Z}$, is given by
$S_{(q)} = \{\gamma_0, \gamma_1, \ldots, \gamma\} $,
where 
 $\gamma_i = \alpha_i$ for $i=1,\ldots,\ell$  and
\begin{align*}
\gamma_0 = \begin{dcases}
- \theta  +q\delta& \quad \text{if $(r^\vee, q) = 1$}, \\
- \theta_{s}+q\delta /r^{\vee}& \quad \text{if $(r^\vee, q) = r^\vee$},
\end{dcases}
\end{align*}
where $\theta_s$ is the highest short root of $\Delta$.

Let
$\phi : \widetilde\h^* \rightarrow \widetilde\h^*$ be the isometry defined to act as the identity on the finite part $\h^*$ and to act by
\begin{align*}
\phi(\Lambda_0) = (1/q)\Lambda_0, \qquad \phi(\delta) = q \delta.
\end{align*}
The adjoint $\phi^* : \widetilde\h \rightarrow \widetilde\h$ acts by
\begin{align*}
\phi^*(K) = (1/q) K, \qquad \phi(D) = q D.
\end{align*}
We have \cite{KacWak89}
\begin{align*}
&\on{Pr}^k=\bigcup_{y\in  \widetilde{W}\atop 
y(S_{(q)} )\subset \widehat{\Delta}^{\text{re}}_+}\on{Pr}^k_y,
\quad \on{Pr}^k_y= y\phi(\widehat{P}^{p-h^{\vee}_\g}_{{+}}+\widehat\rho)-\widehat\rho,\\
&\on{CoPr}^k=\bigcup_{y\in  \widetilde{W}\atop 
y(S_{(q)} )\subset \widehat{\Delta}^{\text{re}}_+}\on{Pr}^k_y,
\quad \on{CoPr}^k_y= y\phi({}^\circ\widehat{P}^{p-h_\g}_{{+}}+\widehat\rho)-\widehat\rho.
\end{align*}
We have denoted by $\widehat{P}^{k}_{{+}}$ the set of dominant integral weights of $\affg$ of level $k$, that is, the set of weights $\lam = k\Lambda_0 + \bar\lam$ such that $\bra \bar\lam,\alpha_i^{\vee}\ket \in \Z_{\ge 0}$ for $i=1,\dots, \ell$, and $\bra \bar\lam,-\theta^{\vee}\ket \le k$. Similarly $\widehat{P}^{k}_{{++}}$ will denote the set of \emph{regular} dominant integral weights, given by $\bra \bar\lam,\alpha_i^{\vee}\ket \in \Z_{> 0}$ for $i=1,\dots, \ell$, and $\bra \bar\lam,-\theta^{\vee}\ket < k$. Furthermore ${}^\circ \widehat{P}^{k}_{{+}}$ is defined similarly to $\widehat{P}^{k}_{{+}}$, the condition $\bra \bar\lam,-\theta^{\vee}\ket \le k$ being replaced by $\bra \lam,-\theta_s^{\vee}+K\ket  \in \Z_{\ge 0}$. In particular, we have
\begin{align*}
\on{Pr}^k_{\Z}=\on{Pr}^k_{1}=\phi(\widehat{P}^{p-h^{\vee}_\g}_{{+}}+\widehat{\rho})-\widehat{\rho},\quad 
\on{CoPr}^k_{\Z}=\on{CoPr}^k_{1}=\phi({}^\circ \widehat{P}^{p-h_\g}_{{+}}+\widehat{\rho})-\widehat{\rho}.
\end{align*}

\begin{Pro}[\cite{KacWak89}]
\label{Proposition:Kac-Wakimoto}
Let $k$ be  principal admissible  of the form \eqref{eq:admissible-n}
and
let $$\lam=y \phi(\nu) - \widehat\rho\in \on{Pr}^k_y$$
with 
$y = t_\beta \ov y$, $\beta \in P^\vee$, $\ov y \in W$,
and $\nu \in {}^\circ\widehat{P}_+^{p-h^\vee}+\widehat\rho$.
Let $T \in \mathbb{R}_{>0}$ and $z \in \h$ such that $\alpha(z) \notin \Z$ for all $\alpha \in \Delta$. 
Then as $T \rightarrow 0^+$ one has 
\begin{align*}
\chi_{L(\lambda)}({\imag}T, -{\imag}T z, 0) \sim b(\lam, z) e^{\pi \G / 12T},
\end{align*}
where
\begin{align*}
\G &= \dim(\g) - \frac{12 |\rho^\vee|^2}{pq} = \left( 1 - \frac{h^\vee_\g}{pq} \right) \dim(\g), \\
\text{and} \quad b(\lam, z) &= \varepsilon(\ov y) |P/ pq Q|^{-\frac{1}{2}} \prod_{\alpha \in \Delta_+} 2 \sin\frac{\pi (\alpha | \nu)}{p} \cdot \frac{ \sin \frac{\pi (\alpha | z-\beta)}{q}}{\sin \pi(\alpha | z)}.
\end{align*}
\end{Pro}

The proof of the proposition below closely follows the arguments of \cite{KacWak89}, with adaptations to the coprincipal case.
\begin{Pro}\label{prop:general.affine.asymp.coprincipal}
Let $k$ be coprincipal admissible of the form \eqref{eq:admissible-n}
and let 
\begin{align*}
\lambda = y \phi(\nu) - \widehat\rho\in \on{CoPr}^k_y,
\end{align*}
with  $y = t_\beta \ov y$, $\beta \in P^\vee$,
$\ov y \in W$, and $\nu \in {}^\circ\widehat{P}_+^{p-h^\vee}+\widehat{\rho}$.  
Let $T \in \mathbb{R}_{>0}$ and $z \in \h$ such that $\alpha(z) \notin \Z$ for all $\alpha \in \Delta$. Then as $T \rightarrow 0^+$ we have
\begin{align*}
\chi_{L(\lambda)}({\imag}T, -{\imag}T z, 0) \sim b(\lam, z) e^{\pi \G / 12T},
\end{align*}
where
\begin{align*}
\G &= \dim(\g) - \frac{12 |\rho^\vee|^2}{pq} = \left( 1 - \frac{h^\vee_{{}^L\g} r^\vee}{pq} \right) \dim(\g), \\
\text{and} \quad b(\lam, z) &= \varepsilon(\ov y) |P^\vee/ pq Q|^{-\frac{1}{2}} \prod_{\alpha \in \Delta_+} 2 \sin\frac{\pi (\alpha^\vee | \nu)}{p} \cdot \frac{ \sin \frac{\pi (\alpha^\vee | z-\beta)}{q}}{\sin \pi(\alpha | z)}.
\end{align*}
\end{Pro}

\begin{proof}
The denominator of
\begin{align}\label{eq:aff.char.frac}
\chi_{L(\lam)}(\tau, z, t) = \frac{A_{\lam+\widehat\rho}(\tau, z, t)}{A_{\widehat\rho}(\tau, z, t)} 
\end{align}
is the standard Weyl denominator. Its asymptotic behaviour is well known \cite[Proposition 13.13]{Kac90} to be
\begin{align}\label{eq:denom.asymptotic-aff}
A_{\widehat{\rho}}({\imag}T, -{\imag}T z, 0) \sim b(\rho, z) T^{-\ell/2} e^{-\pi\dim(\g)/12T}, 
\end{align}
where
\begin{align*}
b(\rho, z) = \prod_{\alpha \in \Delta_+} 2 \sin{\pi (\alpha | z)}.
\end{align*}

We analyse the numerator by writing it in terms of theta functions. Let $m$ be a positive integer for which the lattice $\sqrt{m}Q$ is integral, and let $\widehat\mu = m\Lambda_0 + \mu \in \widehat\h^*$ where $\mu \in P^\vee$. Then we define
\begin{align*}
\Theta_{\widehat\mu, Q}(v) &= e^{-\frac{|\mu|^2}{2\mu(K)} \left<\delta, v\right>} \sum_{\alpha \in Q} e^{\left< t_\alpha \mu, v \right>}.
\end{align*}
The modular transformation of these theta functions is given by 
\begin{align*}
\Theta_{\widehat\mu, Q}({\imag}T, -{\imag}Tz, 0) &= T^{-\ell/2} |P^\vee/ m Q|^{-1/2} \\
& \phantom{............} \times \sum_{\mu' \in P^\vee / m Q} e^{-2\pi \imag (\mu' | \mu)/m} \Theta_{\widehat\mu', Q}\left(-\frac{1}{{\imag}T}, z, -\frac{{\imag}T(z | z)}{2} \right),
\end{align*}
where in the sum $\widehat\mu' = m\Lambda_0 + \mu'$.

In the sum defining $A_{\lam+\widehat\rho}(v)$ the action of the Weyl group $\widehat{W}(\lam)$ is intertwined with the action of the group $\widehat W(S_{(q)}) = W \ltimes t_{q Q}$ via the automorphism $y$. The numerator of \eqref{eq:aff.char.frac} is thus expressed in terms of theta functions as
\begin{align*}
A_{\lam+\widehat\rho}(v)
&= e^{2\pi \imag \tau \frac{|\lam+\widehat\rho|^2}{2(k+h^\vee)}} \sum_{u \in \widehat W(S_{(q)})} \varepsilon(u) e^{\left< y u \phi(\nu), v \right>} \\
&= \varepsilon(\ov y) \sum_{w \in W} \varepsilon(w) \Theta_{q w(\widehat\nu), Q}((1/q)\phi^{*}t_{\beta}^*(v)).
\end{align*}

Now we compute the asymptotic behaviour of
\begin{align*}
\varepsilon(\ov y) \sum_{w \in  W} \varepsilon(w) \Theta_{q w(\widehat\nu), Q}({\imag}T, -{\imag}Tz, t).
\end{align*}
Since $\lam$ is coprincipal the lattice $pq Q$ is integral and we may use the modular transformation formula to write
\begin{align*}
& \sum_{w \in W} \varepsilon(w) \Theta_{q w(\widehat{\nu}), Q}({\imag}T, -{\imag}Tz, 0) \\
&= T^{-\ell/2} |P^\vee/ pq Q|^{-1/2} \\
& \phantom{............} \times \sum_{w \in W} \sum_{\widehat\mu \in P^\vee / (pq) Q} \varepsilon(w) e^{-2\pi \imag (\mu' | q w(\nu))/(pq)} \Theta_{\widehat\mu, Q}\left( -\frac{1}{iT}, z, -\frac{iT(z | z)}{2} \right) \\
&= T^{-\ell/2} |P^\vee/ pq Q|^{-1/2} e^{\pi (pq) T (z | z)}  \\
& \phantom{............} \times \sum_{w \in W} \sum_{\stackrel{\widehat\mu \in P^\vee / (pq) Q}{\text{regular}}} \sum_{\gamma \in \mu/(pq) + Q} \varepsilon(w) e^{-\frac{2\pi \imag}{p} (w(\nu) | \mu)} e^{- \frac{\pi}{T} (pq) (\gamma | \gamma) + 2\pi \imag (pq) (\gamma | z)}. 
\end{align*}
(The summands corresponding to non regular $\widehat{\mu}$ cancel out in the sum over $W$.) In the limit $T \rightarrow 0^+$ the dominant terms in the sum above come from the shortest regular $\mu \in P^\vee$. Such $\mu$ consist precisely of the $W$-orbit of $\rho^\vee$. We therefore obtain the asymptotic
\begin{align*}
\sum_{w \in W} \varepsilon(w) \Theta_{q w(\widehat{\nu}), Q}(iT, -iTz, 0)
\sim {} & T^{-\ell/2} |P^\vee/(pq)Q|^{-1/2} e^{-\frac{\pi}{T} \cdot \frac{|\rho^\vee|^2}{pq}} \\
&\times \sum_{\sigma \in W} \left[ \sum_{w \in W} \varepsilon(w) e^{-\frac{2\pi \imag}{p}(\sigma w(\nu) | \rho^\vee)}\right] e^{2\pi \imag (\rho^\vee | \sigma(z))}.
\end{align*}
The Weyl denominator formula asserts that
\begin{align}\label{eq:Weyl.denom.fla.check}
\sum_{w \in W} \varepsilon(w) e^{(w(\rho^\vee) | \lam)} = \prod_{\alpha \in \Delta_+} \left( e^{(\alpha | \lam)/2} - e^{-(\alpha | \lam)/2} \right).
\end{align}
(There is a more standard version of the formula, from which \eqref{eq:Weyl.denom.fla.check} is obtained by applying to the Langlands dual Lie algebra and noting that the duality exchanges $\rho$ with $\rho^\vee$ up to a factor of $\sqrt{r^\vee}$.) Using \eqref{eq:Weyl.denom.fla.check} we reduce the double sum over $W$ to
\begin{align*}
\prod_{\alpha \in \Delta_+} 4 \sin\frac{\pi (\alpha^\vee | \nu)}{p} \sin \pi (\alpha^\vee | z).
\end{align*}

To deduce the asymptotic behaviour of $A_{\lam + \widehat\rho}$ we note that
\begin{align*}
(1/q) \phi^*t_{\beta}^*(2\pi \imag (-\tau D + z + tK)) = 2\pi \imag \left( -\tau D + (z + \tau \nu^{-1}(\beta))/q + (t + \tau \tfrac{|\beta|^2}{2}) K/q^2 \right).
\end{align*}
It then follows that
\begin{align}\label{eq:numerator.asymptotic}
\begin{split}
A_{\lam + \widehat\rho}({\imag}T, -{\imag}Tz, 0)
\sim {} & \varepsilon(\ov y) T^{-\ell/2} |P^\vee/(pq)Q|^{-1/2} e^{-\frac{\pi}{T} \cdot \frac{|\rho^\vee|^2}{pq}} \\
&\times \prod_{\alpha \in \Delta_+} 4 \sin\frac{\pi (\alpha^\vee | \nu)}{p} \sin \frac{\pi (\alpha^\vee | z-\beta)}{q},
\end{split}
\end{align}
and, combining with \eqref{eq:denom.asymptotic-aff}, finally
\begin{align*}
\chi_{L(\lam)}({\imag}T, -{\imag}Tz, 0)
\sim {} & \varepsilon(\ov y) |P^\vee/(pq)Q|^{-1/2} e^{\frac{\pi}{12 T} \cdot \left(\dim(\g) - \frac{12|\rho^\vee|^2}{pq}\right)} \\
&\times \prod_{\alpha \in \Delta_+} 2 \sin\frac{\pi (\alpha^\vee | \nu)}{p} \cdot \frac{ \sin \frac{\pi (\alpha^\vee | z-\beta)}{q}}{\sin \pi(\alpha | z)}.
\end{align*}
The second form for $\G$ given in the theorem statement follows from the Freudenthal-de Vries strange formula and the fact that Langlands duality exchanges $\rho$ and $\rho^\vee$.
\end{proof}

\begin{Co} 
\label{Co:asymptotic_data_L}
Let $k$ be an admissible number of the form \eqref{eq:admissible-n}.
The admissible affine vertex algebra $L_k(\g)$ admits \asym,
and so does any simple ordinary representation $L(\lam)$ of $L_k(\g)$.
\begin{enumerate}
\item 
For $\lam\in \on{Pr}^k_\Z$, 
\begin{align*}
&\G_{L(\lam)} =  \left(1- \dfrac{h_\g^\vee}{pq} \right)\dim \g,\quad \Weight_{L(\lam)}=0,
\\
&\A_{L(\lam)} =\dfrac{1}{q^{|\Delta_{+}|} 
{|P/(p q)Q^\vee|}^{\frac{1}{2}} } \prod\limits_{\alpha \in \Delta_{+}} 
2 \sin \dfrac{\pi (\lam+\rho | \alpha) }{p},\\
&\on{qdim}_{L(k\Lam_0)}L(\lam)=\prod\limits_{\alpha \in \Delta_{+}} \frac{(\lam+\rho|\alpha)_t}{(\rho|\alpha)_t},
\end{align*}
where $n_t=(t^n-t^{-n})/(t-t^{-1})$,
$t=e^{\pi \imag /p}$.
\item For $\lam\in \on{CoPr}^k_\Z$,
\begin{align*}
&\G_{L(\lam)}= \left(1-\dfrac{r^\vee  h^\vee_{^L\!\g} }{p q}\right) \dim \g ,\quad \Weight_{L(\lam)}=0,
\\
&\A_{L(\lam)} 
= \dfrac{(r^\vee_\g)^{|\Delta_+^{\textup{short}}|}}{ 
q^{|\Delta_+|}{|P^\vee/(pq) Q|}^{\frac{1}{2}}}
\prod_{\alpha \in \Delta_+} 2 \sin \dfrac{\pi (\lam+\rho | \alpha^\vee)}{p},\\
&\on{qdim}_{L(k\Lam_0)}L(\lam)=\prod\limits_{\alpha \in \Delta_{+}} \frac{(\lam+\rho|\alpha^{\vee})_t}{(\rho|\alpha^{\vee})_t}.
\end{align*}
\end{enumerate}
\end{Co}
\begin{proof}
The assertion follows from 
Propositions \ref{Proposition:Kac-Wakimoto} and 
\ref{prop:general.affine.asymp.coprincipal}
 by taking the limit $z \rightarrow 0$ with $\beta=0$  and $\bar y=1$.
 Let us explain the details for the coprincipal case.
 The normalised character $\chi_{L(\lam)}(\tau)$ is just the specialization $\chi_{L(\lam)}(\tau, 0, 0)$. To prove the first assertion we therefore apply Proposition \ref{prop:general.affine.asymp.coprincipal}
with $\beta = 0$, $\ov y = 1$, $\nu = \lam+\rho$, and we let $z$ tend to $0$ along a ray disjoint from the hyperplanes $(\alpha | z) = 0$. Comparing the definition of the asymptotic dimension $\A$ 
(Definition \ref{Def:asymp.datum} above) with that of $b(\lam, z)$ yields 
\begin{align*}
\A_{L(\lam)}
&= \lim_{z \rightarrow 0} b(\lam, z) \\
&= |P^\vee/ pq Q|^{-1/2} \prod_{\alpha \in \Delta_+} 2 \sin\frac{\pi (\alpha^\vee | \lam+\rho)}{p} \cdot \frac{|\alpha^\vee|/q}{|\alpha|} \\
&= |P^\vee/ pq Q|^{-1/2} \frac{(r^\vee)^{|\Delta_+^{\text{short}}|}}{q^{|\Delta_+|}} \prod_{\alpha \in \Delta_+} 2 \sin\frac{\pi (\alpha^\vee | \rho)}{p}.
\end{align*}
Here we used l'Hopital's rule. 
\end{proof}

In the above we have $\A_{L(\lam)}>0$ since 
$0<(\lam+\rho,\alpha)=
\frac{(\alpha|\alpha)}{2}( \lam+\rho|\alpha^{\vee}) <p$
for $\lam\in \on{Pr}^k$, $\alpha\in \Delta_+$,
and
$0<(\lam+\rho,\alpha^\vee)
 <p$
for $\lam\in \on{CoPr}^k$, $\alpha\in \Delta_+$.

Note that $|P /(pq) Q^\vee | = (pq)^\ell |P /Q^\vee|$ 
and 
$|P^\vee/ pq Q^\vee| = \left( \frac{pq}{r^\vee}\right)^\ell  |P^\vee/ r^\vee Q|$ 
if $r^\vee | q$. 
The values of $|P /Q^\vee|$ and $|P^\vee/ r^\vee Q|$, 
as well as other useful data, 
are collected in Table \ref{Tab:data_simple} for each simple Lie algebra. 

{\tiny 
\begin{table}[h]
\begin{tabular}{|c|ccccccccc|}
\hline 
&&&&&&&&&  \\
$\g$ & $A_\ell$, $\ell \ge 1$ & $B_\ell$, $\ell \ge 2$  & $C_\ell$, $\ell \ge 3$ & $D_\ell$, $\ell \ge 4$ 
& $E_6$ & $E_7$ & $E_8$ & $F_4$ & $G_2$ \\
&&&&&&&&&  \\
\hline 
&&&&&&&&&  \\[0.1em]
$\dim \g$ & $\ell (\ell+2)$ & $\ell(2\ell+1)$ & $\ell(2\ell+1)$ & $\ell(2\ell-1)$ & 
$78$ & $133$ & $248$ & $52$ & $14$ \\[0.5em]
$h_\g$ & $\ell+1$ & $2\ell$ & $2\ell$ &  $2\ell-2$ & $12$ & $18$ & $30$ & $12$ & $6$ \\[0.5em]
$h_\g^\vee$ & $\ell+1$ & $2\ell-1$ & $\ell+1$ & $2\ell-2$ & $12$ & $18$ & $30$ & $9$ & $4$ \\[0.5em]
$|P^\vee/Q^\vee|$ & $\ell+1$ & $2$ & $2$ & $4$ & $3$ & $2$ & $1$ & $1$ & $1$ \\[0.5em]
$|P /Q^\vee|$ & $\ell+1$ & $4$ & $2^{\ell}$ & $4$ & $3$ & $2$ & $1$ & $4$ & $3$ \\[0.5em]
$|P^\vee/ r^\vee Q|$ & $\ell+1$ & $2^{\ell}$ & $4$ & $4$ & $3$ & $2$ & $1$ & $4$ & $3$ \\[0.5em]
$|\Delta_+|$ & $\frac{\ell(\ell+1)}{2}$ 
& $\ell^2$ & $\ell^2$ & $\ell(\ell-1)$ & $36$ & $63$ & $120$ & $24$ & $6$ \\[0.5em]
$|\Delta_+^{\rm short}|$ &  
& $\ell$ & $\ell(\ell-1)$ &  &  &  &  & $12$ & $3$ \\[0.5em]
\hline
\end{tabular} \\[1em] 
\caption{\footnotesize Some data for simple Lie algebras} 
\label{Tab:data_simple} 
\end{table}
}
\begin{Th} 
\label{Th:main}
Let $k$ be admissible,
and let $V$ be a conical self-dual conformal vertex algebra
equipped with a vertex algebra homomorphism
$\varphi:V^k(\g)\ra V$
such that $\varphi(\omega_{V^k(\g)})\in V_2$
and $(\omega_V)_{(2)}\varphi(\omega_{V^k(\g)})=0$.
Assume that 
$V$ is a quotient of a vertex algebra 
$\tilde{V}$ that
admits \asym\ 
such that 
\begin{align*}
\Weight_{\tilde{V}}=0 
\quad \text{ and }\quad 
\G_{\tilde{V}} = \G_{L_{k}(\g)}.
\end{align*}
Then $\varphi$ factors through a finite extension 
$L_k(\g)\hookrightarrow V$. 
If moreover 
\begin{align*}
\chi_{\tilde{V}}(\tau) \sim \chi_{L_{k}(\g)}(\tau), 
\quad \text{as }\tau\downarrow 0, 
\end{align*}
that is, 
\begin{align*}
\A_{\tilde{V}} = \A_{L_{k}(\g)}, \quad
\Weight_{\tilde{V}}=0
\quad \text{ and }\quad 
\G_{\tilde{V}} = \G_{L_{k}(\g)}.
\end{align*}
Then $\varphi$ factors through an isomorphism
$L_k(\g)\cong V$. 
\end{Th}
\begin{proof}
By Proposition \ref{Pro:conformal}, $\varphi$ is conformal.
Hence, by Theorem \ref{Th:finite_extensions},
$V$ is a direct sum of admissible representations
$L(\lam)$ with 
$\lam\in \on{Pr}^k_{\Z}$  if $k$ is principal
(resp.\ $\lam\in \on{CoPr}^k_{\Z}$  if $k$ is coprincipal). 
So we can write
 $$V=\bigoplus_{\lam\in \on{Pr}^k_{\Z}}L(\lam)^{\+ m_{\lam}},\quad
 (\text{resp.\ }  V=\bigoplus_{\lam\in \on{CoPr}^k_{\Z}}L(\lam)^{\+ m_{\lam}})$$
 with
$m_{k\Lam_0}=1$. 
In particular,
$V$ admits an asymptotic datum
with $\G_V=\G_{L_k(\g)}=\G_{\tilde{V}}$, $\Weight_V=0$,
$\A_V=\sum_{\lam \in \on{Pr}^k_{\Z}}m_{\lam}\A_{L(\lam)}$
 (resp.\  $\A_V=\sum_{\lam \in \on{CoPr}^k_{\Z}}m_{\lam}\A_{L(\lam)}$).
Since $\A_V \le \A_{\tilde{V}}$ 
and
 $\A_{L(\lam)}>0$  for all $\lam\in \on{Pr}^k_{\Z}$  (resp. $\lam\in \on{CoPr}^k_{\Z}$),
  we get the first assertion. 
If in addition $\A_{\tilde{V}}=\A_{L_k(\g)}$, then 
  $\A_V \le \A_{L_k(\g)}$ which forces $m_{k\Lam_0}=1$ and all 
  others $m_{\lam}$ are zero, whence the statement. 
\end{proof}

We now briefly discuss some relationships between primitive ideals, associated varieties, and conformal dimensions. For a two-sided ideal $I$ of
$U(\g)$, we denote $\on{Var}(I)$ be the zero locus
of $\on{gr}I\subset \on{gr}U(\g)=\C[\g^*]$  in~$\g^{*}$. Now let $k$ be an admissible number  with denominator $q$. We distinguish the following weight
\begin{align}
\lam_{\ooh}=\begin{cases}
\rho/q-\rho &\text{if }(q,r^{\vee})=1,\\
\rho^{\vee}/q-\rho &\text{if }(q,r^{\vee})\ne 1.
\end{cases}
\label{eq:def-lam0}
\end{align}
By \cite[Proposition 2.4]{AEM}
\begin{align}
\lam_{\ooh,k}:= \lam_o+k\Lam_0\in \begin{cases}
\on{Pr}^k &\text{if }(q,r^{\vee})=1,\\
\on{CoPr}^k &\text{if }(q,r^{\vee})\ne 1.
\end{cases}
\end{align}
and
\begin{Th}[\cite{AEM}]
\label{Th:short-paper}
For an admissible number $k$, we have
$\on{Var}(J_{\lam_{\ooh}})=\overline{\mathbb{O}}_k$.
\end{Th}

\begin{Rem}
Let $k$ be 
an admissible number and  let $\lam\in \h^*$ such that $\lam+k\Lam_0 \in \on{Adm}^k$. Since 
 $\on{Var}(I_k)=\overline{\mathbb{O}}_k$ by \cite[Theorem 9.5]{A2012Dec}
(recall  \eqref{eq:Zhu-of-affine}) and
$L_\g(\lam)$ is a $\on{Zhu}(L_k(\g))$-module (Theorem \ref{Th:rationality-in-category-O}),
we have
$$\on{Var}(J_\lam)\subset \overline{\mathbb{O}}_k.$$
\end{Rem}
The connection with conformal dimensions is given by:
\begin{Pro}\label{Pro:minimal-conformal-dim}
For an admissible number $k$, the module $L(\lam_{\ooh,k})$ has the minimal 
conformal dimension
among the simple  $L_k(\g)$-modules that belong to $\mc{O}_k$.
\end{Pro}
\begin{proof}
By \cite{Arakawa2016} a simple $L_k(\g)$-module which belongs to $\mc{O}_k$ is an irreducible highest weight module $L(\widehat{\mu})$ of highest weight $\widehat{\mu} = \mu + k\Lambda_0 \in \on{Adm}^k$.

First we consider the principal admissible case $(q,r^{\vee})=1$. We have, cf. Proposition \ref{Proposition:Kac-Wakimoto} above and preceding remarks, $\widehat{\mu} = y \phi(\widehat{\nu}) - \widehat{\rho}$ where $\widehat{\nu} = \nu + p \Lambda_0$ and $\nu \in P_{++}^{p}$, where $y = \ov y t_{-\eta}$ with $\eta \in P_+^{\vee, q}$.

The conformal weight of $L(\widehat\mu)$ is
\[
h_{L(\mu)} = \frac{(\mu | \mu + 2\rho)}{2(k+h^\vee_\g)} = \frac{|\mu + \rho|^2 - |\rho|^2}{2(k+h^\vee_\g)}.
\]
Now we consider
\begin{align*}
|\mu + \rho|^2 = |\overline{y \phi(\nu)}|^2 = |\nu - \frac{p}{q} \eta|^2 = \frac{1}{q^2} |q\nu - p\eta|^2.
\end{align*}
We show that $q\nu - p\eta \in P$ is a regular weight. Indeed suppose $\left<q\nu - p\eta, \alpha_i^\vee\right> = 0$ for some $i = 1, \ldots, \ell$ and put $m = \left<\nu, \alpha_i^\vee\right>$ and $n = \left<\eta, \alpha_i^\vee\right>$. Then $qm - pn = 0$ and so $p$ divides $m$, but this is a contradiction since $0 < m < p$ because $\nu$ is regular.

The regular elements of $P$ of minimal norm are $\rho$ and its images under the finite Weyl group, so the claim is proved in the principal admissible case. The coprincipal admissible case is similar. Since $(q, r^\vee) = r^\vee$ we have $q\nu \in qP \subset P^\vee$ and hence $q\nu - p\eta \in P^\vee$. The regular elements of $P^\vee$ of minimal norm are $\rho^\vee$ and its images under the finite Weyl group, thus proving the claim.
\end{proof}

\begin{Rem}{\ }
\label{Rem:remar-on-min-con-wt}
\begin{enumerate}
\item 
The weight $
\lam_{\ooh,k}$ 
is not the unique element of $\on{Adm}^k$ that gives the minimal conformal dimension unless $k\in \Z_{\geq 0}$.
Indeed, for $\hat\mu=\mu+k\Lam_0\in \on{Adm}^k$, 
$L(\hat{\mu})$ has the minimal conformal dimension
if and only if 
$\mu+\rho=w(\lam_{\ooh}+\rho)$ with $w\in W$ such that 
$w(\alpha)\in \Delta_+$
for $\alpha\in \Delta_+$ with  $(\rho,\alpha^{\vee}) \in q \Z$
(resp.\ $\lam+\rho=w(\rho^\vee)/q$ with $w\in W$ such that $w(\alpha)\in \Delta_+$
for $\alpha\in \Delta_+$ with $(\rho^{\vee},\alpha^{\vee}) \in q \Z$).

\item 
The strange formula implies that 
the asymptotic growth coincides with the effective central charge, that is,
\begin{align*}
\G_{L(\lam)}=c_{L_k(\g)}-24 h_{\text{min}}
\end{align*}
for all $\lam\in \on{Adm}^k$.
\end{enumerate}

\end{Rem}

\section{Asymptotic data of $W$-algebras}
\label{sec:Asymptotics_of_W_algebras} 
Let $f$ be a nilpotent element of $\g$. 
Recall \cite{KacRoaWak03} that a $\frac{1}{2}\Z$-grading 
\begin{align}
\g
= \bigoplus_{j \in \frac{1}{2}\Z} 
\g_\Gamma^j
\label{eq:good-grading}
\end{align}
 is called {\em good} for $f$  if $f \in \g_\Gamma^{-1}$, $\ad f \colon \g_\Gamma^{j} 
\to \g_\Gamma^{j-1}$ is injective for $j \ge \frac{1}{2}$ and 
surjective for $j \le \frac{1}{2}$. The grading is called even if $\g_\Gamma^j = 0$ for
 $j\not\in \Z$. By the Jacobson-Morosov Theorem, the nilpotent element $f$ embeds into an 
 $\sl_2$-triple $(e,h,f)$, and $\g$ thereby inherits a $\frac{1}{2}\Z$-grading induced by the  
eigenvalues of $\ad(h/2)$. 
Such a grading is called a {\em Dynkin grading}. 
All Dynkin gradings are good, but not all good gradings are Dynkin.

Let $x^0_\Gamma$ be the semisimple element 
of $\g$ defining the grading \eqref{eq:good-grading}, that is, 
$$\g_\Gamma^{i} = \{y \in \g \colon [x^0_\Gamma,y] =  i y\}.$$
We can assume that $x^0_\Gamma$ 
is contained in the Cartan subalgebra 
$\h$ and that $\alpha(x^0_\Gamma) \in \frac{1}{2}\Z_{\ge 0}$ 
for all $\alpha \in \Delta_+$. 
For $j \in \frac{1}{2}\Z$, we set 
$$
\Delta_\Gamma^{j} := \{\alpha \in \Delta \colon (\alpha | x^0_\Gamma ) =  j \} =\{\alpha \in \Delta \colon \g^\alpha \subset   \g_\Gamma^j \},
$$
where $\g^\alpha$ is the $\alpha$-root space.
  
Since the bilinear form $(x,y) \mapsto (f | [x,y])$ is non-degenerate 
on $\g_\Gamma^{1/2} \times \g_\Gamma^{1/2}$, the set $\Delta_\Gamma^{1/2}$ 
has even cardinality. 
We set $\Delta^0_{\Gamma,+}= \Delta_\Gamma^0 \cap \Delta_{+}$. 
It is a set of positive 
roots for $\Delta_\Gamma^0$.

\begin{Rem}
\label{Rem:asymptotic dimension}
Unless otherwise specified, $\Gamma$ will always be the Dynkin grading 
associated with $(e,h,f)$. In this case, $x^0_\Gamma$ is $h/2$, 
and we will briefly write $x^0$, $\g^j$, $\Delta^j$, $\Delta_+^0$ for 
$x_\Gamma^0$, $\g_\Gamma^j$, $\Delta_\Gamma^j$, $\Delta_{\Gamma,+}^0$, 
respectively.  
However, in type $A$, even good gradings always exist \cite{ElashviliKac}, 
and it will be convenient to opt for an even good grading that is not necessarily 
the Dynkin grading. 
\end{Rem}

We denote by $\W^k(\g,f)$ the universal affine 
$W$-algebra associated with $\g,f$ at level $k$ 
and a good grading $\g_\Gamma= \bigoplus_{j \in \frac{1}{2}\Z} 
\g_\Gamma^j$: 
$$\W^k(\g,f) = H_{DS,f}^{0}(V^k(\g)),$$
where $H_{DS,f}^\bullet(?)$ is the BRST cohomology functor of 
the quantized Drinfeld-Sokolov
reduction associated with $\g,f$ (\cite{FeiFre90,KacRoaWak03}). The BRST complex of a $\widehat{\g}$-module $M$ takes the form 
\begin{align*}
C^\bullet(M) = M \otimes F^{\text{ch}} \otimes F^{\text{ne}},
\end{align*}
where $F^{\text{ch}}$ denotes the Clifford vertex algebra associated with $\g_\Gamma^{>0} \oplus (\g_\Gamma^{>0})^*$ and its canonical symmetric bilinear form, or in the terminology of \cite{KacRoaWak03} the charged free fermions vertex algebra associated with $\g_\Gamma^{>0} \oplus (\g_\Gamma^{>0})^*$ made into a purely odd vector superspace and canonical skew-supersymmetric form, and $F^{\text{ne}}$ denotes the neutral free fermion vertex algebra associated with $\g_\Gamma^{1/2}$ and bilinear form $(x,y) \mapsto (f | [x,y])$. We omit the detailed definition, referring the reader to (\cite{KacRoaWak03}).

The $W$-algebra $\W^k(\g,f)$ is conformal provided that $k\ne -h_\g^{\vee}$
and its central charge $c_{\W^k(\g,f)}$ is given by
\begin{align}
c_{\W^k(\g,f)}&=c_{V^k(\g)}-\on{dim}G.f-\frac{3}{2}\dim \g_{1/2}+24(\rho|x^0_{\Gamma})-12(k+h^{\vee}_\g)|x^0_{\Gamma}|^2
\label{eq:W.alg.c.formula2}
\\&
=\dim \g_\Gamma^0 -\dfrac{1}{2} \dim \g_\Gamma^{1/2} 
-\dfrac{12}{k+h_\g^\vee} | \rho - (k+h_\g^\vee) x^0_\Gamma |^2
\label{eq:W.alg.c.formula}
\end{align}
(\cite[Theorem 2.2]{KacRoaWak03}).
Although the vertex algebra structure of 
$\W^k(\g,f)$ does not depend on the choice of the good grading of $\g$ 
(\cite[\S3.2.5]{ArakawaKuwabaraMalikov}),
its conformal structure does. In general we denote by $\W_k(\g,f)$ the unique simple graded quotient of $\W^k(\g,f)$.
\begin{Pro}\label{Pro:Dynkin-self-dual}
The vertex algebra  $\W_k(\g,f)$ is self-dual if the grading \eqref{eq:good-grading} is Dynkin.
\end{Pro} 
\begin{proof}
By \cite[Proposition 6.1 and Remark 6.2]{AEkeren19},
it suffices to show that
$(k+h^{\vee})(h|v)-{\rm tr}_{\g_{>0}}(\ad v)=0$ for all $v\in \g^f_0$.
Since $(h|v)={([e,f]|v)}=0$ for $v\in\g^f_0$,
in fact it is enough to show that ${\rm tr}_{\g_{>0}}(\ad v)=0$.
Note that $\g^f_0$ is the centraliser  $\g^\natural$ of the $\mf{sl}_2$-triple $\{e,h,f\}$ in $\g$,
which is a reductive Lie subalgebra of $\g$.
We clearly have ${\rm tr}_{\g_{>0}}(\ad v)=0$ 
 for $v\in [\g^\natural,\g^\natural]$
since $\g_{>0}$ is a finite-dimensional 
representation of 
the semisimple Lie algebra $ [\g^\natural,\g^\natural]$.
On the other hand,
Lemma \ref{lem:centraliser_is_reductive_subalgebra}  below shows that 
 ${\rm tr}_{\g_{>0}}(\ad v)=0$
for an element $v$ in the centre $\mf{z}(\g^\natural)$ of  $\g^\natural$ as well.
\end{proof}

We now discuss the associated varieties of affine $W$-algebras. The associated variety $X_{\W^k(\g,f) }$ of the universal $W$-algebra is isomorphic 
to the Slodowy slice $$\Slo_f=f+\g^{e}, $$
where $\g^{e}$ is the centraliser of $e$ in $\g$ (\cite{DeSole-Kac}).
There is a natural $\C^*$-action on $\Slo_f$ that contracts to $f$ (\cite{Slo80}),
and the associated variety $X_{\W_k(\g,f)}$ of the simple quotient is a $\C^*$-invariant subvariety of $\Slo_f$.

\begin{Th}[{\cite{Arakawa15a}}] 
\label{Th:Associated_variety-DS} 
\ 
\begin{enumerate}
\item For any $M\in \on{KL}_k$ we have $H_{DS,f}^i(M)=0$ for all $i\ne 0$.
In particular, 
the functor $\on{KL}_k\ra \W^k(\g,f)\on{-Mod}$,
$M\mapsto H_{DS,f}^0(M)$, is exact.
\item The  $\W^k(\g,f)$-module $H_{DS,f}^0(M)$ is  ordinary for any finitely generated object
$M\in \on{KL}_k$.
\item
For any quotient $V$ of $V^k(\g)$, the vertex algebra
$H_{DS,f}^{0}(V)$ is a quotient of $\W^k(\g,f)$ provided 
that it is nonzero, and we have 
$$X_{H_{DS,f}^{0}(V)}= X_V \cap \Slo_f,$$
which is a $\C^*$-invariant subvariety of $\Slo_f$ (\cite{Arakawa15a}). 
In particular,
\begin{enumerate}
\item $ H_{DS,f}^0(L_k(\g)) \not=0$ 
if and only if $f \in X_{L_k(\g)}$;
\item 
If $G.f\subset X_{L_k(\g)}\subset \mc{N}$,
then $H_{DS,f}^{0}(V)$ is quasi-lisse 
and so is $\W_k(\g,f)$;
\item 
If $ X_{L_k(\g)}=\overline{G.f}$,
then $H_{DS,f}^{0}(V)$ is lisse 
and so is $\W_k(\g,f)$.
\end{enumerate}

\end{enumerate}
\end{Th}
Regarding the relationship between the simple $W$-algebra $\W_k(\g,f)$ and the reduction $H_{DS,f}^0(L_k(\g))$ of which it is a quotient, there is the following natural conjecture.
\begin{Conj}[{\cite{KacRoaWak03,KacWak08}}]
\label{Conj:isom}
$H_{DS,f}^0(L_k(\g))$ is either zero
or isomorphic to $\W_k(\g,f)$.
\end{Conj}

This has been proved in many cases. Indeed we have:
\begin{Th}[\cite{AEkeren19}]
\label{Th:simplicity}
Conjecture \ref{Conj:isom} holds
if $k$ is an admissible level, $f\in X_{L_k(\g)}=\overline{\O}_k$,
and
 $f$ admits an even good grading. 
 \end{Th}
Let $k$ be an admissible level. By Theorems \ref{Th:admissible-orbits} and \ref{Th:Associated_variety-DS},  $\W_k(\g,f)$ is lisse if $f\in {\O}_k$, and is quasi-lisse if $f\in \overline{\O}_k$. Moreover, 
 under the hypotheses of Theorem \ref{Th:simplicity},
 the associated variety 
of $ \W_k(\g,f)$ is equal to the nilpotent 
Slodowy slice
\begin{align}
 \Slo_{\O_k,f}:=\Slo_f\cap \overline{\O}_k.
\end{align}

\begin{Conj}[\cite{FKW92,KacWak08,Arakawa15a}]
\label{Conj:rationality}
$\W_k(\g,f)$ is rational if $k$ is admissible and $f\in {\O}_k$. 
\end{Conj}
Conjecture  \ref{Conj:rationality}
has been settled positively by different methods in a number of cases, namely 
for $f$ principal \cite{A2012Dec},
for $\g$ of type $A$ \cite{AEkeren19},
for $\g$ of type $ADE$ and $f$ subregular \cite{AEkeren19},
for $\g$ of type $B_2$ and $f$ subregular \cite{JustineFasquel},
for $\g$ of type $B_n$ and $f$ subregular \cite{CreutzigLinshaw},
for $\g$ of type $C_n$ and $f$ minimal \cite{CreutzigLinshaw}. More recently, a general proof has been announced in \cite{McRae}.
    
By \cite{De-Kac06}, the associative algebra $\on{Zhu}(\W^k(\g,f))$ is naturally isomorphic to the finite $W$-algebra \cite{Pre02} $U(\g,f)$ associated with $(\g,f)$. More generally, we have the following assertion. 
\begin{Th}[\cite{A2012Dec}]\label{Th:Zhu-of-W}
There exists an isomorphism
$$\on{Zhu}(H_{DS,f}^{0}(L_k(\g))\cong H_{DS,f}^{0}(\on{Zhu}(L_{k}(\g)),$$
where on the right-hand-side $H_{DS,f}^{0}(?)$ is the finite-dimensional
analogue of the Drinfeld-Sokolov reduction functor that is denoted by 
$M\mapsto M_{\dagger}$ in \cite{Los07}.
\end{Th}

By \cite{Gin08,Los11}, see also \cite[Section 2]{A2012Dec}, we have the following equivalence
\begin{align}\label{eq:non-vanishing}
H_{DS,f}^{0}(U(\g)/I)\ne 0\iff  G.f\subset \on{Var}(I).
\end{align}
From this, together with Theorems~\ref{The:Zhu-of-admissible} and \ref{Th:Zhu-of-W}, we obtain:
\begin{Co}\label{Co:Zhu-of-reduction}
Let $k$ be  an admissible number.
We have
$$\on{Zhu}(H_{DS,f}^{0}(L_k(\g))\cong
 \prod H_{DS,f}^{0}(U(\g)/J_{\bar \lam}),$$
where the product is taken over 
$[\lam]\in [\on{Adm}^k]$ such that $G.f\subset \on{Var}(J_{\bar \lam})$.
\end{Co}
The image $[\omega]$
of the conformal
vector of $\W^k(\g,f)$ in $\on{Zhu}(H_{DS,f}^{0}(L_k(\g))$
acts on $H_{DS,f}^{0}(U(\g)/J_{\bar \lam})$ as the constant multiplication
by 
\begin{align}
h_{\lam}=\frac{|\lam+\rho|^2-|\rho|^2}{2(k+h^{\vee}_\g)}-\frac{k+h^{\vee}_\g}{2}|x_{\Gamma}^0|^2+(x_\Gamma^0,\rho).
\label{eq:conformal-dim-of-W-module}
\end{align}
In particular,  a simple module of  $H_{DS,f}^{0}(L_k(\g))$ corresponding to a simple module of $H_{DS,f}^{0}(U(\g)/J_{\bar \lam})$
via Corollary \ref{Co:Zhu-of-reduction}
has the conformal dimension $h_{\lam}$.

For a $V^k(\g)$-module $M$, we consider the character of its reduction, that is we set
\begin{align*}
\chi_{H^\bullet_{DS, f} (M)}(\tau,z)=\sum_{i\in \Z}(-1)^i{\rm tr}_{H^i_{DS, f}( M)} (e^{2\pi \imag z}q^{L_0 - c/24}),
\quad (\tau,z)\in
\mathbb{H}\times \h^f,
\end{align*}
when it is well-defined. Here
\[
L_0 = L_0^{\text{Sug}} - (x_\Gamma^0)_0 + L_0^{\text{ch}} + L_0^{\text{ne}}.
\]
Note that $\chi_{H^\bullet_{DS, f} (L(\lam))}(\tau,z)$ is well-defined for finitely generated objects $M$ of $\on{KL}_k$, by Theorem \ref{Th:Associated_variety-DS}, 
and we have
\[
\chi_{H^\bullet_{DS, f} (M)}(\tau,z) = {\rm tr}_{H^0_{DS, f}( M)} (e^{2\pi \imag z}q^{L_0 - c/24}).
\]

\begin{Pro}\label{Pro:reduction-of-assymptotics}
Let $M\in \mc{O}_k$ be a $V^k(\g)$-module whose character has asymptotic behaviour 
$$\chi_M(\tau,-\tau z,0)\sim 
\A_M(z) (-{\imag}  \tau)^{\Weight_M}e^{\frac{\pi {\imag} }{12 \tau}\G_M},$$
and suppose that $\chi_{H^\bullet_{DS, f} (M)}(\tau,z)$ is well-defined.
Then 
\begin{align*}
\chi_{H^\bullet_{DS, f} (M)}(\tau,-\tau z)\sim 
\A_{H^\bullet_{DS, f} (M)}(z)(-{\imag}  \tau)^{\Weight_M}e^{\frac{\pi {\imag} }{12 \tau}(\G_M-\dim \mathbb{O}_f)},
\end{align*}
where 
\begin{align*}
\A_{H^\bullet_{DS, f} (M)}(z):=\frac{
\prod\limits_{\alpha \in \Delta_+} 2 \sin{\pi (\alpha | z+x^0_\Gamma)}}{
\prod\limits_{\alpha \in \Delta_{\Gamma, +}^{0} \cup \Delta^{1/2}_{\Gamma, +}} 2 \sin \pi(\alpha | z+x^0_\Gamma)
}\A_M (z+x_{\Gamma}^0).
\end{align*}

\end{Pro}
\begin{proof}
Let us write
\begin{align*}
\chi_M(\tau,z,t)=\frac{(A_{\widehat{\rho}}\chi_M)(\tau,z,t)}{A_{\widehat{\rho}}(\tau,z,t)}.
\end{align*}
The asymptotic behaviour
of $A_{\widehat{\rho}}(\tau,z,t)$  is well known \cite[Proposition 13.13]{Kac90} to be
\begin{align}\label{eq:denom.asymptotic-W}
A_{\widehat{\rho}}(\tau, -\tau z, 0) \sim b(\rho, z)(-{\imag} \tau)^{-\ell/2} e^{-\frac{\pi{\imag} }{12\tau} \dim(\g)}
\end{align}
where
\begin{align*}
b(\rho, z) = \prod_{\alpha \in \Delta_+} 2 \sin{\pi (\alpha | z)}.
\end{align*}
Hence
\begin{align}
&\chi_M(\tau,-\tau z,0)\sim 
\A_M(z) (-{\imag}  \tau)^{\Weight_M}e^{\frac{\pi {\imag} }{12 \tau}\G_M}
\label{eq:aym-numerator}
\\
&\iff (A_{\widehat{\rho}}\chi_M)(\tau,-\tau z,0)\sim 
b(\rho,z)A(\rho,z) (-{\imag}  \tau)^{\Weight_M-\frac{\ell}{2}}
e^{\frac{\pi {\imag} }{12 \tau}(\G_M-\dim \g)}.
\nonumber
\end{align}
Now 
by Theorem \ref{Th:Associated_variety-DS}
and the Euler-Poincar\'{e} principle (see the discussion in \cite[Section 2]{KacWak08})
\begin{align*}
\ch_{H^0_{DS, f}(M)}(\tau, z) &= \lim_{\varepsilon \rightarrow 0} \text{(I)} \cdot \text{(II)}, \\
\text{where} \quad \text{(I)} &= {\rm tr}_M q^{L_0^{\text{Sug}} - (x_\Gamma^0)_0} e^{2\pi \imag  (z + \varepsilon x_\Gamma^0)_0} \\
\text{and} \quad \text{(II)} & = \on{str}_{F^{\text{ch}} \otimes F^{\text{ne}}} q^{L_0^{\text{ch}} + L_0^{\text{ne}}} e^{2\pi \imag ((z + \varepsilon x_\Gamma^0)_0^{\text{ch}} + (z + \varepsilon x_\Gamma^0)_0^{\text{ne}})}.
\end{align*}
Here $L_0^{\text{Sug}}$ denotes the zero mode of the Sugawara conformal vector in $V^k(\g)$ while $L_0^{\text{ch}}$ and $L_0^{\text{ne}}$ denote the conformal vectors in $F^{\text{ch}}$ and $F^{\text{ne}}$ given in {\cite[Section 2.2]{KacRoaWak03}}, or in {\cite[Section 2.1]{KacWak08}}. 
We thus have
\begin{align}\label{eq:char.H0. limit}
\chi_{H^0_{DS, f}(M)}(\tau, z) &= \lim_{\epsilon \rightarrow 0} \frac{(A_{\widehat{\rho}}\ch M)(\tau, z - \tau x_\Gamma^0 + \varepsilon x_\Gamma^0, 0)}{\psi(\tau, z - \tau x_\Gamma^0 + \varepsilon x_\Gamma^0, 0)},
\end{align}
in which
\begin{align*}
\psi(\tau, z, t) = {} & e^{2\pi \imag h^\vee t} \eta(\tau)^\ell \prod_{\alpha \in \Delta_{\Gamma, +}^{0} \cup \Delta^{1/2}_{\Gamma, +}} f(\tau, (\alpha | z)), \\
\text{where} \quad f(\tau, s) &= e^{\pi \imag \tau/6} e^{\pi \imag s} \prod_{n=1}^\infty \left( 1 - e^{2\pi \imag ((n-1)\tau-s)} \right) \left( 1 - e^{2\pi \imag (n\tau+s)} \right).
\end{align*}
We now describe the proof of \eqref{eq:char.H0. limit} briefly, referring the reader to \cite[Section 2.2]{KacWak08} for details. The numerator of \eqref{eq:char.H0. limit} comes directly from the numerator of the expression \eqref{eq:aff.char.frac} for $\chi_{L(\lam)}$. The modification $z \mapsto z - \tau x_\Gamma^0$ 
corresponds to the shift of Virasoro operator $L_0^{\text{Sug}} \mapsto L_0^{\text{Sug}} - (x_\Gamma^0)_0$. 
The Weyl denominator is $\eta(\tau)^\ell$ times a product of theta functions indexed by $\alpha \in \Delta_+$. At the level of characters the effect of the tensor product with $F^{\text{ch}}$ is to cancel those factors associated with $\alpha \in \Delta_{\Gamma}^j$ for $j > 0$. The tensor product with $F^{\text{ne}}$ reintroduces those factors associated with $\alpha \in \Delta_{\Gamma, +}^{1/2}$ (whose cardinality is half that of $\Delta_{\Gamma}^{1/2}$). Ultimately this yields the denominator $\psi(\tau, z - \tau x_\Gamma^0, t)$ of  \eqref{eq:char.H0. limit}.

The function $\psi$ satisfies the asymptotic
\begin{align}\label{eq:psi.asymptotic}
\psi(\tau, -\tau z, 0) \sim \prod_{\alpha \in \Delta_{\Gamma, +}^{0} \cup \Delta^{1/2}_{\Gamma, +}} 2 \sin \pi(\alpha | z)
(-{\imag} \tau)^{-\ell/2} e^{\pi {\imag} \dim(\g^f) / 12\tau} \end{align}
as $\tau \downarrow 0$. The asymptotic behaviour of the numerator $A_{\nu+\widehat{\rho}}(\tau, z - \tau x_\Gamma^0, t)$ was established in the proof of Proposition \ref{prop:general.affine.asymp.coprincipal}. 
The required asymptotic follows from
\eqref{eq:aym-numerator},
\eqref{eq:char.H0. limit}  and \eqref{eq:psi.asymptotic},
together with the fact that
\begin{align*}
\dim \overline{\mathbb{O}}_f=\dim \g-\dim \g^f.
\end{align*}
\end{proof}

\begin{Pro}
\label{Pro:asymptotic_data_H_DS}
Let $k =-h_\g^\vee +\dfrac{p}{q}$ be an admissible level
and let $f \in X_{L_k(\g)}$.
Then $H_{DS,f}^0(L(\lam))$
admits \asym\ 
for all  simple ordinary $L_k(\g)$-modules $L(\lam)$
with $\Weight_{H_{DS,f}^0(L(\lam))}=0$. 
Moreover:  
\begin{enumerate}
\item (\cite{KacWak08})
For  $\lam\in \on{Pr}^k_{\Z}$,
\begin{align*}
&\G_{H_{DS,f}^0(L(\lam))} =\G_{L_k(\g)}-\dim\overline{G.f} 
= \dim \g^{f} -  \dfrac{h_\g^\vee \dim \g}{pq},& \\
&\A_{H_{DS,f}^0(L(\lam))}
= {} 
\dfrac{1}{2^{\frac{|\Delta_\Gamma^{1/2}|}{2}}q^{|\Delta^0_{\Gamma,+}|}
{|P /(pq)Q^\vee|}^{\frac{1}{2}} } & \\
& \hspace{2.25cm} \times \prod\limits_{\alpha \in \Delta_{+}} 
2 \sin \dfrac{\pi (\lam+\rho | \alpha) }{p} 
\prod\limits_{\alpha \in \Delta_{+} \setminus \Delta^0_{\Gamma,+} } 
2 \sin \dfrac{\pi (x_\Gamma^{0} | \alpha) }{q}.& 
\end{align*}

\item For  $\lam\in \on{CoPr}^k_{\Z}$,
\begin{align*}
& \G_{H_{DS,f}^0(L(\lam))} =\G_{L_k(\g)}-\dim\overline{G.f} 
= \dim \g^{f} - \dfrac{r^\vee   h^\vee_{{}^L\!\g} \dim \g}{pq},  & \\
& \A_{H_{DS,f}^0(L(\lam))}
=  \dfrac{(r^\vee_\g)^{|\Delta_+^{\textup{short}} \cap \Delta_\Gamma^0|}}
{2^{\frac{|\Delta_\Gamma^{1/2}|}{2}} q^{|\Delta^0_{\Gamma,+}|}
{|P^\vee/(pq) Q |}^{\frac{1}{2}}} \\
& \hspace{2.25cm} \times \prod\limits_{\alpha \in \Delta_{+}} 
2 \sin \dfrac{\pi (\lam+\rho | \alpha^\vee) }{p}  
\prod\limits_{\alpha \in \Delta_{+} \setminus \Delta^0_{\Gamma,+} } 
2 \sin \dfrac{\pi (x_\Gamma^0 | \alpha^\vee) }{q} . & 
\end{align*}
\end{enumerate}
In particular,
$$\on{qdim}_{H_{DS,f}^0(L_k(\g))}H_{DS,f}^0(L(\lam))=\on{qdim}_{L_k(\g)}L(\lam)
$$
for any simple ordinary $L_k(\g)$-module $L(\lam)$.

\end{Pro}
\begin{proof}
Let $k$ be coprincipal
and let $\lam\in \on{CoPr}^k$.
By Proposition \ref{prop:general.affine.asymp.coprincipal}
and Proposition~\ref{Pro:reduction-of-assymptotics}, we have
\begin{align}
\chi_{H^0_{DS, f}(L(\lam))}({\imag}T, -{\imag}Tz, 0) \sim \frac{\varepsilon(\ov y)e^{-\pi \G_{L_k(\g)} / 12T}}{ |P^\vee/(pq)Q|^{1/2} } \cdot \frac{ \prod_{\alpha \in \Delta_+} 4 \sin\frac{\pi (\alpha^\vee | \nu)}{p} \sin \frac{\pi (\alpha^\vee | z+x_\Gamma^0-\beta)}{q} }{ \prod_{\alpha \in \Delta_{\Gamma, +}^{0} \cup \Delta^{1/2}_{\Gamma, +}} 2 \sin \pi(\alpha | z+x_\Gamma^0),}
\end{align}
where $\nu$, $\bar y$ are  as in Proposition \ref{prop:general.affine.asymp.coprincipal}.

We are interested in the case $\lam\in\on{CoPr}^k_{\Z}$ which corresponds to $\ov y = 1$, $\beta = 0$ and $\nu=\lam+\widehat{\rho}$. We now simplify the products appearing in the asymptotic above to:
\begin{align*}
& \prod_{\alpha \in \Delta_+} 2 \sin\frac{\pi (\alpha^\vee | \lam+\rho)}{p}
\prod_{\alpha \in \Delta^{1/2}_{\Gamma, +}} \frac{1}{2 \sin \pi(\alpha | z+x_\Gamma^0)} \\
&\phantom{........} \prod_{\alpha \in \Delta_{\Gamma, +}^{0}} \frac{2 \sin \frac{\pi (\alpha^\vee | z+x_\Gamma^0)}{q}}{2 \sin \pi(\alpha | z+x_\Gamma^0)}
\prod_{\alpha \in \Delta_{\Gamma, +}^{\ge 1/2}} 2 \sin \frac{\pi (\alpha^\vee | z+x_\Gamma^0)}{q}.
\end{align*}
The second product here becomes simply $2^{-|\Delta_{\Gamma}^{1/2}|/2}$. In the third product the terms $x_\Gamma^0$ are irrelevant since $\alpha \in \Delta_{\Gamma, +}^0$. In the limit $z \rightarrow 0$ this product reduces, by l'Hopital's rule, to
\begin{align*}
\prod_{\alpha \in \Delta_{\Gamma, +}^{0}} \frac{|\alpha^\vee|/q}{|\alpha|}
= \frac{(r^\vee)^{|\Delta_{\Gamma, +}^0 \cap \Delta^{\text{short}}|}}{q^{|\Delta_{\Gamma, +}^0|}}.
\end{align*}
In the fourth product the limit obtains by simply putting $z=0$. Thus we have obtained the stated formula.
\end{proof}

\begin{Th}\label{Th:min-conf-dim}
Let $k$ be admissible,
and let $f\in \g$  be a nilpotent element that admits an even good grading.
\begin{enumerate}
\item Let $f\in \overline{\mathbb{O}}_k$,
so that $H_{DS,f}^0(L_k(\g))\cong \W_k(\g,f)$ (see Theorem \ref{Th:simplicity}).
Then
$h_{\lam_{\ooh}}$ is the minimal conformal dimension among simple positive energy representations
of $\W_k(\g,f)$ (see \eqref{eq:def-lam0} and \eqref{eq:conformal-dim-of-W-module}),
and we have
\begin{align*}
\G_{\W_k(\g,f)}=c_{\W_k(\g,f)}-24 h_{\lam_{\ooh}}.
\end{align*}
 \item Suppose further that  $f\in \mathbb{O}_k$, so that $\W_k(\g,f)$ is lisse.
Then there exists a unique 
simple $\W_k(\g,f)$-module 
that has 
the  minimal conformal dimension $h_{\lam_{\ooh}}$.
\end{enumerate}
\end{Th}
\begin{proof}
First we treat (1). By Theorem \ref{Th:short-paper},
$\on{Var}(J_{\lam_{\ooh}})=\overline{\mathbb{O}}_k$,
which contains the orbit $G.f$ by the assumption.
Hence,
$H_{DS,f}^{0}(U(\g)/J_{\lam_{\ooh}})$ is nonzero by \eqref{eq:non-vanishing}.
It follows from Corollary  \ref{Co:Zhu-of-reduction}
that 
there exists a simple $\W_k(\g,f)$-module, corresponding to a simple $H_{DS,f}^{0}(U(\g)/J_{\lam_{\ooh}})$-module,
which has the minimal conformal dimension $h_{\lam_{\ooh}}$ among simple positive energy representations
of $\W_k(\g,f)$ 
by Proposition \ref{Pro:minimal-conformal-dim} and \eqref{eq:conformal-dim-of-W-module}. Next we have
\begin{align*}
&\G_{\W_k(\g,f)}=\G_{L_k(\g)}-\dim G.f\quad\text{(Proposition \ref{Pro:asymptotic_data_H_DS})}\\
&=c_{V^k(\g)}-24 (h_{\lam_{\ooh}}+\frac{k+h^{\vee}_\g}{2}|x^0_{\Gamma}|^2-(x^0_{\Gamma}|\rho))
-\dim G.f \quad\text{(Remark \ref{Rem:remar-on-min-con-wt} and \eqref{eq:conformal-dim-of-W-module})},
\end{align*}
where $\Gamma$ is a good even grading.
Hence the assertion follows from  \eqref{eq:W.alg.c.formula} and the fact that $\dim\g_{1/2}=0$, which is true by hypothesis.

Now we turn to (2). By \cite{AEkeren19}, each factor $H_{DS,f}^{0}(U(\g)/J_{\bar \lam})$ in Corollary \ref{Co:Zhu-of-reduction} is a simple algebra, which has a unique simple module. Hence the assertion follows from Remark \ref{Rem:remar-on-min-con-wt}.
\end{proof}

In view of  Proposition \ref{Pro:effective.c.c=growth},
Theorem \ref{Th:min-conf-dim} gives a supporting evidence
for 
 Conjecture  \ref{Conj:rationality}.

\section{Collapsing levels for $W$-algebras}
\label{sec:collapsing}
In general, if $\mf{s}$ is a semisimple Lie algebra, we write $\kappa_{\mf{s}}$ 
for the Killing form of $\mf{s}$. For now, we assume that $\mf{s}$ is simple. 
We denote by $h_{\mf{s}}^\vee$ its dual Coxeter number 
and by $(~|~)_{\mf{s}}$ the normalised inner product $\dfrac{1}{2h_{\mf{s}}^\vee} \kappa_{\mf{s}}$ 
so that $(\theta | \theta)=2$ for $\theta$ the highest 
positive root of $\mf{s}$. 

As in Section~\ref{sec:Asymptotics_of_W_algebras}, we fix an $\sl_2$-triple $(e,h,f)$ of $\g$,  
with related notation. In particular, 
for $j \in \frac{1}{2}\Z$, recall that 
$$\g^{j} = \{x \in \g \colon [h,x] = 2 j x\},$$ and $\g^0$ is the centraliser of $h$ in $\g$. 
The centraliser $\g^\natural$ of the $\sl_2$-triple $(e,h,f)$ in $\g$ 
is given by 
\begin{align}
\g^\natural = \g^0 \cap \g^{f} = \g^{e}\cap \g^{f}. 
\end{align} 
The Lie algebra $\g^\natural$ is a reductive subalgebra of $\g$. 
The centraliser $G^\natural$ of the $\sl_2$-triple $(e,h,f)$ 
in $G$ is a maximal reductive subgroup of the centraliser of $f$ in $G$ 
whose Lie algebra is $\g^\natural$.
We can write 
$\g^\natural= \g_0^\natural \oplus [\g^\natural,\g^\natural]$, 
with $\g_0^\natural$ the centre of the reductive Lie algebra $\g^\natural$.    
Denoting by $\g^\natural_1,\ldots,\g^\natural_s$ the simple factors of 
$[\g^\natural,\g^\natural]$, we get 
$$\g^\natural=\bigoplus_{i=0}^s \g^\natural_i.$$

More generally, for $\g_\Gamma= \bigoplus_{j \in \frac{1}{2}\Z} 
\g_\Gamma^j$ a good grading for $f$, we set 
\begin{align}
\g_\Gamma^\natural:=\g_\Gamma^{0} \cap \g^f. 
\end{align}
We note that $\g_\Gamma^\natural$ is not a reductive Lie algebra 
in general.
In fact $\g^\natural$ is a maximal reductive subalgebra of $\g_\Gamma^\natural$. Indeed, by \cite[Lemma 1.2]{ElashviliKac}, $\g^\natural$ is a maximal reductive subalgebra 
of $\g^f$ while, on the other hand, since $\g^\natural$ is reductive and  
contained in $\bigoplus_{j \ge 0} \g_\Gamma^j$, we have $\g^\natural \subset  \g_\Gamma^0$.

Let $i \in \{0,\ldots,s\}$. 
For any element $x \in \g_i^\natural \subset \g_\Gamma^\natural$ 
the adjoint action $\ad x$ restricts to an endomorphism of $\g_\Gamma^{j}$ which we denote $\rho_x^{j}$, for any $j$.
Setting $\kappa_{\g_\Gamma^{j}}(x,x)={\rm tr}(\rho_x^{j} \circ \rho_x^{j} )$, 
one defines an invariant bilinear form on $\g_{i}^\natural$ 
by (see \cite[Theorem 2.1]{KacWak03} for the case where $\Gamma$ is the Dynkin grading):
\begin{align*}
\phi_{\Gamma,i}^\natural(x,x) := k (x|x)_\g  + \dfrac{1}{2} (\kappa_\g(x,x)
- \kappa_{\g_\Gamma^0}(x,x)-\kappa_{\g_\Gamma^{1/2}}(x,x)).
\end{align*}
Thus, for $i \not=0$, there exists a scalar $k_{\Gamma,i}^\natural$ 
such that 
\begin{align}
\phi_{\Gamma,i}^\natural = k_{\Gamma,i}^\natural  (~|~)_{i} 
\quad \text{where}\quad (~|~)_i  :=(~|~)_{\g_i^\natural}. 
 \label{eq:k_natural}
\end{align}

Note that $V^{\phi_{\Gamma,0}^\natural} (\g_0^\natural) \cong M(1)^{\otimes {\rm rank}\,
\phi_{\Gamma,0}^\natural}$,  
where $M(1)$ is the Heisenberg vertex algebra of central charge 1.

Following \cite{Brundan-Goodwin}, we say that two good gradings $\Gamma,\Gamma'$ 
are {\em adjacent} if 
$$\g= \bigoplus_{i^-\le j \le i^+} \g_{\Gamma}^{i} \cap \g_{\Gamma'}^{j},$$
where $i^-$ denotes the largest half-integer strictly smaller than $i$ and 
$i^+$ denotes the smallest half-integer
strictly greater than $i$.

\begin{Lem} 
\label{Lem:maximal reductive subalgebra}
Let $\Gamma,\Gamma'$ be two good gradings for $f$. 
If $\Gamma,\Gamma'$ are adjacent, then 
$\phi_{\Gamma,i}^\natural =\phi_{\Gamma',i}^\natural$ 
for $i=0,\ldots,s$. 
\end{Lem}
\begin{proof}
According to \cite[Theorem 2 and Lemma 26]{Brundan-Goodwin}, 
there are Lagrangian subspaces $\mf{l}_{\Gamma}$ and 
$\mf{l}_{\Gamma'}$ 
of $\g_{\Gamma}^{1/2}$ and $\g_{\Gamma'}^{1/2}$, respectively, 
with respect to the non-degenerate bilinear form $ (f | [ \cdot ,\cdot ])$ 
such that 
\begin{align*}
\mf{l}_\Gamma \oplus \g_\Gamma^ {>1/2} = \mf{l}_{\Gamma'} \oplus \g_{\Gamma'}^ {>1/2}
\quad \text{ and }\quad 
\mf{l}_\Gamma^- \oplus \g_\Gamma^ {<-1/2} = \mf{l}_{\Gamma'}^- \oplus \g_{\Gamma'}^ {<-1/2},
\end{align*}
where $\mf{a}^-$ denotes 
the dual space to the subspace $\mf{a} \subset \g$ with respect to the Killing form 
of $\g$.
Then for any $x \in \g^\natural$, 
\begin{align*}
\kappa_\g(x,x)
- \kappa_{\g_{\Gamma}^0}(x,x)-\kappa_{\g_{\Gamma}^{1/2}}(x,x) 
& =2 \kappa_{\mf{l}_\Gamma \oplus \g_\Gamma^ {>1/2}}(x,x)  & \\
&  = \kappa_\g(x,x)
- \kappa_{\g_{\Gamma'}^0}(x,x)-\kappa_{\g_{\Gamma'}^{1/2}}(x,x). &
\end{align*}
Indeed, from the decompositions 
\begin{align*}
\g & = (\mf{l}_\Gamma^- \oplus \g_\Gamma^ {<-1/2}) \oplus (\mf{k}_{\Gamma}^- \oplus \g_{\Gamma}^0 \oplus 
\mf{k}_{\Gamma}) \oplus (\mf{l}_\Gamma \oplus \g_\Gamma^ {>1/2}),& \\
\g & = (\mf{l}_{\Gamma'}^- \oplus \g_{\Gamma'}^ {<-1/2}) \oplus (\mf{k}_{\Gamma'}^- \oplus \g_{\Gamma'}^0 \oplus 
\mf{k}_{\Gamma'}) \oplus (\mf{l}_{\Gamma'} \oplus \g_{\Gamma'}^ {>1/2}),&
\end{align*}
where $\mf{k}_\Gamma$ (resp.~$\mf{k}_{\Gamma'}$) 
is a Lagrangian complement in $\g_\Gamma^{1/2}$ (resp.~$\g_{\Gamma'}^{1/2}$)
to $\l_\Gamma$ (resp.~$\mf{l}_{\Gamma'}$), 
we deduce that 
\begin{align*}
\kappa_{\g_{\Gamma}^{1/2}}(x,x) = 
\kappa_{\mf{k}_\Gamma \oplus \mf{k}_\Gamma^{-}}(x,x) 
\quad \text{ and } \quad 
\kappa_{\g_{\Gamma'}^{1/2}}(x,x) = 
\kappa_{\mf{k}_{\Gamma'} \oplus \mf{k}_{\Gamma'}^{-}}(x,x).
\end{align*}
This completes the proof.
\end{proof}

Denoting by $\Gamma_D$  the Dynkin grading, 
we set 
$$\phi_i^\natural:= \phi_{\Gamma_D,i}^\natural, 
\quad \text{ and }\quad 
k_i^\natural:= k_{\Gamma_D,i}^\natural \quad \text{ for } \quad i \not=0.
$$  

\begin{Rem}
\label{Rem:maximal reductive subalgebra}
In type $A$, according to \cite[Lemma 26 and Section~6]{Brundan-Goodwin}, there is a chain $\Gamma_0,\ldots,\Gamma_t$ 
of good gradings 
 for $f$ such that $\Gamma_0=\Gamma_D$, $\Gamma_t=\Gamma$, and 
 for any $i \in \{1,\ldots,t\}$, $\Gamma_{i-1}$ and $\Gamma_i$ are adjacent, 
and one of these good gradings is even. 
Hence, in type $A$, we are free to use an even good grading $\Gamma$ 
to compute $k_i^\natural$, 
which is convenient (see the proof of Lemma \ref{Lem:k_natural-type-A}). 
\end{Rem}

Set 
\begin{align}
V^{k^\natural}(\g^\natural) := V^{\phi_0^\natural} (\g_0^\natural) \otimes 
\bigotimes_{i=1}^s V^{k_i^\natural} (\g_i^\natural).
\end{align}
By \cite[Theorem 2.1]{KacWak03}, there exists an embedding 
\begin{align}
\iota \colon V^{k^\natural}(\g^\natural)
\longhookrightarrow \W^k(\g,f)
\end{align}
of vertex algebras.  
We have
\begin{align}
\iota(\omega_{V^{k^\natural}(\g^\natural)})\in \W^k(\g,f)_2,\quad  (\omega_{\W^k(\g,f)})_{(2)}\iota(\omega_{V^{k^\natural}(\g^\natural)})=0.
\end{align}

We denote by $\mc{V}^k(\g^\natural)$ the image in $\W^k(\g,f)$ of the embedding $\iota$, 
and by $\mc{V}_k(\g^\natural)$ the image of $\mc{V}^k(\g^\natural)$ 
by the canonical projection $\pi \colon \W^k(\g,f) \twoheadrightarrow \W_k(\g,f)$. 
\begin{Def}[\cite{AdaKacMos18}]  
\label{Def:collapsing}
If $ \W_k(\g,f) \cong \mc{V}_k(\g^\natural)$, we say that the level 
$k$ is {\em collapsing}. 
\end{Def}
\begin{Lem}
The level $k$ is collapsing if and only if 
$$ \W_k(\g,f) \cong L_{k^\natural}(\g^\natural),$$
where $L_{k^\natural}(\g^\natural)$ stands for 
$\displaystyle{\bigotimes_{i=0}^s L_{k_i^\natural} (\g_i^\natural)}$.  
Equivalently, $k$ is collapsing for $\W^k(\g,f)$
if and only
if there exists a surjective vertex algebra homomorphism
\begin{align*}
\W^k(\g,f)\twoheadrightarrow  L_{k^\natural}(\g^\natural).
\end{align*}
\end{Lem}
For example, if $ \W_k(\g,f) \cong \C$, then $k$ is collapsing. 
\begin{proof}
If $\W_k(\g,f) \cong \mc{V}_k(\g^\natural)$, 
then $\W_k(\g,f)$ is isomorphic to the quotient of $\bigotimes_{i=0}^s V^{k_i^\natural} (\g_i)$ by 
the kernel of the composition 
$\pi \circ \iota$. Since $\W_k(\g,f)$ is simple we deduce 
that this quotient is isomorphic to ${\bigotimes_{i=0}^s L_{k_i^\natural} (\g_i)}=L_{k^\natural}(\g^\natural)$. 
Conversely, if $ \W_k(\g,f) \cong \bigotimes_{i=0}^s L_{k_i^\natural} (\g_i)$, 
then $\pi \circ \iota$ factorises through $\bigotimes_{i=0}^s L_{k_i^\natural} (\g_i)$, 
and so $ \W_k(\g,f)$ is isomorphic to the image 
of this induced map, so $ \W_k(\g,f)\cong \mc{V}_k(\g^\natural)$. 
\end{proof}

We close this section with a lemma which is used above in the proof of Proposition \ref{Pro:Dynkin-self-dual}.
\begin{Lem}
\label{lem:centraliser_is_reductive_subalgebra}
The centre $\mf{z}(\g^\natural)$ of the reductive Lie algebra $\g^\natural$  consists of semisimple elements 
of $\g$. 
Moreover, for any $x \in \mf{z}(\g^\natural)$, we have 
${\rm tr}_{\g^{f}}(\ad x) =0$ and ${\rm tr}_{\g_{>0}}(\ad x) =0$, 
where $\ad_x$ stands for the endomorphism of $\g^{f}$ 
(resp.~$\g_{>0}$) induced from the adjoint action of $x$ 
acting on $\g^{f}$ 
(resp.~$\g_{>0}$).
\end{Lem}

\begin{proof}
As previously mentioned, $\g^\natural$ is a maximal reductive subalgebra 
of $\g^{f}$ (\cite[Lemma 1.2]{ElashviliKac}). 
Let $\mf{t}$ be a maximal torus of $\g^{0}=\g^{h}$ 
containing a maximal torus of $\g^\natural$ and 
set 
$$\mf{t}^{\natural}:=\mf{t} \cap \g^\natural= 
 \mf{t}\cap \g^{f}.$$ 
 We intend to show that $\mf{z}(\g^\natural)$ 
is contained in $\mf{t}^{\natural}$. 
 
 We use the decomposition in $\mf{t}^{\natural}$-weight spaces of $\g^{\natural}$
following \cite[Section 2]{Brundan-Goodwin}. 
For $\alpha \in (\mf{t}^{\natural})^*$ and $n \ge 0$, let $L(\alpha,n)$ 
denote the irreducible $\mf{t}^{\natural}\oplus \mf{s}$-modules 
of dimension $(n+1)$ on which $\mf{t}^{\natural}$ acts by the weight $\alpha$, 
where $\mf{s}\cong\sl_2$ is the Lie algebra generated by $e,h,f$. 
Since each $L(\alpha,n)$ contains a nonzero vector annihilated by 
$f$, the set of weights of $\mf{t}^{\natural}$ on $\g$ 
is also the set of weights of $\mf{t}^\natural$ on $\g^{f}$.
Let $\Phi_f \subset (\mf{t}^{\natural})^*$ 
be the set of all nonzero weights of $\mf{t}^{\natural}$ on 
$\g^f$. 
We have the following decomposition 
\begin{align}
\label{eq:decomposition_gf}
\g^{f} = \g^{\mf{t}^{\natural}}\cap \g^f \oplus 
\bigoplus\limits_{\alpha \in \Phi_f \atop i\le 0} 
\g^f (\alpha,i),
\end{align}
where 
$\g^{\mf{t}^{\natural}}$ is the centraliser of 
$\mf{t}^{\natural}$ in $\g$ (it is a Levi subalgebra of $\g$) 
and 
$$\g^f (\alpha,i)= \{x \in \g^f \colon [h,x]=i x 
\text{ and } [t,x] = \alpha(t)x \text{ for all } t \in  \mf{t}^{\natural}\}.$$
This decomposition is
compatible with the decomposition 
$\g^f = \g^\natural \oplus \g^{f}_{<0}$, 
and we have 
$$\g^\natural  = \mf{t}^{\natural} \oplus 
\bigoplus\limits_{\alpha \in \Phi_f^\circ} 
\g^f (\alpha,0),$$
where $\Phi_f^\circ$ denotes the set of all  
nonzero element of $\Phi_f$. 

Let now $x \in \mf{z}(\g^\natural)$ that we write 
$x=x_{0}+\sum_{\alpha \in \Phi_f^\circ} x_\alpha$ 
relatively to the above decomposition, with 
$x_0 \in \mf{t}^{\natural} $ and $x_\alpha \in \g^f (\alpha,0)$ 
for $\alpha \in \Phi_f^\circ$. 
Let $\alpha' \in \Phi_f^\circ$ and pick $t' \in \mf{t}^{\natural}$ 
such that $\alpha'(t')\not=0$. 
From the equalities 
\begin{align*}
0 &= [y,x] = [y,x_0]+\sum_{\alpha \in \Phi_f^\circ} [y,x_\alpha] 
= 0 + \sum_{\alpha \in \Phi_f^\circ \atop 
\alpha \not=\alpha'}  \alpha(t) x_\alpha 
+ \alpha'(t') x_{\alpha'},
\end{align*}
we deduce that $x_{\alpha'}=0$. 
Since this is true for each $\alpha \in \Phi_f^\circ$, 
we get that $x=x_0 \in  \mf{t}^{\natural}$. 
In particular, $x$ is a semisimple elements of $\g$. 

To show the second assertion, first note that we have the following 
decomposition 
\begin{align}
\label{eq:decomposition_g}
\g= \bigoplus_{\alpha \in \Phi_f} 
\bigoplus_{n \ge 0} m(\alpha,n) L(\alpha,n),
\end{align} 
where $m(\alpha,n) \not=0$ for all $\alpha \in \Phi_f$. 
Since $\Phi_f$ is a restricted root system in the sense 
of  \cite[Section 2]{Brundan-Goodwin}, 
for $\alpha \in (\mf{t}^\natural)^*$, $\alpha \in \Phi_f$ 
 implies that $-\alpha \in \Phi^f$ 
 and the corresponding root vector spaces 
 have the same dimension. 
So the second assertion 
follows from the decompositions 
\eqref{eq:decomposition_gf} and  \eqref{eq:decomposition_g}. 
\end{proof}

\section{The main strategy}
\label{sec:strategy} 
We describe in this section our general strategy to find new (admissible) collapsing levels  
using Theorem \ref{Th:main}. 
The main difficulty is to find potential candidates for $k$ and $f$. 
We explain below our strategy to achieve this. 

\subsection{Associated variety}
\label{sub:associated_variety}
If $k$ is collapsing for $\W^k(\g,f)$ then certainly
\begin{align}
X_{\W_k(\g,f)} \cong X_{L_{k^\natural}(\g^\natural)}
\end{align}
as Poisson varieties. As we saw in Theorem \ref{Th:Associated_variety-DS} the vertex algebra $\W_k(\g,f)$ is quasi-lisse whenever $f \in \overline{\mathbb{O}}_k$. 
\begin{Lem}
\label{Lem:k_0-condition}
Suppose that $\W_k(\g,f)$ is quasi-lisse. If $k$ is collapsing then
$\phi_0^\natural$ is identically zero on $\g_0^\natural$. 
In particular, if $k$ is admissible and $f\in \overline{\mathbb{O}}_k$, then $k$ can only be collapsing if $\phi_0^\natural=0$. 
\end{Lem}
Our convention is that $\phi_0^\natural=0$ when $\g_0^\natural=\{0\}$.
\begin{proof}
The associated variety $X_{L_{k^\natural}(\g^\natural)}
$ is a subvariety of  $(\g_0^\natural)^*\times(\g_1^\natural)^*\times \dots \times (\g_s^\natural)^*$
and the symplectic leaves of $(\g_0^\natural)^*\times(\g_1^\natural)^*\times \dots \times (\g_s^\natural)^*$
are the coadjoint orbits of $G_0^\natural\times \cdots \times G_s^\natural$,
where $G_i^\natural$ is the adjoint group of $\g_i^\natural$. 
Recall that $V^{\phi_0^{\natural}}(\g_0^\natural)$ is a Heisenberg vertex algebra 
of rank $\dim \g_0^\natural$. The associated variety 
of its simple quotient is 
$\C^{ {\rm rank}\,\phi_0^\natural}$, provided that $\phi_0^{\natural}\not=0$. 
Hence, $X_{L_{k^\natural}(\g^\natural)}$ has finitely many symplectic leaves if and only if 
it is contained in
$\mc{N}_{\g_1}\times \dots \times \mc{N}_{\g_s}$,
where  $\mc{N}_{\g_i}$ is the nilpotent cone of $\g_i^\natural$.
In particular, we must have 
$X_{L_{\phi_0^{\natural}}(\g_0^\natural)}=\{0\}$. 
This happens if and only if $\phi_0^{\natural}=0$.
\end{proof}
We are thus led to consider levels $k$ for which $\phi_0^\natural = 0$ (further conditions on $k$ will appear later). Now, for each $i>0$, the bilinear form $\phi^\natural_i$ equals the standard form on the simple component $\g_i^\natural$, times a well-defined level $k_i^\natural$, as per \eqref{eq:k_natural}. These levels $k_i^\natural$ can then be expressed as functions of the level $k$, as is done below in Tables~\ref{Tab:A_Slodowy}, \ref{Tab:C_Slodowy} and \ref{Tab:BD_Slodowy} for classical types, and Tables~\ref{Tab:Data-G2}--\ref{Tab:Data-E8} for exceptional types. In the tables we also list $k_0^\natural$ for those cases in which $\g_0^\natural \neq 0$, and since there is no standard normalised form on $\g_0^\natural$ we now explain what we mean by this. If the centre $\g_0^\natural$ is one dimensional, then we choose nonzero $x\in \g_0^\natural$ and set $k_{0}^\natural = \phi_{0}^\natural(x,x)$. Then $k_0^\natural$ is well-defined, as a function of $k$, up to multiplication by a nonzero scalar. This is sufficient for our purposes since, by Lemma \ref{Lem:k_0-condition}, we require $k_0^\natural=0$. For the vast majority of pairs $(\g, f)$ relevant for us, we have $\dim(\g_0^\natural) \leq 1$ in fact. Only two exceptions occur (one in type $E_6$ and one in type $E_7$), and in both these cases $\g^\natural=\g_0^\natural\cong \C^2$. These cases are easily handled separately, and in the tables we abuse notation by also writing $k_{0}^\natural$ for the value of $\phi_{0}^\natural(x,x)$ for some nonzero $x\in \g_0^\natural$. 

We now call $k^\natural$ {\em admissible} if $k_0^\natural=0$ and $k_i^\natural$ is admissible for $\g_i^\natural$, for each $i=1,\ldots,s$.  
\begin{Rem}
It may happen that $k$ is admissible while $k^\natural$ 
is not, even if $\g_0^\natural=0$. 
For example, for $\g=G_2$, $f$ in the nilpotent orbit labelled $\tilde{A}_1$ (of dimension 8) 
and $k = -4+4/7$ which is a principal admissible 
level for $G_2$, 
then $k^{\natural} = k +3/2=-2+1/14$ is not admissible for $\g^\natural \cong \sl_2$ 
(see Table~\ref{Tab:Data-G2}). 
However, under additional conditions on $f$ and $k$, we expect that $k^\natural$ is admissible if $k$ is admissible. See Conjecture~\ref{Conj:k_natural_adm_conjecture} below. 
\end{Rem}
Recall 
\begin{align*}
\Slo_{\O_k,f}=\Slo_f\cap \overline{\O}_k,
\end{align*}
the nilpotent Slodowy slice. Now we assume $k$ and $k^\natural$ are both admissible, where $f\in \overline{\mathbb{O}}_k$, and we assume that $\W_k(\g,f)=H_{DS,f}^0(L_k(\g))$ (verified if $f$ has a good even grading, and in general if Conjecture \ref{Conj:isom} holds). Then according to Theorem \ref{Th:admissible-orbits}
and Theorem~\ref{Th:Associated_variety-DS},
\begin{align}\label{eq:slo.equals.nilp}
\Slo_{\O_k,f}= \overline{\O}_{k_1^\natural}
\times \cdots \times  \overline{\O}_{k_s^\natural},
\end{align}
where $\O_{k_1^\natural}, \ldots,\O_{k_s^\natural}$ are nilpotent orbits 
in $\g_1^\natural,\ldots,\g_s^\natural$, respectively. 
This motivates the following definition:

\begin{Def}
\label{def:collapsing_Slodowy_slices}
We say that a nilpotent Slodowy slice $\Slo_{\O,f}$ is {\em collapsing} 
if $\Slo_{\O,f}$ is isomorphic 
to a product of nilpotent orbit closures in $\g^\natural$. 
\end{Def}

Based on the above analysis, we consider 
pairs $(\O,\O')$ such that 
\begin{enumerate}
\item $\overline{\O}=\overline{\O}_k$ is the associated 
variety of some simple admissible affine vertex algebra $L_k(\g)$, 
\item $\O' \subset \overline{\O}$, 
\item for $f \in \O'$, $\Slo_{\O,f}$ 
is a collapsing nilpotent Slodowy slice 
which is the associated 
variety $X$ of the simple admissible affine vertex algebra $L_{k^\natural}(\g^\natural)$.  (In particular, $X$ is a product of nilpotent orbit closures in $\g^{\natural}$.)
\end{enumerate}

The nilpotent orbit $\O_k$, for admissible $k= -h_\g^\vee+p/q$, 
is given by \cite[Tables 2--10]{Arakawa15a} 
and depends only on the denominator $q$: 
it is described in term of the corresponding partition 
of $n$ for $\g$ simple of classical type, that is, $\g=\sl_n$, $\sp_n$ or $\so_n$, and in the Bala-Carter classification 
for $\g$ simple of exceptional type. 

Motivated by the determination of generic singularities of the nilpotent closures $\overline{\O}$ in simple 
Lie algebras $\g$ of classical types, Kraft and Procesi \cite{KraftProcesi79,KraftProcesi82} described the smooth equivalences 
of singularities between Slodowy slices $\Slo_{\O_1,f_1}$ and $\Slo_{\O_2,f_2}$, where $\Slo_{\O_2,f_2}$ is obtained 
from $\Slo_{\O_1,f_1}$ by the {\em row/column removal rule} 
(see Lemmas~\ref{lem:erasing_row_type-A} and~\ref{lem:erasing_row_type-BCD}). 
It turns out that these  smooth equivalences actually yield isomorphisms of varieties,  
see \cite[Proposition 7.3.2]{Li17}\footnote{This is also mentioned 
without detail in \cite[\S1.8.1]{FuJutLev17}.}. Thus the row/column removal rule, combined with \cite{Arakawa15a}, yields many pairs $(\O,\O')$ satisfying the conditions above in classical types. We consider all such pairs.

For the simple Lie algebras of exceptional types, the authors of 
\cite{FuJutLev17} determine the isomorphism type of most of the Slodowy slices $\Slo_{\O,f}$ for which $G.f$ is a minimal degeneration of $\O$ (i.e., those for which $G.f$ is a maximal orbit in $\overline{\O} \setminus \O$), and also for some of the Slodowy slices for which it is not a minimal degeneration. For the remaining cases, they obtain some weaker information. From their work, together with \cite{Arakawa15a}, one possibility would be to proceed by finding pairs $(\O,\O')$ satisfying the above conditions, as in the classical cases. 

In fact our work is more exhaustive in a sense, since in exceptional types we are able to directly exploit the description of the $k_i^\natural$ in Tables \ref{Tab:Data-G2}--\ref{Tab:Data-E8} and the central charge (see \S\ref{sub:central_charge} below) to detect all possible collapsing levels. In this way, we find many nontrivial isomorphisms between a Slodowy slice $\Slo_{\O,f}$ and a product of nilpotent 
orbit closures in $\g^\natural$. Many of these were observed in \cite{FuJutLev17} already, though others seem to be new. On the other hand, of course, many of  isomorphisms between Slodowy slices obtained in \cite{FuJutLev17} do not correspond to collapsing levels, and so do not appear in our work.

We remark that all the isomorphisms between nilpotent orbits and nilpotent Slodowy slices that we obtain from collapsing levels, are automatically isomorphisms of Poisson varieties. 

Roughly speaking, a collapsing Slodowy slice will yield a collapsing level if the asymptotic data of the vertex algebras corresponding to the two sides of \eqref{eq:slo.equals.nilp} can be shown to coincide. In fact asymptotic data is rather difficult to compute, so the next step in our strategy is to rule out spurious cases by comparing central charges. This is described in \S\ref{sub:central_charge} below.

We close this subsection with a discussion of admissibility of the levels $k_i^\natural$. 
\begin{Conj}
\label{Conj:k_natural_adm_conjecture}
If $k$ is admissible and if $f \in \overline{\O}_k$ 
is such that $\Slo_{\O_k,f}$ is collapsing, 
then $k^\natural$ is admissible, provided that $k_0^\natural=0$.
\end{Conj}

We will verify the conjecture in the classical cases in the cases where 
$\Slo_{\O_k,f}$ is collapsing and the isomorphism between 
$\Slo_{\O_k,f}$ and a product of nilpotent orbit closures in $\g^\natural$ 
is obtained from the row/column removal rule of Kraft-Procesi 
(see Lemma \ref{Lem:k_natural-type-A} and Lemma \ref{Lem:k_natural-type-BCD}). 
We feel that these cases exhaust all possible cases of 
collapsing nilpotent Slodowy slices 
so that, together with Lemma \ref{Lem:k_natural_adm_conjecture} below,  
it would complete the proof of the conjecture. 

\begin{Lem}
\label{Lem:k_natural_adm_conjecture}
Conjecture \ref{Conj:k_natural_adm_conjecture} is true 
if $\g$ is simple of exceptional type. 
\end{Lem}

\begin{proof}
Assume that $k$ is admissible.
It is known that $\Slo_{\O_k,f}$ is equidimensional 
of dimension $\dim \O_k -\dim G.f$ (\cite[Corollary 1.3.8]{Gin08}). 
Moreover, by the main result of \cite{AraMor16b}, 
$\Slo_{\O_k,f}$ is irreducible since $k$ is admissible. 

If $\Slo_{\O_k,f}$ is collapsing then in particular it is contained 
in the nilpotent cone $\mathcal{N}_{\g^\natural}$ of $\g^\natural$, and so $\dim \Slo_{\O,f} \le \dim \mathcal{N}_{\g^\natural} = \dim \g^\natural - {\rm rk}\,\g^\natural$.

The semisimple type of $\g^\natural$ and the values of the $k_i^\natural$'s 
are computed in Tables \ref{Tab:Data-G2}--\ref{Tab:Data-E8}. Hence, fixing $f$, we first consider nilpotent orbits $\O$ containing $\overline{G.f}$ such that 
$\dim \O - \dim G.f \le \dim \g^\natural - {\rm rk}\,\g^\natural$ 
and such that $\O=\O_k$ for some admissible level $k$ for $\g$. This heavily restricts the possibilities for the denominator $q$ of $k$. Then we can ask whether for such $q$, the corresponding level $k^\natural$ is admissible for $\g^\natural$. Let us illustrate with an example. 

Assume that $\g=F_4$ and that $f$ belongs to the nilpotent orbit 
labelled $B_3$ in the Bala-Carter classification (dimension 42). 
According to Table  \ref{Tab:Data-F4}, $\g^\natural \cong A_1$. Hence $\dim \Slo_{\O,f} \le \dim A_1 - {\rm rk}\,A_1 = 2$, so $42 \leq \dim \O_k \leq 44$ from which we see that $\O_k$ must be $B_3$, $C_3$ or $F_4(a_2)$. 

Recall that $k=-9+p/q$ with $(p,q)=1$ and $p\ge 9$ if $q$ is odd, 
and $p\ge 12$ if $q$ is even. By \cite{Arakawa15a}, if $\O_k = B_3$ then $q=8$, if $\O_k = F_4(a_2)$ then $q=7$ or $q=10$, and finally the orbit $C_3$ cannot occur as $\O_k$. Now we examine admissibility of 
$$k^\natural = 8k+60 =-2+\frac{2 (4 p - 5 q)}{q}$$ 
in each case.
\begin{itemize}
\item[$\ast$] If $q=7$, then $2 (4 p - 5 q) \ge 2 (4 \times 9 - 5\times 7) =2$, 
so $k^\natural$ is admissible.

\item[$\ast$] If $q=8$, then $2 (4 p - 5 q)/q = (4 p - 5 \times 8)/4$ 
and $4 p - 5 q \ge 4 \times 13 - 5\times 8 =12$ 
so $k^\natural$ is admissible.

\item[$\ast$] If $q=10$, then $2 (4 p - 5 q)/q = (4 p - 5 \times 10)/5$ 
and $4 p - 5 q \ge 4 \times 13 - 5\times 10 =2$ 
so $k^\natural$ is admissible.
\end{itemize}

In most cases considerations of orbit dimension, like those above, suffice to reach the desired conclusion. In some cases, however, a more detailed analysis of the Hasse diagrams of $\g$ and $\g^\natural$ is necessary (we use the diagrams of \cite{FuJutLev17}). 
This is the case, for example, for the minimal nilpotent orbits of $F_4$ and $E_7$. 
We explain the verification for these two cases. 

Let us first consider the minimal nilpotent orbit of $F_4$ (dimension $16$). 
In this case $\g^\natural \cong C_3$. Since the nilpotent cone of $C_3$ has dimension $18$ we have $16 \le \dim\O_k \le 34$. These inequalities are satisfied by $6$ nilpotent orbits, but of these only the orbits $A_1$ (the minimal orbit itself)  
and $A_2 + \widetilde{A}_1$ can occur as $\O_k$. 

If $\O_k = A_2+\tilde{A}_1$ then $q=4$ (here $k = -9 + p/q$ as above), and if $\O_k = A_1$ then $q=2$. Now we examine admissibility of the level
$$k^\natural = k+5/2 =-4+\frac{2 p - 5 q}{2 q},$$ 
in each case.
\begin{itemize}
\item[$\ast$] If $q=2$, then $(2p-5q)/(2q) = (p - 5)/2$  
and $p - 5 \ge 13 - 5 = 8 \ge 6=h_{C_3}$ 
so $k^\natural$ is admissible.

\item[$\ast$] If $q=4$, then $(2p - 5 q)/(2q) = (p - 10)/4$  
but $p \ge 13$ gives us only $p - 10 \ge 3$, not $p - 10 \ge 6$ as required for admissibility.
\end{itemize}
To exclude the second case, we observe instead that $\dim \O_{A_2+\tilde{A}_1}-\dim \O_{A_1}=18$. So if $\Slo_{\O_{A_2+\tilde{A}_1},A_1}$ were collapsing, 
then necessarily $\Slo_{\O_{A_2+\tilde{A}_1},A_1}$ would be isomorphic 
to the nilpotent cone $\mathcal{N}_{C_3}$ of $C_3$ since it is 
irreducible. 
In particular, it should contain a nilpotent $G^\natural$-orbit of dimension 16 
(corresponding to the subregular nilpotent orbit of $C_3\cong \sp_6$). 
Here $G^\natural$ denotes the centraliser in $F_4$ of the $\sl_2$-triple associated 
with $f$ in the minimal nilpotent orbit $A_1$ of $F_4$ whose Lie algebra is $\g^\natural$. 
But from the Hasse diagram of the $F_4$ we see that this is not possible. 
Indeed, 
the Hasse diagram of the $G^\natural$-action on $\Slo_{\O_{A_2+\tilde{A}_1},A_1}$ 
is just the \emph{interval} between the orbits labelled ${A_2+\tilde{A}_1}$ and $A_1$ 
in the Hasse diagram of $F_4$, and the dimension of the corresponding  
$G^\natural$-orbits would be $18, 14, 12, 6, 0$.  
Since it does not coincide with the Hasse diagram of $\mathcal{N}_{C_3}$, this case is ruled out.

Consider now the minimal nilpotent orbit $A_1$ of $E_7$ (dimension 34). 
In this case $\g^\natural \cong D_6$. 
Since the nilpotent cone of $D_6$ has dimension $60$ we have $34 \le \dim\O_k \le 94$. 
Only the orbits $4 A_1$ 
and $2 A_2 + A_1$ can satisfy this conditions and occur as $\O_k$. 

If $\O_k = 4 A_1$ then $q=2$ (here $k = -18 + p/q$), 
and if $\O_k = 2 A_2+A_1$ then $q=3$. Now we examine admissibility of the level
$$k^\natural = k+4 =-10+(p-4q)/q,$$ 
in each case.
\begin{itemize}
\item[$\ast$] If $q=2$, then $(p-4q)/q = (p - 8)/2$  
and $p - 8 \ge 19 - 8 = 11 \ge 19=h_{D_6}^\vee$  
so $k^\natural$ is admissible.

\item[$\ast$] If $q=3$, then $(p-4q)/q= (p - 12)/3$  
but $p \ge 19$ gives us only $p - 12 \ge 7$, not $p - 12 \ge 10$ as required for admissibility.
\end{itemize}
To exclude the second case, we observe instead that 
$\dim \O_{2 A_2+A_1}-\dim \O_{A_1}=56$. 
In $\mc{N}_{D_6}$ there are exactly two nilpotent orbits of dimension 56: 
there are associated with the partition $(7,5)$ and $(9,1^3)$. 
So if $\Slo_{\O_{A_2+\tilde{A}_1},A_1}$ were collapsing, 
then necessarily $\Slo_{\O_{A_2+\tilde{A}_1},A_1}$ would be isomorphic 
to one of the nilpotent orbit closures of $D_6$ of dimension 56 since it is irreducible. 
Both of them contain the nilpotent $G^\natural$-orbit of dimension 54, associated 
with the partition $(7,3,1^2)$ while, looking at the Hasse diagram 
of the $G^\natural$-action on $\Slo_{\O_{2 A_2+A_1},A_1}$ in $E_7$ 
we see that this is not possible.  

The rest of the verifications are left to the reader. 
\end{proof}

\subsection{Central charge} 
\label{sub:central_charge}
Let $k$ be an admissible level for $\g$. We recall (1) that the nilpotent orbit $X_{L_k(\g)} \subset \mathcal{N}_\g$ completely determines the denominator $q$ of the level $k=-h_\g^\vee+p/q$ (cf.~Theorem \ref{Th:admissible-orbits}), and (2) that the $k_i^\natural$, defined by \eqref{eq:k_natural}, are all polynomials of degree one in $k$ or, equivalently, in $p$. 

Assume that both $k$ and $k^\natural$ are admissible. Then both $L_{k^\natural}(\g^\natural)$ and $\W_k(\g,f)$ are conformal vertex algebra since admissible levels are never critical. 
Denoting by $c_V$ the central charge of a conformal vertex algebra $V$, 
we recall that 
$$c_{L_{k}(\g)}=  \dfrac{k  \dim \g }{k + h_\g^\vee}.$$ 

If $\g_0^\natural = 0$ then obviously
\begin{align}\label{eq:cLk.formula}
c_{L_{k^\natural}(\g^\natural)} =  \sum\limits_{i=1}^s  c_{L_{k_i^\natural} (\g_i^\natural)}.
\end{align}
In this case the possible values of the numerator $p$ of $k+h_\g^\vee$ for admissible $k$ 
are determined as solutions of
\begin{align}
\label{eq:central_charge}
c_{H_{DS,f}^0(L_k(\g))} = c_{L_{k^\natural}(\g^\natural)},
\end{align}
considered as an equation in an unknown $p$. Recall that $c_{H_{DS,f}^0(L_k(\g))}$ is given in equation (\ref{eq:W.alg.c.formula}). If there are no solutions in $p$ to the equation \eqref{eq:central_charge}, then $k$ is not collapsing. If there are solutions (with $p \ge h_\g$ or $p \ge h_\g^\vee$ so as to ensure that $k=-h_\g^{\vee}+p/q$ is admissible) then we proceed to the next step (\S\ref{sub:Asymptotic growth}).

If $\g_0^\natural \neq 0$ then, by Lemma \ref{Lem:k_0-condition}, the level $k$ can only be collapsing if $\phi_0^\natural = 0$. In particular $c_{L_{k^\natural}(\g^\natural)}$ continues to be given by (\ref{eq:cLk.formula}). Now the condition $\phi_0^\natural = 0$ entirely determines $k$, and the equation \eqref{eq:central_charge}, 
rather than determining $p$, is simply either true or false. If it is false then $k$ is not collapsing. If it is true then we proceed to the next step (\S\ref{sub:Asymptotic growth}) as before.

The data needed to compute the levels $k_i^\natural$ in term of $k$, 
for $i=0,\ldots,s$, are collected in 
Tables~\ref{Tab:A_Slodowy}--\ref{Tab:BD_Slodowy} 
for the classical types,  
and in 
Tables \ref{Tab:Data-G2}--\ref{Tab:Data-E8} 
for the exceptional types. 
The data for the exceptional types have been 
obtained using the software \texttt{GAP4}.

\subsection{Asymptotic growth and asymptotic dimension}
\label{sub:Asymptotic growth}
The first and second steps (\S\ref{sub:associated_variety} and 
\S\ref{sub:central_charge}) allow us to detect possible  
values for admissible $k$ and $k^\natural$. 
As this point, we apply 
Theorem \ref{Th:main}, or rather, the following powerful consequence of it.

\begin{Pro}
\label{Pro:asymptotics-and-collapsing}
Let $k$ and $k^{\natural}$ be admissible,
$f\in \overline{\mathbb{O}}_k$.
Suppose that 
\begin{align*}
&\G_{H^0_{DS,f}(L_k(\g))}=\G_{L_{k^\natural}(\g^\natural)}=\sum_{i=1}^s\G_{L_{k_i}(\g_i)},\\
& \A_{H^0_{DS,f}(L_k(\g))}=\A_{L_{k^\natural}(\g^\natural)}=\prod_{i=1}^s\A_{L_{k_i}(\g_i)}.
\end{align*}
Then $k$ is collapsing. 
\end{Pro}
\begin{proof}
The assertion 
follows immediately by applying Theorem \ref{Th:main}, taking for $\tilde{V} \rightarrow V$ the map $H^0_{DS,f}(L_k(\g)) \rightarrow \W_k(\g,f)$. Thus $L_{k^\natural}(\g^\natural) \cong \W_k(\g,f)$ and we conclude that $k$ is collapsing as required. 
\end{proof}

This proposition ensures that it is enough to compare the asymptotic growths  
and the asymptotic dimension of the vertex algebras 
$H_{DS,f}^{0}(L_k(\g))$ and $L_{k^\natural}(\g^\natural)$ 
that we compute using Corollary~\ref{Co:asymptotic_data_L}. 
This is the goal of next sections. 
In the classical cases, we sometimes 
directly use  the asymptotic growths 
(when $f$ admits an even good grading) 
to detect 
possible values of $k$ instead of the central charge argument, 
because the equation given by the central charge 
is often difficult to solve. 

\begin{Rem} 
\label{Rem:lisse_conjecture_true}
In Proposition \ref{Pro:asymptotics-and-collapsing}
suppose further that $f\in \mathbb{O}_k$ so that $\W_k(\g,f)$ is lisse. 
Then we get
that $$\W_k(\g,f)\cong H_{DS,f}^0(L_k(\g))$$
without the assumption that $f$ admits an even good grading.
Indeed,
$f\in \mathbb{O}_k$ implies that
 $H_{DS,f}^0(L_k(\g))$ is lisse
by Theorem \ref{Th:Associated_variety-DS}.
Hence  
$L_{k^\natural}(\g^{\natural})$ must be integrable and
the homomorphism
$\varphi$ of Theorem \ref{Th:main} must factors through the embedding
$L_{k^\natural}(\g^{\natural})\hookrightarrow H_{DS,f}^0(L_k(\g))$
(\cite{DonMas06}).
In particular,
$H_{DS,f}^0(L_k(\g))$ is a direct sum of integrable representations of the affine Kac-Moody algebra
associated with $\g^{\natural}$.
It follows that
the proof of Theorem \ref{Th:main}
goes through to obtain that 
$H_{DS,f}^0(L_k(\g))\cong L_{k^\natural}(\g^{\natural})$.
\end{Rem}

If the  isomorphism $\W_k(\g,f)\cong H_{DS,f}^0(L_k(\g))$ holds,
which is the case when 
 $f$ admits an even good grading, 
Proposition \ref{Pro:asymptotics-and-collapsing} 
gives a necessary and sufficient condition 
for admissible $k$ to be collapsing.
Unfortunately, in general, it gives only a sufficient condition. 

The following proposition will be useful to obtain explicit decompositions 
of finite extensions 
of admissible simple affine vertex algebras.

\begin{Pro}
\label{Pro:fin_ext-assoc} 
Let $k$ and $k^{\natural}$ be admissible,
$f\in \overline{\mathbb{O}}_k$.
Suppose that the associated varieties 
of $H^0_{DS,f}(L_k(\g))$ and $L_{k^\natural}(\g^\natural)$ 
have the same dimension and are not isomorphic. 
Then $k$ is not collapsing. 
\end{Pro}
\begin{proof}
The assumption ensures that $H^0_{DS,f}(L_k(\g))$ is nonzero. 
Hence, 
$\W_k(\g,f)$ is a quotient of $H^0_{DS,f}(L_k(\g))$ 
and its associated variety 
is a Zariski closed subvariety of that of $H^0_{DS,f}(L_k(\g))$. 
On the other hand, since $k$ is admissible, the associated variety 
of $H^0_{DS,f}(L_k(\g))$ is irreducible, $\overline{\O_k}$ being 
unibranch (see \cite{AraMor16b}). 
If $k$ were collapsing, then the associated variety of $\W_k(\g,f)$ would be isomorphic to 
the (irreducible) variety $X_{L_{k^\natural}(\g^\natural)}$,  
and so, would have the same dimension as $X_{H^0_{DS,f}(L_k(\g))}$ by the hypothesis. 
Then it would be isomorphic to the variety $X_{H^0_{DS,f}(L_k(\g))}$, 
since $X_{H^0_{DS,f}(L_k(\g))}$ is irreducible. 
This contradicts the non-isomorphism hypothesis. 
\end{proof}

\section{Some useful product formulas}
\label{sec:identities}

We recall the well-known identity  
\begin{align}
\label{eq:sin_formula}
\prod\limits_{j=1}^{n-1}  2 \sin  \dfrac{j \pi}{n} = n
\end{align}
and its immediate consequence
\begin{align}\label{eq:sin_formula.odd}
\prod_{j=1}^{n} 2\sin{\frac{(j-1/2) \pi}{n}} = 2.
\end{align}
In this section we record some further identities similar to these which will be very helpful when we come to apply Corollary~\ref{Co:asymptotic_data_L}.

Firstly from \eqref{eq:sin_formula} we deduce
\begin{align}
\label{eq:sin_formula2}
\prod_{j=1}^{n-1} \left( 2 \sin \dfrac{j\pi}{n} \right)^{n-j} 
& = \left(  \prod\limits_{j=1}^{\lfloor \frac{n-1}{2} \rfloor} 2\sin \dfrac{j \pi}{n} \right)^{n} 
 =   \left(\prod\limits_{j=1}^{n-1}2 \sin \dfrac{j \pi}{n} \right)^{\frac{n}{2}}  = n^{\frac{n}{2}}.
\end{align}

Next, we have the following identities. 

\begin{Lem} 
[{\cite{KacPeterson84,Kac90}}]
\label{Lem:main_identities}
{\ }
\begin{enumerate}
\item We have 
$ {\displaystyle \prod\limits_{\alpha \in \Delta_{+} } 
2 \sin \dfrac{\pi (\rho  | \alpha) }{h_\g^\vee} = 
{|P/Q^\vee|}^{\frac{1}{2}} \left( h_\g^\vee\right)^{\frac{\ell}{2}}.}$
\item 
For $\g$ simple, not of type $C_\ell$, $G_2$, $F_4$,  
we have
${\displaystyle \prod\limits_{\alpha \in \Delta_{+} } 
2 \sin \dfrac{\pi (\rho  | \alpha) }{h_\g^\vee+1} =
 \left( h_\g^\vee +1\right)^{\frac{\ell}{2}}.}$
\item We have 
${\displaystyle  \prod\limits_{\alpha \in \Delta_{+} } 
2 \sin \dfrac{\pi (\rho  | \alpha^\vee) }{h_\g } = \prod\limits_{\alpha \in \Delta_{+} } 
2 \sin \dfrac{\pi (\rho^\vee  | \alpha) }{h_\g } =|P^\vee/Q^\vee|^{\frac{1}{2}}  (h_\g)^{\frac{\ell}{2}}}$. 
\item We have
$ {\displaystyle \prod\limits_{\alpha \in \Delta_{+} } 
2 \sin \dfrac{\pi (\rho  | \alpha^\vee) }{h_\g +1}= \prod\limits_{\alpha \in \Delta_{+} } 
2 \sin \dfrac{\pi (\rho^\vee  | \alpha) }{h_\g +1} =
 \left( h_\g  +1 \right)^{\frac{\ell}{2}}.}$
\end{enumerate}
\end{Lem}

The lemma is probably known. 
We provide a proof for the convenience of the reader, and to clear up an ambiguity 
from \cite{KacPeterson84} (see Remark \ref{Rem:KacPeterson84}). 
\begin{proof}
As a rule, in this proof, 
we write $\rho_{\g}$ (resp.~ $\rho_{\g}^\vee$) 
for  the half-sum of positive roots (resp.~coroots) of $\g$. 
For $p \in \Z_{>0}$, 
set 
$$\varPi_\g(p):={\displaystyle \prod\limits_{\alpha \in \Delta_{+} } 
2 \sin \dfrac{\pi (\rho_\g  | \alpha) }{p}}, 
\qquad \varPi^\vee_\g(p):={\displaystyle \prod\limits_{\alpha \in \Delta_{+} } 
2 \sin \dfrac{\pi (\rho_\g  | \alpha^\vee) }{p}}.$$
Since $\{   {\rm ht}(\alpha) \colon \alpha \in \Delta_+\}= 
\{   {\rm ht}(\alpha^\vee) \colon \alpha \in \Delta_+\}$, 
note that for any $p \ge h_\g$, 
$$ \varPi^\vee_\g(p)=
\prod\limits_{\alpha \in \Delta_{+} } 
2 \sin \dfrac{\pi (\rho_\g^\vee  | \alpha) }{p}.$$
We use the data of Table \ref{Tab:data_simple}.

(1) The identity is established in \cite[Chapter 13, (13.8.1)]{Kac90} 
using modular invariance properties. 
It can also be checked using a case-by-case argument 
exploiting identities \eqref{eq:sin_formula} and 
\eqref{eq:sin_formula2}. 

(2) 
We check the identity using a case-by-case argument. 

$\ast$ Type $A_\ell$. 
We have $h_{A_\ell}^\vee +1=\ell+2$. 
Since $(\rho_{A_\ell} | \alpha) = \on{ht}(\alpha)$, we easily that 
\begin{align*} 
\varPi_{A_\ell}(\ell+2) & = 
\varPi_{A_{\ell+1}}(\ell+2) 
\left( \prod\limits_{j=1}^{\ell+1} 2 \sin \dfrac{j \pi}{\ell+2}
\right)^{-1}.
\end{align*} 
So by (1) applied to $A_{\ell+1}$, we obtain the expected statement, 
$
\varPi_{A_\ell}(\ell+2)   
 =(\ell+2)^{\ell/2} ,$ 
using the identity \eqref{eq:sin_formula}. 

$\ast$ Type $B_\ell$ and $D_\ell$. 
We have $h_{B_\ell}^\vee=2\ell-1$ and  $h_{D_\ell}^\vee+1=2\ell-1$. 
We first show the statement for $D_\ell$. We have to show that 
\begin{align*} 
\varPi_{D_\ell}(2\ell -1)
= (2\ell-1)^{\ell/2}. 
\end{align*} 
By (1) applied to $B_\ell$, we have 
\begin{align*}
\varPi_{B_\ell}(2\ell -1) 
=2 (2\ell-1)^{\ell/2}. 
\end{align*}  
So it suffices to show that the ratio $\dfrac{\varPi_{B_\ell}(2\ell -1)}{\varPi_{D_\ell}(2\ell -1)}$ 
equals $2$. 
Observing that $(\rho_{D_\ell} |\alpha)={\rm ht}(\alpha)$, 
we get that 
\begin{align*} 
\varPi_{D_\ell}(2\ell -1)= 
\prod\limits_{j=1}^{\ell-1} \left(2 \sin \dfrac{j \pi }{2\ell-1}\right)^{\ell-j} 
\times 
\prod\limits_{j=1}^{\ell-1} 2 \sin \dfrac{j \pi }{2\ell-1} 
\times 
\prod\limits_{i=1}^{\ell-2} \prod\limits_{j=2i+1} ^{\ell+(i-1)} 2\sin \dfrac{j \pi}{2\ell-1}.  
\end{align*} 
On the other hand, observing that $(\rho_{B_\ell} |\alpha)={\rm ht}(\alpha)$ if $\alpha$ 
is long and $(\rho_{B_\ell} |\alpha)=\frac{1}{2}{\rm ht}(\alpha)$ if $\alpha$ is short, 
we get that 
\begin{align*}
\varPi_{B_\ell}(2\ell -1) & 
=  \prod\limits_{j=1}^\ell 2 \sin \dfrac{(2j-1)\pi}{2(2\ell-1)} 
\times \prod\limits_{j=1}^{\ell-1} \left(2 \sin \dfrac{j \pi }{2\ell-1}\right)^{\ell-j} 
\times 
\prod\limits_{i=1}^{\ell-1} \prod\limits_{j=2i} ^{\ell+(i-1)} 2\sin \dfrac{j \pi}{2\ell-1}.
\end{align*} 
Using the identity \eqref{eq:sin_formula}, 
we show that 
\begin{align*}
& \prod\limits_{j=1}^\ell 2 \sin \dfrac{(2j-1)\pi}{2(2\ell-1)}   = 2,  
\quad  \prod\limits_{j=1}^{\ell-1} 2 \sin \dfrac{j \pi }{2\ell-1} =
 \prod\limits_{j=1}^{\ell-1} 2\sin \dfrac{2j \pi}{ 2\ell-1} = (2\ell-1)^{1/2}. & 
\end{align*}
From this, we obtain that $\dfrac{\varPi_{B_\ell}(2\ell -1)}{\varPi_{D_\ell}(2\ell -1)}=2$ 
as desired. 

We now turn to the statement for $B_\ell$. We have 
to show that 
\begin{align*} 
\varPi_{B_\ell}(2\ell) 
= (2\ell)^{\ell/2}. 
\end{align*} 
By (1) applied to $D_{\ell +1}$, we have 
\begin{align*}
\varPi_{D_{\ell+1}}(2\ell)  
=2 (2\ell)^{(\ell+1)/2}. 
\end{align*}   
So it suffices to show that the ratio $\dfrac{\varPi_{D_{\ell+1}}(2\ell)}{\varPi_{B_\ell}(2\ell) }$ 
equals $2 (2\ell)^{1/2}$. As before, computing the heights of roots, we obtain  that 
\begin{align*}
\dfrac{\varPi_{D_{\ell+1}}(2\ell)}{\varPi_{B_\ell}(2\ell) } &  =  
\prod\limits_{j=1}^\ell 2 \sin \dfrac{ j \pi}{2\ell} 
\left( \prod\limits_{j=1}^{\ell} 2 \sin \dfrac{(2j-1) \pi }{4\ell} \right) ^{-1} 
\prod\limits_{j=1}^{\ell} 2\sin \dfrac{(2j-1) \pi}{ 2\ell}. 
\end{align*} 
Using \eqref{eq:sin_formula} 
we  show that 
\begin{align*}
\prod\limits_{j=1}^\ell 2 \sin \dfrac{ j \pi}{2\ell} = 2 \ell^{1/2},  
\quad \prod\limits_{j=1}^{\ell} 2 \sin \dfrac{(2j-1) \pi }{4\ell}  =2^{1/2}, 
\quad 
\prod\limits_{j=1}^{\ell} 2\sin \dfrac{(2j-1) \pi}{ 2\ell} = 2, & 
\end{align*}
whence $\dfrac{\varPi_{D_{\ell+1}}(2\ell)}{\varPi_{B_\ell}(2\ell) }=2 (2\ell)^{1/2},$ 
as desired. 

\smallskip

$\ast$ Types $E_6, E_7, E_8$.
By direct calculations, 
we easily obtain that 
\begin{align*}
\varPi_{E_6}(13) 
 & =   \prod\limits_{k=1}^6 \left( 2\sin \dfrac{k\pi}{13}\right)^{6}  = \left( \prod\limits_{k=1}^{12}2  \sin \dfrac{k\pi} {19}\right)^{6/2}  =13^3. &
\end{align*}
Similarly, we get
\begin{align*}
\varPi_{E_7}(19)
& =19^{7/2}, \qquad 
\varPi_{E_8}(31)
 =31^{4}.&
\end{align*} 

\noindent

(3) and (4). 
We prove both identities together. 

By (2), it suffices to check the statement 
for the non simply-laced cases. 
We easily check that 
\begin{align*}
& \varPi^\vee_{G_2}(6) = 
6, \quad  
\varPi^\vee_{G_2}(7)=7, \quad 
\varPi^\vee_{F_4}(12)
=12^{2}, \quad 
\varPi^\vee_{F_4}(13)=13^{2}.& 
\end{align*}
It remains to consider the cases where $\g$ has type $B_\ell$ or $C_\ell$.

$\ast$ Type $B_\ell$.  
Let us first prove the identity (4). 
By (1) applied to $B_{\ell+1}$, we 
have 
\begin{align*}
\varPi_{B_{\ell+1}}(2\ell+1)
=2 (2\ell +1)^{(\ell+1)/2}. 
\end{align*}
Hence it suffices to show that 
$$\dfrac{\varPi_{B_{\ell+1}}(2\ell+1)}{\varPi^\vee_{B_{\ell}}(2\ell+1)}=2(2\ell+1)^{1/2}.$$ 
Using the computations of (2) and the identity \eqref{eq:sin_formula}, we easily obtain 
the expected equality.  

Let us now prove the identity (3) for $B_\ell$. 
By (2) applied to $B_\ell$, we have 
\begin{align*}
\varPi_{B_{\ell}}(2\ell)
= (2\ell)^{1/2}. 
\end{align*}
Hence it suffices to show that 
$$\dfrac{\varPi^\vee_{B_{\ell}}(2\ell)}{\varPi_{B_{\ell}}(2\ell)}=2^{1/2}$$
since $|P^\vee/Q^\vee|=2$ for the type $B_\ell$. 
Using \eqref{eq:sin_formula}  and the computations of (2), we easily 
obtain the expected equality.

$\ast$ Type  $C_\ell$. 
Notice that $h_{C_\ell}= h_{B_\ell}= 2 \ell$. 
Hence it suffices to show that 
$$\dfrac{\varPi_{B_{\ell}}^\vee(2\ell)}{\varPi_{C_{\ell}}^\vee(2\ell)}=1
\text{ and } \quad 
\dfrac{\varPi_{B_{\ell}}^\vee(2\ell+1)}{\varPi_{C_{\ell}}^\vee(2\ell+1)}=1.
$$
since $|P^\vee/Q^\vee|=2$ for the types $B_\ell$ and $C_\ell$.
Again using  \eqref{eq:sin_formula}, we easily obtain 
the expected equalities. 
This concludes the proof of the lemma. 
\end{proof}

\begin{Rem}
\label{Rem:KacPeterson84}
The identities of the lemma are also stated in \cite[Proposition 4.30]{KacPeterson84}.  
However, contrary to what should follow from (4.30.2) of \cite{KacPeterson84}, 
identity (2) does not hold for the types $C_\ell$, $\ell \ge 3$, $G_2$ and $F_4$. 
For these types, it seems that there is no pleasant formula for 
${\displaystyle \prod\limits_{\alpha \in \Delta_{+}(C_\ell)} 2 \sin \dfrac{\pi (\rho  | \alpha) }{h_\g^\vee+1}}$.
\end{Rem}

In the next two sections  
we study collapsing levels in the classical cases, implementing the strategy described in Section~\ref{sec:strategy}. 

\section{Collapsing levels for type $\sl_n$} 
\label{sec:sl_n}

Let $n\in \Z_{>0}$. 
In this section, it is assumed that $\g$ is the simple Lie algebra $\mf{sl}_n$, 
the Lie algebra 
of traceless $n \times n$ matrices with coefficients in $\C$. 
The Killing form of $\g=\sl_n$ 
is given by $\kappa_\g(x,y)= 2n {\rm tr}(xy)$ 
and $(x|y)_\g={\rm tr}(xy)$. 

Denote by $\P(n)$ the set 
of partitions of $n$. 
As a rule, unless otherwise specified, we write 
an element $\bs{\lambda}$ 
of $\P(n)$ 
as a decreasing sequence $\bs{\lambda}=(\lambda_1,\ldots,\lambda_r)$ 
omitting zeroes. 
Thus, 
$$
\lambda_1 \ge \cdots \ge \lambda_r \ge 1\quad 
\text{ and }\quad  
\lambda_1 + \cdots + \lambda_r  = n.
$$ 

Let us denote by $\ge$ the partial order on $\P(n)$ relative 
to dominance. More precisely, given
$\bs{\lambda} = ( \lambda_{1},\cdots ,\lambda_{r}),
\bs{\mu} = (\mu_{1},\dots ,\mu_{s}) \in \P (n)$,  
we have $\bs{\lambda} \geqslant \bs{\mu}$ if 
$
\sum_{i=1}^{k} \lambda_{i} \geqslant \sum_{i=1}^{k} \mu_{i}
$
for $1\leqslant k\leqslant \min (r,s)$.

By \cite[Theorem~5.1.1]{CMa}, nilpotent orbits of $\sl_{n}$ are 
parametrised by $\P(n)$. For $\bs{\lambda}\in\P(n)$, 
we shall denote by 
$\O_{\bs{\lambda}}$ the corresponding nilpotent orbit of 
$\sl_n$. 
If $\bs{\lambda}, \bs{\mu} \in \P (n)$, then 
$\O_{\bs{\mu}} \subset \overline{\O}_{\bs{\lambda}}$ if and only if
$\bs{\mu} \leqslant \bs{\lambda}$.

\begin{Def} \label{Def:degeneration}
Let $\bs{\lam} \in \P(n)$.  A {\em degeneration} of $\bs{\lam}$ is 
an element $\bs{\mu} \in \P(n)$ such that 
$\O_{\bs{\mu}} \subsetneq \overline{\O}_{\bs{\lam}}$, 
that is, $\bs{\mu} < \bs{\lam}$. 
A  degeneration $\bs{\mu}$  of $\bs{\lam}$ is said to be {\em minimal} 
if $\O_{\bs{\mu}}$ is open in 
$ \overline{\O}_{\bs{\lam}} \setminus \O_{\bs{\lam}}$. 
\end{Def}

Fix $\bs\lam \in  \P (n)$. 
As proved in \cite{ElashviliKac}  
the set of good gradings for $f \in\O_{\bs\lam}$ are in bijection with 
the set of {\em pyramids of shape $\bs{\lam}$}. 
We refer to \cite{Brundan-Goodwin} for the precise construction of 
pyramids associated with good gradings. 

For our purpose, let us just recall that a pyramid is a diagram consisting of $n$ boxes each of size 
2 units by 2 units drawn in the upper half of the $xy$-place, with midpoints having integer coordinates. 
By the coordinates of the 
box $i$, we mean the coordinates of its midpoint.
We will also speak of the row number of a
box, by which we mean its $y$-coordinate, and the column number of a box, meaning its $x$-coordinate.
A pyramid of shape $\bs\lam$ consists of $r$ rows, with the $i^{\text{th}}$ row consisting of $\lam_i$ horizontally consecutive boxes.
The rows are positioned so that the boxes of the 
first row are centred on the $y$-axis and 
have $y$-coordinate $1$, the boxes of the second row have $y$-coordinate $3$, etc. 
In addition, no box of row $i$ is permitted to have $x$-coordinate smaller than the minimal $x$-coordinate of the boxes in row $i-1$, nor may its $x$-coordinate be greater than the maximal 
$x$-coordinates of the boxes in row $i-1$. 
We obtain the Dynkin pyramid (corresponding to the Dynkin grading) when the boxes 
of any row are centred around the $y$-axis. 
Note that a good grading is even if and only if the $x$-coordinates of all 
boxes have the same parity.  This is always the case for the left-adjusted and 
the right-adjusted 
pyramids of shape $\bs\lam$. 
For the Dynkin pyramid, this happens if and only if all parts $\lam_i$ have the same 
parity. 

We number the boxes of the pyramid from top right to bottom left 
and so that col$(1) \ge {\rm col}(2) \ge \ldots \ge {\rm col}(n)$. 
Then one can choose a representative $f$ in $\O_\lam$ as follows. 
Set $f=\sum_{i,j} e_{i,j}$, where the sum is over all $i,j \in \{1,\ldots,n\}$ such 
that ${\rm row}(i)  = {\rm row}(j)$ and 
${\rm col}(j)= {\rm col}(i)+2$. 
Set 
$$x_\Gamma^0 = \eps I_n + \tilde{x}_\Gamma^0,$$
where 
$\tilde{x}_\Gamma^0=\sum_{i=1}^n \frac{1}{2}{\rm col}(i) e_{i,i}$ 
and $\eps$ is chosen so that ${\rm tr}(x_\Gamma^0)=0$. 
Here $e_{i,j}$ stands for the $(i,j)$-matrix unit. 
Then $f \in \O_{\bs\lam}$ and 
$$\g^{i}_\Gamma :=  \{x \in \g \colon [x_\Gamma^0,x] = i x \}$$
yields a good grading 
$\displaystyle{\g=\bigoplus\limits_{j \in\Z} \g^{j}_\Gamma}$ 
for $f$ (see \cite{ElashviliKac} or \cite[\S6]{Brundan-Goodwin}). 

As an example, we represent in Figure \ref{Fig:exemple_pyr} the left-adjusted, 
the Dynkin and the right-adjusted pyramids for the partition $\bs\lam=(3^2,2)$. 

\begin{figure}[ht]
\setlength\unitlength{0.0175cm}
\begin{picture}(600,60)

\put(100,0){\line(1,0){60}}
\put(100,20){\line(1,0){60}}
\put(100,40){\line(1,0){60}}
\put(100,60){\line(1,0){40}}

\put(100,0){\line(0,1){60}}
\put(120,0){\line(0,1){60}}
\put(140,0){\line(0,1){60}}
\put(160,0){\line(0,1){40}}

\put(150,30){\makebox(0,0){\Tiny{{$1$}}}}
\put(150,10){\makebox(0,0){\Tiny{{$2$}}}}
\put(130,50){\makebox(0,0){\Tiny{{$3$}}}}
\put(130,30){\makebox(0,0){\Tiny{{$4$}}}}
\put(130,10){\makebox(0,0){\Tiny{{$5$}}}}
\put(110,50){\makebox(0,0){\Tiny{{$6$}}}}
\put(110,30){\makebox(0,0){\Tiny{{$7$}}}}
\put(110,10){\makebox(0,0){\Tiny{{$8$}}}}

\put(130,0){\makebox(0,0){\Tiny{{$\bullet$}}}}

\put(300,0){\line(1,0){60}}
\put(300,20){\line(1,0){60}}
\put(300,40){\line(1,0){60}}
\put(310,60){\line(1,0){40}}

\put(300,0){\line(0,1){40}}
\put(320,0){\line(0,1){40}}
\put(340,0){\line(0,1){40}}
\put(360,0){\line(0,1){40}}
\put(310,40){\line(0,1){20}}
\put(330,40){\line(0,1){20}}
\put(350,40){\line(0,1){20}}

\put(350,30){\makebox(0,0){\Tiny{{$1$}}}}
\put(350,10){\makebox(0,0){\Tiny{{$2$}}}}
\put(340,50){\makebox(0,0){\Tiny{{$3$}}}}
\put(330,30){\makebox(0,0){\Tiny{{$4$}}}}
\put(330,10){\makebox(0,0){\Tiny{{$5$}}}}
\put(320,50){\makebox(0,0){\Tiny{{$6$}}}}
\put(310,30){\makebox(0,0){\Tiny{{$7$}}}}
\put(310,10){\makebox(0,0){\Tiny{{$8$}}}}

\put(330,0){\makebox(0,0){\Tiny{{$\bullet$}}}}

\put(500,0){\line(1,0){60}}
\put(500,20){\line(1,0){60}}
\put(500,40){\line(1,0){60}}
\put(520,60){\line(1,0){40}}

\put(500,0){\line(0,1){40}}
\put(520,0){\line(0,1){60}}
\put(540,0){\line(0,1){60}}
\put(560,0){\line(0,1){60}}

\put(550,30){\makebox(0,0){\Tiny{{$2$}}}}
\put(550,10){\makebox(0,0){\Tiny{{$3$}}}}
\put(530,50){\makebox(0,0){\Tiny{{$4$}}}}
\put(530,30){\makebox(0,0){\Tiny{{$5$}}}}
\put(530,10){\makebox(0,0){\Tiny{{$6$}}}}
\put(550,50){\makebox(0,0){\Tiny{{$1$}}}}
\put(510,30){\makebox(0,0){\Tiny{{$7$}}}}
\put(510,10){\makebox(0,0){\Tiny{{$8$}}}}

\put(530,0){\makebox(0,0){\Tiny{{$\bullet$}}}}
\end{picture}
\caption{Left-adjusted, Dynkin and right-adjusted pyramids of shape $\bs\lam=(3^2,2)$}
\label{Fig:exemple_pyr}
\end{figure}
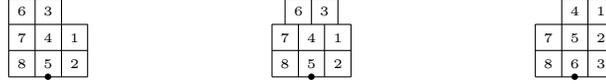

The following lemma is a refinement of \cite[Propositions 4.4 and 5.4]{KraftProcesi81}. 
We refer to \cite[Propositions 7.3.1 and 7.3.2]{Li17} for a proof. 

\begin{Lem}[row/column removal rule in $\sl_n$]
\label{lem:erasing_row_type-A}
Let ${\bs{\lam}} \in \P(n)$ and ${\bs{\mu}}$ a degeneration 
of $ {\bs{\lam}}$. 
Assume that the first $l$ rows and 
the first $m$ columns of $ {\bs{\lam}}$ and $ {\bs{\mu}}$ 
coincide.  
Denote by $\bs{\lam}'$ and $\bs{\mu}'$ the partitions obtained by erasing 
these $l$ common rows and $m$ common columns. Then 
\begin{align*}
\Slo_{\O_{\bs{\lam}},f} \cong \Slo_{\O_{{\bs{\lam}'}},f'},
\end{align*}
as algebraic varieties, 
with $f \in \O_{{\bs{\mu}}}$ and $f' \in \O_{{\bs{\mu}'}}$. 
In particular, if $f'=0$, then 
$\Slo_{\O_{\bs{\lam}},f} \cong \overline{\O}_{{\bs{\lam}'}}$.  
\end{Lem}

By \cite{Arakawa15a}
if $k$ is an admissible level 
for $\sl_n$, we have 
$\O_k=\O_{\bs{\lam}}$, 
where  
\begin{align}
\label{eq:Ok_sl_n}
\bs{\lam}=(q^{\widetilde m},\widetilde s), \quad
\quad 0 \le \widetilde s \le \widetilde q-1.
\end{align}

\begin{Lem}
\label{Lem:choice_of_partitions}
Let $\bs{\mu}$ be a partition of $n$ such that $\O_{\bs{\mu}} \subset  \overline{\O}_{\bs{\lam}}$.   
Let $\bs{\lam}'$ and $\bs{\mu}'$ be the partitions obtained from $\bs{\lam}$ and $\bs{\mu}$ 
by erasing all common rows and columns of $\bs{\lam}$ and $\bs{\mu}$. 
Then, $\bs{\mu}'$ corresponds to the zero nilpotent orbit of $\sl_{|\bs{\mu}'|}$, that is, $\bs{\mu}'=(1^{|\bs{\mu}'|})$ 
if and only if $\bs\mu$ is of one of the following types:  
\begin{itemize}
\item[(a)] $\bs{\mu}=(q^{\widetilde m},\widetilde s)=\bs{\lam}$, 
\item[(b)] $\bs{\mu}=(q^m,1^s)$ with $0 \le m \le \widetilde m$ and $s \ge 0$, 
\item[(c)] $\bs{\mu}=(q^{\widetilde m -1},(q-1)^2)$ and $\widetilde s = q-2$. 
\end{itemize}
Here $|\bs{\mu}'|$ stands for the sum of the parts of $\bs{\mu}'$. 
\end{Lem}

\begin{proof} 
Since $\O_{\bs{\mu}} \subset  \overline{\O}_{\bs{\lam}}=\overline{\O}_k$, one can 
write $\bs{\mu}=(q^m,\bs{\nu})$, with $0 \le m \le \widetilde m$ and $\bs{\nu} =({\nu}_1,\ldots,{\nu}_t)$ 
with $\nu_1 < q$. 
Assume that $\bs{\mu} \not = \bs{\lam}$, the case $\bs{\mu} = \bs{\lam}$ being obvious. 

Let $\bs{\lam}''$ and $\bs{\mu}''$ be the partitions obtained from $(q^{\widetilde m},\widetilde s)$ and $\bs{\mu}$ 
by erasing all common rows of $\bs{\lam}=(q^{\widetilde m},\widetilde s)$ and $\bs{\mu}$. 
Then $\bs{\lam}''=(q^{\widetilde m - m},\widetilde s)$ and $\bs{\mu}''=\bs{\nu}$. 
The partition $\bs{\mu}''$ corresponds to the zero nilpotent orbit of $\sl_{|\bs{\mu}''|}$ 
if and only if $\bs{\nu} = (1^{|\bs{\nu} |})$. 
This leads to the partitions of type (b). 
We illustrate in Figures \ref{Fig:removal-1}  
the row removal rule in the case where 
$\bs{\lam}=(3^3,1)$ and $\bs{\mu}=(3^2,1^4)$ is of type (b).

{\tiny
\begin{center}
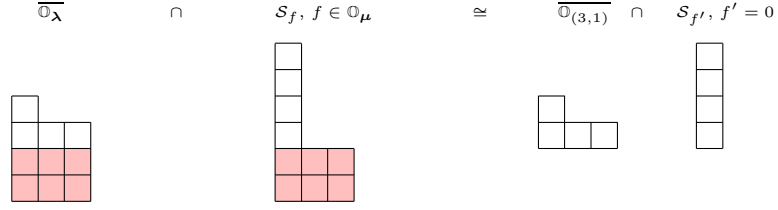
\begin{figure}[ht]
\setlength\unitlength{0.0175cm}
\begin{picture}(800,160)

{\color{pink}\polygon*(100,0)(160,0)(160,40)(100,40)}

\put(100,0){\line(1,0){60}}
\put(100,20){\line(1,0){60}}
\put(100,40){\line(1,0){60}}
\put(100,60){\line(1,0){60}}
\put(100,80){\line(1,0){20}}

\put(100,0){\line(0,1){80}}
\put(120,0){\line(0,1){80}}
\put(140,0){\line(0,1){60}}
\put(160,0){\line(0,1){60}}

\put(120,140){$\overline{\O}_{\bs{\lam}}$}
\put(220,140){$\cap$}
\put(300,140){$\Slo_f$, $f \in \O_{\bs{\mu}}$}


{\color{pink}\polygon*(300,0)(360,0)(360,40)(300,40)}

\put(300,0){\line(1,0){60}}
\put(300,20){\line(1,0){60}}
\put(300,40){\line(1,0){60}}
\put(300,60){\line(1,0){20}}
\put(300,80){\line(1,0){20}}
\put(300,100){\line(1,0){20}}
\put(300,120){\line(1,0){20}}

\put(300,0){\line(0,1){120}}
\put(320,0){\line(0,1){120}}
\put(340,0){\line(0,1){40}}
\put(360,0){\line(0,1){40}}

\put(450,140){$\cong$}

\put(500,40){\line(1,0){60}}
\put(500,60){\line(1,0){60}}
\put(500,80){\line(1,0){20}}

\put(500,40){\line(0,1){40}}
\put(520,40){\line(0,1){40}}
\put(540,40){\line(0,1){20}}
\put(560,40){\line(0,1){20}}

\put(620,40){\line(1,0){20}}
\put(620,60){\line(1,0){20}}
\put(620,80){\line(1,0){20}}
\put(620,100){\line(1,0){20}}
\put(620,120){\line(1,0){20}}

\put(620,40){\line(0,1){80}}
\put(640,40){\line(0,1){80}}

\put(515,140){$\overline{\O}_{(3,1)}$}
\put(570,140){$\cap$}
\put(605,140){$\Slo_{f'}$, $f' =0$}

\end{picture}
\caption{\footnotesize{Row removal rule for $\bs{\lam}=(3^3,1)$ and $\bs{\mu}=(3^2,1^4)$}} 
\label{Fig:removal-1}
\end{figure}
\end{center}}

Consider now the common columns of $\bs{\lam}$ and $\bs{\mu}$.
Observe that $\bs{\lam}=(q^{\widetilde m},\widetilde s)$ and $\bs{\mu}=(q^m,{\nu}_1,\ldots,{\nu}_t)$ have at least one common column if only if $\widetilde m +1=m+t$, that is, 
$$t = \widetilde m - m +1.$$ 
Assume that this condition holds. 

We illustrate in Figures \ref{Fig:removal-2} and \ref{Fig:removal-3} 
the row/column removal rule in the case where 
$\bs{\lam}=(5^3,2)$, $\bs{\mu}=(5^2,4,3)$, and 
in the case where $\bs{\lam}=(5^3,3)$, $\bs{\mu}=(5^2,4^2)$. 

{\tiny
\begin{center}
\begin{figure}[ht]
\setlength\unitlength{0.0175cm}
\begin{picture}(800,160)

{\color{pink}\polygon*(100,0)(200,0)(200,40)(140,40)(140,80)(100,80)}

\put(100,0){\line(1,0){100}}
\put(100,20){\line(1,0){100}}
\put(100,40){\line(1,0){100}}
\put(100,60){\line(1,0){100}}
\put(100,80){\line(1,0){40}}

\put(100,0){\line(0,1){80}}
\put(120,0){\line(0,1){80}}
\put(140,0){\line(0,1){80}}
\put(160,0){\line(0,1){60}}
\put(180,0){\line(0,1){60}}
\put(200,0){\line(0,1){60}}

\put(140,120){$\overline{\O}_{\bs{\lam}}$}
\put(250,120){$\cap$}
\put(310,120){$\Slo_f$, $f \in \O_{\bs{\mu}}$}


{\color{pink}\polygon*(300,0)(400,0)(400,40)(340,40)(340,80)(300,80)}

\put(300,0){\line(1,0){100}}
\put(300,20){\line(1,0){100}}
\put(300,40){\line(1,0){100}}
\put(300,60){\line(1,0){80}}
\put(300,80){\line(1,0){60}}

\put(300,0){\line(0,1){80}}
\put(320,0){\line(0,1){80}}
\put(340,0){\line(0,1){80}}
\put(360,0){\line(0,1){80}}
\put(380,0){\line(0,1){60}}
\put(400,0){\line(0,1){40}}

\put(450,120){$\cong$}

\put(500,40){\line(1,0){60}}
\put(500,60){\line(1,0){60}}

\put(500,40){\line(0,1){20}}
\put(520,40){\line(0,1){20}}
\put(540,40){\line(0,1){20}}
\put(560,40){\line(0,1){20}}

\put(620,40){\line(1,0){40}}
\put(620,60){\line(1,0){40}}
\put(620,80){\line(1,0){20}}

\put(620,40){\line(0,1){40}}
\put(640,40){\line(0,1){40}}
\put(660,40){\line(0,1){20}}

\put(515,120){$\overline{\O}_{(3)}$}
\put(570,120){$\cap$}
\put(605,120){$\Slo_{f'}$, $f' \in \O_{(2,1)}$}

\end{picture}
\caption{\footnotesize{Row/column removal rule for $\bs{\lam}=(5^3,2)$ and $\bs{\mu}=(5^2,4,3)$}} 
\label{Fig:removal-2}
\end{figure}
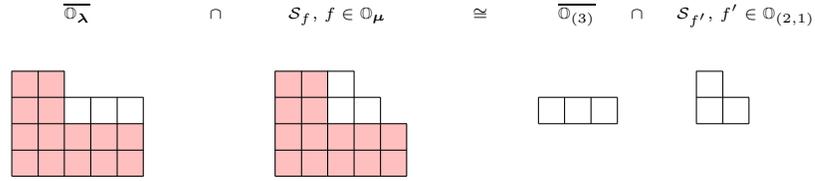
\end{center}}

{\tiny
\begin{center}
\begin{figure}[ht]
\setlength\unitlength{0.0175cm}
\begin{picture}(800,160)

{\color{pink}\polygon*(100,0)(200,0)(200,40)(160,40)(160,80)(100,80)}

\put(100,0){\line(1,0){100}}
\put(100,20){\line(1,0){100}}
\put(100,40){\line(1,0){100}}
\put(100,60){\line(1,0){100}}
\put(100,80){\line(1,0){60}}

\put(100,0){\line(0,1){80}}
\put(120,0){\line(0,1){80}}
\put(140,0){\line(0,1){80}}
\put(160,0){\line(0,1){80}}
\put(180,0){\line(0,1){60}}
\put(200,0){\line(0,1){60}}

\put(140,120){$\overline{\O}_{\bs{\lam}}$}
\put(250,120){$\cap$}
\put(310,120){$\Slo_f$, $f \in \O_{\bs{\mu}}$}


{\color{pink}\polygon*(300,0)(400,0)(400,40)(360,40)(360,80)(300,80)}

\put(300,0){\line(1,0){100}}
\put(300,20){\line(1,0){100}}
\put(300,40){\line(1,0){100}}
\put(300,60){\line(1,0){80}}
\put(300,80){\line(1,0){80}}

\put(300,0){\line(0,1){80}}
\put(320,0){\line(0,1){80}}
\put(340,0){\line(0,1){80}}
\put(360,0){\line(0,1){80}}
\put(380,0){\line(0,1){80}}
\put(400,0){\line(0,1){40}}

\put(450,120){$\cong$}

\put(520,40){\line(1,0){40}}
\put(520,60){\line(1,0){40}}

\put(520,40){\line(0,1){20}}
\put(540,40){\line(0,1){20}}
\put(560,40){\line(0,1){20}}

\put(620,40){\line(1,0){20}}
\put(620,60){\line(1,0){20}}
\put(620,80){\line(1,0){20}}

\put(620,40){\line(0,1){40}}
\put(640,40){\line(0,1){40}}

\put(515,120){$\overline{\O}_{(2)}$}
\put(570,120){$\cap$}
\put(605,120){$\Slo_{f'}$, $f' = 0$}

\end{picture}
\caption{\footnotesize{Row/column removal rule for $\bs{\lam}=(5^3,3)$ and $\bs{\mu}=(5^2,4^2)$}} 
\label{Fig:removal-3}
\end{figure}
\end{center}}

We obtain that 
$\bs{\lam}' = ((q-\widetilde s)^{\widetilde m -m})$ and $\bs{\mu}'=({\nu}_1-\widetilde s ,\ldots,{\nu}_t-\widetilde s)$. 
Furthermore, $\bs{\mu}'$ corresponds to the zero nilpotent orbit of $\sl_{|\bs{\mu}'|}$ if and only if ${\nu}_1-\widetilde s=\cdots={\nu}_t -\widetilde s=1$. 
If so, then necessarily $(\widetilde m -m) (q-\widetilde s)=t$ 
since $\bs{\lam}'$ and $\bs{\mu}'$ are partitions of the same integer, whence 
\begin{align}
\label{eq:partition_condition}
(\widetilde m -m) (q-\widetilde s -1)=1.
\end{align}
using $t = \widetilde m - m +1$. 
But condition \eqref{eq:partition_condition} holds if and only if 
$q-\widetilde s -1=1$ and $\widetilde m -m=1$. This leads to the partitions of type (c). 

Conversely, it is easy to verify that if $\bs{\mu}$ is of type (a), (b) or (c) 
then $\bs{\mu}'$ corresponds to the zero 
nilpotent orbit of $\sl_{|\bs{\mu}'|}$. 
\end{proof}

In view of Lemma \ref{Lem:choice_of_partitions}, 
we describe the centraliser $\g^\natural$ 
and the values of the $k_i^\natural$'s for particular   
$\sl_2$-triples $(e,h,f)$ of $\sl_n$.

\begin{Lem}
\label{Lem:k_natural-type-A} 
Let $f$ be a nilpotent element of $\sl_n$ associated with $\bs\mu \in \P(n)$, 
with $\bs\mu$ 
as in the first column of Table~\ref{Tab:A_Slodowy}. 
Then the centraliser $\g^\natural$ of an $\sl_2$-triple 
$(e,h,f)$   
and the values of the $k_i^\natural$'s are given 
by Table~\ref{Tab:A_Slodowy}. 
Moreover, if $k$ is admissible, then so is $k^\natural$, 
provided that $k_0^\natural =0$. 
\end{Lem}
In Table \ref{Tab:A_Slodowy}, the numbering of the levels $k_i^\natural$'s 
follows the order in which the simple factors of $\g^\natural$ 
appears.

\begin{proof}
By Lemma \ref{Lem:maximal reductive subalgebra} 
and Remark \ref{Rem:maximal reductive subalgebra}, in order 
to describe $\g^\natural$ and compute the $k_i^\natural$'s, 
one can use the left-adjusted pyramid of shape $\bs{\mu}$. 
This pyramid always corresponds to an even good grading for $f$. 

\smallskip

(1) Consider first the case $\bs{\mu}= (q^m,1^s)$, 
with $s$ possibly zero. 
The pyramid consists of a $q \times m$ rectangle 
surmounted by a vertical strip of dimension $1 \times s$ on the left as in Figure~\ref{Fig:sl_n}. 
For example, for $\bs{\mu}=(5^3,1^4)$ we get the pyramid as in Figure~\ref{Fig:sl-ex}.
{\tiny
\begin{center}
\begin{figure}
\setlength\unitlength{0.0125cm}
\begin{minipage}[l]{.46\linewidth} 
\begin{picture}(120,210)
\put(100,0){\line(1,0){150}}
\put(100,30){\line(1,0){150}}
\put(100,60){\line(1,0){150}}
\put(100,90){\line(1,0){150}}
\put(100,120){\line(1,0){150}}
\put(100,150){\line(1,0){30}}
\put(100,180){\line(1,0){30}}
\put(100,210){\line(1,0){30}}
\put(100,0){\line(0,1){210}}
\put(130,0){\line(0,1){210}}
\put(250,0){\line(0,1){120}}
\put(220,0){\line(0,1){120}}
\put(190,0){\line(0,1){120}}
\put(160,0){\line(0,1){120}}
\put(175,0){\makebox(0,0){\Tiny{{$\bullet$}}}}
\put(115,48){\makebox(0,0){\Tiny{{$\vdots$}}}}
\put(235,15){\makebox(0,0){\Tiny{{$m$}}}}
\put(235,75){\makebox(0,0){\Tiny{{$2$}}}}
\put(235,48){\makebox(0,0){\Tiny{{$\vdots$}}}}
\put(235,105){\makebox(0,0){\Tiny{{$1$}}}}
\put(115,15){\makebox(0,0){\Tiny{{$n$}}}}
\end{picture}
\caption{\footnotesize{Pyramid for $(q^m,1^s)$}} 
\label{Fig:sl_n}
\end{minipage} \hfill \begin{minipage}[l]{.46\linewidth}
\setlength\unitlength{0.0125cm} 
{\tiny 
\begin{picture}(120,210)
\put(100,0){\line(1,0){150}}
\put(100,30){\line(1,0){150}}
\put(100,60){\line(1,0){150}}
\put(100,90){\line(1,0){150}}
\put(100,120){\line(1,0){30}}
\put(100,150){\line(1,0){30}}
\put(100,180){\line(1,0){30}}
\put(100,210){\line(1,0){30}}
\put(100,0){\line(0,1){210}}
\put(130,0){\line(0,1){210}}
\put(250,0){\line(0,1){90}}
\put(220,0){\line(0,1){90}}
\put(190,0){\line(0,1){90}}
\put(160,0){\line(0,1){90}}
\put(175,0){\makebox(0,0){\Tiny{{$\bullet$}}}}
\put(235,75){\makebox(0,0){\Tiny{{$1$}}}}
\put(235,45){\makebox(0,0){\Tiny{{$2$}}}}
\put(235,15){\makebox(0,0){\Tiny{{$3$}}}}
\put(205,75){\makebox(0,0){\Tiny{{$4$}}}}
\put(205,45){\makebox(0,0){\Tiny{{$5$}}}}
\put(205,15){\makebox(0,0){\Tiny{{$6$}}}}
\put(175,75){\makebox(0,0){\Tiny{{$7$}}}}
\put(175,45){\makebox(0,0){\Tiny{{$8$}}}}
\put(175,15){\makebox(0,0){\Tiny{{$9$}}}}
\put(145,75){\makebox(0,0){\Tiny{{$10$}}}}
\put(145,45){\makebox(0,0){\Tiny{{$11$}}}}
\put(145,15){\makebox(0,0){\Tiny{{$12$}}}}
\put(115,195){\makebox(0,0){\Tiny{{$13$}}}}
\put(115,165){\makebox(0,0){\Tiny{{$14$}}}}
\put(115,135){\makebox(0,0){\Tiny{{$15$}}}}
\put(115,105){\makebox(0,0){\Tiny{{$16$}}}}
\put(115,75){\makebox(0,0){\Tiny{{$17$}}}}
\put(115,45){\makebox(0,0){\Tiny{{$18$}}}}
\put(115,15){\makebox(0,0){\Tiny{{$19$}}}}
\end{picture}}
\caption{\footnotesize{Pyramid for $(5^3,1^4)$}} 
\label{Fig:sl-ex}
\end{minipage}
\end{figure}
\end{center}
} 
From the pyramid, we easily see that 
\begin{align*}
\g^{0}_\Gamma& =\left\lbrace  
\on{diag}(x_1,\ldots,x_{q-1},y) \in \sl_n \colon 
x_i \in  \mf{gl}_m,\, y \in \mf{gl}_{m+s} \right\rbrace & \\
& \cong \C^{q-1} \times (\mf{sl}_m)^{q-1} \times \mf{sl}_{m+s},
\end{align*}
and 
\begin{align*} 
\g^\natural  & =  \left\{ 
\on{diag}(x,\ldots,x,y)  \in \sl_n \colon 
x \in \mf{gl}_m, \, 
y = \on{diag}( x,x'),\, x' \in \mf{gl}_{s} \right\} \subset \g^{0}_\Gamma &\\
& \cong \begin{cases} 
\C \times \mf{sl}_m \times \mf{sl}_s & \text{ if } s\not=0,\\
 \mf{sl}_m & \text{ if } s=0.  
 \end{cases} & 
\end{align*} 
First, 
pick $t =\on{diag}(x,\ldots,x,y) 
\in \g^\natural_1 \cong \sl_m$,  
with $x=\on{diag}(1,-1,0 ,\ldots,0) \in \mf{gl}_m$ and 
$y = \on{diag}(x,x')\in \mf{gl}_{m+s}$ 
with 
$x'=0$. 
We have 
\begin{align*}
k (t|t)_\g  + (\kappa_\g(t,t)- \kappa_{\g^0_\Gamma}(t,t))/2  
= 2 q k  + 2( q n - q m - s) ,& 
\end{align*}
and $ (t|t)_1^\natural  = 2 $, 
whence $k^\natural_1 = q k +  q n - n.$
This terminates the case $s=0$. 

Assume now $s\not=0$. 
Pick $t= \on{diag}(x,\ldots,x,y) 
 \in \g^\natural_0$, 
with $x=\on{diag}(-s,\ldots,-s) \in \mf{gl}_m$  
and $y =\on{diag}(x,x') \in \mf{gl}_{m+s}$ 
with $x'=\on{diag}\left(qm,\ldots,q m\right)$.  
The projection of $t$ onto the semisimple part of $\g_\Gamma^0$ 
is ${\rm diag}(0,\ldots,0,z)$ with 
$$z = {\rm diag}\left(-\frac{sn}{m+s},\ldots,-\frac{sn}{m+s}, 
\frac{mn}{m+s},\ldots,\frac{mn}{m+s}\right) \in \sl_{m+s}.$$
From this, we get 
\begin{align*}
& k (t|t)_\g  + (\kappa_\g(t,t)
- \kappa_{\g^0_\Gamma}(t,t))/2
= smn (qk + qn -n). &
\end{align*}
Choose $(~|~)_0$ so that $(t|t)_0= smn $, 
whence $k^\natural_0 = q k + q n - n$.  

Pick finally $t =\on{diag}(x,\ldots,x,y) 
 \in \g^\natural_2 \cong \sl_s$,  
where $x=0 \in \mf{gl}_m$ and $y = \on{diag}(x,x')\in \mf{gl}_{m+s}$ 
with $x'=\on{diag}(1,-1,0,\ldots,0)$. 
We have 
\begin{align*}
k (t|t)_\g  + (\kappa_\g(t,t)
- \kappa_{\g^0_\Gamma}(t,t))/2 
=  2 k + (4 n- 4(m+s))/2
\end{align*}
and $(t|t)_2^\natural=2$, 
whence $k^\natural_2 = k+ q m  - m.$

\smallskip

(2) Assume now that $\bs\mu=(q^m,(q-1)^2)$. 
Here, we easily 
see that the embedding $\g^\natural   \hookrightarrow 
\g^0_\Gamma$ is as follows. 
We have 
\begin{align*}
\g_\Gamma^0 & = \left\{{\rm diag}(x,y_1,\ldots,y_{q-1})  \in \sl_n \colon x \in \mf{gl}_{m}, \, y_i  
\in \mf{gl}_{m+2} \right\} & \\
& \cong  \C^{q-1} \cong (\sl_{m}) \times (\sl_{m+2})^{q-1}, & 
\end{align*}
\begin{align*} \g^\natural   &= \left\{ 
{\rm diag}(x,y,\ldots, y) \in \sl_n \colon x \in \mf{gl}_m, \, y = \on{diag}(z,x) 
 \in \mf{gl}_{m+2}, \, z \in \mf{gl}_2 \in  \right\} & \\
& \cong \C \times \sl_m \times \sl_2. &
\end{align*}
We compute $k_0^\natural,k_1^\natural,k_2^\natural$ 
similarly as the previous case. So we omit the details. 

\smallskip

(3) Finally, assume that $\bs\mu=(q^m,s)$, with $m \ge 0$ and $0 < s \le q-1$. 
Here, we easily 
see that the embedding $\g^\natural   \hookrightarrow 
\g^0_\Gamma$ is as follows. 
We have 
\begin{align*}
\g^0_\Gamma & = \left\{{\rm diag}(x_1,\ldots,x_s, 
y_1,\ldots,y_{q-s}) \in \sl_n \colon x_i \in \mf{gl}_{m+1}, \, y_j  
\in \mf{gl}_{m} \right\} & \\
& \cong (\mf{gl}_{m+1})^s \times (\mf{gl}_m)^{q-s} & 
\end{align*}
and 
\begin{align*} \g^\natural  &= \left\{ 
{\rm diag}(x,\ldots,x, y,\ldots,y)  \in \sl_n \colon y \in \mf{gl}_m, \, x= 
\on{diag}(\lam , y) \in \mf{gl}_{m+1}\right\} & \\
& \cong \C \times \sl_m. &
\end{align*}
We compute $k_0^\natural,k_1^\natural$ 
as in the first case. So we omit the details. 

The last assertion of the lemma is then easy to verify. 
Indeed, the conditions $k_0^\natural=0$ (when $k_0^\natural$ appears) implies $p=n$. 
Then $k_1^\natural=0$ and $k_2^\natural= -s+s/q$ (when  $k_2^\natural$ appears), 
which is admissible for $\sl_s$. 
If $\g_0^\natural =0$, then $\bs\mu=(q^m)$ and 
$k_1^\natural = p-n = -m + p-m(q-1)$ which is admissible for $\sl_m$ 
since $p-m(q-1) \ge m$ is equivalent to the condition $p \ge n$. 
\end{proof}

{\tiny 
\begin{table}
\begin{tabular}{llll}
\hline
&&& \\[-0.5em]
 $\bs\mu$  
& $\g^\natural=\bigoplus_i \g^\natural_i$  
& $k_i^\natural$ 
& conditions \\
&&& \\[-0.5em]
\hline 
&&& \\[-0.5em]
$(q^{m},{s})$ & 
 $ \C \times\sl_{m}$ &  $k^\natural_0 = q k + q n  - n$ &   $1 \le {s} \le q-1$ \\ 
&& $k^\natural_1 = q k + q n  - n$ &    \\[0.5em]
  $(q^{m})$ & $\mf{sl}_{m}$ &  $k^\natural_1 = q k + q n  - n$  & \\[0.5em]
$(q^m,1^s)$  
& $\C \times \mf{sl}_m \times \mf{sl}_s$ 
&  $k^\natural_0 = q k + q n  - n$  
& $s >0$  \\
&& $k^\natural_1 = q k + q n  - n$  &  \\
&& $k^\natural_2 = k +  q m - m$ &   \\[0.2em]
$(q^m,(q-1)^2)$  
& $\C \times \mf{sl}_m \times \mf{sl}_2$
&  $k^\natural_0 = q k + q n  - n$  
&   \\
&& $k^\natural_1 = q k + q n  - n$  &  \\
&& $k^\natural_2 = (q-1) (k + n -m-2) $ &   \\[0.2em]
\hline
\end{tabular} \\[1em] 
\caption{\footnotesize{Centralisers of some $\sl_2$-triples $(e,h,f)$ in $\sl_{n}$, 
with $f \in \O_{\bs\mu}$}} 
\label{Tab:A_Slodowy}
\end{table}
}

\begin{Rem}
Besides $\g$, $f$ and $k$, the construction of $H^0_{DS, f}(L_k(\g))$ takes an auxiliary choice of good grading on $\g$ compatible with $f$. For all types other than $\sl_n$ we use the Dynkin grading, but in this section we sometimes work with good gradings different than the Dynkin grading. While it is known that $H^0_{DS, f}(L_k(\g))$, as a vertex algebra, is independent of the choice of good grading, its conformal structure, and therefore in principle its asymptotic datum, can depend on this choice. Since Theorem \ref{Th:main} requires self-duality, and by Proposition \ref{Pro:Dynkin-self-dual} self-duality holds for the Dynkin grading, it is important to verify that the asymptotic datum is independent of the choice of good grading in the cases relevant for us. For $\G$ this is evident from the formula (see Proposition \ref{Pro:asymptotic_data_H_DS}). As for $\A$, for the pairs $(k, f)$ that arise in our search for collapsing levels for $\sl_n$ 
the relationship between the orbits $\O_k$ and $G. f$ is such that the independence 
of $\A$ on the choice of good grading can be seen via combinatorial arguments. 
Namely, we have the following lemma. 
\end{Rem}

\begin{Lem}
\label{Lem:asymptotic_dimension_good_gradings}
Assume that $\bs{\mu}$ is a partition of type (a), (b), (c) as in 
Lemma~\ref{Lem:choice_of_partitions}, 
and pick $f \in \O_{\bs{\mu}} \subset \overline{\O}_k$.  
Then the expression given for $\A_{H^0_{DS, f}(L_k(\g))}$ in Proposition~\ref{Pro:asymptotic_data_H_DS} is independent of the choice of good grading $\Gamma$.
\end{Lem}

\begin{proof}
The expression in question is the product of a term manifestly independent of $\Gamma$ with
\begin{align}\label{eq:Gamma.dep.part}
\A' = \frac{1}{2^{\frac{|\Delta_\Gamma^{1/2}|}{2}} q^{|\Delta_{\Gamma, +}^0|}} \prod_{\alpha \in \Delta_+ \backslash \Delta_{\Gamma, +}^0} 2 \sin{\frac{\pi (x_\Gamma^0 | \alpha)}{q}}.
\end{align}

Now we consider $k = -n + p/q$ an admissible level for $\g = \mathfrak{sl}_n$ and the Hamiltonian reduction $H^0_{DS, f}(L_k(\g))$ for some $f \in \overline{\O}_k$. 

We first consider the partitions $(q^m, s)$ 
of $n$. 

The pyramids associated with $\bs\lambda = (q^m, s)$ all consist of a $q \times m$ rectangle surmounted by a horizontal strip of dimension $s \times 1$. In Figures \ref{Fig:asdim-pyr-1} and \ref{Fig:asdim-pyr-2} we illustrate two pyramids associated with $n=16$, $q=5$. In general there are $2(q-s)+1$ distinct pyramids associated with $\bs\lambda$.
{\tiny
\begin{center}
\begin{figure}[ht]
\setlength\unitlength{0.0175cm}
\begin{minipage}[l]{.46\linewidth} 
\begin{picture}(150,160)
\put(100,0){\line(1,0){100}}
\put(100,20){\line(1,0){100}}
\put(100,40){\line(1,0){100}}
\put(100,60){\line(1,0){100}}

\put(100,0){\line(0,1){60}}
\put(120,0){\line(0,1){60}}
\put(140,0){\line(0,1){60}}
\put(160,0){\line(0,1){60}}
\put(180,0){\line(0,1){60}}
\put(200,0){\line(0,1){60}}

\put(100,60){\line(0,1){20}}
\put(120,60){\line(0,1){20}}
\put(100,80){\line(1,0){20}}

\put(110,70){\makebox(0,0){{\Tiny{$i$}}}}

\end{picture}
\caption{\footnotesize{An even good grading}} 
\label{Fig:asdim-pyr-1}
\end{minipage}
\hfill
\begin{minipage}[l]{.46\linewidth} 
\begin{picture}(150,160)
\put(100,0){\line(1,0){100}}
\put(100,20){\line(1,0){100}}
\put(100,40){\line(1,0){100}}
\put(100,60){\line(1,0){100}}

\put(100,0){\line(0,1){60}}
\put(120,0){\line(0,1){60}}
\put(140,0){\line(0,1){60}}
\put(160,0){\line(0,1){60}}
\put(180,0){\line(0,1){60}}
\put(200,0){\line(0,1){60}}

\put(110,60){\line(0,1){20}}
\put(130,60){\line(0,1){20}}
\put(110,80){\line(1,0){20}}

\put(120,70){\makebox(0,0){{\Tiny{$i$}}}}
\end{picture}
\caption{\footnotesize{An odd good grading}} 
\label{Fig:asdim-pyr-2}
\end{minipage}
\end{figure}
\end{center}}

We rewrite \eqref{eq:Gamma.dep.part} as
\begin{align*}
\A' =  \prod_{\stackrel{i < j}{x_i=x_j}} \frac{1}{q} \cdot \prod_{\stackrel{i < j}{|x_i-x_j|=1}} \frac{1}{2^{1/2}} \cdot \prod_{\stackrel{i < j}{|x_i-x_j| \ge 1}} 2\sin \frac{\pi |x_i - x_j|}{q}.
\end{align*}
To show that $\A'$ is independent of $\Gamma$ we first note that if boxes $i$ and $j$ lie both in the $q \times m$ rectangle or else both in the $s \times 1$ horizontal strip, then the contribution of the pair $(i, j)$ to the product above is the same for all choices of $\Gamma$. So we may ignore the contributions of such pairs.

Now we let $i$ be a box in the $s \times 1$ strip and we consider the contribution to the product as $j$ runs along one of the rows of the $q \times m$ rectangle. In the configurations depicted by 
Figures~\ref{Fig:asdim-pyr-1} and \ref{Fig:asdim-pyr-2} this contribution is, respectively,
\begin{align*}
\frac{1}{q} \prod_{j=1}^{n-1} 2\sin{\frac{j \pi}{q}} \qquad \text{or} \qquad \frac{1}{2} \prod_{j=1}^{n} 2\sin{\frac{(j-1/2) \pi}{n}}.
\end{align*}
Indeed for any of the $q$ possible positions of box $i$ aligned with the $q \times m$ rectangle we obtain the first of these products simply because $\sin{((q-j)\pi/q)} = \sin{(j \pi/q)}$. Similarly all remaining positions of box $i$ yield the second product. But finally both products are equal to $1$ due to identities \eqref{eq:sin_formula} and \eqref{eq:sin_formula.odd}, and so all products are equal. Therefore the value of $\A'$ does not depend on $\Gamma$. 

Similar arguments establish independence of $\A$ on $\Gamma$ for nilpotent elements 
associated with partitions of type $(q^m, 1^s)$ or $(q^m,(q-1)^2)$. 
\end{proof}

We are now in a position to state our results on collapsing levels for $\sl_n$. 
We have
\begin{align}
\W_k(\g,f)\cong H_{DS,f}^0(L_k(\g))
\label{eq:simple}
\end{align}
for an admissible level $k$ and $f\in \overline{\mathbb{O}}_k$,
since any nilpotent element of $\sl_n$ admits an even good grading.
Moreover, 
$\W_k(\g,f)$ is rational if $f\in \O_k$ (\cite{AEkeren19}).

First, we consider the case when $\W_k(\g,f)$ is rational. 
\begin{Th}  
\label{Th:main_sl_n-1}
Assume that $k = - h^{\vee}_{\g} + p/q = -n +p/q$ 
is admissible for $\g=\sl_n$. 
Pick a nilpotent element 
$f \in\O_k$, so that $\W_k(\g,f)$ is rational. 
Then
$k$ is collapsing if and only if
$n\equiv 0,\pm 1\pmod{q}$
and 
$p=\begin{cases} h^{\vee}_{\g}&\text{if }n\equiv \pm 1\pmod{q}\\
 h^{\vee}_{\g}+1&\text{if }n\equiv 0\pmod{q}.
\end{cases}$ \newline
Furthermore
\begin{enumerate}
\item If $n\equiv \pm 1\pmod{q}$,
then
$$\W_{-n +n/q}(\mf{sl}_n,f)  \cong \C.$$ 
\item  If $n\equiv 0\pmod{q}$,
$$\W_{-n + (n+1)/q}(\mf{sl}_n,f)  \cong L_{1}(\mf{sl}_{\widetilde m}),$$
where $\widetilde m$ 
is defined by \eqref{eq:Ok_sl_n}. 
\end{enumerate}
\end{Th}
\begin{proof} Fix a nilpotent element 
$f \in \O_k=\O_{(q^{m},s)}$, with $m:=\widetilde m$ 
and $s:=\widetilde s$ in the notation of \eqref{eq:Ok_sl_n}.

(a) Case $ s \not=0$. 

According to Table \ref{Tab:A_Slodowy}, 
we have $\g^\natural \cong \C \times \sl_{ m}$, 
and $k_0^\natural =  k_1^\natural = qk + q n-n = p-n$. 
If $k$ is collapsing, then necessarily $k_0^\natural =0$, 
whence $p=n$. 
Assume from now that $p=0$, whence 
$k_0^\natural =  k_1^\natural=0$. 
By \eqref{eq:simple},  
if $k$ is collapsing, 
then $\W_k(\sl_n,f)\cong \C$ and 
the asymptotic growth $\G_{\W_k(\sl_n,f)}=\G_{H^0_{DS,f}(L_k(\sl_n)}$ 
must be $0$. 
By Corollary~\ref{Co:asymptotic_data_L}, 
we have $\G_{L_{k^\natural}(\g^\natural))} =0$ while by Proposition \ref{Pro:asymptotic_data_H_DS}, 
\begin{align*}
\G_{H^0_{DS,f}(L_k(\sl_n))} & = \dim \g^{f} - \dfrac{h_\g^\vee  \dim \sl_n}{p q} 
= ( s-1) \left(\dfrac{q-( s+1)}{q}\right)
\end{align*}
since 
$\dim \g^{f} = { m}^2(q-{ s}) +({ m}+1)^2 { s}-1$. 
Therefore, $\G_{H^0_{DS,f}(L_k(\sl_n))} =0$ if and only if ${ s}=1$ or ${ s}=q-1$. 
The case ${ s}=1$ will be dealt with in Theorem \ref{Th:main_sl_n-2} with $s=1$. 
We consider here only the case ${ s}=q-1$. 

Our aim is to show that  $\W_{-n + n/q}(\mf{sl}_n,f)  \cong \C$, 
for $f$ corresponding 
to the partition $(q^{ m},q-1)$. 
By Proposition \ref{Pro:asymptotics-and-collapsing} and the above computation it suffices to show 
that the asymptotic dimension of $H^0_{DS,f}(L_k(\sl_n))$ is~$1$. 

Let $\g=\bigoplus_i \g_\Gamma^{i}$ be the even good grading for $f$ 
corresponding to the left-adjusted pyramid of shape $(q^{{ m}}, s)$ 
as in the proof of Lemma \ref{Lem:k_natural-type-A}. 
By Proposition \ref{Pro:asymptotic_data_H_DS} and Lemma \ref{Lem:main_identities} (1), 
we get 
\begin{align}
\label{eq:Ampl-1} 
\A_{H^0_{DS,f}(L_k(\sl_n))} & = 
\dfrac{1}{q^{|\Delta_{\Gamma,0}^+|} 
q^{\frac{n-1}{2}} } 
\prod_{\alpha \in \Delta_+ \setminus \Delta_{\Gamma,+}^{0}} 
2 \sin \dfrac{\pi (x_\Gamma^0 |\alpha)}{q} , & 
\end{align}
with $n=q  m + q-1$ and 
$|\Delta_{\Gamma,0}^+|= \dfrac{{ m}(n-1)}{2}$. 
Indeed, by Lemma \ref{Lem:asymptotic_dimension_good_gradings}, the asymptotic dimension does not depend 
on the good grading for such an $f$.  
For $i \in \{1,\ldots,n-1\}$, let $x_i$ be the $x$-coordinate of the box labelled $i$ 
in the pyramid. 
Note that for $j \in\Z$, we have 
\begin{align}
\label{eq:cardinality_pyramid-sln}
\# \{\alpha \in \Delta_+ \colon (x_\Gamma^0|\alpha) = j \}= 
\# \{ (i,l) \in \{1,\ldots,n-1\} \colon i \le l,\, |x_i - x_{l} |/2 = j \}.
\end{align}
In this way, using \eqref{eq:cardinality_pyramid-sln} 
and the identity \eqref{eq:sin_formula2}, 
we obtain that
\begin{align*}
\prod_{\alpha \in \Delta_+ \setminus \Delta_{\Gamma,+}^{0}} 
2 \sin  \dfrac{\pi (x_\Gamma^0 |\alpha)}{q} & 
= \prod_{j=1}^{q-1} \left( 2 \sin \dfrac{j \pi}{q}\right)^{({ m}+1)(n-j({ m}+1))}  
 =q^{\frac{q({ m}+1)^2}{2}-({ m}+1)} .&
\end{align*}
Combining this with \eqref{eq:Ampl-1}, we verify that $\A_{H^0_{DS,f}(L_k(\sl_n)}=1$, 
as desired. This concludes this case. 

(b) Case ${ s} =0$. 

According to Table \ref{Tab:A_Slodowy}, we have 
$\g^\natural  \cong \sl_{ m}$ 
and $k^\natural=qk +qn -n$. 
If $k$ is collapsing, then by \eqref{eq:simple},  
we must have $\G_{L_{k^\natural}(\sl_{{ m}})}=\G_{H^0_{DS,f}(L_k(\sl_n))}$.  
We have 
\begin{align*}
\G_{L_{k^\natural}(\sl_{{ m}})} & = ({ m}^2-1) 
\left(1-\dfrac{{ m}}{p-{ m}(q-1)}\right), \\
\G_{H^0_{DS,f}(L_k(\sl_n))} & = q{ m}^2-1-\dfrac{n(n^2-1)}{p q}
\end{align*}
since 
$\dim \g^{f} = { m}^2 q -1$. 
Solving the equation 
$$ ( m^2-1) \left(1-\dfrac{{ m}}{p-{ m}(q-1)}\right)= q{ m}^2-1-\dfrac{n(n^2-1)}{p q}$$
with unknown $p$  
we obtain that $p$ must be either equal to $n+1$ or $n-1$. 
Only the case $p=n+1$ is greater than $h_\g^\vee=n$.

From now, it is assumed that $p=n+1$, whence $k_1^\natural=1 = -{ m} + ({ m}+1)/1$.  
We apply Proposition \ref{Pro:asymptotics-and-collapsing} to prove that $k$ is collapsing. 
It is enough to show that 
$L_{1}(\mf{sl}_{ m})$ and $H^0_{DS,f}(L_k(\sl_n))$ share 
the same asymptotic dimension. 
By Corollary~\ref{Co:asymptotic_data_L} and Lemma \ref{Lem:main_identities} (2), \begin{align}
\label{eq:Amp:L1}
\A_{L_{1}(\mf{sl}_{ m})} & = 
\frac{1}{\sqrt{m}} 
\end{align} 
On the other hand, by Proposition \ref{Pro:asymptotic_data_H_DS} and Lemma \ref{Lem:main_identities} (2), 
\begin{align}
\label{eq:Amp:W-1a}
& \A_{H_{DS,f}^0(L_{-n + (n+1)/q}(\mf{sl}_n))}  = \dfrac{1}{q^{|\Delta_{\Gamma,+}^0|} 
q^{\frac{n-1}{2}} n^{\frac{1}{2}}} 
\prod\limits_{\alpha\in \Delta_+ \setminus\Delta_{\Gamma,+}^0} 
2 \sin \dfrac{\pi (x_\Gamma^0|\alpha) }{q} ,&
\end{align} 
with $n=q  m$ and  
$|\Delta_{\Gamma,+}^0|= q{ m}({ m}-1)/2$. 
Moreover, computing the cardinality of the sets 
$\{\alpha \in \Delta_+ \colon (x_\Gamma^0|\alpha) = j \}$ 
as in the previous case, 
we obtain by \eqref{eq:sin_formula2} that 
\begin{align*}
\prod\limits_{\alpha\in \Delta_+ \setminus\Delta_{\Gamma,+}^0} 
2 \sin\dfrac{\pi (x_\Gamma^0|\alpha) }{q} = 
q^{qm^2/2} . 
\end{align*} 
Combining this with \eqref{eq:Amp:W-1a}, we get that
$\A_{H_{DS,f}^0(L_{-n + (n+1)/q}(\mf{sl}_n))} =\frac{1}{\sqrt{m}} $ 
as expected. This completes the proof. 
\end{proof}

We now consider the partitions $\bs\mu \in \P(n)$ 
as in Lemma \ref{Lem:choice_of_partitions-C} 
of type (b) and (c). 
This leads us to the following results.

\begin{Th} 
\label{Th:main_sl_n-2}
Assume that $k = - n + p/q$ 
is admissible  for $\g=\sl_n$. 
\begin{enumerate}
\item 
Pick a nilpotent element 
$f \in \overline{\O}_k$ corresponding to the partition $(q^m,1^s)$, 
with $m\ge 0$ and $s >0$. 
Then $k$ is collapsing if and only if $p=n=h^{\vee}_{\sl_n}$. 
Moreover, 
$$\W_{-n + n/q}(\mf{sl}_n,f)  \cong L_{-s+s/q}(\mf{sl}_s).$$
\item Assume that $\widetilde s =q-2$ and pick a nilpotent element 
$f\in \overline{\O}_k $ corresponding to the partition 
$(q^{m},(q-1)^2)$, with $m=\widetilde m-1$ 
in the notation of \eqref{eq:Ok_sl_n}. 
Then $k$ is collapsing if and only if $p=n=h^{\vee}_{\g}$. 
Moreover, 
$$\W_{-n + n/q}(\mf{sl}_n,f)  \cong L_{-2+2/q}(\mf{sl}_2).$$
\end{enumerate}
\end{Th}

\begin{Rem}
\label{Rem:compatible_sl_n}
For $s=1$ in (1) the formula has to be understood as 
$$\W_{-n + n/q}(\mf{sl}_n,f)  \cong \C.$$
Since for $s=1$,  $\G_{L_{-s+s/q}(\mf{sl}_s)} = \frac{(q-1)(s^2-1)}{q}= 0$ and 
$\A_{L_{-s+s/q}(\mf{sl}_s)}=q^{-(s^2-1)/2}=1$ (see the below proof), 
the formulas make sense 
and are compatible with Theorem~\ref{Th:main_sl_n-1}. 
\end{Rem}

\begin{proof}
As in the proof of Theorem \ref{Th:main_sl_n-1}, we let $\g=\bigoplus_i \g_\Gamma^{i}$ be the 
even good grading for $f$ corresponding to the left-adjusted pyramid associated with the partition 
of $f$.  

(1) Fix a nilpotent element 
$f \in\overline{\O}_k$ corresponding to the partition $(q^m,1^s)$. 
According to Table \ref{Tab:A_Slodowy}, we have 
$\g^\natural  \cong \C\times \sl_{m} \times \sl_{s}$,  
$k_0^\natural=k_1^\natural = p-n$ and 
$k_2^\natural  = k+qm - m$. 
If $k$ is collapsing, then necessarily $k_0^\natural=0$, that is, $p=n$. 
We assume from now on that $p=n$.  
Hence $k_2^\natural  = k + qm  - m  = - s +s/q$, 
which is an admissible level for $\sl_s$.  
By Proposition~\ref{Pro:asymptotic_data_H_DS}, we get 
\begin{align*}
\G_{H_{DS,f}^0(L_{-n + n/q}(\mf{sl}_n))} & 
 =   \left(1 -\dfrac{1}{q}\right) (s^2-1)  =\G_{L_{-s+s/q}(\mf{sl}_s)} & 
\end{align*} 
since $\dim \g^f = (m+s)^2 + (q-1)m^2-1$.  
By Corollary~\ref{Co:asymptotic_data_L} and Lemma \ref{Lem:main_identities} (1), 
we have 
\begin{align*}
\A_{L_{-s+s/q}(\mf{sl}_s)} & = 
\dfrac{1}{q^{s(s-1)/2} q^{(s-1)/2}}
= \dfrac{1}{q^{(s^2-1)/2}} .& 
\end{align*}
On the other hand, by Proposition \ref{Pro:asymptotic_data_H_DS}, we have 
\begin{align}
 \label{eq:Amp-sln-2}
& \A_{H_{DS,f}^0(L_{-n+n/q}(\mf{sl}_n))}  = 
 \dfrac{1}{q^{| \Delta_{\Gamma,+}^0 |} 
q^{(n-1)/2} } 
\prod\limits_{\alpha\in \Delta_+ \setminus \Delta_{\Gamma,+}^0 } 
2 \sin \left(\dfrac{\pi (x_\Gamma^0|\alpha) }{q}\right) ,& 
\end{align} 
with $|\Delta_{\Gamma,+}^0 | =\frac{(m+s)(m+s-1)}{2}+ \frac{m(q-1)(m-1)}{2}$. 
Indeed, by Lemma \ref{Lem:asymptotic_dimension_good_gradings}, the asymptotic dimension does not depend 
on the good grading for such an $f$. 
Using the left-adjusted pyramid of shape $(q^m,1^s)$, we easily see that 
\begin{align}
\label{eq:sin_Delta_+-2}
\prod\limits_{\alpha\in \Delta_+\setminus \Delta_{\Gamma,+}^0 } 
2 \sin \left(\dfrac{\pi (x_\Gamma^0|\alpha) }{q}\right) 
= & 
\prod\limits_{j=1}^{q-1} 
\left( 2 \sin
\dfrac{j \pi}{q} \right)^{m^2(q-j)+sm} =  q^{qm^2/2+sm} & 
\end{align}
using \eqref{eq:sin_formula2}.  
Combining \eqref{eq:Amp-sln-2} and \eqref{eq:sin_Delta_+-2}, we 
conclude that  
$$\A_{H_{DS,f}^0(L_{-n+n/q}(\mf{sl}_n))}
= \dfrac{1}{q^{(s^2-1)/2}} =\A_{L_{-s+s/q}(\mf{sl}_s)},$$
as desired. By Proposition \ref{Pro:asymptotics-and-collapsing} it follows that $k$ 
is collapsing. 

(2) Fix a nilpotent element 
$f \in\overline{\O}_k$ corresponding to the partition $(q^m,(q-1)^2)$. 
According to Table \ref{Tab:A_Slodowy}, we have 
$\g^\natural  \cong \C \times \sl_{m} \times \sl_{2}$,  
$k_0^\natural=k_1^\natural =  p-n$ and 
$k_2^\natural  =(q-1)( k+n-m-2)$. 
If $k$ is collapsing, then necessarily $k_0^\natural=0$, that is, $p=n$. 
We assume from now on that $p=n$.  
Hence $k_2^\natural  = (q-1)(n/q-m-2)=-2+2/q$, 
which is an admissible level for $\sl_2$. 
Note that $q$ is odd since $(q,n)=1$ 
and $n=qm +2(q-1)$.   
By Proposition~\ref{Pro:asymptotic_data_H_DS}, we get  
\begin{align*}
\G_{H_{DS,f}^0(L_{-n + n/q}(\mf{sl}_n))} & 
 =  3\left(1-\frac{1}{q}\right)  =\G_{L_{-2+2/q}(\mf{sl}_2)} & 
\end{align*} 
since $\dim \g^f = m^2+(m+2)^2(q-1)-1$.  
Moreover, from (1) we know that 
\begin{align*}
\A_{L_{-2+2/q}(\mf{sl}_2)} = \dfrac{1}{q^{3/2}} . 
\end{align*}
On the other hand, by Proposition \ref{Pro:asymptotic_data_H_DS} 
and Lemma \ref{Lem:asymptotic_dimension_good_gradings}, we have 
\begin{align*}
& \A_{H_{DS,f}^0(L_{-n+n/q}(\mf{sl}_n))}  = 
 \dfrac{1}{q^{| \Delta_{\Gamma,+}^0 |} 
q^{(n-1)/2} } 
\prod\limits_{\alpha\in \Delta_+ \setminus \Delta_{+}^0 } 
2 \sin \left(\dfrac{\pi (x_\Gamma^0|\alpha) }{q}\right) ,& 
\end{align*} 
with $|\Delta_{\Gamma,+}^0 | =\frac{m(m-1)}{2}+\frac{(m+1)(m+2)(q-1)}{2}$.  
Using the left-adjusted pyramid of shape $(q^m,(q-1)^2)$, we easily see that 
\begin{align*}
\prod\limits_{\alpha\in \Delta_+\setminus \Delta_{+}^0 } 
2 \sin \left(\dfrac{\pi (x^0|\alpha) }{q}\right) 
= & 
\prod\limits_{j=1}^{q-1} 
\left( 2 \sin
\dfrac{j \pi}{q} \right)^{m(m+2) + (m+2)^2 (q-j-1)}  . & 
\end{align*}
As in the previous cases, we 
conclude using \eqref{eq:sin_formula} and \eqref{eq:sin_formula2} that 
$$\A_{H_{DS,f}^0(L_{-n+n/q}(\mf{sl}_n))}=\A_{L_{-2+2/q}(\mf{sl}_2)},$$
as desired. By Proposition \ref{Pro:asymptotics-and-collapsing} it follows that $k$ 
is collapsing. 
\end{proof}

\begin{Rem}As has been observed 
in the above proof, if $k$ is collapsing for $f \in \O_{(q^{ m})}$, 
with $n = q  m$, then necessarily 
$k=-n+(n+1)/q$ or $k=-n+(n-1)/q$. 
Only the first case leads to an admissible level. 
However, one may ask whether the following holds\footnote{This statement 
has been recently established in \cite{Adamovic-et-al_New-collapsing}.}:
$$\W_{-n + (n-1)/q}(\mf{sl}_n,f)  \cong L_{-1}(\mf{sl}_{ m}).$$
(The two above vertex algebras have the same central charge.) 
\end{Rem}

We propose the following conjectural extension of Theorem \ref{Th:main_sl_n-2}. 

\begin{Conj}
\label{Conj:collapsing_sln}
Let  $f \in \sl_n$ be a nilpotent element associated with  
a partition $(q^m, \bs\nu)$,  where 
$1<m\le \widetilde m$ in the notation of \eqref{eq:Ok_sl_n},  
and $\bs\nu=(\nu_1,\ldots,\nu_t)$ is a partition 
of $s:= n - q m$ such that $\nu_1 < q$.  
Then 
$$\W_{-n+n/q}(\sl_n,f) \cong \W_{-s+s/q}(\sl_s,f'),$$ 
where $f'$ is a nilpotent element in $\sl_s$ associated 
with the partition $\bs\nu$. 
 \end{Conj} 
 Note that Conjecture~\ref{Conj:collapsing_sln} has been proven in the special 
case where $n=7$, $q=3$ and $s=4$ 
by Francesco Allegra \cite{Allegra19}. 
This case is in fact a particular case of Theorem~\ref{Th:main_sl_n-2} 
used with $f$ corresponding to the partitions $(3^2,1)$, $(3,2^2)$ and $(3,1^4)$. 
It seems that this conjecture has been stated  in \cite{XieYan2}.

The associated variety of 
$\W_{-n+n/q}(\sl_n,f)$ is 
$\Slo_{\O_{(q^{\widetilde m},\widetilde s)},f}$ while the associated variety of 
$\W_{-s+s/q}(\sl_s,f')$ is 
$\Slo_{\O_{(q^{\widetilde m -m},\widetilde s)},f'}$. 
These two nilpotent Slodowy slices are isomorphic 
by Lemma \ref{lem:erasing_row_type-A}. 
The following proposition gives further 
evidence for Conjecture \ref{Conj:collapsing_sln}. 
\begin{Pro}
In the above notations, 
the vertex algebras 
$\W_{-n+n/q}(\sl_n,f)$ and 
$\W_{-s+s/q}(\sl_s,f')$ have the same 
asymptotic growth.  
\end{Pro}

\begin{proof} 
Write $\bs{\mu}=(\mu_1,\ldots,\mu_r)$ the partition $(q^m,\bs\nu)$ corresponding to $f$. 
Notice that 
$\G_{\W_{-n+n/q}(\sl_n,f)}=\dim \g^f -\frac{n}{n q}  \dim \sl_n  = 
 \sum_{i=1}^{q}( \mu_i^*)^2- 1 -\frac{(n^2-1)}{q}$ 
 and 
$\G_{\W_{-s+s/q}(\sl_s,f')}=\dim \g^{f} - \frac{s}{s q} \dim \sl_s = 
 \sum_{i=1}^{q}( \mu_i^*-m)^2 - 1 -\frac{(s^2-1)}{q},$
 where $(\mu_1^*,\ldots,\mu_{q}^*)$ is the dual partition to $\bs\mu$. 
But we readily verify from the values of the $\mu_i$'s that the following identity holds: 
$ \sum_{i=1}^{q}( \mu_i^*)^2 - \frac{(n^2-1)}{q} 
 =  \sum_{i=1}^{q}( \mu_i^*-m)^2 - \frac{(s^2-1)}{q},$ 
 whence the lemma.
 \end{proof}

\section{Collapsing levels for types $\sp_n$ and $\so_n$} 
\label{sec:classical}
Let $n\in \Z_{> 0}$. 
We study in this section collapsing levels for $\sp_n$ and $\so_n$. 

\subsection*{Notations for $\sp_{n}$} 
We realise $\g=\mf{sp}_{n}$ as the set of $n$-size square matrices 
$x$ such that $ x^T J_n + J_n x =0$ where 
$J_n$ is the anti-diagonal 
matrix given by 
$$J_n := \begin{pmatrix} 
0 & U_{n/2} & \\
- U_{n/2} & 0
\end{pmatrix},$$
where for $m \in \Z_{\ge 0}$,  $U_m$ stands for 
the $m$-size square matrix 
with unit on the anti-diagonal.
For an $m$-size square matrix $x$, we write $\widehat{x}$ for the matrix 
$U_m x^T U_m$, where $x^T$ is the transpose 
matrix of $x$.  
Thus, 
\begin{align*}
\sp_n = \left\{ \begin{pmatrix} a &  b \\
c & - \hat{a} \\
\end{pmatrix} \colon a,b,c \in \mf{gl}_{n/2}, \, b =  \hat{b}, \, c=  \hat{c}\right\}. 
\end{align*}
Writing $e_{i,j}$ for the $i,j$-matrix unit as in the $\sl_n$ case, 
the following
matrices give a Chevalley basis for $\g$: 
\begin{align*}
\{e_{i,j}- e_{-j,-i}\}_{1 \le i,j \le n/2}\cup \{e_{i,-j} & + e_{j,-i},e_{-i,j}+ e_{-j,i} \}_{1 \le i<j \le n/2} \\
&\cup \{e_{k,-k},e_{-k,k}\}_{1 \le k \le n/2}.
\end{align*}
Let $\sigma_{i,j} \in\{\pm 1\}$ denote the $e_{i,j}$-coefficient of the unique element 
in the above basis that involves $e_{i,j}$.
The Killing form of $\g=\sp_n$ 
is given by $\kappa_\g(x,y)
= (n+2) {\rm tr}(xy)$ and $(x|y)_\g ={\rm tr}(xy)$. 

Set 
$$\P_{-1}(n) :=\{\bs{\lambda} \in \P(n)\colon \text{number of parts of size $\lambda$ is even, for each odd number $\lambda$}\},$$ 
By \cite[Theorem~5.1.3]{CMa}, nilpotent orbits of $\sp_{n}$ 
are parametrised by $\P_{-1}(n)$. 
For $\bs{\lambda}=(\lambda_{1},\dots ,\lambda_{r})\in \P_{-1}(n)$, 
we shall denote by 
$\O_{-1;\bs{\lambda}}$, 
or simply by $\O_{\bs{\lambda}}$ when there is no possible confusion,  
the corresponding nilpotent orbit of $\sp_{n}$. 
As in the case of $\sl_{n}$, if $\bs{\lambda}, \bs{\mu} \in \P_{-1} (n)$, then 
$\O_{-1;\bs{\mu}} \subset \overline{\O}_{-1;\bs{\lambda}}$ if and only if
$\bs{\mu} \leqslant \bs{\lambda}$. 

Given $\bs{\lambda}\in\P(n)$, there exists a unique 
$\bs{\lambda}^{-} \in \P_{-1}(n)$ 
such that $\bs{\lambda} ^{-} \le \bs{\lambda} $, 
and if $\bs{\mu}\in\P_{-1}(n)$ verifies 
$\bs{\mu} \le \bs{\lambda} $, then 
$\bs{\mu}\le \bs{\lambda} ^{-}$. 

\subsection*{Notations for $\so_n$}
We realise $\so_n$ as the set of $n$-size square matrices 
$x$ such that $ x^T K_n + K_n x =0$, where 
$K_n$ is the anti-diagonal 
matrix given by 
$$K_n := \begin{pmatrix} 
0 & 0 & U_{n/2}  \\
0 & 2 & 0\\
U_{n/2}  & 0 & 0
\end{pmatrix} 
\text{ if }n\text{ is odd}, \qquad K_n := \begin{pmatrix} 
0 & U_{n/2} \\
U_{n/2}  & 0 
\end{pmatrix} \text{ if }n \text{ is even}.$$ 
Thus we get that 
\begin{align*}
\so_n = \left\{ \begin{pmatrix} a & u & b \\
v & 0 & - \hat{u}\\
c & -\hat{v} & -\hat{a} \\
\end{pmatrix} \colon a,b,c \in \mf{gl}_{n/2}, \, 
b = - \hat{b}, \, c= - \hat{c}\right\} \text{ if }n\text{ is odd},
\end{align*}
\begin{align*}
\so_n  = \left\{ \begin{pmatrix} a &  b \\
c & -\hat{a} \\
\end{pmatrix} \colon a,b,c \in \mf{gl}_{n/2}, \, b = - \hat{b}, \, c= - \hat{c}\right\} \text{ if }n\text{ is even}. 
\end{align*}.
Writing $e_{i,j}$ for the $i,j$-matrix unit as in the $\sl_n$ case, 
we see that the following
matrices give a Chevalley basis for $\g$ 
(omitting the last family if $n$ is even): 
\begin{align*}
\{e_{i,j}- e_{-j,-i}\}_{1 \le i,j \le n/2}\cup \{e_{i,-j}- e_{j,-i},e_{-j,i}- e_{-i,j} \}_{1 \le i <j \le n/2}  \\
\cup \{2e_{k,0}- e_{0,-k},e_{0,k}-2 e_{-k,0}\}_{1 \le k \le n/2}.
\end{align*}
As in the $\mf{sp}_{n}$ case, let $\sigma_{i,j} \in\{\pm 1\}$ denote the 
$e_{i,j}$-coefficient of the unique element 
in the above basis that involves $e_{i,j}$.

The Killing form of $\g=\so_n$ 
is given by $\kappa_\g(x,y)= (n-2) {\rm tr}(xy)$ 
and $(x|y)_\g ={\rm tr}(xy)/2$. 
Set 
$$\P_{1}(n) :=\{\bs{\lambda} \in \P(n)\colon \text{number 
of parts of size $\lambda$ is even, for each even $\lambda$}\}.$$
By \cite[Theorems 5.1.2 and 5.1.4]{CMa}, 
nilpotent orbits of $\so_{n}$ 
are parametrised by $\P_1(n)$, with the exception that each 
{\em very even} 
partition $\bs{\lambda} \in\P_{1}(n)$ (i.e., $\bs{\lambda}$ has only even parts) 
corresponds to two nilpotent orbits.
For $\bs{\lambda}\in \P_1(n)$, not very even, we shall denote by 
$\O_{1;\bs\lambda}$, 
or simply by $\O_{\bs{\lambda}}$ when there is no possible confusion, 
the corresponding nilpotent orbit of $\so_n$. 
For very even $\bs{\lambda}\in \P_1(n)$, we shall denote by 
$\O_{1;\bs{\lambda}}^{I}$ 
and $\O_{1;\bs{\lambda}}^{I\!I}$ the two corresponding nilpotent orbits 
of $\so_n$.  
In fact, their union forms a single $O(n)$-orbit. 
Thus nilpotent orbits of $\mf{o}_{n}$ 
are parametrised by $\P_1(n)$. 

If $\bs{\lambda},\bs{\mu}\in\P_{1} (n)$, then 
$\overline{\O}_{1;\bs{\mu}}^{\bullet} \subsetneq
\overline{\O}_{1;\bs{\lambda}}^{\bullet}$ if and only if
$\bs{\mu} < \bs{\lambda}$, 
where $\O_{1;\bs{\lambda}}^{\bullet}$ is either 
$\O_{1;\bs{\lambda}}$, $\O_{1;\bs{\lambda}}^{I}$ or 
$\O_{1;\bs{\lambda}}^{II}$ according 
to whether $\bs{\lambda}$ is very even or not.

Given $\bs{\lambda} \in\P(n)$, there exists a unique 
$\bs{\lambda}^{+}
\in \P_{1}(n)$ such that $\bs{\lambda}^{+} \le \bs{\lambda} $, and if 
$\bs{\mu}\in\P_{1}(n)$ verifies $\bs{\mu} \le \bs{\lambda} $, then 
$\bs{\mu}\le \bs{\lambda} ^{+}$.

\begin{Def} \label{Def:eps_degeneration}
Assume that $\bs{\lam} \in \P_{\eps}(n)$, for $\eps \in \{\pm 1\}$. 
An {\em $\eps$-degeneration} of $\bs{\lam}$ is 
an element $\bs{\mu} \in \P_{\eps}(n)$ such that 
$\O_{\eps; \bs{\mu}}\subsetneq \overline{\O}_{\eps; \bs{\lam}}$, 
that is, $\bs{\mu} < \bs{\lam}$. 
A $\eps$-degeneration $\bs{\mu}$ of $\bs{\lam}$ is said to be {\em minimal} 
if $\O_{\eps;\bs{\mu}}$ is open in 
$ \overline{\O}_{\eps;\bs{\lam}} \setminus \O_{\eps;\bs{\lam}}$. 
\end{Def}

\subsection{Symplectic and orthogonal pyramids}
\label{sub:sympl_pyramids}
As in the $\sl_n$ case, there is a bijection between the set of good gradings 
of $\sp_n$ or $\so_n$ compatible with a given nilpotent element 
and the set of {\em some} pyramids of shape the corresponding partition 
of $\P_\eps(n)$, with $\eps \in \{\pm 1\}$. 
Such pyramids are called {\em symplectic pyramids} for $\sp_n$ 
($\eps=-1$) and {\em orthogonal pyramids} for $\so_n$ ($\eps=1$). 
For $\sp_n$ and $\so_n$, 
we will be only using symplectic or orthogonal Dynkin pyramids. 
This is a diagram consisting of $n$ boxes each of size $2$ units by
$2$ units drawn in the $xy$-plane. As in the $\sl_n$ case, the coordinates of a box are the coordinates of its
midpoint, and the row and column numbers of a box mean its $y$- and $x$-coordinate, respectively, but there are a few differences. 

Let us first explain what are the  
{symplectic} Dynkin pyramids  
of shape $\bs\lam \in \P_{-1}(n)$. 
The parts of $\bs\lam$ indicate the number of
boxes in each row, and the rows are added to the diagram 
so that 
we get a symmetric pyramid with respect to the point $(0,0)$. 
The only complication is that if some (necessarily even) part $\lam_i$ 
of $\bs\lam$ 
has odd multiplicity, then the first time a row of this length is added to the diagram it is split
into two halves, the right half is added to the next free row in the upper half plane in columns
$1, 3,\ldots,\lam_i - 1$ 
and the left half is added to the lower half plane in a centrally symmetric way.
We refer the exceptional rows arising in this way to as {\em skew rows}. 
The missing boxes in skew
rows are drawn as a box with a cross through it.
We number the boxes 
of the symplectic Dynkin pyramid 
with labels $1,\ldots,n/2,-n/2,\ldots,-1$ in such a way that $i$ and 
$-i$ appear in centrally symmetric boxes, for $i=1,\ldots,n/2$. 
As a rule, we will number the first $n/2$ boxes of symplectic Dynkin pyramids  
from top right to bottom left. 

For example, we represent in Figure \ref{Fig:sp_n-pyr1} 
and Figure \ref{Fig:sp_n-pyr2} the numbered 
symplectic Dynkin pyramids of shape $\bs\lam=(5^2,1^2)$ 
and $\bs\lam=(5^2,4,2)$, respectively.

{\tiny
\begin{center}
\begin{figure}[h]
\setlength\unitlength{0.0175cm}
\begin{minipage}[l]{.46\linewidth} 
\hspace{2cm}\begin{picture}(0,100)

\put(40,0){\line(1,0){20}}
\put(0,20){\line(1,0){100}}
\put(0,40){\line(1,0){100}}
\put(0,60){\line(1,0){100}}
\put(40,80){\line(1,0){20}}

\put(40,0){\line(0,1){80}}
\put(60,0){\line(0,1){80}}
\put(0,20){\line(0,1){40}}
\put(20,20){\line(0,1){40}}
\put(80,20){\line(0,1){40}}
\put(100,20){\line(0,1){40}}

\put(90,50){\makebox(0,0){{\Tiny{$1$}}}}
\put(90,30){\makebox(0,0){{\Tiny{$2$}}}}
\put(70,50){\makebox(0,0){{\Tiny{$3$}}}}
\put(70,30){\makebox(0,0){{\Tiny{$4$}}}}
\put(50,70){\makebox(0,0){{\Tiny{$5$}}}}
\put(50,50){\makebox(0,0){{\Tiny{$6$}}}}
\put(50,30){\makebox(0,0){{\Tiny{-$6$}}}}
\put(50,10){\makebox(0,0){{\Tiny{-$5$}}}}
\put(30,50){\makebox(0,0){{\Tiny{-$4$}}}}
\put(30,30){\makebox(0,0){{\Tiny{-$3$}}}}
\put(10,50){\makebox(0,0){{\Tiny{-$2$}}}}
\put(10,30){\makebox(0,0){{\Tiny{-$1$}}}}

\put(50,40){\makebox(0,0){{\Tiny{$\bullet$}}}}
\end{picture}
\caption{\footnotesize{Symplectic Dynkin pyramid of shape $(5^2,1^2)$}} 
\label{Fig:sp_n-pyr1}
\end{minipage} \hfill \begin{minipage}[l]{.46\linewidth}
\setlength\unitlength{0.0175cm}
\hspace{2cm}\begin{picture}(0,100)
\put(10,0){\line(1,0){80}}
\put(10,80){\line(1,0){80}}
\put(0,20){\line(1,0){100}}
\put(0,40){\line(1,0){100}}
\put(0,60){\line(1,0){100}}

\put(0,20){\line(0,1){40}}
\put(20,20){\line(0,1){40}}
\put(40,20){\line(0,1){40}}
\put(60,20){\line(0,1){40}}
\put(80,20){\line(0,1){40}}
\put(100,20){\line(0,1){40}}

\put(10,0){\line(0,1){20}}
\put(30,0){\line(0,1){20}}
\put(50,0){\line(0,1){20}}
\put(70,0){\line(0,1){20}}
\put(90,0){\line(0,1){20}}

\put(10,60){\line(0,1){20}}
\put(30,60){\line(0,1){20}}
\put(50,60){\line(0,1){20}}
\put(70,60){\line(0,1){20}}
\put(90,60){\line(0,1){20}}

\put(10,60){\line(1,1){20}}
\put(10,80){\line(1,-1){20}}

\put(70,0){\line(1,1){20}}
\put(70,20){\line(1,-1){20}}

\put(90,50){\makebox(0,0){{\Tiny{$1$}}}}
\put(90,30){\makebox(0,0){{\Tiny{$2$}}}}
\put(70,50){\makebox(0,0){{\Tiny{$4$}}}}
\put(70,30){\makebox(0,0){{\Tiny{$5$}}}}
\put(50,50){\makebox(0,0){{\Tiny{$8$}}}}
\put(50,30){\makebox(0,0){{\Tiny{-$8$}}}}

\put(30,50){\makebox(0,0){{\Tiny{-$5$}}}}
\put(30,30){\makebox(0,0){{\Tiny{-$4$}}}}
\put(10,50){\makebox(0,0){{\Tiny{-$2$}}}}
\put(10,30){\makebox(0,0){{\Tiny{-$1$}}}}

\put(80,70){\makebox(0,0){{\Tiny{$3$}}}}
\put(60,70){\makebox(0,0){{\Tiny{$6$}}}}
\put(40,70){\makebox(0,0){{\Tiny{-$7$}}}}

\put(20,10){\makebox(0,0){{\Tiny{-$3$}}}}
\put(40,10){\makebox(0,0){{\Tiny{-$6$}}}}
\put(60,10){\makebox(0,0){{\Tiny{$7$}}}}

\put(50,40){\makebox(0,0){{\Tiny{$\bullet$}}}}
\end{picture}
\caption{\footnotesize{Symplectic Dynkin pyramid of shape $(5^2,4,2)$}} 
\label{Fig:sp_n-pyr2}
\end{minipage}
\end{figure}
\end{center}}

Let us now explain what are the  
{orthogonal} Dynkin pyramids  
of shape $\bs\lam \in \P_{1}(n)$. 
Assume to start with
that $n$ is even. 
Then the orthogonal 
Dynkin pyramid is constructed as in the symplectic case, adding
rows of lengths determined by the parts of $\bs\lam$ 
working outwards from the $x$-axis starting with
the largest part, in a centrally symmetric way. 
The only difficulty is if some (necessarily odd)
part of $\bs\lam$  appears with odd multiplicity. 
As $n$ is even, the number of distinct parts having odd
multiplicity is even. 
Choose $i_1 < j_1 < \cdots < i_r < j_r$ such that $\lam_{i_1} > \lam_{j_1} >\cdots > \lam_{i_r }> \lam_{j_r}$ are
representatives for all the distinct odd parts of $\bs\lam$ having odd multiplicity. 
Then the first time
the part $\lam_{i_s}$ needs to be added to the diagram, the part $\lam_{j_s}$ 
is also added at the same time, so
that the parts $\lam_{i_s}$ and $\lam_{j_s}$ of $\bs\lam$ 
contribute two centrally symmetric rows to the diagram, one
row in the upper half plane with boxes in columns $1-\lam_{j_s}, 3-\lam_{j_s},\ldots,
\lam_{i_s}-1$ and the other 
row in the lower half plane with boxes in columns 
$1 -\lam_{i_s} , 3-\lam_{i_s},\ldots,\lam_{j_s}-1$. 
We will refer
to the exceptional rows arising in this way as {\em skew rows}. 
We number the boxes exactly as in the symplectic case. 

For example, we represent in Figure \ref{Fig:so_n-pyr1} 
and Figure \ref{Fig:so_n-pyr2} the numbered 
orthogonal Dynkin pyramids of shape $\bs\lam=(4^2,1^2)$ 
and $\bs\lam=(3,1^3)$, respectively. 

{\tiny
\begin{center}
\begin{figure}[h]
\setlength\unitlength{0.0175cm}
\begin{minipage}[l]{.46\linewidth} 
\hspace{2cm}\begin{picture}(0,100)

\put(40,0){\line(1,0){20}}
\put(10,20){\line(1,0){80}}
\put(10,40){\line(1,0){80}}
\put(10,60){\line(1,0){80}}
\put(40,80){\line(1,0){20}}

\put(10,20){\line(0,1){40}}
\put(30,20){\line(0,1){40}}
\put(50,20){\line(0,1){40}}
\put(70,20){\line(0,1){40}}
\put(90,20){\line(0,1){40}}
\put(40,0){\line(0,1){20}}
\put(60,0){\line(0,1){20}}
\put(40,60){\line(0,1){20}}
\put(60,60){\line(0,1){20}}

\put(80,50){\makebox(0,0){{\Tiny{$1$}}}}
\put(80,30){\makebox(0,0){{\Tiny{$2$}}}}
\put(60,50){\makebox(0,0){{\Tiny{$3$}}}}
\put(60,30){\makebox(0,0){{\Tiny{$4$}}}}
\put(50,70){\makebox(0,0){{\Tiny{$5$}}}}
\put(50,10){\makebox(0,0){{\Tiny{-$5$}}}}
\put(40,50){\makebox(0,0){{\Tiny{-$4$}}}}
\put(40,30){\makebox(0,0){{\Tiny{-$3$}}}}
\put(20,50){\makebox(0,0){{\Tiny{-$2$}}}}
\put(20,30){\makebox(0,0){{\Tiny{-$1$}}}}

\put(50,40){\makebox(0,0){{\Tiny{$\bullet$}}}}
\end{picture}
\caption{\footnotesize{Orthogonal Dynkin pyramid of shape $(4^2,1^2)$}} 
\label{Fig:so_n-pyr1}
\end{minipage} \hfill \begin{minipage}[l]{.46\linewidth}
\setlength\unitlength{0.0175cm}
\hspace{2cm}\begin{picture}(0,100)
\put(40,0){\line(1,0){20}}
\put(20,20){\line(1,0){60}}
\put(20,40){\line(1,0){60}}
\put(20,60){\line(1,0){60}}
\put(40,80){\line(1,0){20}}

\put(40,0){\line(0,1){80}}
\put(60,0){\line(0,1){80}}
\put(20,20){\line(0,1){40}}
\put(80,20){\line(0,1){40}}

\put(20,20){\line(1,1){20}}
\put(20,40){\line(1,-1){20}}
\put(60,40){\line(1,1){20}}
\put(60,60){\line(1,-1){20}}

\put(70,30){\makebox(0,0){{\Tiny{$1$}}}}
\put(50,70){\makebox(0,0){{\Tiny{$2$}}}}
\put(50,50){\makebox(0,0){{\Tiny{$3$}}}}
\put(50,30){\makebox(0,0){{\Tiny{-$3$}}}}
\put(50,10){\makebox(0,0){{\Tiny{-$2$}}}}
\put(30,50){\makebox(0,0){{\Tiny{-$1$}}}}

\put(50,40){\makebox(0,0){{\Tiny{$\bullet$}}}}
\end{picture}
\caption{\footnotesize{Orthogonal Dynkin pyramid of shape $(3,1^3)$}} 
\label{Fig:so_n-pyr2}
\end{minipage}
\end{figure}
\end{center}}

If $n$ is odd, there is one additional consideration. 
There must be some odd part appearing
with odd multiplicity. Let $\lam_i$ 
 be the largest such part, and put $\lam_i$ boxes into the zeroth row in
columns $1-\lam_i, 3- \lam_i,\ldots,\lam_i -1$; 
 we also treat this zeroth row as a skew row. 
Now remove the
part $\lam_i$ from $\bs\lam$, to obtain a partition of an even number. The remaining parts are then added
to the diagram exactly as in the case $n$ even. 
We number the boxes 
exactly as in the symplectic Dynkin pyramid 
with labels $1,\ldots,n/2,0,-n/2,\ldots,-1$ in such a way that $i$ and 
$-i$ appear in centrally symmetric boxes, for $i=1,\ldots,n/2$, 
except that there is here a box numbered $0$. 

We represent in Figure \ref{Fig:so_n-pyr3} 
and Figure \ref{Fig:so_n-pyr4} two more examples: the numbered 
orthogonal Dynkin pyramids of shape $\bs\lam=(5^3,1^2)$ 
and $\bs\lam=(4^2,3)$, respectively. 

{\tiny
\begin{center}
\begin{figure}[h]
\setlength\unitlength{0.0175cm}
\begin{minipage}[l]{.46\linewidth} 
\hspace{2cm}\begin{picture}(0,100)

\put(40,0){\line(1,0){20}}
\put(0,20){\line(1,0){100}}
\put(0,40){\line(1,0){100}}
\put(0,60){\line(1,0){100}}
\put(0,80){\line(1,0){100}}
\put(40,100){\line(1,0){20}}

\put(40,0){\line(0,1){100}}
\put(60,0){\line(0,1){100}}
\put(0,20){\line(0,1){60}}
\put(20,20){\line(0,1){60}}
\put(80,20){\line(0,1){60}}
\put(100,20){\line(0,1){60}}

\put(90,70){\makebox(0,0){{\Tiny{$1$}}}}
\put(90,50){\makebox(0,0){{\Tiny{$2$}}}}
\put(90,30){\makebox(0,0){{\Tiny{$3$}}}}
\put(70,70){\makebox(0,0){{\Tiny{$4$}}}}
\put(70,50){\makebox(0,0){{\Tiny{$5$}}}}
\put(70,30){\makebox(0,0){{\Tiny{$6$}}}}
\put(50,90){\makebox(0,0){{\Tiny{$7$}}}}
\put(50,70){\makebox(0,0){{\Tiny{$8$}}}}
\put(50,50){\makebox(0,0){{\Tiny{$0$}}}}
\put(50,30){\makebox(0,0){{\Tiny{-$8$}}}}
\put(50,10){\makebox(0,0){{\Tiny{-$7$}}}}
\put(30,70){\makebox(0,0){{\Tiny{-$6$}}}}
\put(30,50){\makebox(0,0){{\Tiny{-$5$}}}}
\put(30,30){\makebox(0,0){{\Tiny{-$4$}}}}
\put(10,70){\makebox(0,0){{\Tiny{-$3$}}}}
\put(10,50){\makebox(0,0){{\Tiny{-$2$}}}}
\put(10,30){\makebox(0,0){{\Tiny{-$1$}}}}

\end{picture}
\caption{\footnotesize{Orthogonal Dynkin pyramid of shape $(5^3,1^2)$}} 
\label{Fig:so_n-pyr3}
\end{minipage} \hfill \begin{minipage}[l]{.46\linewidth}
\setlength\unitlength{0.0175cm}
\hspace{2cm}\begin{picture}(0,100)

\put(0,10){\line(1,0){80}}
\put(0,30){\line(1,0){80}}
\put(0,50){\line(1,0){80}}
\put(0,70){\line(1,0){80}}

\put(0,10){\line(0,1){20}}
\put(0,50){\line(0,1){20}}
\put(20,10){\line(0,1){20}}
\put(20,50){\line(0,1){20}}
\put(40,10){\line(0,1){20}}
\put(40,50){\line(0,1){20}}
\put(60,10){\line(0,1){20}}
\put(60,50){\line(0,1){20}}
\put(80,10){\line(0,1){20}}
\put(80,50){\line(0,1){20}}

\put(10,30){\line(0,1){20}}
\put(30,30){\line(0,1){20}}
\put(50,30){\line(0,1){20}}
\put(70,30){\line(0,1){20}}

\put(70,60){\makebox(0,0){{\Tiny{$1$}}}}
\put(50,60){\makebox(0,0){{\Tiny{$4$}}}}
\put(30,60){\makebox(0,0){{\Tiny{-$5$}}}}
\put(10,60){\makebox(0,0){{\Tiny{-$2$}}}}

\put(60,40){\makebox(0,0){{\Tiny{$3$}}}}
\put(40,40){\makebox(0,0){{\Tiny{$0$}}}}
\put(20,40){\makebox(0,0){{\Tiny{-$3$}}}}

\put(70,20){\makebox(0,0){{\Tiny{$2$}}}}
\put(50,20){\makebox(0,0){{\Tiny{$5$}}}}
\put(30,20){\makebox(0,0){{\Tiny{-$4$}}}}
\put(10,20){\makebox(0,0){{\Tiny{-$1$}}}}

\end{picture}
\caption{\footnotesize{Orthogonal Dynkin pyramid of shape $(4^2,3)$}} 
\label{Fig:so_n-pyr4}
\end{minipage}
\end{figure}
\end{center}}

We can now fix a 
choice of an $\sl_2$-triple $(e,h,f)$ 
such that $f \in \O_{-1;\bs{\lam}}$ by 
setting $f =\sum_{i,j} \sigma_{i,j} e_{i,j}$ for all $1 \le i,j \le n$, 
where the sum of over all pairs $i,j$ of boxes such 
in the Dynkin
pyramid such that 
\begin{align*} 
& \text{either }   {\rm col}(j)= {\rm col}(i)+2 \text{ and } 
{\rm row}(i)={\rm row}(j), & \\
& \text{or }  {\rm col}(j)=1, \; {\rm col}(i) = -1 
\text{ and } {\rm row}(i)= - {\rm row}(j) & \\ 
& \hspace{0.5cm} \text{ is a skew-row in the upper half plane}, & 
\end{align*}
and $h = \sum_{i=1}^{n} {\rm col}(i) e_{i,i}$  
(so $x^0=  \sum_{i=1}^n 
\frac{1}{2}{\rm col}(i) e_{i,i}$).

The following lemma is a refinement of \cite[Theorem 12.3]{KraftProcesi82}. 
We refer to \cite[Propositions 8.5.1 and 8.5.2]{Li17} for a proof. 

\begin{Lem}[row/column removal rule for $\sp_n$ and $\so_n$]
\label{lem:erasing_row_type-BCD}
Let ${\bs{\lam}} \in \P_\eps(n)$, 
for $\eps\in\{\pm 1\}$, and $ {\bs{\mu}}$ a degeneration 
of ${\bs{\lam}}$. 
Assume that the first $l$ rows and the first $m$ columns 
of ${\bs{\lam}}$ and ${\bs{\mu}}$ 
coincide and denote by $\bs{\lam}'$ and $\bs{\mu}'$ the partitions obtained by erasing 
these $l$ common  rows and common  $m$ columns. 
Then 
\begin{align*} 
\Slo_{\O_{\eps; {\bs{\lam}}},f} \cong \Slo_{\O_{\eps;{\bs{\lam}'}},f'},
\end{align*}
as algebraic varieties, 
with $f \in \O_{\eps; {\bs{\mu}}}$ and $f' \in \O_{\eps;{\bs{\mu}'}}$. 
In particular, if $f'=0$, then 
$\Slo_{\O_{\eps; {\bs{\lam}}},f} \cong \overline{\O}_{\eps;{\bs{\lam}'}}$.  
\end{Lem}

By \cite[Tables 2 and 3]{Arakawa15a}, the nilpotent orbit $\O_k$, for $k$ admissible for $\g=\sp_n$ 
or $\g=\so_n$, is described as follows: 
\begin{itemize}
\item If $k$ is a principal admissible level for $\sp_n$ 
(that is, $q$ is odd), then $\O_k=\O_{-1;\bs{\widetilde{\lam}}^-}$, 
where $\bs{\widetilde{\lam}}=(q^{\widetilde m},\widetilde s) \in \P(n)$, with $0 \le  \widetilde s\le q-1$.

\item If $k$ is a coprincipal admissible level for $\sp_n$ 
(that is, $q$ is even), then 
$\O_k=\O_{-1;\bs{\widetilde{\lam}}^-}$ where 
$\bs{\widetilde{\lam}}=(\frac{q}{2}+1,(\frac{q}{2})^{\widetilde m},\widetilde s) \in \P(n)$, 
with  $0 \le  \widetilde s\le \frac{q}{2}-1$. 

\item If $k$ is a principal admissible level for $\so_n$ 
(that is, either $n$ is even, or both $n$ and $q$ are odd), 
then $\O_k=\O_{1;\bs{\widetilde{\lam}}^+}$, 
where $\bs{\widetilde{\lam}}=(q+1,q^{\widetilde m},\widetilde s) \in \P(n)$, with $0 \le  \widetilde s\le q-1$. 

\item If $k$ is a coprincipal admissible level for $\so_n$ 
(that is, $n$ is odd and $q$ is even), 
then $\O_k=\O_{1;\bs{\widetilde{\lam}}^+}$, 
where $\bs{\widetilde{\lam}}=(q^{\widetilde m},\widetilde s) \in \P(n)$, with $0 \le  \widetilde s\le q-1$. 
\end{itemize}

The following lemma is the analog of Lemma \ref{Lem:choice_of_partitions} 
for $\sp_n$. 
\begin{Lem}[Case $\sp_n$]
\label{Lem:choice_of_partitions-C}
Let $\bs\lam \in \P_{-1}(n)$ be such that $\O_k=\O_{-1,\bs\lam}$.  
Fix a partition $\bs{\mu} \in \P_{-1}(n)$ of $n$ such that $\O_{\bs{\mu}} \subset  \overline{\O}_{\bs{\lam}}$.   
Let $\bs{\lam}'$ and $\bs{\mu}'$ be the partitions obtained from $\bs{\lam}$ and $\bs{\mu}$ 
by erasing all common rows and columns of $\bs{\lam}$ and $\bs{\mu}$. 
Then, $\bs{\mu}'$ corresponds to the zero nilpotent orbit of $\sp_{|\bs{\mu}'|}$, that is, $\bs{\mu}'=(1^{|\bs{\mu}'|})$ 
if and only if $\bs\mu$ is of one of the following types:  
\begin{itemize}
\item[(a)] $\bs{\mu}=\bs{\lam}$, 
\item[(b)] $\bs{\mu}=(q^m,1^s)$, with $q$ odd, $m$ even, $s\ge 0$, 
\item[(c)] $\bs{\mu}=(q^{\widetilde m-1},q-1,1^{\widetilde s+1})$, with $q$ odd, 
\item[(d)] $\bs{\mu}=(q^{\widetilde m-1},(q-2)^{2})$ and $\widetilde s = q-4$, with $q$ odd, 
\item[(e)] $\bs{\mu}=((\frac{q}{2})^{m},1^s)$, with $q$ even, $\frac{q}{2},s$ even, $m$ odd or even.
\item[(f)] $\bs{\mu}=((\frac{q}{2})^{\widetilde m},(\frac{q}{2}-1)^2)$ 
and $\widetilde s=\frac{q}{2}-4$, with $q$ even, 
$\frac{q}{2}$ even,
\item[(g)] $\bs{\mu}=(\frac{q}{2}+1,(\frac{q}{2})^{m},1^s)$, with $q$ even, $\frac{q}{2}$ odd, $m,s$ even, 
\item[(h)] $\bs{\mu}=((\frac{q}{2})^{\widetilde m +2}, \widetilde s)$, with $q$ even, $\frac{q}{2}$ odd. 
\item[(i)] $\bs{\mu}=(\frac{q}{2}+1,(\frac{q}{2})^{m},\frac{q}{2}-1,1^s)$, with $q$ even, $\frac{q}{2}$ odd, $m,s$ even, 
\item[(j)] $\bs{\mu}=(\frac{q}{2}+1,(\frac{q}{2})^{\widetilde m-1},(\frac{q}{2}-2)^2)$, 
and $\widetilde s=\frac{q}{2}-4$, with $q$ even, 
$\frac{q}{2}$ odd, 
\item[(k)] $\bs{\mu}=((\frac{q}{2}-1)^2)$, $\widetilde m=0$  
and $\widetilde s=\frac{q}{2}-3$, with $q$ even, 
$\frac{q}{2}$ odd. 
\end{itemize}
Here $|\bs{\mu}'|$ stands for the sum of the parts of $\bs{\mu}'$. 
\end{Lem}

\begin{proof}
We argue as in the proof of Lemma \ref{Lem:choice_of_partitions}. 
One can assume that $\bs\mu\not=\bs\lam$, the case where $\bs\mu=\bs\lam$ being obvious. 
According to the above description of $\O_k$, 
five types of partitions for $\bs\lam$ 
can be distinguished. 
We consider the different types.

(1) $\bs\lam=(q^{\widetilde m},\widetilde s)$, with $q$ odd, $\widetilde m$ even and $0 \le \widetilde s \le q-1$. 
This case is very similar to that of $\sl_n$. 
This leads to the partitions of type (b), 
the case that $\bs\mu$ is of type $(q^{\widetilde m-1},(q-1)^2)$ being 
excluded since $q$ and $\widetilde m-2$ are odd.

(2) $\bs\lam=(q^{\widetilde m-1},q-1,\widetilde s +1)$, with $q$ odd, $\widetilde m-1$ even, $2 \le \widetilde s+1 \le q-1$. 
Assume first that $\bs\mu=(q^{\widetilde m-1},q-1,\bs\nu)$, with $\bs\nu=(\nu_t,\ldots,\nu_t)$ and $\nu_1 \le q-1$. 
Then $\bs\lam'=(\widetilde s +1)$ and $\bs\mu'= \bs\nu$. 
So the only possibility for that $\bs\mu'$ corresponds to the zero 
orbit is that $\nu_i=1$ for any $i$, that is, $\bs\mu$ is of the form (c). 

Assume now that $\bs\mu=(q^{m},\bs\nu)$, with $0 \le m \le \widetilde m-1$ even, 
$\bs\nu=(\nu_t,\ldots,\nu_t)$ and $\nu_1 \le q-1$. 
Let $\bs{\lam}''$ and $\bs{\mu}''$ be the partitions obtained 
by erasing all common rows of $\bs{\lam}$ and $\bs{\mu}$. 
Then $\bs\lam''=(q^{\widetilde m-1-m},q-1,\widetilde s +1)$ and $\bs\mu''= \bs\nu$. 
If $\nu_i=1$ for all $i$, then we get a partition of type  (b). 
Otherwise, observe that $\bs\lam''$ and $\bs\mu'$ have at least one common column 
if and only if $t=\widetilde m-m+1$. 
If so, then $\nu_t > \widetilde s+1$ and 
$\bs\lam'=((q-\widetilde s-1)^{\widetilde m-1-m},q-2-\widetilde s)$, $\bs\mu'=(\nu_1-\widetilde s-1, 
\ldots,\nu_t-\widetilde s-1)$. 
The partition $\bs\mu'$ corresponds to the zero orbit if and only if 
$\nu_i-\widetilde s-1=1$ for all $i$. 
Then 
$$\widetilde m-m+1=(\widetilde m-1-m) (q-\widetilde s-1) + q-2-\widetilde s,$$
that is, 
$$2=(\widetilde m-m) (q-\widetilde s-2).$$
Since $\widetilde m-m$ is odd, the only possibility is that $\widetilde m-m=1$, that is, 
$m=\widetilde m-1$ 
and 
$q-\widetilde s-2=2$, that is, $\widetilde s =q-4$. 
Hence $t=2$ and $\nu_1=\nu_2=q-2$. So $\bs\mu$ of the form (d). 

(3) $\bs\lam=((\frac{q}{2})^{\widetilde m+1},\widetilde s +1)$, with $q$ even, 
$\frac{q}{2}$ even, $\widetilde m$ odd or even 
and $0 \le \widetilde s+1 \le \frac{q}{2}-1$. 
This case is very similar to the $\sl_n$ one. So we conclude similarly. 
This leads to the partitions of type (e) or (f). 

(4) $\bs\lam=(\frac{q}{2}+1,(\frac{q}{2})^{\widetilde m},\widetilde s)$, with $q$ even, 
$\frac{q}{2}$ odd, $\widetilde m$ even 
and $1 \le \widetilde s \le \frac{q}{2}-1$. 
Assume first that $\bs\mu=(\frac{q}{2}+1,(\frac{q}{2})^{m},\bs\nu)$, with $0 \le m \le \widetilde m$, 
$\bs\nu=(\nu_t,\ldots,\nu_t)$ and $\nu_1 \le \frac{q}{2}-1$. 
Let $\bs{\lam}''$ and $\bs{\mu}''$ be the partitions obtained 
by erasing all common rows of $\bs{\lam}$ and $\bs{\mu}$. 
Then $\bs\lam''=(q^{\widetilde m-m},\widetilde s )$ and $\bs\mu''= \bs\nu$. 
So we are led to the case (3), but the partition 
 $\bs\mu=(\frac{q}{2}+1,(\frac{q}{2})^{\widetilde m-1},(\frac{q}{2}-1)^2)$ is excluded 
since both $\frac{q}{2}$ and $\widetilde m-1$ are odd. So we get only the type (g). 
 
Assume now that $\bs\mu=((\frac{q}{2})^{m},\bs\nu)$, with $0 \le m \le \widetilde m+2$, 
$\bs\nu=(\nu_t,\ldots,\nu_t)$ and $\nu_1 \le \frac{q}{2}-1$. 
Then $\bs\lam$ and $\bs\mu$ have no common rows. Moreover, $\bs\lam$ and $\bs\mu$ have at least one 
common column if and only if $m+t=\widetilde m+2$. 
If so, then $\nu_t > \widetilde s$ and we have 
$\bs\lam'=(\frac{q}{2}+1-\widetilde s,(\frac{q}{2}-\widetilde s)^{\widetilde m})$,  
$\bs\mu'=((\frac{q}{2}-\widetilde s)^{m},\nu_1-\widetilde s, 
\ldots,\nu_t-\widetilde s)$. 
The later corresponds to the zero orbit if and only if either $\frac{q}{2}-\widetilde s=1$ 
and $\nu_i-\widetilde s=1$ for all $i$, whence  
$m+t =\widetilde m+2$ and $\widetilde s=\frac{q}{2}-1$ and, 
necessarily, $\bs\mu$ has type (h) with $\widetilde s=0$. 
Or $m=0$ and $\nu_i-\widetilde s=1$ for all $i$, whence 
$$\widetilde m +2 = \widetilde m\left(\frac{q}{2}-\widetilde s\right)+\frac{q}{2}-1-\widetilde s,$$
that is,  
$$2= ( \widetilde m+1)\left(\frac{q}{2}-1-\widetilde s\right).$$
Hence for parity reasons we get that $\widetilde m=0$, 
$\bs\lam=(\frac{q}{2}+1,\frac{q}{2}-3)$ and $\bs\mu=((\frac{q}{2}-1)^2)$. 
This yields to the partition of type (k).

(5) $\bs\lam=(\frac{q}{2}+1,(\frac{q}{2})^{\widetilde m-1},\frac{q}{2}-1,\widetilde s+1)$, with $q$ even, 
$\frac{q}{2}$ odd, $\widetilde m-1$ even 
and $2 \le \widetilde s+1 \le \frac{q}{2}-1$. 

Assume first that the first row of $\bs\mu$ if $\frac{q}{2}+1$. 
By erasing the first (common) row of $\bs\lam$ and $\bs\mu$ 
we go to the situation (3) with $\frac{q}{2}$ is place of $q$. 
So we get partitions of types (g), (i) or (j). 

Assume now that $\bs\mu=((\frac{q}{2})^{m},\bs\nu)$, with $0 \le m \le \widetilde m+1$, 
$\bs\nu=(\nu_t,\ldots,\nu_t)$ and $\nu_1 \le \frac{q}{2}-1$. 
Then $\bs\lam$ and $\bs\mu$ have no common rows. Moreover, $\bs\lam$ and $\bs\mu$ have at least one 
common column only if $m+t=\widetilde m+2$. 
If so, either $\nu_t= \widetilde s+1$ and then necessarily $\bs\mu$ is of the type (h), 
or $\nu_t > \widetilde s+1$ and we have 
$\bs\lam'=(\frac{q}{2}-\widetilde s,(\frac{q}{2}-\widetilde s-1)^{\widetilde m-1},\frac{q}{2}-2-\widetilde s)$ and 
$\bs\mu'=((\frac{q}{2}-\widetilde s-1)^{m},\nu_1-\widetilde s-1, 
\ldots,\nu_t-\widetilde s-1)$. 
The later corresponds to the zero nilpotent orbit if and only if either $\frac{q}{2}-\widetilde s-1=1$ 
and $\nu_i-\widetilde s-1=1$ for all $i$, whence 
$m+t =\widetilde m+1$ and $\widetilde s+1=\frac{q}{2}-1$, 
and $\bs\mu$ has type (h) with $\widetilde s=\frac{q}{2}-2$, 
or $m=0$ and $\nu_i-\widetilde s-1=1$ for all $i$, 
whence 
$$\widetilde m+1= (\widetilde m+1)\frac{q}{2} +\widetilde s+1,$$
that is, 
$$0=(\widetilde m+1)\left(\frac{q}{2}-1\right) +\widetilde s+1,$$
whence, $\frac{q}{2}=1$ and $\widetilde s+1=0$ which is impossible. 

Conversely, we easily that all the partitions of type (a)--(k) verify the desired conditions. 
\end{proof}

We now state the analog of Lemma \ref{Lem:choice_of_partitions} 
for $\so_n$. 

\begin{Lem}[Case $\so_n$]
\label{Lem:choice_of_partitions-BD}
Let $\bs\lam \in \P_{1}(n)$ be such that  
$\O_k=\O_{1;\bs\lam}$.  
Fix a partition $\bs{\mu}$ of $n$ such that $\O_{\bs{\mu}} \subset  \overline{\O}_{\bs{\lam}}$.   
Let $\bs{\lam}'$ and $\bs{\mu}'$ be the partitions obtained from $\bs{\lam}$ and $\bs{\mu}$ 
by erasing all common rows and columns of $\bs{\lam}$ and $\bs{\mu}$. 
Then, $\bs{\mu}'$ corresponds to the zero nilpotent orbit of $\sp_{|\bs{\mu}'|}$, that is, $\bs{\mu}'=(1^{|\bs{\mu}'|})$ 
if and only if $\bs\mu \in \P_1(n)$ is of one of the following types:  
\begin{itemize}
\item[(a)] $\bs{\mu}=\bs{\lam}$, 
\item[(b)] $\bs{\mu}=(q^m,1^s)$, with $q$ odd or even, $m$ odd or even, $s\ge 0$, 
\item[(c)] $\bs{\mu}=(q^{\widetilde m},(q-1)^{2})$, $\widetilde s = q-3$ 
and $\widetilde m$ odd, with $q$ odd, 
\item[(d)] $\bs{\mu}=(3^{\widetilde m-1},2^{4})$, $q=3$, $\widetilde s=1$ 
and $\widetilde m$ even,
\item[(e)] $\bs{\mu}=(q^{\widetilde m+1},2^{2})$, $\widetilde s=3$ 
and $\widetilde m+1$ odd, with $q$ odd, 
\item[(f)] $\bs{\mu}=(q^{\widetilde m},(q-1)^2,1)$, $\widetilde s = q-2$ 
and $\widetilde m$ even, with $q$ odd, 
\item[(g)] $\bs{\mu}=(q^{\widetilde m -1},q-1,1^s)$, with $q$ even, $s\ge 0$,  
\item[(h)] $\bs{\mu}=(q^{\widetilde m -1},q-1,2^2)$ and $\widetilde s=3$, with $q$ even,  
\item[(i)] $\bs{\mu}=(q^{\widetilde m -1},(q-2)^2,1)$ and $\widetilde s = q-3$, with $q$ even, 
\item[(j)] $\bs{\mu}=(q+1,q^{m},1^s)$, with $q,m$ even and $s\ge 0$, 
\item[(k)] $\bs{\mu}=(q+1,q^{m},q-1,1^s)$, with $q,m$ even and $s\ge 0$, 
\item[(l)] $\bs{\mu}=(q+1,q^{\widetilde m-1},q-1,2^2)$ and $\widetilde s=3$, with $q$ even,  
\item[(m)] $\bs{\mu}=(q+1,q^{\widetilde m-1},(q-2)^2,1)$ and $\widetilde s = q-3$, with $q$ even, 
\item[(n)] $\bs{\mu}=(q^{\widetilde m+1},q-1,1)$ and $\widetilde s=q-1$, with $q$ even.  
\end{itemize}
Here $|\bs{\mu}'|$ stands for the sum of the parts of $\bs{\mu}'$. 
\end{Lem}

\begin{proof}
The proof is very similar to that of Lemma \ref{Lem:choice_of_partitions-C}. 
The details are left to the reader. 
\end{proof}

In particular, by Lemma \ref{lem:erasing_row_type-BCD}, if $\bs{\mu}$ is a partition of $n$ 
as in Lemma \ref{Lem:choice_of_partitions-C} (resp. Lemma~\ref{Lem:choice_of_partitions-BD}), 
the nilpotent Slodowy slice $\Slo_{\O_k,f}$ is isomorphic 
to a nilpotent orbit closure in $\sp_{s}$ (resp.~$\so_s$) 
for $f \in \O_{\varepsilon;\bs{\mu}}$ with $\varepsilon=-1$ (resp.~$\eps=1$) 
and $s:=|\bs{\mu}'|$. 
If $\bs{\mu}$ is 
of type (a), note that $\Slo_{\O_k,f} \cong \{f\}$. 

In view of Lemmas \ref{Lem:choice_of_partitions-C} and 
\ref{Lem:choice_of_partitions-BD}, 
we describe the centraliser $\g^\natural$ 
and the values of the $k_i^\natural$'s for particular   
$\sl_2$-triples $(e,h,f)$ of $\sp_n$ and $\so_n$.

\begin{Lem}
\label{Lem:k_natural-type-BCD}
Let $\eps=-1$ (resp.~$\eps=1$), 
and $f \in \O_{\eps;\bs\mu}$ a nilpotent element of $\sp_n$ 
(resp.~$\so_n$)  with $\bs\mu$ 
as in the first column of Table~\ref{Tab:C_Slodowy} 
(resp.~Table~\ref{Tab:BD_Slodowy}). 
Then the centraliser $\g^\natural$ of an $\sl_2$-triple 
$(e,h,f)$   
and the  $k_i^\natural$'s are given 
by Table~\ref{Tab:C_Slodowy} 
(resp.~Table~\ref{Tab:BD_Slodowy}).
Moreover, if $k$ is admissible, then so is $k^\natural$, 
provided that $k_0^\natural=0$. 
\end{Lem}
In Tables~\ref{Tab:C_Slodowy} 
and Table~\ref{Tab:BD_Slodowy}, the numbering of the levels $k_i^\natural$'s 
follows the order in which the simple factors of $\g^\natural$ 
appear.

\begin{Rem}
\label{Rem:small_values}
In Tables~\ref{Tab:C_Slodowy} 
and Table~\ref{Tab:BD_Slodowy}, the conclusions 
for small $s$ or $m$ remain valid, up to possible changes of numbering for the factors $\g_i^\natural$. 
More specifically, if the factor $\g_i^\natural$ for some $i$ is isomorphic to $\mf{so}_m$ with $m=2$,  
then $\g_i^\natural \cong \C$, so $i$ must be replaced by $0$ and the value of $k_i^\natural$ remains valid.

If the factor $\g_i^\natural$ for some $i$, let's say $i=1$, is isomorphic to $\mf{so}_m$ 
with $m=4$, then $\g_1^\natural$ is not simple and has to replaced by $\g_1^\natural \times \g_2^\natural$, 
with $\g_1^\natural \cong \sl_2$,  $\g_2^\natural \cong \sl_2$. 
Moreover, we have to read $k_2^\natural = k_1^\natural=\cdots$ 
instead of $k_1^\natural=\cdots$, and the value of $k_1^\natural$ remains valid. 

If the factor $\g_i^\natural$ for some $i$ is isomorphic to $\mf{so}_m$ or $\sl_m$ with $m=1$, then this 
factor does not appear (we use the convention that $\sl_1=\mf{so}_1 = 0$), 
and we just forget about $k_i^\natural$. 

Finally, if the factor $\g_i^\natural$ for some $i$ is isomorphic to $\mf{so}_{3} \cong \sl_2$, 
$\mf{so}_{6} \cong \sl_3$ or $\mf{sp}_{2} \cong \sl_2$ then the conclusions remain valid without any change. 
\end{Rem}

\begin{proof}
We proceed as in Lemma \ref{Lem:k_natural-type-A}. 
In order to describe $\g^\natural$ we use the symplectic (resp.~orthogonal) 
Dynkin pyramid of shape $\bs{\mu}$. 
Indeed, as explained in \S\ref{sub:sympl_pyramids} such a pyramid allows 
to construct an $\sl_2$-triple $(e,h,f)$ with $f \in \O_{\bs\mu}$. 
One easily can compute $\g^0=\g^{h}$ from the pyramid  
and, hence, $\g^\natural = \g^0 \cap \g^f$.  

\smallskip

\noindent 
{\bf Case $\sp_n$.}
We detail the proof for 
$\bs\mu= (q^m,1^s) \in \P_{-1}(n)$. 
The other cases are dealt with similarly, and the verifications 
are left to the reader. There are two cases. 

(1) $\bs\mu= (q^m,1^s) \in \P_{-1}(n)$, with $q$ odd, $m$ even, 
$s$ even.  
In this case, the orbit $\O_{-1;\bs\mu}$ is even, and 
the symplectic Dynkin pyramid associated with the partition $\bs\mu$ is 
as in Figure~\ref{Fig:sp_n-pyr1} if $q=5$, $m=s=2$. 
From the pyramid, we obtain that  
\begin{align*}
\g^0 & = \left\{{\rm diag}(x_1, \ldots,x_{(q-1)/2}, 
y, - \widehat{x_{(q-1)/2}}, \ldots,\widehat{x_1} \colon x_i \in \mf{gl}_m, 
 \, y \in \mf{sp}_{m+s} \right\} & \\
& \cong  (\mf{gl}_m)^{(q-1)/2}  \times \mf{sp}_{m+s}, &\\
\g^\natural & = \left\{{\rm diag}(x, \ldots,x,y, 
- \widehat{x}, \ldots, - \widehat{x}) \right. \colon & \\
& \qquad  \left. x = \begin{pmatrix} 
a & b \\ c & -\hat{a}\end{pmatrix}, \, 
y =   \begin{pmatrix} 
a & 0 & b \\ 
0 & z & 0 \\
c & 0 & -\hat{a}
\end{pmatrix}, \,  b= \hat{b}, \, c= \hat{c},\, 
z \in \mf{sp}_s \right\}  \subset \g^0 & \\
& \cong \mf{sp}_m \times \mf{sp}_s. & 
\end{align*} 
To compute $k_1^\natural$, pick $t \in \g^\natural$ with 
$a= {\rm diag}(1,\ldots,0)$, $b=c=0$ and $z=0$ 
in the above description of $\g^\natural$. 
We obtain that $(t | t)_1^\natural = 2$, $(t|t)_\g = 2 q$, 
$\kappa_\g(t,t) = 4 q \left(\frac{n}{2}+1\right)$, 
 $\kappa_{\g^0}(t,t) = 2 m (q-1) +4\left(\frac{m+s}{2}+1\right)$. 
Therefore, $k_1^\natural  = 
  q k + q \left(\frac{n}{2}+1 \right) -  \left(\frac{n}{2}+1\right)$. 
To compute $k_2^\natural$, pick $t \in \g^\natural$ with $z={\rm diag}(1,0,\ldots,0,-1)$ 
and $x=0$ 
in the above description of $\g^\natural$. 
We obtain that $(t | t)_1^\natural = 2$, $(t|t)_\g = 2$, 
$\kappa_\g(t,t) = 4  \left(\frac{n}{2}+1\right)$, 
 $\kappa_{\g^0}(t,t) = 4\left(\frac{m+s}{2}+1\right)$. 
Therefore, $k_2^\natural  =  k +  \left(\frac{n}{2}+1 \right) -  \left(\frac{m+s}{2}+1\right)$.

(2) $\bs{\mu}=(q^m,1^s)$, with $q$ even, $m$ odd or even, 
$s$ even. 
The orbit $\O_{-1;\bs{\mu}}$ is odd, and 
we have
\begin{align*}
\g^0   & = \left\{{\rm diag}(x_1, \ldots,x_{q/2},y, - \widehat{x_{q/2}}, \ldots,\widehat{x_1} 
\colon x_i \in \mf{gl}_m,  \, y \in \mf{sp}_s \right\}&  \\
 &  \cong  (\mf{gl}_m)^{q/2}  \times \mf{sp}_{s}, & \\
\g^\natural  & = \left\{{\rm diag}(x, \ldots,x,y, - \widehat{x}, \ldots,\widehat{x} 
\colon x \in \mf{so}_m,  \, y \in \mf{sp}_s \right\} \subset \g^0 & \\
&  \cong \mf{so}_m \times \mf{sp}_s. & 
\end{align*}
If $m=2$ or $m>4$, 
we compute $k_1^\natural$ by picking $t \in \g^\natural$ with $x={\rm diag}(1,0,\ldots,0,-1)$ 
and $y=0$ in the above description of $\g^\natural$. 
We obtain that $(t | t)_1^\natural = 1$, $(t|t)_\g = 2 q$, 
$\kappa_\g(t,t) = 4 q \left(\frac{n}{2}+1\right)$, 
 $\kappa_{\g^0}(t,t) = 2 m q$, 
  $\kappa_{1/2}(t,t) = 2 s$. 
Therefore, 
$k_1^\natural  = 
  2 q k + 2 q \left(\frac{n}{2}+1 \right) - n,$
whence the expected result. 
To compute $k_2^\natural$, pick $t \in \g^\natural$ with $x=0$  
and $y={\rm diag}(1,0,\ldots,0,-1)$ 
in the above description of $\g^\natural$. 
We obtain that $(t | t)_2^\natural = 2$, $(t|t)_\g = 2$, 
$\kappa_\g(t,t) = 4  \left(\frac{n}{2}+1\right)$, 
 $\kappa_{\g^0}(t,t) = 4\left(\dfrac{s}{2}+1\right)$,  
  $\kappa_{1/2}(t,t) = 2 m$. 
Therefore, 
$ k_2^\natural  =  k +  \left(\frac{n}{2}+1 \right) -  \left(\dfrac{m}{2}+\dfrac{m}{2}+1\right),$
whence the expected result. 

Assume $m=4$. 
The isomorphism 
\begin{align}
\label{eq:iso_so4}
\mf{so}_4
\cong \sl_2 \times \sl_2
\end{align}
is obtained through the assignments 
$$t_1= {\rm diag}(1,1,-1,-1)  \mapsto h_1, \quad t_2= {\rm diag}(1,-1,1,-1)  \mapsto h_2,$$ 
$$e_{1,2} - e_{-2,-1}  \mapsto e_1, \quad e_{1,-2} - e_{2,-1}  \mapsto e_2,\quad 
e_{2,1} - e_{-1,-2}  \mapsto f_1, \quad e_{-2,1} -e_{-1,2}  \mapsto f_2,$$ 
with ${\rm span}(h_i,e_i,f_i) \cong \sl_2$. 
To compute $k_1^\natural$, choose 
$t \in \g^\natural$ such that  
$x= {\rm diag}(1,1,-1,-1)$ and $y=0$ 
in the above description of $\g^\natural$. 
We obtain that $k_1^\natural  = 
  2 q k + 2 q \left(\frac{n}{2}+1 \right) - n$. 
To compute $k_2^\natural$, choose 
$t \in \g^\natural$ such that  
$x= {\rm diag}(1,-1,1,-1)$ and $y=0$ 
in the above description of $\g^\natural$. 
In the same way we obtain that 
$k_2^\natural  =2 q k + 2 q \left(\frac{n}{2}+1 \right) - n$. 
The computation of $k_3^\natural$ in this case 
works as in the case $m>4$ for $k_2^\natural$. 
This terminates this case following Remark \ref{Rem:small_values}.

\smallskip

\noindent 
{\bf Case $\so_n$.}
As for the $\sp_n$, we detail the proof only for 
$\bs\mu= (q^m,1^s) \in \P_{1}(n)$. 
The other cases are dealt with similarly, and the verifications 
are left to the reader. There are two cases. 

(1) $\bs{\mu} = (q^m,1^s)$ with $q$ odd. 

(a) Assume first that $m$ is even. 
Then the orbit $\O_{1;\bs\mu}$ is even, and 
from the pyramid, we easily obtain that 
\begin{align*} 
\g^0 &=\left\{ {\rm diag}(x_1,\ldots,x_{(q-1)/2},y,-\widehat{x_{(q-1)/2}},\ldots,-\widehat{x_1}) 
\colon x_i \in \mf{gl}_m, y \in \mf{so}_{m+s}\right\} &\\ 
&\cong (\mf{gl}_m)^{(q-1)/2} \times \mf{so}_{m+s}, & 
\end{align*}
\begin{align*}  
\g^\natural &= \left\{{\rm diag}(x,\ldots,x,
 y,-\widehat{x},\ldots,-\widehat{x}) 
\colon x= \begin{pmatrix} a & b \\ c & -\widehat{a}
\end{pmatrix},  \, b = - \widehat{b}, \, c = - \widehat{c}, \right. & \\
& \hspace{6cm} \left. y = \begin{pmatrix} 
a & 0 & b \\
0 & z & 0 \\
c & 0  & - \widehat{a}
\end{pmatrix}  \, z \in \mf{so}_s \right\}  \subset \g^0 & \\
& \cong \mf{so}_{m} \times \mf{so}_s.&
\end{align*}

Assume  $m=2$ or $m>4$. 
In order to compute $k_1^\natural$, pick
$t \in \g^\natural$ with $a= {\rm diag}(1,0,\ldots,0)$ and $b=c=0$, $z=0$  
in the above description of $\g^\natural$. 
Then $(t| t)_{1}^\natural = 1$, $(t | t )_\g = q$, 
$\kappa_{\g} (t,t) = 2 q(n-2)$ and $\kappa_{\g^0} (t,t) = 
2m (q-1) + 2(m+s-2)$. Therefore 
$k_1^\natural = qk + q(n-2) -(n-2).$
In order to compute $k_2^\natural$, choose 
$t \in \g^\natural$ such that 
$z= {\rm diag}(1,\ldots,-1)$ and $a=b=c=0$  
in the above description of $\g^\natural$. 
Then 
$(t| t)_{1}^\natural = 1$, $(t | t )_\g = 1$, 
$\kappa_{\g} (t,t) = 2(n-2)$ and $\kappa_{\g^0} (t,t) = 
2(m+s-2)$. Therefore 
$k_2^\natural = k + n-m-s.$
If $m=4$, using the isomorphism \eqref{eq:iso_so4} as in the $\mf{sp}_n$ case, 
we easily verify that $k_1^\natural = k_2^\natural = qk + q(n-2) -(n-2)$ 
and $k_3^\natural =k + n-m-s$. 
So the same conclusions hold following Remark \ref{Rem:small_values}. 

(b) Assume now that $m$ is odd. 
Two cases have to be distinguished.  

$\ast$ If $s$ is even, then the orbit $\O_{1;\bs\mu}$ is even 
and the Dynkin pyramid is  
as in 
Figure~\ref{Fig:so-ex-2} if $q=7$, $m=3$ and $s=2$. 
{\tiny
\begin{center}
\begin{figure}
\setlength\unitlength{0.0175cm} 
\hspace{-2cm}\begin{picture}(140,140)
\put(160,0){\line(1,0){20}}
\put(160,20){\line(1,0){20}}
\put(100,40){\line(1,0){140}}
\put(100,60){\line(1,0){140}}
\put(100,80){\line(1,0){140}}
\put(100,100){\line(1,0){140}}
\put(160,120){\line(1,0){20}}
\put(160,140){\line(1,0){20}}
\put(100,40){\line(0,1){60}}
\put(120,40){\line(0,1){60}}
\put(140,40){\line(0,1){60}}
\put(160,0){\line(0,1){140}}
\put(180,0){\line(0,1){140}}
\put(200,40){\line(0,1){60}}
\put(220,40){\line(0,1){60}}
\put(240,40){\line(0,1){60}}
\put(230,50){\makebox(0,0){\Tiny{$3$}}}
\put(230,70){\makebox(0,0){\Tiny{$2$}}}
\put(230,90){\makebox(0,0){\Tiny{$1$}}}
\put(210,50){\makebox(0,0){\Tiny{$6$}}}
\put(210,70){\makebox(0,0){\Tiny{$5$}}}
\put(210,90){\makebox(0,0){\Tiny{$4$}}}
\put(190,50){\makebox(0,0){\Tiny{$9$}}}
\put(190,70){\makebox(0,0){\Tiny{$8$}}}
\put(190,90){\makebox(0,0){\Tiny{$7$}}}
\put(170,50){\makebox(0,0){\Tiny{-$12$}}}
\put(170,70){\makebox(0,0){\Tiny{$0$}}}
\put(170,90){\makebox(0,0){\Tiny{$12$}}}
\put(150,50){\makebox(0,0){\Tiny{-$7$}}}
\put(150,70){\makebox(0,0){\Tiny{-$8$}}}
\put(150,90){\makebox(0,0){\Tiny{-$9$}}}
\put(130,50){\makebox(0,0){\Tiny{-$4$}}}
\put(130,70){\makebox(0,0){\Tiny{-$5$}}}
\put(130,90){\makebox(0,0){\Tiny{-$6$}}}
\put(110,50){\makebox(0,0){\Tiny{-$1$}}}
\put(110,70){\makebox(0,0){\Tiny{-$2$}}}
\put(110,90){\makebox(0,0){\Tiny{-$3$}}}
\put(170,110){\makebox(0,0){\Tiny{$11$}}}
\put(170,130){\makebox(0,0){\Tiny{$10$}}}
\put(170,10){\makebox(0,0){\Tiny{-$10$}}}
\put(170,30){\makebox(0,0){\Tiny{-$11$}}}
\end{picture}
\caption{\footnotesize{orthogonal Dynkin Pyramid for $(7^3,1^4)$}} \label{Fig:so-ex-2}
\end{figure}
\end{center}
}

From the pyramid, we obtain that 
\begin{align*} 
\g^0 & = \left\{ {\rm diag}(x_1,\ldots,x_{(q-1)/2},y,-\widehat{x_{(q-1)/2}},\ldots,-\widehat{x_1}) 
\colon x_i \in \mf{gl}_m, y \in \mf{so}_{m+s}\right\} &\\ 
& \cong  (\mf{gl}_m)^{(q-1)/2} \times \mf{so}_{m+s},& \\ 
\g^\natural &= \left\{{\rm diag}(x,\ldots,x,
 y,-\widehat{x},\ldots,-\widehat{x}) 
\colon  x= \begin{pmatrix} a & u & b \\ v & 0 & - \widehat{u} & \\ 
c & -\widehat{v} & -\widehat{a}
\end{pmatrix} \in \mf{so}_m,  \, a,b,c \in \mf{gl}_{(m-1)/2}, \right. &\\
&\qquad\quad \left. b = - \widehat{b}, \, c = - \widehat{c}, \, y = \begin{pmatrix} 
a & 0 & u & 0 & b \\
0 & a' & 0 & b' & 0 \\ 
v & 0 & 0 &  0 &  -\widehat{u}\\ 
0 & c' & 0  & -\widehat{a'} & 0 \\
c & 0  & -\widehat{v} & 0 & - \widehat{a}
\end{pmatrix} \in \mf{so}_{m+s}, \, b' = - \widehat{b'}, \, c' = - \widehat{c'}  \right\} & \\
& \cong \mf{so}_{m} \times \mf{so}_s. &
\end{align*}
We compute the $k_i^\natural$'s as in the case where $m$ is even. 

$\ast$ If $s$ is odd, 
then the orbit $\O_{1;\bs\mu}$ is even and the Dynkin pyramid is  
as 
Figure~\ref{Fig:so-ex-3}
if $q=7$, $m=3$ and $s=3$. 
{\tiny
\begin{center}
\begin{figure}
\setlength\unitlength{0.0175cm}
\hspace{-2cm}\begin{picture}(140,140)
\put(160,0){\line(1,0){20}}
\put(160,20){\line(1,0){20}}
\put(100,20){\line(1,0){140}}
\put(100,40){\line(1,0){140}}
\put(100,60){\line(1,0){140}}
\put(100,80){\line(1,0){140}}
\put(100,100){\line(1,0){140}}
\put(160,120){\line(1,0){20}}

\put(100,20){\line(0,1){80}}
\put(120,20){\line(0,1){80}}
\put(140,20){\line(0,1){80}}
\put(160,0){\line(0,1){120}}
\put(180,0){\line(0,1){120}}
\put(200,20){\line(0,1){80}}
\put(220,20){\line(0,1){80}}
\put(240,20){\line(0,1){80}}

\put(100,60){\line(1,1){20.1}}
\put(120,60){\line(1,1){20.1}}
\put(140,60){\line(1,1){20.1}}
\put(180,40){\line(1,1){20.1}}
\put(200,40){\line(1,1){20.1}}
\put(220,40){\line(1,1){20.1}}
\put(120,60){\line(-1,1){20.1}}
\put(140,60){\line(-1,1){20.1}}
\put(160,60){\line(-1,1){20.1}}
\put(200,40){\line(-1,1){20.1}}
\put(220,40){\line(-1,1){20.1}}
\put(240,40){\line(-1,1){20.1}}

\put(230,50){\makebox(0,0){{}}}
\put(230,30){\makebox(0,0){{$3$}}}
\put(210,30){\makebox(0,0){{$6$}}}
\put(190,30){\makebox(0,0){{$9$}}}
\put(230,70){\makebox(0,0){{$2$}}}
\put(230,90){\makebox(0,0){{$1$}}}
\put(210,50){\makebox(0,0){{}}}
\put(210,70){\makebox(0,0){{$5$}}}
\put(210,90){\makebox(0,0){{$4$}}}
\put(190,50){\makebox(0,0){{}}}
\put(190,70){\makebox(0,0){{$8$}}}
\put(190,90){\makebox(0,0){{$7$}}}
\put(170,50){\makebox(0,0){{-$11$}}}
\put(170,70){\makebox(0,0){{$11$}}}
\put(170,90){\makebox(0,0){{$10$}}}
\put(170,110){\makebox(0,0){{$12$}}}
\put(150,50){\makebox(0,0){{-$8$}}}
\put(150,30){\makebox(0,0){{-$7$}}}
\put(150,70){\makebox(0,0){{}}}
\put(150,90){\makebox(0,0){{-$9$}}}
\put(130,50){\makebox(0,0){{-$5$}}}
\put(110,50){\makebox(0,0){{-$2$}}}
\put(130,30){\makebox(0,0){{-$4$}}}
\put(130,90){\makebox(0,0){{-$6$}}}
\put(110,30){\makebox(0,0){{-$1$}}}
\put(110,90){\makebox(0,0){{-$3$}}}
\put(170,10){\makebox(0,0){{-$12$}}}
\put(170,30){\makebox(0,0){{-$10$}}}
\end{picture}
\caption{\footnotesize{orthogonal Dynkin Pyramid for $(7^3,1^3)$}} 
\label{Fig:so-ex-3} 
\end{figure}
\end{center}
}
From the pyramid we obtain that 
\begin{align*} 
\g^0 & = \left\{ {\rm diag}(x_1,\ldots,x_{(q-1)/2},y,-\widehat{x_{(q-1)/2}},\ldots,-\widehat{x_1}) 
\colon x_i \in \mf{gl}_m, y \in \mf{so}_{m+s}\right\} &\\ 
& \cong  (\mf{gl}_m)^{(q-1)/2} \times \mf{so}_{m+s},& 
\end{align*}
and 
\begin{align*}  
\g^\natural &= \left\{{\rm diag}(x,\ldots,x,
 y,-\widehat{x},\ldots,-\widehat{x}) 
\colon  x= \begin{pmatrix} a & u & b \\ 2v & 0 & - 2\widehat{u} & \\ 
c & -\widehat{v} & -\widehat{a}
\end{pmatrix}, \, a,b,c \in\mf{gl}_{(m-1)/2}, \, b = - \widehat{b}, \, c = - \widehat{c}, \, \right. &\\
&\hspace{6cm} \left. y = \begin{pmatrix} 
w & 0 & z & -z & 0 & 0 \\
0 & a & u & u & 0 & 0 \\
t & - v & 0 & 0 & - \widehat{u}  & \widehat{z}\\
- t & - v & 0 & 0 & - \widehat{u}  & - \widehat{z}\\
0 & 0 &  \widehat{v} & \widehat{v} & 0 & 0 \\
0 & 0 & \widehat{t} & - \widehat{t} & 0 & \widehat{w}
\end{pmatrix},\, w\in \so_{s-1}  \right\} & \\
& \cong \mf{so}_{m} \times \mf{so}_s. &
\end{align*}
The first statements then easily follows. 
We compute the $k_i^\natural$'s as in the case where $m$ is even. 
 
(2) $\bs\mu=(q^m,1^s)$ with $q$ even, $m$ even and $s>0$ arbitrary. 
Here, the orbit $\O_{1;\bs\mu}$ is odd, 
and we have 
\begin{align*}
\g^{0} = & \left\{{\rm diag}(x_1,\ldots,x_{q/2},y,-\widehat{x_{q/2}},
\ldots,-\widehat{x_1}) \right.\colon& \\
 & \qquad \qquad \left.  x_i \in \mf{gl}_{m}, \, y \in \mf{so}_s \right\} \cong 
(\mf{gl}_m)^{q/2} \times \mf{so}_s, &\\
\g^\natural = & \left\{{\rm diag}(x,0,\ldots,0,x,y,-\widehat{x}, 
0,\ldots,0,-\widehat{x}) \right. \colon & \\
 & \qquad \qquad \left.  x = \begin{pmatrix} a & b \\ 
 c & - \hat{a}
 \end{pmatrix}, b = \hat{b}, \, c=\hat{c}, \, y \in \mf{so}_s \right\} \cong 
\sp_{m} \times \mf{so}_s. & 
\end{align*}
Let $\g_1^\natural$ be the simple component of $\g^\natural$ isomorphic to $\sp_{m}$. 
To compute $k_1^\natural$, pick $t \in \g^\natural$ with $a= {\rm diag}(1,\ldots,0)$, 
$b=c=0$, $y=0$ in the above description of $\g^\natural$. 
Then $(t|t)_{\g^\natural}=2$, $(t|t)_\g = q$, 
$\kappa_\g(t,t)= 2q (n-2)$, 
 $\kappa_{\g^{0}}(x,x)= 2 m q$, 
 $\kappa_{\g^{1/2}}(t,t)= 2 s$. 
Therefore 
 $ 2 k_1^\natural 
= q k + q (n-2) -  n,$ 
 whence the expected result. 
 
Assume $s=3$ or $s>4$. 
Let $\g_2^\natural$ be the simple component of $\g^\natural$ isomorphic to $\so_{s}$. 
To compute $k_2^\natural$, pick $y={\rm diag}(1,\ldots,-1)$ and $x=0$ 
in the above description of $\g^\natural$. 
Then $(t|t)_{\g^\natural}=1$, $(t|t)_\g = 1$, 
$\kappa_\g(t,t)= 2 (n-2)$, 
 $\kappa_{\g^{0}}(x,x)= 2 (s-2)$,  
 $\kappa_{\g^{1/2}}(t,t)= 2 m$. 
Therefore 
 $  k_2^\natural = k + n-m-s,$ 
 whence the expected result. 
 
 We compute similarly $k_0^\natural$ for $s=2$ in which 
 case the $\mf{so}_s$ component is the centre $\g_0^\natural$ of $\g^\natural$. 
 Also, using the isomorphism $\mf{so}_4 \cong \sl_2 \times \sl_2$ 
 as in (1), $m=4$ case, we easily obtain that 
 $  k_2^\natural =k_3^\natural = k + n-m-s.$

It remains to check the last assertion. 
We do it for $\g=\sp_m$ and $\bs\mu=(q^m,1^s)$, the other cases can be checked easily as well. 
In this case, $\g^\natural \cong 
\mf{sp}_m \times \mf{sp}_s$, $k_1^\natural=qk + q(\frac{n}{2}+1) - (\frac{n}{2}+1)$ 
and $k_2^\natural= k + \frac{n}{2}- \frac{m}{2} - \frac{s}{2}$. 
Therefore $k_1^\natural = p - (\frac{n}{2}+1)$ which is a nonnegative integer. 
In particular it is admissible for $\mf{sp}_m$. 
On the other hand, $k_2^\natural= - (\frac{s}{2}+1) + \frac{p-q \frac{m}{2}}{q}$, 
with $(p-q \frac{m}{2},q)=1$ and 
$p-q \frac{m}{2} \ge \frac{s}{2}+1$.  
So $k_2^\natural$ is admissible for $\mf{sp}_s$. 
\end{proof}

{\tiny 
\begin{table}
\begin{tabular}{llll}
\hline
&&& \\[-0.5em]
$\bs\mu$ & $\g^\natural=\bigoplus_i \g^\natural_i$ 
& $k_i^\natural$ 
& 
conditions \\
&&& \\[-0.5em]
\hline 
&&& \\[-0.5em]
$(q^{m},{s})$ 
& $\mf{sp}_{ m}$ & $k_1^\natural=qk + q(\frac{n}{2}+1) - (\frac{n}{2}+1)$ & $q$ odd, $m$ even, 
$0 \le {s} \le q-1$ even  \\[0.5em]  
$(q^m,1^s)$ 
& $\mf{sp}_m \times \mf{sp}_s$ & $k_1^\natural=qk + q(\frac{n}{2}+1) - (\frac{n}{2}+1)$ & 
$q$ odd, $m$ even, $s\ge 0$ even  \\
&& $k_2^\natural= k + \frac{n}{2}- \frac{m}{2} - \frac{s}{2}$ & \\[0.5em] 
$(q^{m},q-1,{s})$ 
& $\mf{sp}_{m}$ & $k_1^\natural=qk + q(\frac{n}{2}+1)- (\frac{n}{2}+1)$ &  $q$ odd, $m$ even, 
$0 \le  {s} \le q-1$ even \\[0.5em]
 $(q^{m},q-1,1^{s})$ & 
 $\mf{sp}_{m} \times \mf{sp}_{s}$ & $k_1^\natural=qk + q(\frac{n}{2}+1) - (\frac{n}{2}+1)$ & 
 $q$ odd, $m$ even, $s \ge 0$ even  \\
&& $k_2^\natural= k + \frac{n}{2}- \frac{m}{2} - \frac{s}{2} -\frac{1}{2}$  & \\[0.5em] 
$(q^{m},(q-2)^2)$ & 
 $\mf{sp}_{m} \times \mf{sl}_{2}$ & $k_1^\natural=qk + q(\frac{n}{2}+1) - (\frac{n}{2}+1)$ & 
 $q$ odd, $m$ even  \\
&& $k_2^\natural= (q-2)(k + \frac{n}{2} - \frac{m}{2}) -1$  & \\[0.5em] 
$(\frac{q}{2}+1,(\frac{q}{2})^{m},{s})$ 
& $\mf{sp}_{m}$ & $k_1^\natural=\frac{q}{2}k + \frac{q}{2} (\frac{n}{2}+1) - \frac{n+1}{2}$ & 
$q$ even, $\frac{q}{2}$ odd, $m $ even, $0 \le  {s} \le q-1$ even  \\[0.5em] 
$(\frac{q}{2}+1,(\frac{q}{2})^m,1^s)$ 
& $\mf{sp}_m \times \mf{sp}_s$ & $k_1^\natural=\frac{q}{2} k + \frac{q}{2} (\frac{n}{2}+1) - \frac{n+1}{2}$ & 
$q$ even, $\frac{q}{2}$ odd, $m$ even, $s \ge 0$ even \\
&& $k_2^\natural= k + \frac{n}{2}- \frac{m}{2} - \frac{s}{2} -\frac{1}{2} $ &  \\[0.5em] 
$(\frac{q}{2}+1,(\frac{q}{2})^{m},\frac{q}{2}-1,{s})$ 
& $\sp_{m}$ & $k_1^\natural= \frac{q}{2} k + \frac{q}{2}(\frac{n}{2}+1) - \frac{n+1}{2}$ 
& $q$ even, $\frac{q}{2}$ odd, 
$m $ even, $2 \le {s} \le q-1$ even   \\[0.5em] 
$(\frac{q}{2}+1,(\frac{q}{2})^{m},\frac{q}{2}-1,1^{ s})$ 
& $\mf{sp}_{m} \times \mf{sp}_{s}$ & $k_1^\natural=\frac{q}{2} k 
+  \frac{q}{2}(\frac{n}{2}+1) - \frac{n+1}{2}$ & 
$q$ even, $\frac{q}{2}$ odd, $m$ even, $s \ge 0$ even \\
&& $k_2^\natural= k + \frac{n}{2}- \frac{m}{2} - \frac{s}{2} - 1 $ &  \\[0.5em] 
$(\frac{q}{2}+1,(\frac{q}{2})^{m},(\frac{q}{2}-2)^2)$ 
& $\mf{sp}_{m} \times \mf{sl}_{2}$  & $k_1^\natural= \frac{q}{2} k + \frac{q}{2}(\frac{n}{2}+1) - \frac{n+1}{2}$ & 
$q$ even, $\frac{q}{2}$ odd, $m$ even \\
&& $k_2^\natural=(\frac{q}{2}-2)(k+\frac{n}{2}-\frac{m}{2}-\frac{1}{2})-1$ &  \\[0.5em] 
$((\frac{q}{2})^{m},{s})$ 
& $\mf{so}_{m}$ & $k_1^\natural={q} k + {q} (\frac{n}{2}+1) -n$   & $q$ even, $\frac{q}{2}$ even, 
$m$ odd or even, $0 \le  {s} \le q-1$ even  \\[0.5em] 
 $((\frac{q}{2})^m,1^s)$  
& $\mf{so}_m \times \mf{sp}_s$ & $k_1^\natural={q} k + {q} (\frac{n}{2}+1) -n$ &  $q$ even, $\frac{q}{2}$ even, 

$s \ge 0$ even \\
&& $k_2^\natural= k + \frac{n}{2}- \frac{m}{2} - \frac{s}{2} $ &   \\[0.5em] 
$((\frac{q}{2})^m,(\frac{q}{2}-1)^2)$ & 
 $\mf{so}_{m} \times \mf{sl}_{2}$ & 
 $k_1^\natural=qk + q(\frac{n}{2}+1) - n$ & 
$q$ even, $\frac{q}{2}$ even, 
$m$ odd or even  \\
&& $k_2^\natural= (\frac{q}{2}-1)(k + \frac{n}{2} - \frac{m}{2}) -1$  & \\[0.5em] 
\hline
\end{tabular}
\caption{\footnotesize{Centralisers of some $\sl_2$-triples $(e,h,f)$ in $\sp_{n}$, 
with $f \in \O_{-1;\bs\mu}$}}
\label{Tab:C_Slodowy}
\end{table}
}

\bigskip

{\tiny 
\begin{table}
\begin{tabular}{llll}
\hline
&&& \\[-0.5em]
$\bs\mu$   
& $\g^\natural=\bigoplus_i \g^\natural_i$ 
& $k_i^\natural$ 
& 
conditions \\
&&& \\[-0.5em]
\hline 
&&& \\[-0.5em] 
$(q^{m},{s})$ 
& $\mf{so}_{m}$ 
& $k_1^\natural = qk + q(n-2) - (n-2)$ & $q$ odd, $m$ odd or even, 
$1 \le {s} \le q-1$ odd  \\[0.5em]
$(q^m,1^s)$ & $\mf{so}_{m} \times \mf{so}_s$  
& $k_1^\natural = qk + q(n-2) - (n-2)$ & $q$ odd, $m$ odd or even, $s \ge 1$ \\ 
&& $k_2^\natural = k + n - m - s$ &  \\[0.5em]
$(q^m,(q-1)^2)$ & $\mf{so}_{m} \times \mf{sl}_2$  
& $k_1^\natural = qk + q(n-2) - (n-2)$ & $q$ odd, $m$ odd or even\\ 
&& $k_2^\natural = \frac{q-1}{2}(k + n - m - 4)$ &  \\[0.5em]
$(q^m,(q-1)^2,1)$ &  $\mf{so}_{m} \times \mf{sl}_2$  
& $k_1^\natural = qk + q(n-2) - (n-2)$ & $q$ odd, $m$ even\\ 
&& $k_2^\natural = \frac{q-1}{2}(k + n - m - 4)$ &  \\[0.5em]

$(q^{m},{s},1)$ 
& $\mf{so}_{m}$ 
& $k_1^\natural = qk + q(n-2) - (n-2)$ & $q$ odd, $m$ odd or even, $3 \le {s} \le q-1$ odd \\[0.5em]
$(q^m,2^2)$ & $\mf{so}_{m} \times \mf{sl}_2$  
& $k_1^\natural = qk + q(n-2) - (n-2)$ & $q$ odd, $m$ odd or even \\ 
&& $k_2^\natural = k + n - m - 4$ &  \\[0.5em]
$(3^m,2^4)$ &   $\mf{so}_{m} \times \mf{sp}_4$ 
& $k_1^\natural = 3 k + 3 (n-2) - (n-2)$  & $m$ odd \\ 
& & $k_2^\natural = k + 2m +2$ &  \\[0.5em]
$(q+1,q^{m},{s})$  
& $\mf{sp}_{m}$ & 
$k_1^\natural =\frac{qk}{2} + \frac{q(n-2)}{2} - \frac{(n-1)}{2}$ & $q$ even, $m$ even, 
$1 \le  {s} \le q-1$ odd \\[0.5em]  
$(q+1,q^m,1^s)$ 
& $\mf{sp}_m \times \mf{so}_s$ &
$k_1^\natural =\frac{qk}{2} + \frac{q(n-2)}{2} - \frac{(n-1)}{2}$ & $q$ even, $m$ even, 
$s\ge 1$  odd \\ 
&& $k_2^\natural = k+n - m -s-1$ & \\[0.5em] 
$(q+1,q^{m},q-1,{s},1)$ &  
$\mf{sp}_{m}$ & 
$k_1^\natural =\frac{qk}{2} + \frac{q(n-2)}{2} - \frac{(n-1)}{2}$  &  $q$ even, $m$ even, 
$3 \le {s} \le q-1$ odd  \\[0.5em]
$(q+1,q^{m},q-1,1^s)$ 
& $\mf{sp}_{m} \times \mf{so}_s$ & 
$k_1^\natural =\frac{qk}{2} + \frac{q(n-2)}{2} - \frac{(n-1)}{2}$ & $q$ even, $m$ even, $s\ge 2$ 
even \\
&& $k_2^\natural = k+n - m -s-2$& \\[0.5em] 

$(q+1,q^m,q-1,2^2)$ & $\mf{sp}_{m} \times \mf{sl}_2$  
& $k_1^\natural =\frac{qk}{2} + \frac{q(n-2)}{2} - \frac{(n-1)}{2}$ 
& $q$ even, $m$ even \\ 
&& $k_2^\natural = k+n-m-6$ 
&  \\[0.5em]

$(q+1,q^m,(q-2)^2,1)$ & $\mf{sp}_{m} \times \mf{sl}_2$  
& $k_1^\natural =\frac{qk}{2} + \frac{q(n-2)}{2} - \frac{(n-1)}{2}$
& $q$ even, $m$ even \\ 
&& $k_2^\natural = \frac{q-2}{2}(k + n - m - 5)-\frac{1}{2}$ 
&  \\[0.5em]
$(q^{m},{s})$  
& $\mf{sp}_{ m}$ & 
$k_1^\natural =\frac{qk}{2} + \frac{q(n-2)}{2} - \frac{n}{2}$ & $q$ even, $m$ even, 
$1 \le {s} \le q - 1$ odd  \\[0.5em]
$(q^m,1^s)$  
& $\mf{sp}_{m} \times \mf{so}_s$ 
& $k_1^\natural =\frac{qk}{2} + \frac{q(n-2)}{2} - \frac{n}{2}$ &  $q$ even, $m$ even, $s\ge 0$ \\
&& $k_2^\natural = k + n - m -s$  &  \\[0.5em] 
$(q^{m},q-1,{s},1)$ 
& $\mf{sp}_m$ 
& $k_1^\natural = \frac{qk}{2} + \frac{q(n-2)}{2} - \frac{n}{2}$  &  $q$ even, $m$ even, 
$3 \le {s} \le q-1$ odd   \\[0.5em] 
$(q^{m},q-1,1^s)$ 
& $\mf{sp}_m \times \mf{so}_s$ 
& $k_1^\natural = \frac{qk}{2} + \frac{q(n-2)}{2} - \frac{n}{2}$ &  $q$ even, $m$ even, $s \ge 0$ \\
&& $k_2^\natural = k + n - m -s-1$ & \\[0.5em] 
$(q^m,q-1,2^2)$ & $\mf{sp}_{m} \times \mf{sl}_2$  
& $k_1^\natural = \frac{qk}{2} + \frac{q(n-2)}{2} - \frac{n}{2}$ & $q$ even, $m$ even \\ 
&& $k_2^\natural = k+n-m-1$ & \\[0.5em] 
$(q^m,(q-1)^2,1)$ & $\mf{sp}_{m} \times \mf{sl}_2$  
& $k_1^\natural = \frac{qk}{2} + \frac{q(n-2)}{2} - \frac{n}{2}$ & $q$ even, $m$ even \\ 
&& $k_2^\natural = \frac{q-1}{2}(k + n - m - 4)$ &  \\[0.5em]
$(q^m,(q-2)^2,1)$ &  $\mf{sp}_{m} \times \mf{sl}_2$
& $k_1^\natural = \frac{qk}{2} + \frac{q(n-2)}{2} - \frac{n}{2}$ & $q$ even, $m$ even \\ 
& & $k_2^\natural = \frac{q-2}{2}(k + n - m - 4)-\frac{1}{2}$   &  \\[0.5em]
\hline
\end{tabular}
\caption{\footnotesize{Centralisers of some $\sl_2$-triples $(e,h,f)$ in $\so_{n}$, 
with $f \in \O_{1;\bs\mu}$}} 
\label{Tab:BD_Slodowy}
\end{table}
}

We are now in a position to state our results on collapsing levels for $\sp_n$. 
First, we consider the case where $\W_k(\g,f)$ is lisse, that is, $f \in \O_k$.

\begin{Th} 
\label{Th:main_sp_n-1}    
Assume that $k = -h_\g^{\vee} + p/q=- \left(\frac{n}{2}+1\right) + p/q$ 
is an admissible level for $\g=\mf{sp}_n$. 
Pick a nilpotent element $f \in \O_k$ so that $\W_k(\g,f)$ is lisse. 
\begin{enumerate} 
\item Assume that $k$ is principal. 
If $p=h_{\sp_n}^{\vee}=\frac{n}{2}+1$, then for generic\footnote{Here, by {\em generic} $q$, we mean all $q$ but a finite number.} $q$, 
$k$ is collapsing if and only if $n \equiv 0,-1 \mod q$. 
If $n \equiv 0,-1 \mod q$, then for generic $q$, 
$k$ is collapsing if and only if $p=h_{\sp_n}^{\vee}$. 
Moreover, if $n \equiv 0,-1 \mod q$, then
$$\W_{-h_{\sp_n}^{\vee} + h_{\sp_n}^{\vee}/q}(\mf{sp}_n,f)  \cong H_{DS,f}^0(L_{-h_{\sp_n}^{\vee} + h_{\sp_n}^{\vee}/q}(\mf{sp}_n))\cong \C.$$
\item Assume that $k$ is a coprincipal admissible level for $\g=\mf{sp}_n$. 
If $p=h_{\sp_n}+1=n+1$, then for generic $q$, 
$k$ is collapsing if and only if $n \equiv 0,1  \mod \frac{q}{2}$. 
If  $n \equiv 0,1  \mod \frac{q}{2}$, then for generic $q$, 
$k$ is collapsing if and only if $p=h_{\sp_n}+1$
Moreover, if $n \equiv 0,1  \mod \frac{q}{2}$ with $\frac{q}{2}$ odd, then 
$$\W_{-h_{\sp_n}^{\vee} + (h_{\sp_n}+1)/q}(\mf{sp}_n,f)  \cong H_{DS,f}^0(L_{-h_{\sp_n}^{\vee} + (h_{\sp_n}+1)/q}(\mf{sp}_n)) \cong \C,$$
and if $n \equiv 0,1  \mod \frac{q}{2}$ with $\frac{q}{2}$ even, then 
$$\W_{-h_{\sp_n}^{\vee} + (h_{\sp_n}+1)/q}(\mf{sp}_n,f) \cong H_{DS,f}^0(L_{-h_{\sp_n}^{\vee} + (h_{\sp_n}+1)/q}(\mf{sp}_n)) \cong L_{1}(\mf{so}_{m}).$$ 
\end{enumerate}
\end{Th}

Note that, in the above statements, the isomorphisms $\W_k(\g,f) \cong H_{DS,f}^0(L_k(\g))$ 
holds due to Remark~\ref{Rem:lisse_conjecture_true}. 

\begin{proof}
Let $\bs\lam$ be the partition corresponding to $f \in \O_k$. 
In the below proof, we always exploit the {symplectic Dynkin pyramid} 
of shape $\bs\lam$ as described in \S\ref{sub:sympl_pyramids}. 
Letting $I:=\{1,\ldots,\frac{n}{2},-\frac{n}{2},\ldots,-1\}$ the set of labels, 
we notice that for $j \in \frac{1}{2}\Z_{> 0}$, 
\begin{align}
\label{eq:cardinality_pyramid}
 \# \{\alpha \in \Delta_+ \colon (x^0|\alpha) = j \} = 
\# \{ (i,l) \in I \colon 0 < i \le |l| , \, |\on{col}(i) - \on{col}(l) |/2 = j \}. 
\end{align} 
In this way, the pyramid allows us to compute the central charge 
and the asymptotic dimension of $H_{DS,f}^0(L_k(\g))$. 
Indeed, the central charge of $H_{DS,f}^0(L_k(\g))$ is given by
$$c_{H_{DS,f}^0(L_k(\g))} = \dim \g^0 -\frac{1}{2} \dim \g^{1/2} - 12\left(\frac{q}{p} |\rho|^2 -2 (\rho|x^0)
+\frac{p}{q} |x^0|^2\right),$$
with 
\begin{align*}
|x^0|^2 & = \frac{1}{h_\g^\vee} \sum_{\alpha \in\Delta_+} (x^0 | \alpha)^2, \qquad 
 (\rho|x^0)  = \frac{1}{2} \sum_{\alpha \in\Delta_+} (x^0 | \alpha), &
\end{align*}
while $|\rho|^2$ is computed using the strange formula:
$$|\rho|^2=\dfrac{h_\g^\vee \dim \g}{12} = \dfrac{n(n+1)(n+2)}{48}.$$
On the other hand, \eqref{eq:cardinality_pyramid} enables to compute 
the term $\prod\limits_{\alpha\in \Delta_+ \setminus \Delta_{+}^0 } 
2 \sin \left(\dfrac{\pi (x^0|\alpha) }{q}\right)$ in 
the asymptotic dimension of $H_{DS,f}^0(L_k(\g))$ in the principal case. 

For the coprincipal case, note that 
 \begin{align*}
\prod\limits_{\alpha\in \Delta_+ \setminus \Delta_{+}^0 } 
2 \sin \left(\dfrac{\pi (x^0|\alpha^\vee) }{q}\right) 
= \prod\limits_{\alpha\in \Delta_+^{\rm long} \setminus \Delta_{+}^0 } 
2 \sin \left(\dfrac{\pi (x^0|\alpha) }{q}\right) 
\prod\limits_{\alpha\in \Delta_+^{\rm short} \setminus \Delta_{+}^0 } 
2 \sin \left(\dfrac{2 \pi (x^0|\alpha) }{q}\right), 
\end{align*}
and that for $j \in \frac{1}{2}\Z_{>0}$, 
\begin{align}
\label{eq:cardinality_pyramid-co}
& \# \{\alpha \in \Delta_+^{\rm short} \colon (x^0|\alpha) = j \}= 
\# \{ (i,l) \in I \colon  0 < i < |l| ,  \, |\on{col}(i) - \on{col}(l) |/2 = j \},& \\\nonumber
& \# \{\alpha \in \Delta_+^{\rm long} \colon (x^0|\alpha) = j \}= 
\# \{ (i,l) \in I \colon i >0, \, l = - i , \, |\on{col}(i) - \on{col}(l) |/2 = j \},&
\end{align}

We now follow the strategy of Section~\ref{sec:strategy}.

\smallskip

\noindent 
(1) In this part, $q$ is assumed to be odd. 
From the description of $\O_k$, 
either $\O_k=\O_{(q^{\widetilde m}, \widetilde s)}$ or $\O_k=\O_{(q^{\widetilde m-1},q-1,{\widetilde s}+1)}$. 
We consider the two cases separately. 

(a) Assume  that $\O_k=\O_{(q^{\widetilde m}, \widetilde s)}$, 
with $\widetilde m,\widetilde s$ even, and set $m:=\widetilde m$, 
$s:= \widetilde s$ for simplicity.  
According to Table \ref{Tab:C_Slodowy}, we have 
$\g^\natural  \cong \mf{sp}_{m}$, and 
$k^\natural = p-(\frac{n}{2}+1)$. 
Using the symplectic pyramid  of shape $(q^{m},s)$, 
we establish for $j=1,\ldots,(q-1)/2$ that 
\begin{align*}
&\# \{\alpha \in \Delta_+ \colon (x^0|\alpha)=2j-1\}  = \left(\dfrac{q-2j+1}{2}\right)\times m^2 
+\begin{cases} 
s/2-j+1 &  \text{ if } j \le s/2\\
0 & \text{ else.}
\end{cases},  & \\
&\# \{\alpha \in \Delta_+ \colon (x^0|\alpha)=2j\}  = \left(\dfrac{q-2j-1}{2}\right)\times m^2 +\dfrac{m(m+1)}{2} 
+\begin{cases} 
s/2-j & \text{ if } j \le s/2 \\
0 & \text{ else.}
\end{cases}.& 
\end{align*}

The cardinality of roots $\alpha \in \Delta_+$ such that $(x^0|\alpha)$ is a half-integer 
can be computed as follows. 

$\ast$ If $s-1 \le \frac{q-1}{2}-\frac{s}{2}+1$, then 
\begin{align*}
& \# \{\alpha \in \Delta_+ \colon (x^0|\alpha)=(2j-1)/2\}  & \\
& \hspace{2.5cm}= 
\begin{cases} sm & \text{ if } j=1,\ldots,s-1, \\ 
\left(\frac{s}{2}+\left\lfloor \frac{j}{2}\right\rfloor\right)m & \text{ if } j=s,\ldots,\frac{q-1}{2}-\frac{s}{2}+1,\\
\left(\frac{s}{2}-i+\left\lfloor \frac{j}{2}\right\rfloor\right)m & \text{ if } 
j=\frac{q-1}{2}-\frac{s}{2}+1+i \text{ with }i=1,\ldots,s-1.
\end{cases}&
\end{align*}

$\ast$ If $s-1 > \frac{q-1}{2}-\frac{s}{2}+1$, then 
\begin{align*}
& \# \{\alpha \in \Delta_+ \colon (x^0|\alpha)=(2j-1)/2\}  & \\
& \hspace{2.5cm}= 
\begin{cases} sm & \text{ if } j=1,\ldots,\frac{q-1}{2}-\frac{s}{2}+1,\\
 (s-i)m & \text{ if } j=\frac{q-1}{2}-\frac{s}{2}+1+i\text{ with }i=1,\ldots,\frac{3 s}{2}-2-\frac{q-1}{2},\\
 \left(\frac{s}{2}-i+\left\lfloor \frac{j}{2}\right\rfloor\right)m
 & \text{ if }j=\frac{q-1}{2}-\frac{s}{2}+1+i\text{ with }i=\frac{3 s}{2}-1-\frac{q-1}{2},\ldots,s-1.
 \end{cases}&
\end{align*}
From this, we get an expression of 
$c_{H_{DS,f}^0(L_k(\g))}$ depending on $q,m,s$. 
On the other hand 
$$c_{L_{k^\natural}(\g^\natural)}  
=\frac{m ( m+1) (2 - 2 p + m q + s)}{2 (-2 p + m ( q-1) + s)}.$$

Fix $p=h_{\g}^\vee = \frac{n}{2}+1$. 
For generic $q$, the only solutions of the equation 
$c_{H_{DS,f}^0(L_k(\g))} =c_{L_{k^\natural}(\g^\natural)}$ 
with unknown $s$  
are $s=0$ and  $s=q-1$. 
Now if we fix $s=0$, for generic $q$, 
the only solutions of the equation $c_{H_{DS,f}^0(L_k(\g))} =c_{L_{k^\natural}(\g^\natural)}$ 
with unknown $p$ are, 
$$p = \frac{n}{2}+1\quad \text{ or }\quad p = \dfrac{n+1}{2}.$$ 
Only $p = \frac{n}{2}+1$ is greater than $h_\g^\vee$. 
If we fix $s=q-1$, for generic $q$, 
the only solution of the equation $c_{H_{DS,f}^0(L_k(\g))} =c_{L_{k^\natural}(\g^\natural)}$ 
with unknown $p$ is $p = \frac{n}{2}+1$.

$\ast$ Assume $p = \frac{n}{2} +  1$ and $s=0$. 
In this case, $k_1^\natural=0$ and $\G_{H_{DS,f}^0(L_{k}(\mf{sp}_n))}=0$. 
We aim to apply Proposition \ref{Pro:asymptotics-and-collapsing}. 
By Proposition~\ref{Pro:asymptotic_data_H_DS} and Lemma \ref{Lem:main_identities} (1), 
we have 
 \begin{align}
 \label{eq:ampl-sp_n-1}
& \A_{H_{DS,f}^0(L_{ -(\frac{n}{2}+1) + (\frac{n}{2}+1)/q}(\mf{sp}_n))} 
=  \dfrac{1}{q^{|\Delta_+^0|} 
q^{\frac{n}{4}}} 
\prod\limits_{\alpha\in \Delta_+ \setminus \Delta_{+}^0 } 
2 \sin \left(\dfrac{\pi (x^0|\alpha) }{q}\right) & 
\end{align} 
with $n=qm$ 
and $|\Delta_+^0| =  \frac{(q-1)m(m-1)}{4}+\frac{m^2}{4}$, 
since the orbit $\O_{-1;\bs\lam}$ is even. 
By \eqref{eq:cardinality_pyramid} and the previous computations,  
we show that 
\begin{align}
\label{eq:sin_Delta+-sp_n-1}
& \prod\limits_{\alpha\in \Delta_+ \setminus \Delta_{+}^0 } 
2 \sin \left(\dfrac{\pi (x^0|\alpha) }{q}\right)   
=  q^{\frac{qm^2}{4}+\frac{m}{4}},& 
\end{align}
using the identities \eqref{eq:sin_formula} and \eqref{eq:sin_formula2}. 
Combining \eqref{eq:ampl-sp_n-1} and \eqref{eq:sin_Delta+-sp_n-1} we conclude that 
$\A_{H_{DS,f}^0(L_{ -(\frac{n}{2}+1) + (\frac{n}{2}+1)/q}(\mf{sp}_n))}
=1$ as desired. 

$\ast$ Assume $p = \frac{n}{2} +  1$ and $s=q-1$.  
Then $k_1^\natural=0$ and $\G_{H_{DS,f}^0(L_{k}(\mf{sp}_n))}=0$. 
As in the previous case, we aim to apply Proposition \ref{Pro:asymptotics-and-collapsing}. 
By Proposition~\ref{Pro:asymptotic_data_H_DS} and Lemma~\ref{Lem:main_identities} (1), 
we have 
 \begin{align}
 \label{eq:ampl-sp_n-odd}  
 \A_{H_{DS,f}^0(L_{ -(\frac{n}{2}+1) + (\frac{n}{2}+1)/q}(\mf{sp}_n))}  
 = \dfrac{1}{2^{\frac{|\Delta^{1/2}|}{2}} q^{|\Delta_+^0|} 
q^{\frac{n}{4}}} 
\prod\limits_{\alpha\in \Delta_+ \setminus \Delta_{+}^0 } 
 \sin \left(\dfrac{\pi (x^0|\alpha) }{q}\right) & 
\end{align} 
with $n=qm+q-1$, $|\Delta_+^0| =\frac{(q-1)m(m-1)}{4} +\frac{m^2}{4}$ 
and $|\Delta^{1/2}| = (q-1)m$. 
Using the identities \eqref{eq:sin_formula} and \eqref{eq:sin_formula2} 
and the previous computations, we show that 
\begin{align}
\prod\limits_{\alpha\in \Delta_+ \setminus \Delta_{+}^0 } 
 \sin \left(\dfrac{\pi (x^0|\alpha) }{q}\right) 
& 
  = \left(\dfrac{1}{2}\right)^{\frac{(q-1)^2 m}{2}} \left(\dfrac{q}{2^{q-1}}\right)^{\frac{q(m^2+1)+(m-1)}{4}} . &
\label{eq:sin_Delta+-sp_n-odd}
\end{align}
Combining \eqref{eq:ampl-sp_n-odd} and \eqref{eq:sin_Delta+-sp_n-odd}, we 
conclude that $\A_{H_{DS,f}^0(L_{ -(\frac{n}{2}+1) + (\frac{n}{2}+1)/q}(\mf{sp}_n))}=1$
as desired. 

(b) Assume now that $\O_k=\O_{(q^m,q-1,s)}$, with $m := \widetilde m-1$, 
$s:=\widetilde s+1$.  
We argue as in (a). According to Table~\ref{Tab:C_Slodowy}, we have 
$\g^\natural  \cong \mf{sp}_m$, and 
$k_1^\natural = p-(\frac{n}{2}+1)$. 
We compute the central charges of 
$H_{DS,f}^0(L_k(\g))$ and $L_{k^\natural}(\g^\natural)$ similarly as in (a). 
Fix $p=h_{\g}^\vee = \frac{n}{2}+1$. 
Solving the equation $c_{H_{DS,f}^0(L_k(\g))} =c_{L_{k^\natural}(\g^\natural)}$ 
with unknown $s$, we get that for generic $q$, the only solutions  
are $0$ or $q+1$, the second one being excluded.  
Fix $s=0$.
Solving the equation $c_{H_{DS,f}^0(L_k(\g))} =c_{L_{k^\natural}(\g^\natural)}$ 
with unknown $p$, we get that for generic $q$, the only solutions  are
$$p = \frac{n}{2}\quad \text{ or }\quad p = \dfrac{n}{2}+1.$$ 
Only $p = \frac{n}{2}+1$ is greater than $h_\g^\vee$. 
But the case $p = \frac{n}{2}+1$ and $s=0$ has been 
already dealt with in part (a). 

\smallskip

\noindent 
(2) In this part, $q$ is assumed to be even. 
It follows from the description of $\O_k$. 
that either $\O_k=\O_{((\frac{q}{2})^{\widetilde m+1}, \widetilde s+1 )}$ 
with even $\frac{q}{2}$, or 
$\O_k=\O_{(\frac{q}{2}+1,(\frac{q}{2})^{\widetilde m}, \widetilde  s )}$ 
with odd $\frac{q}{2}$, 
or  $\O_k=\O_{(\frac{q}{2}+1,(\frac{q}{2})^{\widetilde m -1 },\frac{q}{2}-1, \widetilde  s +1)}$ 
with odd $\frac{q}{2}$. 
 
(a) Assume that $\O_k=\O_{((\frac{q}{2})^{\widetilde m+1}, \widetilde s+1 )}$, 
with $\frac{q}{2}$ even, and set $m:=\widetilde m+1$, $s:=\widetilde s+1$ for simplicity. 
According to Table \ref{Tab:C_Slodowy}, we have 
$\g^\natural  \cong \mf{so}_{m}$, and 
$k^\natural  = p- n$. 
The orbit $\O_{k}$ is even. 
Hence one can directly use the asymptotic growth.  
We have 
\begin{align*}
\G_{H_{DS,f}^0(L_{k}(\mf{sp}_n))} & = 
\frac{1}{2} (m^2 \left(\frac{q}{2}- s \right) +(m+1)^2  s ) 
- \dfrac{(n-1)n(n+2)}{p q} ,& 
\end{align*} 
with $n=\frac{q m}{2}+ s $, 
while 
\begin{align*}
\G_{L_{k^\natural}(\g^\natural)} & = \dfrac{m(m-1)}{2} 
\left( 1 - \dfrac{m -2}{p - \frac{qm}{2}- s 
+ m -2}\right) . &
 \end{align*}
Fixing $ s   =0$, we find that, for generic $q$, 
the only solution of the equation  $\G_{H_{DS,f}^0(L_{k}(\mf{sp}_n))} =  \G_{L_{k^\natural}(\g^\natural)}$ with unknown $p$ 
is $p =n+1$. 
Fixing $p = n +  1$, 
the only solutions of the equation  $\G_{H_{DS,f}^0(L_{k}(\mf{sp}_n))} =  \G_{L_{k^\natural}(\g^\natural)}$ with unknown $ s  $ 
are $ s =0$ or $ s =\frac{q}{2}+1$, the second case being excluded. 

From now on it is assumed that ${ s  }=0$ and $p =n+  1$. 
Thus $k^\natural=1$ and $\G_{H_{DS,f}^0(L_{k}(\mf{sp}_n))}=\frac m{2}$. 
By Corollary~\ref{Co:asymptotic_data_L}, 
 $\A_{L_1(\mf{sp}_m)} = \dfrac{1}{2}$, 
and by Proposition~\ref{Pro:asymptotic_data_H_DS},
 \begin{align*}
& \A_{H_{DS,f}^0(L_{ -(\frac{n}{2}+1) + (n+1)/q}(\mf{sp}_n))} 
=  \dfrac{2^{|\Delta_+^{\rm short} \cap \Delta_+^0| }} 
{q^{|\Delta_+^0| } 
{\left( \frac{q}{2} \right)}^{\frac{n}{4}} \sqrt{4}} 
\prod\limits_{\alpha\in \Delta_+ \setminus \Delta_{+}^0 } 
2 \sin \left(\dfrac{\pi (x^0|\alpha^\vee) }{q}\right) & 
\end{align*} 
with $|\Delta_+^0| = |\Delta_+^{\rm short} \cap \Delta_+^0| 
= \frac{q \widetilde m(\widetilde m -1)}{8}$ since $\O_{k}$ is even. 
Using the symplectic pyramid of shape $((\frac{q}{2})^{m})$, 
we establish for $j=1,\ldots, q/4$ that 
\begin{align*}
& \# \{\alpha \in \Delta_+^{\rm short} 
\colon (x^0|\alpha)=2j-1\}  = m^2 \left(\frac{q}{4}-j\right) +\frac{{m} ({m}+1)}{2},  & \\
& \# \{\alpha \in \Delta_+^{\rm long} 
\colon (x^0|\alpha)=2j-1\} = \widetilde{m},  & \\
& \# \{\alpha \in \Delta_+^{\rm short} 
\colon (x^0|\alpha)=2j\}   =m^2 \left(\frac{q}{4}-j \right).& 
\end{align*}
From this, we deduce as in previous computations that 
$$\A_{H_{DS,f}^0(L_{ -(\frac{n}{2}+1) + (n+1)/q}(\mf{sp}_n))}
=\frac{1}{2}$$ 
whence the statement by Proposition \ref{Pro:asymptotics-and-collapsing}.

(b) Assume that $\O_k=\O_{(\frac{q}{2}+1,(\frac{q}{2})^{\widetilde m},  \widetilde s )}$ 
with $\frac{q}{2}$ odd, $m:=\widetilde m$ even, $s:= \widetilde s$ even. 
According to Table \ref{Tab:C_Slodowy}, we have 
$\g^\natural  \cong \mf{sp}_{m}$, and 
$k^\natural = \frac{p}{2}- \frac{n+1}{2}$. 
Using the symplectic pyramid of shape 
$(\frac{q}{2}+1,(\frac{q}{2})^{m},s)$, 
we compute $\# \{\alpha \in \Delta_+^{\rm short} 
\colon (x^0|\alpha)=j\}$ and 
$\# \{\alpha \in \Delta_+^{\rm long} 
\colon (x^0|\alpha)=j\}$ for each $j$. 
For $s=0$, we establish that for $j=1,\ldots,q/2$  

$\# \{\alpha \in \Delta_+^{\rm short} 
\colon (x^0|\alpha)=(2j-1)/2\} = m \left(\frac{q}{2}-(j-1) \right)$,

\smallskip

$\# \{\alpha \in \Delta_+^{\rm short} \colon (x^0|\alpha)=2j-1\} = \frac{q+2}{4}-j + 
\left(\frac{q+2}{4}-j \right) m^2$,   

\smallskip

$\# \{\alpha \in \Delta_+^{\rm long} \colon (x^0|\alpha)=2j-1\} = 1$, 

\smallskip

$\# \{\alpha \in \Delta_+^{\rm short} \colon (x^0|\alpha)=2j\} = \frac{q+2}{4}-j + \left(\frac{q+2}{4}-j\right) m^2 
+ \frac{ m( m-1)}{2}$,

\smallskip

\noindent 
and that for  $j=1,\ldots, (q+2)/4-1$, 

$\# \{\alpha \in \Delta_+^{\rm long} \colon (x^0|\alpha)=2j\} =m$. 

\smallskip

For $s\not=0$, we establish that 
\begin{align*}
&\# \{\alpha \in \Delta_+^{\rm short} 
\colon (x^0|\alpha)=(2j-1)/2\} =m \left(\frac{q}{2}-(j-1) \right)& \\
& \qquad\qquad  +\begin{cases} ms & \text{if } j \le \frac{q+2-2s}{4}, \\
 m(s-\frac{q+2-2s}{4}+j) & \text{if } \frac{q+2-2s}{4} +1 \le j \le  \frac{q+2-2s}{4} +  s-1,\\
0 & \text{if } \frac{q+2-2s}{4} + s \le j \le \frac{q}{2}, 
\end{cases}&\\
&\# \{\alpha \in \Delta_+^{\rm short} 
\colon (x^0|\alpha)= 2j-1 \} = \left(\frac{q+2}{4}-j \right) m^2& \\
& \qquad \qquad 
+\begin{cases} \frac{q+2-2s}{4} -(2j-1) +3 + 4\left(\frac{s}{2}-j\right)& \text{if } j \le \frac{q+2-2s}{4}, \\
 2 & \text{if } \frac{q+2-2s}{4} +1 \le j \le  \frac{q+2-2s}{4} +  s-1,\\
1 & \text{if } \frac{q+2-2s}{4} + s \le j \le \frac{q}{2}, 
\end{cases}&\\
&\# \{\alpha \in \Delta_+^{\rm long} \colon (x^0|\alpha)=2j-1\} 
 = \begin{cases} 2 & \text{if } 2j-1 \le s-1, \\
1 & \text{otherwise,}
\end{cases}&\\
&\# \{\alpha \in \Delta_+^{\rm short} 
\colon (x^0|\alpha)= 2j \} = \left(\frac{q+2}{4}-j\right) m^2 
+ \frac{ m( m-1)}{2} & \\
& \qquad \qquad 
+\begin{cases} \frac{q+2-2s}{4} -2j +2 + 4\left(\frac{s}{2}-j\right)& \text{if } j \le \frac{q+2-2s}{4}, \\
 2 & \text{if } \frac{q+2-2s}{4} +1 \le j \le  \frac{q+2-2s}{4} +  s-1,\\
1 & \text{if } \frac{q+2-2s}{4} + s \le j \le \frac{q}{2},
\end{cases}&\\
&\# \{\alpha \in \Delta_+^{\rm long} \colon (x^0|\alpha)=2j\} 
 = m. &
\end{align*}

Fixing $p = n +  1$, 
the only solutions of the equation  $c_{H_{DS,f}^0(L_{k}(\mf{sp}_n))} =c_{L_{k^\natural}(\g^\natural)}$ with unknown $ s  $ 
are $ s =0$ and $ s =\frac{q}{2}-1$. 
Fixing $ s =0$ or $s=\frac{q}{2}-1$, we find that for generic $q$ 
the only admissible 
solution of the equation  $c_{H_{DS,f}^0(L_{k}(\mf{sp}_n))} =c_{L_{k^\natural}(\g^\natural)}$ with unknown $p$ 
is $p =n+1$. 

$\ast$ Assume that $p=n+1$ and $s =0$. 
We have $k^\natural=0$, $\G_{H_{DS,f}^0(L_{k}(\mf{sp}_n))}=0$ 
and  $\A_{L_1(\mf{sp}_m)} =1$.  
Here we have 
$|\Delta_+^0| =\frac{(q-2) m( m -1)}{8}+\frac{(q-2) }{4}
\left(\frac m{2}\right)^{2}$, 
$|\Delta_+^{\rm short} \cap \Delta_+^0| 
= \frac{q  m( m -1)}{8}+\frac{(q-2) }{4}+\frac{ m( m -2)}{4}$, 
$|\Delta^{1/2}_+| = m+  m (q-2)$, 
and using the above computations,  
we obtain  that 
$\A_{H_{DS,f}^0(L_{ -(\frac{n}{2}+1) + (n+1)/q}(\mf{sp}_n))}
=1$, whence the statement by  Proposition \ref{Pro:asymptotics-and-collapsing}. 

$\ast$ Assume that $p=n+1$ and $ s =\frac{q}{2}-1$. 
We have $k^\natural=0$, $\G_{H_{DS,f}^0(L_{k}(\mf{sp}_n))}=0$ 
and  $\A_{L_1(\mf{sp}_m)} =1$.  
Here we have $|\Delta_+^0| =\frac{(q-2) m( m -1)}{8}+\left(\frac m{2}\right)^{2}$, 
$|\Delta_+^{\rm short} \cap \Delta_+^0| 
= \frac{q  m( m -1)}{8}+\frac{ m( m -2)}{4}$, 
$|\Delta^{1/2}_+| = \frac{q m}{2}$  
and using the above computations we deduce that 
$\A_{H_{DS,f}^0(L_{ -(\frac{n}{2}+1) + (n+1)/q}(\mf{sp}_n))}
=1$, whence the statement by Proposition \ref{Pro:asymptotics-and-collapsing}.

(c) Assume that $\O_k=\O_{(\frac{q}{2}+1,(\frac{q}{2})^{m},\frac{q}{2}-1,  s )}$ 
with $\frac{q}{2}$ odd, $m:=\widetilde m -1$ even, 
$s:=\widetilde s +1$ even. 
According to Table \ref{Tab:C_Slodowy}, we have 
$\g^\natural  \cong \mf{sp}_{m}$, and 
$k^\natural = \frac{p}{2}- \frac{n+1}{2}$. 
As in the previous cases, we compute the central charge 
$c_{H_{DS,f}^0(L_{k}(\mf{sp}_n))}$ using the symplectic pyramid 
of shape $(\frac{q}{2}+1,(\frac{q}{2})^{m},\frac{q}{2}-1,  s )$. 
We omit here the details. 
Fixing $p = n +  1$, 
the only solutions of the equation  $c_{H_{DS,f}^0(L_{k}(\mf{sp}_n))}=c_{L_{k^\natural}(\g^\natural)}$ with unknown $ s  $ 
are $ s =0$ or $ s =\frac{q}{2}+1$, the second case being excluded. 
Fixing $ s   =0$, we find that for generic $q$ 
the only admissible 
solution of the equation  $c_{H_{DS,f}^0(L_{k}(\mf{sp}_n))}=c_{L_{k^\natural}(\g^\natural)}$ with unknown $p$ 
is $p =n+1$. 
But the case where $p = n +  1$ and $ s =0$ has been already dealt 
with in part (2) (b). 

This terminates the proof of part (2). 
\end{proof}

\begin{Rem}
Since $\C$ and $L_1(\mf{so}_m)$ are rational,
Theorem \ref{Th:main_sp_n-1}   
confirms Conjecture~\ref{Conj:rationality}
for $\g=\mf{sp}_n$,
$k+h^{\vee}_\g=h^{\vee}_\g/q$ with
$n\equiv 0,-1\pmod{q}$, $q$ odd,
and $k+h^{\vee}_\g=(h^{\vee}_\g+1)/q$ with
$n\equiv 0,1\pmod{q/2}$, $q$ even.
\end{Rem}

\begin{Rem}\label{rem:min.mod.cases.sp}
Although our main goal is to identify collapsing levels, i.e., to prove isomorphisms between certain $W$-algebras and affine vertex algebras, it is worth remarking that our methods can be used to prove other isomorphisms too, as we now illustrate.

We consider the coprincipal admissible level $k = -h_\g^\vee + p/q$, for $\g = \mf{sp}_n$,  where $q$ is twice an odd integer and $p=h+1$. Then $\O_k = \O_{\bs\lam}$ where $\bs\lam =(\frac{q}{2}+1, (\frac{q}{2})^{m}, s)$ for some $m, s$ even. Let $f \in \O_k$. In Theorem \ref{Th:main_sp_n-1} above we showed that if $s=0$ then $H_{DS, f}^0(L_{k}(\g))$ is isomorphic to the trivial vertex algebra $\C$.

If instead $s=2$ then the central charge of $H_{DS, f}^0(L_{k}(\g))$, which in general is given by the formula
\begin{align*}
c = -\frac{s (q-2s-2) \left(qs + q -2s^2 -2s + 4 \right)}{4 q}
\end{align*}
becomes
\begin{align*}
c = 13 - \frac{24}{q} - \frac{3q}{2},
\end{align*}
which is the central charge of the Virasoro minimal model $\on{Vir}_{2, q/2}$. The values of $(x_0 | \alpha)$ for $\alpha \in \Delta_+$ can be read off from the symplectic pyramid of $\bs\lam$ as was done in the proof of Theorem \ref{Th:main_sp_n-1}, to obtain a formula for the asymptotic dimension $\A_{H_{DS, f}^0(L_{k}(\g))}$. On the other hand the asymptotic dimension of $\on{Vir}_{2, q/2}$ is given by formula (\ref{eq:minmod.as.dim}) and the two expressions can be shown to coincide by an elementary, though very tedious, calculation. It follows that 
\begin{align*}
\W_k(\g,f)\cong H_{DS, f}^0(L_{k}(\g)) \cong \on{Vir}_{2, q/2}.
\end{align*}
This gives a yet another evidence for  Conjecture \ref{Conj:isom}
and Conjecture \ref{Conj:rationality}.
\end{Rem}

We now turn to collapsing levels for $\mf{sp}_n$ in the \emph{non} lisse cases, that is, we consider the partitions $\bs\mu \in \P_{-1}(n)$ identified in Lemma \ref{Lem:choice_of_partitions-C} cases (b)--(k). Our study leads us to the following results. 

\begin{Th} 
\label{Th:main_sp_n-2}    
Assume that $k = -h_\g^{\vee} +p/q=- \left(\frac{n}{2}+1\right) + p/q$  
is admissible for $\g=\mf{sp}_n$. 
\begin{enumerate}
\item Assume that $k$ is principal, that is, $q$ is odd. 
\begin{enumerate}
\item Pick a nilpotent element $f \in \overline{\O}_k$ 
corresponding to the partition $(q^m,1^s)$, with $m,s$ even. 
For generic $q$, $k$ is collapsing if and only if $p=h_{\sp_n}^{\vee}$, and 
$$\W_{-h_{\mf{sp}_n}^{\vee} + h_{\mf{sp}_n}^{\vee}/q}(\mf{sp}_n,f)  \cong L_{-h_{\sp_s}^{\vee} + h_{\sp_s}^{\vee} /q}(\mf{sp}_s).$$
\item Pick a nilpotent element $f \in \overline{\O}_k$ 
corresponding to the partition $(q^{m},q-1,1^{s})$, with $s$ even. 
For generic $q$, $k$ is collapsing only if $p=h_{\mf{sp}_n}^{\vee}$ 
and 
the following inclusion is a 
finite extension: 
\begin{align*} 
L_{-h_{\mf{sp}_s}^{\vee}+(h_{\mf{sp}_s}+1)/(2q)}(\mf{sp}_s)  \longhookrightarrow 
\W_{-h_{\mf{sp}_n}^{\vee} + h_{\mf{sp}_n}^{\vee}/q}(\mf{sp}_n,f).
\end{align*}
\end{enumerate}
\item Assume that $k$ is coprincipal, that is, $q$ is even. 
\begin{enumerate}
\item 
Pick a nilpotent element $f \in \overline{\O}_k$ 
corresponding to the partition  
$(\frac{q}{2}+1,(\frac{q}{2})^{m},1^s)$, with $\frac{q}{2}$ odd, $m$ even, $s$ even. 
If $p=h_{\sp_n}+1$, 
then $k$ is collapsing if and only if $s=0$ or $s=2$. 
If $s=2$, then for generic $q$, $k$ is collapsing if and only if $p=h_{\sp_n}+1$. 
Moreover, if $s=2$, then  
$$\qquad \W_{-h_{\mf{sp}_n}^{\vee} + (h_{\sp_n}+1)/q}(\mf{sp}_n,f)  \cong 
L_{-h_{\mf{sp}_s}^{\vee}+h_{\mf{sp}_s}^{\vee}/(q/2)}(\sp_s).$$

\item Pick a nilpotent element $f \in \overline{\O}_k$ 
corresponding to the partition $(\frac{q}{2}+1,(\frac{q}{2})^{m}, \frac{q}{2}-1,1^s)$, with $\frac{q}{2}$ odd 
and $s$ even. 
For generic $q$, $k$ is collapsing only if $p=h_{\sp_n}+1$ and 
the following inclusion is a 
finite extension: 
\begin{align*} 
L_{-h_{\sp_s}^\vee+(h_{\sp_s}+1)/q}(\sp_s) \longhookrightarrow 
\W_{-h_{\sp_n}^\vee + (h_{\sp_n}+1)/q}(\mf{sp}_n,f).
\end{align*}

\item Pick a nilpotent element $f \in \overline{\O}_k$ 
corresponding to the partition $((\frac{q}{2})^m,1^s)$, with $\frac{q}{2}$ even, 
$m$ odd or even, and $s$  even. 
For generic $q$, $k$ is collapsing only if $p=h_{\sp_n}+1$ and 
the following inclusion is a 
finite extension: 
\begin{align*} 
L_1(\so_m) \otimes (L_{-h_{\mf{sp}_s}^{\vee}+(h_{\sp_s}+1)/q}(\mf{sp}_{s})) 
\longhookrightarrow 
\W_{- h_{\mf{sp}_n}^{\vee}+ (h_{\sp_n}+1)/q}(\mf{sp}_n,f) 
\end{align*}
\end{enumerate}
\end{enumerate}
\end{Th}

\begin{Rem}
Similarly to Remark \ref{Rem:compatible_sl_n}, notice that (3) and (4) 
 are compatible with (1) for $s= 0$, and (5), (6), (7) is compatible with (2) with $s= 0$. 
\end{Rem}

\begin{proof}
We argue as in the proof of Theorem \ref{Th:main_sp_n-1}. 
So we omit some details when the computations are very similar to those 
considered in Theorem \ref{Th:main_sp_n-1}. 

\smallskip

\noindent 
(1) In this part, $q$ is odd. 

(a) Fix a nilpotent element 
$f \in\overline{\O}_k$ corresponding to the partition $(q^m,1^s)$. 
According to Table \ref{Tab:C_Slodowy}, we have 
$\g^\natural  \cong \mf{sp}_m \times \mf{sp}_{s} $,  
$k_1^\natural = p-(\frac{n}{2}+1)$ and 
$k_2^\natural  
= - (\frac{s}{2}+1) + (p-\frac{qm}{2})/q$. 
Since $f$ is even, $\W_k(\g,f) \cong H_{DS,f}^0(L_{k}(\mf{sp}_n))$, 
and Proposition \ref{Pro:asymptotics-and-collapsing} gives a necessary and 
sufficient condition for that $k$ is collapsing. 

By Proposition \ref{Pro:asymptotic_data_H_DS}, we have 
\begin{align*}
\G_{H_{DS,f}^0(L_{k}(\mf{sp}_n))} 
& =  \frac{1}{2} ((m+s)^2 +m^2(q-1) +m+s) - \dfrac{n(n+1)(n+2)}{4pq},& 
\end{align*} 
with $n=qm+s$, 
while by  Corollary~\ref{Co:asymptotic_data_L}, 
\begin{align*}
\G_{L_{k^\natural}(\g^\natural)} & = 
\G_{L_{-(\frac{m}{2}+1)+(\frac{m-n}{2}+p)/1}(\mf{sp}_m)}+\G_{L_{- (\frac{s}{2}+1) + (p-\frac{qm}{2})/q}(\mf{sp}_s)}
& \\
&= \dfrac{m(m+1)}{2} \left( 1 - \dfrac{m+2}{m-n+2p}\right) 
+  \dfrac{s(s+1)}{2} \left( 1 - \dfrac{s+2}{(2p- qm)q}\right). &
 \end{align*}
By solving the equation  $\G_{H_{DS,f}^0(L_{k}(\mf{sp}_n))}  =\G_{L_{k^\natural}(\g^\natural)}$ with unknown $p$, 
we obtain that, for generic $q$, the only solutions which are nonnegative integers are 
$$p = \dfrac{n+1}{2}  \quad \text{ and }\quad p = \dfrac{n}{2} +1  .$$
Only the solution $p = \frac{n}{2} +  1$ is greater than $h_\g^\vee$. 
From now on it is assumed that $p = \frac{n}{2} +  1$. 
Then $k_1^\natural=0$ and $k_2^\natural = -(\frac{s}{2}+1) + (\frac{s}{2}+1)/q$. 
We now apply Proposition~\ref{Pro:asymptotics-and-collapsing} to prove that $k$ is collapsing. 
By Corollary~\ref{Co:asymptotic_data_L} and Lemma \ref{Lem:main_identities} (1), 
we have 
\begin{align*}
\A_{L_{ -(\frac{s}{2}+1) + (\frac{s}{2}+1)/q}(\mf{sp}_s)} & = 
\dfrac{1}{q^{s^2/4} q^{s/4}} .& 
\end{align*}
On the other hand, by Proposition \ref{Pro:asymptotic_data_H_DS} and 
Lemma \ref{Lem:main_identities} (1), we have 
 \begin{align}
\label{eq:ampl-sp_n-2}
 \A_{H_{DS,f}^0(L_{ -(\frac{n}{2}+1) + (\frac{n}{2}+1)/q}(\mf{sp}_n))}  = 
 \dfrac{1}{q^{|\Delta_+^0|} 
q^{\frac{n}{4}}} 
\prod\limits_{\alpha\in \Delta_+ \setminus \Delta_{+}^0 } 
2 \sin \left(\dfrac{\pi (x^0|\alpha) }{q}\right) & 
\end{align} 
with $n=qm+s$  and $|\Delta_+^0| =  \frac{(q-1)m(m-1)}{4}+\frac{(m+s)^2}{4}$, 
since the orbit $G.f$ is even. 
Using the pyramid associated with the partition $(q^m,1^s)$ we establish that for $j=1,\ldots,(q-1)/2$, 

\smallskip

$\# \{\alpha \in \Delta_+ \colon (x^0|\alpha)=2j-1\} = \left(\frac{q-2j-1}{2}\right)\times m^2 + m(m+s) ,$

\smallskip

$\# \{\alpha \in \Delta_+ \colon (x^0|\alpha)=2j\}  = \left(\frac{q-2j-1}{2}\right)\times m^2 + m(m+s) +\frac{m(m+1)}{2},$

\smallskip

$\# \{\alpha \in \Delta_+ \colon (x^0|\alpha)=q-2j\} = j m^2 ,$
\smallskip

$\# \{\alpha \in \Delta_+ \colon (x^0|\alpha)=q- 2j+1\} = (j-1) m^2+ \frac{m(m+1)}{2} ,$

\smallskip

We deduce that 
\begin{align}
\label{eq:sin_Delta+-sp_n-2}
\prod\limits_{\alpha\in \Delta_+ \setminus \Delta_{+}^0 } 
2 \sin \left(\dfrac{\pi (x^0|\alpha) }{q}\right) 
= &   q^{\frac{qm^2}{4}+\frac{ms}{2}+\frac{m}{4}}  & 
\end{align}
by \eqref{eq:sin_formula} and \eqref{eq:sin_formula2}. 
Combining \eqref{eq:ampl-sp_n-2} and \eqref{eq:sin_Delta+-sp_n-2}, we 
conclude that 
$$\A_{H_{DS,f}^0(L_{ -(\frac{n}{2}+1) + (\frac{n}{2}+1)/q}(\mf{sp}_n))}
=\A_{L_{ -(\frac{s}{2}+1) + (\frac{s}{2}+1)/q}(\mf{sp}_s)},$$
as desired. 

(b) Fix a nilpotent element 
$f \in \overline{\O}_k$ corresponding to the partition $(q^m,q-1,1^s)$. 
According to Table \ref{Tab:C_Slodowy}, we have 
$\g^\natural  \cong \mf{sp}_m \times \mf{sp}_{s} $,  
$k_1^\natural  = p-(\frac{n}{2}+1)$ and 
$k_2^\natural  
= - (\frac{s}{2}+1) + \frac{2p-mq-q}{2q}$, 
with $ (2p-mq-q,2q)=1$. 
Using the pyramid of shape $(q^m,q-1,1^s)$, 
we establish that for $j=1,\ldots,(q-1)/2$, 

\smallskip

$\# \{\alpha \in \Delta_+ \colon (x^0|\alpha)=\frac{2j-1}{2}\}   = \left( \frac{q-1}{2}\right) m + \left( \frac{q-(2j-1)}{2}\right) m +s$, 

\smallskip

$\# \{\alpha \in \Delta_+ \colon (x^0|\alpha)=\frac{q+2j-2}{2} \}  = \left( \frac{q-(2j-1)}{2}\right) m$, 

\smallskip

$\# \{\alpha \in \Delta_+ \colon (x^0|\alpha)= 2j-1 \} =  \left( \frac{q-(2j-1)}{2}\right) m^2 + \left( \frac{q-(2j-1)}{2}\right)  + 
\begin{cases}
ms & \text{  if } 2j-1 \le \frac{q-1}{2} \\
0 & \text{ otherwise},
\end{cases}$, 

$\# \{\alpha \in \Delta_+ \colon (x^0|\alpha)= 2j \} =  \left( \frac{q- 2j-1}{2}\right) m^2 + \left( \frac{q-2j-1}{2}\right)  +\frac{m(m+1)}{2}
+ \begin{cases}
ms & \text{  if } 2j-1 \le \frac{q-1}{2} \\
0 & \text{ otherwise.}
\end{cases}$

\smallskip

As explained in the proof of Theorem \ref{Th:main_sp_n-1}, 
the above cardinalities allow us to compute the central charge of $H_{DS,f}^0(L_{k}(\mf{sp}_n))$. 
By solving the equation $c_{H_{DS,f}^0(L_{k}(\mf{sp}_n))} =c_{L_{k^\natural}(\g^\natural)} $ with unknown $p$, 
we obtain that 
for generic $q$ the only admissible 
solution is $p = \dfrac{n}{2}+1.$ 
Assume that $p = \frac{n}{2} +  1$. 
Then $k_1^\natural=0$ and $k_2^\natural = -(\frac{s}{2}+1) + \frac{s+1}{2q}$  
which is coprincipal admissible for $\mf{sp}_s$. 
We easily verify that 
$\G_{H_{DS,f}^0(L_{ -(\frac{n}{2}+1) + (\frac{n}{2}+1)/q}(\mf{sp}_n))}
= \G_{L_{-(\frac{s}{2}+1) + \frac{s+1}{2q}}(\mf{sp}_s)}.$
By Corollary~\ref{Co:asymptotic_data_L} and 
Lemma~\ref{Lem:main_identities} (4),  
we have 
\begin{align*}
\A_{L_{ -(\frac{s}{2}+1) + (s+1)/(2q)}(\mf{sp}_s)} & = 
\dfrac{1}{2^{\frac{s}{2}+1} q^{\frac{s}{4}(s+1) }} .& 
\end{align*}
On the other hand, by Proposition \ref{Pro:asymptotic_data_H_DS} and 
Lemma \ref{Lem:main_identities} (1),  
\begin{align*}
& \A_{H_{DS,f}^0(L_{ -(\frac{n}{2}+1) + (\frac{n}{2}+1)/q}(\mf{sp}_n))}  =& \\ 
&\qquad\qquad  \dfrac{1}{2^{|\Delta^{1/2}|/2}q^{|\Delta_+^0|} 
q^{\frac{n}{4}}} 
\prod\limits_{\alpha\in \Delta_+ \setminus \Delta_{+}^0 } 
2 \sin \left(\dfrac{\pi (x^0|\alpha) }{q}\right) & 
\end{align*} 
with $n=qm+q-1+s$, $|\Delta_+^0| =  \frac{(q-1)m(m-1)}{4}+\frac{(m+s)^2}{4}$ 
and $|\Delta^{1/2}|=(q-1)m+s$. 
From the above computations, 
we get that 
$$\A_{H_{DS,f}^0(L_{ -(\frac{n}{2}+1) + (\frac{n}{2}+1)/q}(\mf{sp}_n))}
= 2 \A_{L_{-(\frac{s}{2}+1) + \frac{s+1}{2q}}(\mf{sp}_s)}.$$
Since $\G_{H_{DS,f}^0(L_{ -(\frac{n}{2}+1) + (\frac{n}{2}+1)/q}(\mf{sp}_n))}
= \G_{L_{-(\frac{s}{2}+1) + \frac{s+1}{2q}}(\mf{sp}_s)},$
it follows from Theorem~\ref{Th:main}  
$\W_{ -(\frac{n}{2}+1) + (\frac{n}{2}+1)/q}(\mf{sp}_n,f)$ 
a direct sum of admissible representations of $L_{-(\frac{s}{2}+1) + \frac{s+1}{2q}}(\mf{sp}_s)$.

\smallskip

\noindent 
(2) In this part, $q$ is even.

(a) Pick a nilpotent element $f \in \overline{\O}_k$ 
corresponding to the partition $(\frac{q}{2}+1,(\frac{q}{2})^{m}, 1^s)$, with odd $\frac{q}{2}$. 
According to Table \ref{Tab:C_Slodowy}, we have 
$\g^\natural  \cong \mf{sp}_m \times \mf{sp}_{s} $,  
$k_1^\natural = 
\frac{p}{2} - \frac{n+1}{2}$ and 
$k_2^\natural  
= - (\frac{s}{2}+1) + \frac{2p-mq -q}{2q}$. 

Using the pyramid of shape $(\frac{q}{2}+1,(\frac{q}{2})^{m}, 1^s)$, 
we establish that for $j=1,\ldots,q/2$, 
\begin{align*}
\# \{\alpha \in \Delta_+^{\rm short} \colon (x^0|\alpha)=\frac{2j-1}{2}\}   
= \left( \frac{q+2-2j}{2}\right) m +  
\begin{cases}
s & \text{  if } 2j-1 \le \frac{q}{2} \\
0 & \text{ otherwise},
\end{cases}
\end{align*}
that for $j=1,\ldots,(q+2)/4$, 
\begin{align*}
& \# \{\alpha \in \Delta_+^{\rm short} \colon (x^0|\alpha)=2j-1 \}  
 = \left( \frac{q+2-4j}{4}\right) (m^2+1)+
\begin{cases}
ms & \text{  if } 2j  \le \frac{q-2}{4} \\
0 & \text{ otherwise,}
\end{cases}&\\ 
& \# \{\alpha \in \Delta_+^{\rm long} \colon (x^0|\alpha)=2j-1 \}  
= 1,& 
\end{align*}
and that for $j=1,\ldots,(q-2)/4$, 
\begin{align*}
& \# \{\alpha \in \Delta_+^{\rm short} \colon (x^0|\alpha)= 2j \} 
=  \left( \frac{q+2-4j}{4}\right)(m^2 +1)
+  \frac{m(m-1)}{2} + 
\begin{cases}
ms & \text{  if } 2 j \le \frac{q-2}{4} \\
0 & \text{ otherwise,}
\end{cases}&\\ 
&\# \{\alpha \in \Delta_+^{\rm long} \colon (x^0|\alpha)= 2j \}=m.&
\end{align*}
The above cardinalities allows us to compute the central charge of $H_{DS,f}^0(L_{k}(\mf{sp}_n))$. 

By solving the equation  $c_{H_{DS,f}^0(L_{k}(\mf{sp}_n))} =c_{L_{k^\natural}(\g^\natural)} $ 
with unknown $p$, 
we obtain that, for generic $q$, the only admissible solution is  
$p = n+1$. 
Fixing $p=n+1$, we obtain that the only nonnegative integer 
solutions the equation  $c_{H_{DS,f}^0(L_{k}(\mf{sp}_n))} =c_{L_{k^\natural}(\g^\natural)} $ with unknown $s$ 
are $s=0$ and $s=2$. 
The case where $p=n+1$ and $s=0$ has been already dealt 
with in the proof of Theorem \ref{Th:main_sp_n-1}. 
We now assume that $p =n+  1$ and $s=2$. 
Then $k_1^\natural$, $k_2^\natural=(\frac{s}{2}+1)/(q/2)$ and  
we easily verify that $\G_{H_{DS,f}^0(L_{k}(\mf{sp}_n))} =\G_{L_{k^\natural}(\g^\natural)}$. 
Let us compare the asymptotic dimensions. 
We obtain here that 
\begin{align*}
\A_{L_{ -(\frac{s}{2}+1) + (s/2+1)/(q/2)}(\mf{sp}_s)} & = 
\dfrac{1}{(q/2)^{\frac{s^2}{4}+\frac{s}{3}}},& 
\end{align*}
and that  
\begin{align*}
& \A_{H_{DS,f}^0(L_{ -(\frac{n}{2}+1) + (n+1)/q}(\mf{sp}_n))}  =& \\ 
&\qquad\qquad  \dfrac{2^{|\Delta^{\rm short} \cap \Delta_+^0|}}{2^{|\Delta^{1/2}|/2}q^{|\Delta_+^0|} 
(q/2)^{\frac{n}{4}}\times \sqrt{4}} 
\prod\limits_{\alpha\in \Delta_+ \setminus \Delta_{+}^0 } 
2 \sin \left(\dfrac{\pi (x^0|\alpha) }{q}\right) & 
\end{align*} 
with $n=qm/2+q/2+3$, $|\Delta_+^0| =  \frac{(q-2)m(m-1)}{4}+\frac{(m+2)^2}{4}$, 
$|\Delta^{\rm short} \cap \Delta_+^0|=|\Delta_+^0|-\frac{m+2}{2}$  
and $|\Delta^{1/2}|=qm/2+2$. 
From the above computations, 
we can compute the last factor in $\A_{H_{DS,f}^0(L_{ -(\frac{n}{2}+1) + (n+1)/q}(\mf{sp}_n))}$ 
and we get that $\A_{H_{DS,f}^0(L_{ -(\frac{n}{2}+1) + (n+1)/q}(\mf{sp}_n))}$ equals 
$\A_{L_{ -(\frac{s}{2}+1) + (s/2+1)/(q/2)}(\mf{sp}_s)}$, whence the expected result. 

(b) Pick a nilpotent element $f \in \overline{\O}_k$ 
corresponding to the partition $(\frac{q}{2}+1,(\frac{q}{2})^m,\frac{q}{2}-1,1^s)$, with odd $\frac{q}{2}$. 
According to Table \ref{Tab:C_Slodowy}, we have 
$\g^\natural  \cong \mf{sp}_m \times \mf{sp}_{s} $,  
$k_1^\natural  = \frac{p}{2}-  \frac{n+1}{2}$ and 
$k_2^\natural  
= - (\frac{s}{2}+1) + 
\frac{p-\frac{qm}{2}-q}{q}$. 
Using the pyramid of shape 
$(\frac{q}{2}+1,(\frac{q}{2})^m,\frac{q}{2}-1,1^s)$ we establish that 
for $j=1,\ldots,q/2$, 

\smallskip

$\# \{\alpha \in \Delta_+^{\rm short} \colon (x^0|\alpha)=(2j-1)/2\} = 
m+m(q-2j) + \begin{cases} 2s & \text{ if } j \le (q-2)/4, \\
s & \text{ if } j = (q+2)/4, \\
0 & \text{otherwise,}
\end{cases}$\\
and that for $j=1,\ldots,(q-2)/4$, 

\smallskip

$\# \{\alpha \in \Delta_+^{\rm short} \colon (x^0|\alpha)=2j-1\} = 
2+ \left( \frac{q+2-4j}{4}\right)(m^2+4)+1$ 

\hfill $ + \begin{cases} m s & \text{ if } 2j-1 \le (q-2)/4, \\
0 & \text{otherwise,}
\end{cases}$ 

\smallskip

$\# \{\alpha \in \Delta_+^{\rm long} \colon (x^0|\alpha)=2j-1\} = 2,$ 

\smallskip

$\# \{\alpha \in \Delta_+^{\rm long} \colon (x^0|\alpha)=q/2\} = 1,$

\smallskip

$\# \{\alpha \in \Delta_+^{\rm short} \colon (x^0|\alpha)=2j\} = 
2+ \left( \frac{q-2-4j}{4}\right)(m^2+4)+\frac{m(m-1)}{2}$

\hfill $+ \begin{cases} m s & \text{ if } 2j \le (q-2)/4 \\
0 & \text{otherwise,}
\end{cases},$

\smallskip

$\# \{\alpha \in \Delta_+^{\rm long} \colon (x^0|\alpha)=2j\} = m,$

\smallskip

By solving the equation  $c_{H_{DS,f}^0(L_{k}(\mf{sp}_n))} 
=c_{L_{k^\natural}(\g^\natural)} $ with unknown $p$, 
we obtain that, for generic $q$, the only solution is $p =  n+1$. 
From now on we assume that $p=n+1$. 
Then $k_1^\natural=0$ and $k_2^\natural = -(\frac{s}{2}+1) + \frac{s+1}{q}$ 
and we have 
$\G_{\W_{-(\frac{n}{2}+1) + \frac{n+1}{q}}(\mf{sp}_n,f) }= 
\G_{L_{-(\frac{s}{2}+1) + \frac{s+1}{q}}(\mf{sp}_s) }$. . 
Moreover,  
\begin{align*}
\A_{L_{ -(\frac{s}{2}+1) + \frac{s+1}{q}}(\mf{sp}_s)} & = 
\frac{2^{\frac{s}{2}(\frac{s}{2}-1)}}{q^{\frac{s^2}{4}}(q/2)^{\frac{s}{4}}\sqrt{4}},& 
\end{align*}
and, as in the previous case, we compute $\A_{H_{DS,f}^0(L_{ -(\frac{n}{2}+1) + \frac{n+1}{q}}(\mf{sp}_n))}$. 
Here we get that 
$$\A_{H_{DS,f}^0(L_{ -(\frac{n}{2}+1) + \frac{n+1}{q}}(\mf{sp}_n))}
=2 \A_{L_{ -(\frac{s}{2}+1) +\frac{s+1}{q}}(\mf{sp}_s)},$$
and we conclude as in (1) (b). 

(c) Fix a nilpotent element $f \in \overline{\O}_k$ 
corresponding to the partition $((\frac{q}{2})^m,1^s)$, with even $\frac{q}{2}$. 
According to Table \ref{Tab:C_Slodowy}, we have 
$\g^\natural  \cong \mf{so}_m \times \mf{sp}_{s} $,  
$k_1^\natural = p-n $ and 
$k_2^\natural  
= - (\frac{s}{2}+1) + (p-\frac{qm}{2})/q$. 
Using the pyramid associated with the partition $((\frac{q}{2})^m,1^s)$ we establish that 
for $j=1,\ldots,q/4$, 
\smallskip

$\# \{\alpha \in \Delta_+^{\rm short} \colon (x^0|\alpha)=(2j-1)/2\} = 
ms,$

\smallskip

$\# \{\alpha \in \Delta_+^{\rm short} \colon (x^0|\alpha)=2j-1\} = 
m^2 \left( \frac{q}{4}-1\right)+\frac{m(m-1)}{2},$

\smallskip

$\# \{\alpha \in \Delta_+^{\rm long} \colon (x^0|\alpha)=2j-1\} = m,$\\
and that 
for $j=1,\ldots,q/4-1$,

$\# \{\alpha \in \Delta_+^{\rm short} \colon (x^0|\alpha)=2j\} = 
m^2 \left( \frac{q}{4}-1\right).$

\smallskip

By solving the equation  $c_{H_{DS,f}^0(L_{k}(\mf{sp}_n))} 
=c_{L_{k^\natural}(\g^\natural)} $ with unknown $p$, 
we obtain that, for generic $q$, the only solutions are 
$$p = n  \quad \text{ and }\quad p = n+1  .$$
Only the solution $p = n +  1$ leads to an admissible level. 
From now on it is assumed that $p = n+  1$. 
Then $k_1^\natural=1$ and $k_2^\natural = -(\frac{s}{2}+1) + \frac{s+1}{q}$ 
and we have 
$\G_{\W_{-(\frac{n}{2}+1) + \frac{n+1}{q}}(\mf{sp}_n,f) }=
\G_{L_{1}(\mf{so}_m) }+ 
\G_{L_{-(\frac{s}{2}+1) + \frac{s+1}{q}}(\mf{sp}_s) } =\frac{m}{2}+\frac{s (2 + q - 2 s + q s)}{2 q}$. 

By Corollary~\ref{Co:asymptotic_data_L} and Lemma~\ref{Lem:main_identities} (1), 
\begin{align*}
\A_{L_{1}(\mf{so}_m) \otimes L_{ -(\frac{s}{2}+1) + \frac{s+1}{2}}(\mf{sp}_s)} & = \frac{1}{2}\times 
\frac{2^{\frac{s}{4}(s-1)-1}}{q^{\frac{s}{4}(s+1)}}.& 
\end{align*}
On the other hand, by Proposition \ref{Pro:asymptotic_data_H_DS} and 
Lemma \ref{Lem:main_identities} (1), 
 \begin{align*}
 \A_{H_{DS,f}^0(L_{ -(\frac{n}{2}+1) + \frac{n+1}{q}}(\mf{sp}_n))}  = 
 \dfrac{2^{|\Delta_+^{\rm short} \cap \Delta_+^0|}}{2^{|\Delta_+^{1/2}|/2}q^{|\Delta_+^0|} 
q^{\frac{n}{4}}} 
\prod\limits_{\alpha\in \Delta_+ \setminus \Delta_{+}^0 } 
2 \sin \left(\dfrac{\pi (x^0|\alpha) }{q}\right) & 
\end{align*} 
with $n=qm+s$, $|\Delta_+^{\rm short} \cap \Delta_+^0| =\frac{m(m-1)q}{8}+\frac{s(s-2)}{4} $, 
$|\Delta_+^0| =\frac{m(m-1)q}{8}+\left(\frac{s}{2}\right)^2 $ 
and $|\Delta_+^{1/2}|=ms$. 
Using the above computations, we get  that 
$$\A_{H_{DS,f}^0(L_{ -(\frac{n}{2}+1) + \frac{n+1}{2}}(\mf{sp}_n))}
=2 \A_{L_1(\so_m) \otimes L_{ -(\frac{s}{2}+1) +\frac{s+1}{2}}(\mf{sp}_s)},$$
and we conclude as in (1) (b). 
\end{proof}

One can specify the decomposition of the finite extension (1) (b) 
in Theorem~\ref{Th:main_sp_n-2}  
assuming that Conjecture \ref{Conj:isom} holds, i.e:
$$\W_{-h_{\mf{sp}_n}^{\vee} + h_{\mf{sp}_n}^{\vee}/q}(\mf{sp}_n,f)\cong 
H_{DS,f}^0(L_{-h_{\mf{sp}_n}^{\vee} + h_{\mf{sp}_n}^{\vee}/q}(\mf{sp}_n)).$$ 
(The same arguments work for the the finite extensions in (2) (b) and (c).)
Keep the notations of the proof. 
Denoting by $L_{k_2^\natural}(\sp_s;\lam)$ the highest irreducible 
representation of $L_{k_2^\natural}(\sp_s)$ of admissible weight $\lam$, 
we get that 
$$H_{DS,f}^0(L_{ -(\frac{n}{2}+1) + (\frac{n}{2}+1)/q}(\mf{sp}_n)) = \bigoplus_{\Delta_\lam 
\in \frac{1}{2} \Z} L_{-(\frac{s}{2}+1) + \frac{s+1}{2q}}(\sp_s;\lam)^{\otimes m_\lam},$$
where $\Delta_\lam$ is the lowest $L_0$-eigenvalue  
of $ L_{-(\frac{s}{2}+1) + \frac{s+1}{2q}}(\sp_s;\lam)$. 
Moreover, the sum is finite.

We have $p_2- h_{\sp_s} = (s+1)-s=1$, where $k_2^\natural+h_{\sp_s}^\vee=p_2/q_2$. 
Recall that 
$$\Delta_\lam(\sp_s)= 
\dfrac{( \lam | \lam +2\rho_{\sp_s} )}{2(k_2^\natural+h_{\sp_s}^\vee)}.$$
For generic $q$, we observe that $\lam =\varpi_1$ is the only fundamental weight 
for which $\Delta_\lam \in \frac{1}{2}\Z$. 
One the other hand, we easily verify using Proposition \ref{Pro:asymptotic_data_H_DS} that 
$$\on{qdim}(L_{ -(\frac{s}{2}+1) + (s+1)/(2q)}(\mf{sp}_s ;\varpi_1))=1,$$ that is, 
$$\A_{L_{ -(\frac{s}{2}+1) + (s+1)/(2q)}(\mf{sp}_s ;\varpi_1)}=\A_{L_{ -(\frac{s}{2}+1) + (s+1)/(2q)}(\mf{sp}_s)}.$$
As a result, 
$$m_0  \A_{L_{ -(\frac{s}{2}+1) + (\frac{s}{2}+1)/q}(\sp_s;0)}+ 
m_{\varpi_1} \A_{L_{ -(\frac{s}{2}+1) + (\frac{s}{2}+1)/q}(\sp_s;\varpi_1)} = 
2 \A_{L_{ -(\frac{s}{2}+1) + (\frac{s}{2}+1)/q}(\sp_s;0)}.$$ 
But $m_0$ must be at most $1$. 
So, if  
$$\W_{-h_{\mf{sp}_n}^{\vee} + h_{\mf{sp}_n}^{\vee}/q}(\mf{sp}_n,f)\cong 
H_{DS,f}^0(L_{-h_{\mf{sp}_n}^{\vee} + h_{\mf{sp}_n}^{\vee}/q}(\mf{sp}_n)),$$
we get  
$$\W_{-h_{\mf{sp}_n}^{\vee} + h_{\mf{sp}_n}^{\vee}/q}(\mf{sp}_n,f) 
\cong L_{-h_{\mf{sp}_s}^{\vee}+(h_{\mf{sp}_s}+1)/(2q)}(\mf{sp}_s) 
\oplus L_{-h_{\mf{sp}_s}^{\vee}+(h_{\mf{sp}_s}+1)/(2q)}(\mf{sp}_s; \varpi_1).$$ 
Arguing as above, we conjecture explicit  
decompositions for the other finite extensions:

\begin{Conj}
\label{Conj:dec_finite_extensions-sp_n}
For generic $q$, we have the following finite extensions:
\begin{align*} 
&\hspace{1cm} \W_{-h_{\mf{sp}_n}^{\vee} + h_{\mf{sp}_n}^{\vee}/q}(\mf{sp}_n,f) 
\cong L_{-h_{\mf{sp}_s}^{\vee}+(h_{\mf{sp}_s}+1)/(2q)}(\mf{sp}_s) 
\oplus L_{-h_{\mf{sp}_s}^{\vee}+(h_{\mf{sp}_s}+1)/(2q)}(\mf{sp}_s; \varpi_1).&\\
&\hspace{1cm} \W_{-h_{\sp_n}^\vee + (h_{\sp_n}+1)/q}(\mf{sp}_n,f)   
\cong 
L_{-h_{\sp_s}^\vee+(h_{\sp_s}+1)/q}(\sp_s) \oplus L_{-h_{\sp_s}^\vee+(h_{\sp_s}+1)/q}(\sp_s; \varpi_1).&\\
& \hspace{1cm}  \W_{- h_{\mf{sp}_n}^{\vee}+ (h_{\sp_n}+1)/q}(\mf{sp}_n,f)  
& \\
& \hspace{1cm} \quad \cong 
L_1(\so_m) \otimes (L_{-h_{\mf{sp}_s}^{\vee}+(h_{\sp_s}+1)/q}(\mf{sp}_{s}))  
\oplus (L_1(\so_m;\varpi_1) \otimes (L_{-h_{\mf{sp}_s}^{\vee}+(h_{\sp_s}+1)/q}(\mf{sp}_{s};\varpi_1)) ).&
\end{align*}
\end{Conj}

\begin{Rem}As it has been observed 
in the proof of Theorem \ref{Th:main_sp_n-1}  
(1), if $k$ is collapsing for $f \in \O_{(q^{m},1^s)}$, 
then necessarily 
$p=\frac{n}{2}+1$ or $p=\frac{n+1}{2}$. 
Only the first case verifies that $p \ge h_\g^\vee$. 
However, one may wonder whether the following holds:
$$
H_{DS,f}^0(L_{-(\frac{n}{2}+1)+\frac{n+1}{2q}}(\mf{sp}_n))
\cong 
L_{-1/2} (\mf{sp}_m) \otimes L_{-(\frac{s}{2}+1)+\frac{s+1}{2q}}(\mf{sp}_{s}).$$ 
(The two above vertex algebras have the same central charge.) 
\end{Rem}

We now state our main results on collapsing levels 
for $\so_n$. Here also we start with the lisse case, that is, $f \in \O_k$. 

\begin{Th}
\label{Th:main_so_n-1}
Assume that $k = -h_\g^\vee+p/q=-(n-2)+p/q$ 
is admissible for $\g=\mf{so}_n$. 
Pick a nilpotent element $f \in \O_k$ so that $\W_k(\g,f)$ is lisse. 
\begin{enumerate} 
\item Assume that $q$ is odd. 
If $p=h_{\so_n}^\vee=n-2$, then for generic 
$q$, $k$ is collapsing if and only if $n \equiv 0,1 \mod q$. 
If $n \equiv 0,1 \mod q$,  then for generic 
$q$, $k$ is collapsing if and only if $p=h_{\so_n}^\vee$. 
Moreover, if $n \equiv 0,1 \mod q$, then 
$$\W_{-h_{\so_n}^\vee+h_{\so_n}^\vee/q}(\so_n,f) 
\cong H_{DS,f}^0(L_k(\g)) \cong \C.$$
\item Assume that $n$ and $q$ are even. 
If $p=h_{\so_n}+1=n-1$, 
then for generic $q$, $k$ is collapsing if and only if $n \equiv 0,2 \mod q$. 
If $n \equiv 0,2 \mod q$, then for generic $q$, $k$ is collapsing if and only if  
$p=h_{\so_n}+1$. 
Moreover, if $n \equiv 0,2 \mod q$, then 
$$\W_{-h_{\so_n}^\vee+(h_{\so_n}+1)/q}(\so_n,f)  \cong H_{DS,f}^0(L_k(\g))\cong \C.$$
\item Assume that $n$ is odd and that $q$ is even. 
If $p=h_{\so_n}+1=n$, then for generic $q$, $k$ is collapsing if and only if $n \equiv -1,1 \mod q$. 
If $n \equiv -1,1 \mod q$, then for generic $q$, $k$ is collapsing if and only if $p=h_{\so_n}+1$. 
Moreover, if $n \equiv -1,1 \mod q$, then 
$$\W_{-h_{\so_n}^\vee+(h_{\so_n}+1)/q}(\so_n,f) 
\cong H_{DS,f}^0(L_k(\g)) \cong \C.$$
\end{enumerate} 
\end{Th}

\begin{proof}
We argue as in the proof of Theorem \ref{Th:main_sp_n-1}. 
We exploit 
here the orthogonal Dynkin pyramid of shape $\bs\lam$ corresponding to $f \in \O_k$. 
Here, we set $I=\{1,\ldots,\frac{n}{2},-\frac{n}{2},\ldots,-1\}$ if $n$ is even, 
and $I=\{1,\ldots,\frac{n}{2},0,-\frac{n}{2},\ldots,-1\}$ if $n$ is odd. 
Moreover, for $j \in \frac{1}{2}\Z_{> 0}$, 
\begin{align*}
& \# \{\alpha \in \Delta_+^{\rm long} \colon (x^0|\alpha) = j \}= 
\# \{ (i,l) \in I \colon  0 < i \le |l| ,  \, |\on{col}(i) - \on{col}(l) |/2 = j \},&\\ 
& \# \{\alpha \in \Delta_+^{\rm short} \colon (x^0|\alpha) = j \}= 
\# \{ (i,l) \in I \colon i >0, \, l = 0 , \, |\on{col}(i) - \on{col}(l) |/2 = j \},&
\end{align*}
and 
$$\{\alpha \in \Delta_+ \colon (x^0|\alpha) = j \}=
\{\alpha \in \Delta_+^{\rm long} \colon (x^0|\alpha) = j \} \cup 
\{\alpha \in \Delta_+^{\rm short} \colon (x^0|\alpha) = j \},$$ 
with $\{\alpha \in \Delta_+^{\rm short} \colon (x^0|\alpha) = j \}=\varnothing$ 
if $n$ is even.  
From this, we compute the central charge and the asymptotic dimension 
of $H_{DS,f}^0(L_k(\g))$ 
using the pyramid exactly as in the case where $\g=\sp_n$. 
Since the proof is very similar to that of  Theorem~\ref{Th:main_sp_n-1}, 
we omit the details. 
\end{proof}

\begin{Rem}\label{rem:min.mod.cases.so}
As in Remark \ref{rem:min.mod.cases.sp} we now comment on a few isomorphisms between $W$-algebras obtained with similar methods as in the proof of Theorem \ref{Th:main_so_n-1} above.

Let $q>3$ be even and not divisible by $3$, let $n$ be odd, and $k = -h_\g^\vee + p/q$ a coprincipal admissible level for $\g = \mf{so}_n$, where $p=h_\g+1$. Then $\O_k = \O_{\bs\lam}$ where $\bs\lam = (q^m, s)$. If we choose $n$ so that $s = 3$ and take $f \in \O_k$ then we obtain an isomorphism
\[
\W_k(\g,f) \cong H_{DS, f}^0(L_{k}(\g)) \cong \on{Vir}_{3, q/2},
\]
proved by comparison of asymptotic growth and asymptotic dimension.

If $q$ and $m$ are odd so that $n = mq + 3$ is even, and $k = -h_\g^\vee + p/q$ is the principal admissible level for $\g = \mf{so}_n$ where $p=h_\g^\vee$, then $\O_k = \O_{\bs\lam}$ where $\bs\lam = (q^m, 3)$. If we choose $f \in \O_k$ then we obtain an isomorphism
\[
\W_k(\g,f) \cong H_{DS, f}^0(L_{k}(\g)) \cong \on{Vir}_{2, q},
\]
again proved by comparison of asymptotic growth and asymptotic dimension.
\end{Rem}

\begin{Th}
\label{Th:main_so_n-2}
Assume that $k = -h_\g^\vee+p/q=-(n-2)+p/q$ 
is admissible for $\g=\mf{so}_n$. 
\begin{enumerate} 
\item Assume that $q$ is odd so that $k$ is principal.  
\begin{enumerate} 
\item Pick a nilpotent element 
$f \in\overline{\O}_k$ corresponding to the partition $(q^m,1^s)$ 
with $s\ge 3$. 
For generic $q$, $k$ is collapsing only if $p=h_{\so_n}^\vee$ or $p=h_{\so_n}^\vee+1$. 
Moreover, 
$$\W_{-h_{\so_n}^\vee +h_{\so_n}^\vee/q}(\mf{so}_{n},f) \cong L_{-h_{\so_s}^\vee +h_{\so_s}^\vee/q}(\mf{so}_s),
$$
and (for $m \ge 3$) 
we have the following inclusion is a finite extension:
$$L_{1}(\mf{so}_m) \otimes L_{-h_{\so_s}^\vee +(h_{\so_s}^\vee+1)/q}(\mf{so}_s) 
\longhookrightarrow \W_{-h_{\so_n}^\vee+(h_{\so_n}^\vee+1)/q }(\mf{so}_{n},f).$$
\item Pick a nilpotent element 
$f \in\overline{\O}_k$ corresponding to the partition $(q^m,(q-1)^2)$. 
For generic $q$, $k$ is collapsing if and only if $p=h_{\so_n}^\vee$ and,  
$$\W_{-h_{\so_n}^\vee +h_{\so_n}^\vee/q}(\mf{so}_{n},f) \cong L_{-2 +2/q}(\sl_2),
$$
\end{enumerate}

\item Assume that $q$ and $n$ are even so that $k$ is principal.  
\begin{enumerate}  
\item Pick a nilpotent element 
$f \in\overline{\O}_k$ corresponding to the partition $(q+1,q^m,1^s)$, 
with even $m$, odd $s$. 
Then $k$ is collapsing if and only if $p=h_{\so_n}+1$ and, 
$$\W_{-h_{\so_n}^\vee +(h_{\so_n}+1)/q}(\mf{so}_{n},f) \cong L_{-h_{\so_s}^\vee +(h_{\so_s}+1)/q}(\mf{so}_s).$$

\item Pick a nilpotent element 
$f \in\overline{\O}_k$ corresponding to the partition $(q+1,q^m,q-1,1^s)$, 
with even $m,s$, $s >2$. 
Then for generic $q$, $k$ is collapsing only if $p=h_{\so_n}+1$. 
Moreover, for generic $q$, 
we have the following inclusion is a finite extension:
$$L_{-h_{\so_s}^\vee +(h_{\so_s}+1)/q}(\mf{so}_s)  \longhookrightarrow 
\W_{-h_{\so_n}^\vee +(h_{\so_n}+1)/q}(\mf{so}_{n},f).$$
\end{enumerate}

\item Assume that $n$ is odd and $q$ is even 
so that $k$ is coprincipal. 
\begin{enumerate}  
\item Pick a nilpotent element 
$f \in\overline{\O}_k$ corresponding to the partition $(q^m,1^s)$, 
with even $m$ and odd $s$. 
Then for generic $q$, $k$ is collapsing only if $p=h_{\so_n}+1$. 
Moreover,
$$\W_{-h_{\so_n}^\vee+(h_{\so_n}+1)/q}(\mf{so}_{n},f) \cong 
L_{-h_{\so_s}^\vee +(h_{\so_s}+1)/q}(\mf{so}_s).$$
\item Pick a nilpotent element 
$f \in\overline{\O}_k$ corresponding to the partition $(q^m,q-1,1^s)$, 
with even $m,s$, $s>2$. 
Then for generic $q$, $k$ is collapsing only if $p=n$. 
Moreover, for generic $q$, 
we have the following inclusion is a finite extension:
$$L_{-h_{\so_s}^\vee +(h_{\so_s}+1)/q}(\mf{so}_s) 
\longhookrightarrow \W_{-h_{\so_n}^\vee +(h_{\so_n}+1)/q}(\mf{so}_{n},f).$$
\end{enumerate}
\end{enumerate} 
\end{Th}

\begin{proof}
This proof is very similar to that of Theorem \ref{Th:main_sp_n-2}, 
except that, obviously, we use here orthogonal Dynkin pyramids. 
We omit the details. 
\end{proof}

The following conjecture is similar to Conjecture \ref{Conj:dec_finite_extensions-sp_n}.
\begin{Conj}
\label{Conj:dec_finite_extensions-so_n}
For generic $q$, we have the following finite extensions:
\begin{align*}
&\hspace{1cm} \W_{-h_{\so_n}^\vee+(h_{\so_n}^\vee+1)/q }(\mf{so}_{n},f) 
\cong & \\  
& \hspace{1cm} \qquad 
L_{1}(\mf{so}_m) \otimes L_{-h_{\so_s}^\vee +(h_{\so_s}^\vee+1)/q}(\mf{so}_s) \oplus L_{1}(\mf{so}_m;\varpi_1) \otimes L_{-h_{\so_s}^\vee +(h_{\so_s}^\vee+1)/q}(\mf{so}_s;\varpi_1).&\\
&\hspace{1cm} \W_{-h_{\so_n}^\vee +(h_{\so_n}+1)/q}(\mf{so}_{n},f)  
\cong  L_{-h_{\so_s}^\vee +(h_{\so_s}+1)/q}(\mf{so}_s) 
\oplus L_{-h_{\so_s}^\vee +(h_{\so_s}+1)/q}(\mf{so}_s;\varpi_1).& \\
&\hspace{1cm} \W_{-h_{\so_n}^\vee +(h_{\so_n}+1)/q}(\mf{so}_{n},f) 
\cong  L_{-h_{\so_s}^\vee +(h_{\so_s}+1)/q}(\mf{so}_s) \oplus 
 L_{-h_{\so_s}^\vee +(h_{\so_s}+1)/q}(\mf{so}_s;\varpi_1).& 
\end{align*}
\end{Conj}

\begin{Rem}
Similarly to Remark \ref{Rem:compatible_sl_n}, notice that (4)  
with $s=0$ or $s=1$ is compatible with (1), (5)  
with $s=1$ is compatible with (2), (7)  
with $s=1$ is compatible with (3), and  (8)  
with $s=0$ is compatible with (3).
\end{Rem}

\begin{Rem} \label{Rem:non_adm}
It might be also interesting to consider 
the case where $(n-2,q) \not=1$. 
For example, consider the Lie algebra $\mf{so}_{27}$ 
(type $B_{13})$ and $f \in \O_{(5^4,1^7)}$. 
We find that $k= -25+\dfrac{25}{5}=-20$ is not admissible, 
and $k_1^\natural= - 5+\dfrac{5}{5}=-4$ is not admissible either. 
One can wonder  whether 
$$\W_{-20}(\mf{so}_{27},f) \cong L_{-4}(\mf{so}_{7}).$$
Other examples are
$$\W_{-18}(\mf{so}_{23},f) \cong L_{-6}(\mf{so}_{9}).$$
$$\W_{-6}(\mf{so}_{11},f) \cong L_{-2}(\mf{so}_{5}).$$
The last example is interesting because one knows 
\cite[Theorem 7.1]{AraMor16b} that, for $\ell \ge 3$, 
$$X_{ L_{-2}(\mf{so}_{2\ell+1})}=\overline{\O}_{short}.$$
\end{Rem}

\begin{Conj} 
\label{Conj:exhaustive}
\begin{enumerate}
\item The cases covered by Theorems \ref{Th:main_sl_n-1} 
and \ref{Th:main_sl_n-2} give the exhaustive list of pairs $(f,k)$ 
where $f$ is a nilpotent element of $\sl_n$ 
and $k$ is an admissible 
collapsing levels for $\sl_n$. 
\item 
The cases covered by Theorems \ref{Th:main_sp_n-1} 
and \ref{Th:main_sp_n-2} give the exhaustive list of pairs $(f,k)$ 
where $f$ is a nilpotent element of $\sp_n$ 
and $k$ is an admissible 
collapsing levels for $\mf{sp}_n$. 
\item The cases covered by Theorems \ref{Th:main_so_n-1} 
and \ref{Th:main_so_n-2} give the exhaustive list of pairs $(f,k)$ 
where $f$ is a nilpotent element of $\so_n$ 
and $k$ is an admissible 
collapsing levels for $\mf{so}_n$. 
\end{enumerate}
\end{Conj}

\section{Collapsing levels in the exceptional types}
\label{sec:exceptional}

In this section we state our main results and conjectures 
concerning collapsing levels in the exceptional types. As in the preceding sections, our proofs  
follow the strategy described in Section~\ref{sec:strategy};  
considering pairs $(\O_k,G.f)$ such that $\Slo_{\O_k,f}$ is collapsing. 
Data on nilpotent orbits, $\mathfrak{sl}_2$-triples and centralisers, which we will use throughout this section, is recorded in Tables~\ref{Tab:Data-G2}, \ref{Tab:Data-F4}, \ref{Tab:Data-E6},
\ref{Tab:Data-E7-a}-\ref{Tab:Data-E7}, \ref{Tab:Data-E8-a}-\ref{Tab:Data-E8}. 

The results of this section are organised by type, the isomorphisms in type $E_6$, $E_7$, $E_8$, $G_2$ and $F_4$ presented in Theorems \ref{Th:main_E6}, \ref{Th:main_E7}, \ref{Th:main_E8}, \ref{Th:main_G2} and \ref{Th:main_F4} respectively. In place of $f$ we write the label of $G.f$ in the Bala-Carter classification, and $\g$ and $\g^\natural$ are denoted by their types. The results are summarized in Tables \ref{Tab:main_results-E6}, \ref{Tab:main_results-E7}, \ref{Tab:main_results-E8}, \ref{Tab:main_results-G2}, and \ref{Tab:main_results-F4}. We present a complete proof only for Theorem \ref{Th:main_E6}, the others being very similar.

In the tables, we indicate for each triple 
$(\O_k,G.f,\overline{\O^{\natural}}\cong \Slo_{\O_k,f})$, the values of $p/q=k+h_\g^\vee$ , 
$k_i^\natural +h_{\g_i^\natural}^\vee$, for $i=1,\ldots,s$    
(if $i>1$, we write in the first column 
$k_1^\natural +h_{\g_1^\natural}^\vee$, and in second column $k_2^\natural +h_{\g_2^\natural}^\vee$, etc.), 
the central charge $c_V$,  
the asymptotic growth $\G_V$ and the asymptotic dimension $\A_V$ 
with $V$ being either $H_{DS,f}^0(L_k(\g))$ or $L_{k^\natural}(\g^\natural)$. 
Then we write in the last column the symbol~$\checkmark$ if all invariants match, 
we write ^^ ^^ fin.~ext.'' if we expect that 
$\W_k(\g,f)$ 
is a finite 
extension of  $L_{k^\natural}(\g^\natural)$ (usually, 
this happens when all invariants for $H_{DS,f}^0(L_k(\g))$ 
match except the asymptotic dimension, 
and more details are furnished in the corresponding theorem). 
As a rule, when one of the invariants does not coincide, we write in first position the invariant 
corresponding to $H_{DS,f}^0(L_k(\g))$. 

Nilpotent orbits are given by their the Bala-Carter classification 
in the exceptional types. 
Nilpotent orbits in classical types  
are given in term of partitions. 
Note that the associated variety of $L_k(\g)$ determines the possible denominators  
$q$ of $k+h_\g^\vee$ since $k$ is admissible (see Theorem \ref{Th:admissible-orbits}). 
So the values of $q$ are always among these possible denominators. 

We also indicate in the table the isomorphism type of $\Slo_{\O_k,f}$ 
(in the third column) 
when it is known (if so, it is a product of nilpotent orbits in $\g^\natural$). 
When the isomorphism $\overline{\O^{\natural}}\cong \Slo_{\O_k,f}$ 
comes from a minimal degeneration, 
then we always obtain a minimal nilpotent orbit closure \cite{FuJutLev17}. 
Following Kraft and Procesi \cite{KraftProcesi79,KraftProcesi82}, 
we refer to the minimal nilpotent 
orbit $\O_{\text{min}}$ of a simple Lie algebra 
by the lower case letters for the ambient simple Lie algebra: 
$a_k$, $b_k$, $c_k$, $d_k$ $(k\ge 4)$, $g_2$, $f_4$, $e_6$, $e_7$, $e_8$. 
Similarly, we refer to the minimal 
special nilpotent orbit\footnote{There is 
 an order-reversing map $d$ on the set of nilpotent orbits in $\g$ that becomes
an involution when restricted to its image \cite{Spaltenstein82}. 
Orbits in the image of $d$ are called
{\em special}. There is a unique minimal special 
nilpotent orbit, which is of dimension $2 h_\g -2$. 
Note that the minimal nilpotent orbit of $\g$ has dimension 
$2 h_\g^\vee-2$.} for the ambient simple Lie algebra 
as $a_k^{sp}$, $b_k^{sp}$, $c_k^{sp}$, $d_k^{sp}$ $(k\ge 4)$, 
$g_2^{sp}$, $f_4^{sp}$, $e_6^{sp}$, $e_7^{sp}$, $e_8^{sp}$. 
The nilpotent cone of a Lie algebra of type $X$ 
will denoted by $\mc{N}_X$.  
In Table \ref{Tab:main_results-F4}, the letter $m$ refers to a non-normal type  
of singularity which is neither a simple surface singularity 
nor a minimal singularity  (see \cite[\S1.8.4]{FuJutLev17}).
As for the singularity  $a_2^+$ appearing also in Table \ref{Tab:main_results-F4} 
it refers to the singularity $a_2$ together 
with the action of a
subgroup $K \subset {\rm Aut}(\sl_3)$ which lifts a 
subgroup of the Dynkin diagram of $\sl_3=A_2$ (see \cite[\S1.8.2]{FuJutLev17}).

When the isomorphism type of $\Slo_{\O_k,f}$ is not known (to the best of our knowledge) 
we write ^^ ^^ unknown''. 

\begin{Th}
\label{Th:main_E6}
The following isomorphisms hold, providing 
collapsing levels for $\g=E_6$. 
{\footnotesize
\begin{align*}
& \W_{-12+12/13}(E_6,E_6) \cong \C, && 
\W_{-12+13/12}(E_6,E_6) \cong \C, & \\
& \W_{-12+13/9}(E_6,E_6(a_1)) \cong \C,  &&
\W_{-12+13/6}(E_6,E_6(a_3)) \cong \C, &\\ 
& \W_{-12+13/6}(E_6,A_5) \cong L_{-2+2/3}(A_1), && 
 \W_{-12+12/7}(E_6,D_4) \cong L_{-3+3/7}(A_2), &\\ 
& \W_{-12+13/6}(E_6,D_4) \cong L_{-2+4/3}(A_2), &&
\W_{-12+12/5}(E_6,A_4) \cong L_{-2+2/5}(A_1), &\\ 
& \W_{-12+13/3}(E_6,2A_2+A_1) \cong \C, && 
\W_{-12+13/3}(E_6,2A_2)\cong L_{-4+7/3}(G_2), &\\
& \W_{-12+13/2}(E_6,3 A_1) \cong L_1(A_2). &&  &
\end{align*}}
Moreover, the following inclusions are finite extensions:
 {\footnotesize
\begin{align*}
& L_{-2+3/14}(A_1) \longhookrightarrow \W_{-12+12/7}(E_6,A_5),&& 
&L_{-3+4/3}(A_2)\otimes 
L_{-3+4/3}(A_2)  
\longhookrightarrow \W_{-12+13/3}(E_6,A_2), & \\ 
& L_{-6+7/2}(A_5) \longhookrightarrow \W_{-12+13/2}(E_6,A_1).&&
\end{align*}
}
\end{Th}

\begin{proof}
We detail below only a few cases. 
The chief tool to prove the isomorphisms is Theorem~\ref{Th:main}, using data summarised in Table \ref{Tab:main_results-E6}. 
 
$\bullet$ Assume $q \ge 12$. Let $p\ge 12$ and $(p,q)=1$. Computing the central charge, we observe that $c_{\W_{-12+p/q}(E_6,E_6)}=0$ 
if and only if $(p,q)=(12,13)$ or $(p,q)=(13,12)$. 
Moreover, in both these cases $\G_{\W_{-12+p/q}(E_6,E_6)}=0$ and $\A_{\W_{-12+p/q}(E_6,E_6)}=1$. By Theorem \ref{Th:main} the first two isomorphisms follow. 

$\bullet$ Assume that $q=6$ or $7$, and pick $f \in A_5$. 
According to Table \ref{Tab:Data-E6}, we have $\g^\natural \cong \sl_2$ 
and $k^\natural = k+17/2$. Computing the central charge, we easily verify that if $k$ 
is collapsing then necessarily $(p, q) = (13, 6)$ or $(p, q) = (12, 7)$. 

$\ast$ Assume first that $(p, q) = (13, 6)$. Then $k^\natural = -2+2/3$. 
Since $k = -12+13/6$ and $k^\natural = -2+2/3$ are admissible 
for $E_6$ and $A_1$, respectively, it suffices to apply Theorem 
\ref{Th:main} to prove that $\W_{-12+13/6}(E_6,A_5)$ 
and $L_{-2+2/3}(A_1)$ are isomorphic. 

We easily check using Corollary~\ref{Co:asymptotic_data_L}  
and Proposition \ref{Pro:asymptotic_data_H_DS} 
that $L_{-2+2/3}(A_1)$ and $H^0_{DS,A_5}(L_{-12+13/6}(E_6))$ 
share the same asymptotic growth of $2$. 
Let us compare their asymptotic dimensions. 
By  Corollary~\ref{Co:asymptotic_data_L}  and Lemma \ref{Lem:main_identities} (1), 
we have 
$$\A_{L_{-2+2/3}(A_1)} = \frac{1}{3 \sqrt{3}},$$ 
while by Proposition \ref{Pro:asymptotic_data_H_DS} and Lemma \ref{Lem:main_identities} (2), 
\begin{align*} 
\A_{H^0_{DS,A_5}(L_{-12+13/6}(E_6)) }&= 
\dfrac{1} 
{2^{|\Delta^{1/2}|/2} 6^{|\Delta_+^0|} 6^{3} \sqrt{3}} 
\prod\limits_{\alpha \in \Delta_+ \setminus \Delta_+^0} 
2 \sin \dfrac{\pi (x^0  | \alpha) }{6}, &\\
\end{align*} 
with $|\Delta^{1/2}|=6$ and $|\Delta_+^0|=1$. 
Computing the cardinality of $\{\alpha \in \Delta_+ \colon (x^0  | \alpha)=j \}$ 
for $j>0$, we verify that 
$$\prod\limits_{\alpha \in \Delta_+ \setminus \Delta_+^0} 
2 \sin \dfrac{\pi (x^0  | \alpha) }{6}= 2^7 3^3,$$ 
whence $\A_{H^0_{DS,A_5}(L_{-12+13/6}(E_6)) }= \frac{1}{3 \sqrt{3}}$, 
as required. 

$\ast$ Assume now that $(p, q) = (12, 7)$. 
Here, we obtain that 
$$\G_{H^0_{DS,A_5}(L_{-12+12/7}(E_6))}= \G_{L_{-2+3/14}(A_1)}$$ 
and that 
$$\A_{H^0_{DS,A_5}(L_{-12+12/7}(E_6))}= 2\A_{L_{-2+3/14}(A_1)}.$$ 
Moreover, we can verify that the central charge of 
$H^0_{DS,A_5}(L_{-12+12/7}(E_6))$ and $L_{-2+3/14}(A_1)$ 
are both equal to $-25$. 
By Theorem~\ref{Th:main}, $\W_{-12+12/7}(E_6,A_5)$ is a finite extension of $L_{-2+3/14}(A_1)$. 

We argue similarly in all the other cases. 
\end{proof}

{\tiny 
\begin{table}
\begin{tabular}{lllllllll}
\hline
&& &&& &&& \\[-0.5em]
$\O_k$ & $G.f$ & $\Slo_{\O_k,f}$ & $\frac{p}{q}=k+h_{\g}^\vee$ & $k^\natural+h_{\g^\natural}^\vee$ & $c_V$ & $\G_V$ & $\A_V$ & comments \\
&& &&& &&& \\[-0.5em]
\hline
&& &&& &&& \\[-0.5em]
$E_6$ & $E_6$ & $\{0\}$ & $12/13, 13/12$ &  & $0$ &  $0$ & $1$& $\checkmark$ \\[1em]
$E_6(a_1)$ & $E_6(a_1)$ & $\{0\}$ & $13/9$ &  & $0$ &  $0$ & $1$& $\checkmark$\\[2em]
 $E_6(a_3)$ & $E_6(a_3)$ & $\{0\}$ & $13/6$ &  & $0$ &  $0$ & $1$& 
 $\checkmark$\\[1em]
&  $A_5$ & $a_1$ & $13/6$ & $2/3$ & $-6$ &  2 & $\dfrac{1}{3\sqrt{3}} $& 
$\checkmark$\\
&& & $12/7$ & $3/14$ & $-25$ & $\dfrac{20}{7}$ & $\dfrac{2}{28\sqrt{7}} 
\not=   \dfrac{1}{28\sqrt{7}}$ & fin.~ext.  \\[2em]
 & $D_4$ & $\mc{N}_{A_2}$ & $12/7$ & $3/7$ & $-48$ & $\dfrac{48}{7}$ & $\dfrac{1}{7^4}$ 
 & $\checkmark$  \\
&&  & $13/ 6$ & $4/3$ & $-10$ & 6 &  $\dfrac{1}{81\sqrt{3}}$& $\checkmark$  \\[1em]
 $A_4+A_1$ & $A_4$ & $a_1$ & $12/5$ & $2/5$ & $-12$ &  $\dfrac{12}{5}$ & 
$\dfrac{1}{5\sqrt{5}}$ & $\checkmark$\\[1em]
$D_4(a_1)$ & $2 A_2$ & $g_2^{sp}$ & $12/4$ &    & $-42$ &   & 
 & $k=-9$, $k^\natural=-3$\\
&& &&& &&& not admissible \\[1em]
$2 A_2+A_1$ & $2A_2+A_1$ & $\{0\} \subset A_1$ & $13/3$ & $2/1$ & $0$ &  $0$ & 
$1$ & $\checkmark$ \\[1em]
 & $2A_2$ & $g_2$ & $13/3$ & $7/3$ & $-10$ &  $6$ & 
$\dfrac{1}{27\sqrt{3}}$ & $\checkmark$ 
\\[1em]
 & $A_2$ & 
unknown in $A_2\times A_2$ & 
$13/3$ & 
$4/3$ & 
$-20$ &  
$12$ & 
$3 \times \left(\dfrac{1}{81\sqrt{3}}\right)^2 \not=\left(\dfrac{1}{81\sqrt{3}}\right)^2$ 
&  fin.~ext.
\\[1em]
$3 A_1$ & $3 A_1$ & $\{0\} \subset A_1\times A_2$ & $13/2$ & $4/1$ & $2$ & $2$  & $\dfrac{1}{\sqrt{3}}$ 
 & $\checkmark$ \\[1em]
 & $2A_1$ & $b_3$ & $12/2$ &    & $- 14$ &   & 
 & $k=-6$, $k^\natural=-2$\\
 && &&& &&& not admissible \\[1em]
 & $A_1$ & {unknown in $A_5$} & $13/2$ & $7/2$ & $- 25$ & $20$  & $\dfrac{2}{2^{17}\sqrt{3}} 
 \not=\dfrac{1}{2^{17}\sqrt{3}}$
 & fin.~ext. \\[1em]
\hline
\end{tabular}\\[1em]
\caption{\footnotesize{Main asymptotic data in type $E_6$}}
\label{Tab:main_results-E6}
\end{table}
}
\begin{Rem}
\label{Rem:birational_equivalent_E6} 
\begin{enumerate}
\item The associated 
variety of 
$H_{DS,A_1}^0(L_{-12+13/2}(E_6))$ 
and that 
of $L_{-6+7/2}(A_5)$ have the same dimension, but they are not isomorphic. 
Indeed the former is the nilpotent Slodowy 
slice $\Slo_{3 A_1,A_1}$ in $E_6$   
while the latter is the closure of the nilpotent orbit of $A_5=\sl_6$ attached to the partition 
$(2^3)$. 
These two varieties are not isomorphic 
since the number of nilpotent $G^\natural$-orbits in $\overline{\O}_{(2^3)}$ is $4$, while 
if $\Slo_{3 A_1,A_1}$ had a dense $G^\natural$-orbit then the number of $G^\natural$-orbits in $\Slo_{3 A_1,A_1}$ would be $3$, as we can see from the Hasse diagram of $E_6$. 
\item Similarly, the nilpotent Slodowy slice 
$\Slo_{2A_2+A_1,A_2}$ (associated variety of  
$H_{DS,A_2}^0(L_{-12+13/3}(E_6))$ 
is not isomorphic to the product $\mc{N}_{A_2}\times \mc{N}_{A_2}$  
(associated variety of $L_{-3+4/3}(A_2)\otimes 
L_{-3+4/3}(A_2)$)
and these two varieties have the same dimension. 
\end{enumerate}
\end{Rem}

\begin{Pro}
\label{Pro:fin_dec-E6}
The following decompositions hold:
 {\footnotesize
\begin{align*}  
& \W_{-12+13/2}(E_6,A_1) 
\cong L_{-6+7/2}(A_5) \oplus L_{-6+7/2}(A_5;\varpi_3), &\\
& \W_{-12+13/3}(E_6,A_2) \cong \left( L_{-3+4/3}(A_2)\otimes 
L_{-3+4/3}(A_2)\right) & \\
&\hspace{1cm}\oplus \left(L_{-3+4/3}(A_3;\varpi_1) 
\otimes L_{-3+4/3}(A_3;\varpi_1) \right)
\oplus \left(L_{-3+4/3}(A_3;\varpi_2) 
\otimes L_{-3+4/3}(A_3;\varpi_2) \right). &\\
\end{align*}
}
\end{Pro}

\begin{proof}
$\bullet$ Assume first that $(p, q)=(13, 2)$, and let $f \in A_1$. 
In this case we have 
\begin{align*}
& \G_{H^0_{DS,A_1}(L_{-12+13/2}(E_6))} = \G_{L_{-6+7/2}(A_5)}, 
\quad \A_{H^0_{DS,A_1}(L_{-12+13/2}(E_6))} = 2\A_{L_{-6+7/2}(A_5)}.&
\end{align*}
Moreover, we verify that the central charges of 
$H^0_{DS,A_1}(L_{-12+13/2}(E_6))$ and $L_{-6+7/2}(A_5)$ 
are both equal to $-25$. 
Arguing 
as in the proof of Theorem \ref{Th:main_sp_n-2} (1)~(b), 
we obtain that 
$$
\W_{-12+13/2}(E_6,A_1) = \bigoplus_{\Delta_\lam 
\in \frac{1}{2} \Z} L_{-6+7/2}(A_5;\lam),$$
where $\Delta_\lam$ is the lowest $L_0$-eigenvalue  
of $ L_{-6+7/2}(A_1;\lam)$. 
We require
$$\Delta_\lam= \dfrac{ ( \lam| \lam+2 \rho )}{2(k^\natural + 6)}= \dfrac{( \lam | \lam+2 \rho)}{7} 
\in \frac{1}{2} \Z$$
where $\lam = \sum\limits_{i=1}^5 \lam_i \varpi_i$, {assuming that all $\lam _i \in \Z_{\ge 0}$}, with $1 = p - h_{A_5}^\vee \ge \sum\limits_{i=1}^5 \lam_i \ge 0$. 
The only possibilities are $\lam =0$ or $\lam =\varpi_3$. 

On the other hand, by Proposition \ref{Pro:asymptotic_data_H_DS}, 
we easily see that 
$$\A_{L_{-6+7/2}(A_5; \varpi_3)} = \A_{L_{-6+7/2}(A_5)}.$$
As a result, 
$$m_0  \A_{L_{-6+7/2}(A_5; 0)}+ m_1 \A_{L_{-6+7/2}(A_5; \varpi_3)} = 
2 \A_{L_{-6+7/2}(A_5; 0)}.$$ 
But $m_0$ must be at most $1$. 
We conclude that either 
$\W_{-12+13/2}(E_6,A_1)  \cong 
L_{-6+7/2}(A_5) \oplus L_{-6+7/2}(A_5; \varpi_3)$, 
or that $-12+13/2$ is collapsing. 
But this is impossible by Proposition \ref{Pro:fin_ext-assoc} 
and Remark \ref{Rem:birational_equivalent_E6}. 

$\bullet$ Assume now $(p,q)=(13,3)$ and $f \in A_2$. 
We have $\W_{-12+13/3}(E_6,A_2) \cong H_{DS,f}^{0}(L_{-12+13/3}(E_6))$ since $f$ is even. 
Arguing as in the previous case, we obtain 
here that the only possible weights are $\lam =0$, $\lam =\varpi_1$ or $\lam =\varpi_2$. 
Since $H_{DS,f}^{0}(L_{-12+13/3}(E_6)) \cong \W_{-12+13/3}(E_6,A_2)$ 
is a finite direct sum of admissible $L_{-3+4/3}(A_2) 
\otimes L_{-3+4/3}(A_2)$-modules, 
we obtain here from the equalities 
\begin{align*}
 m_0  \A_{L_{-3+4/3}(A_2;0) \otimes 
L_{-3+4/3}(A_2;0)}+ m_1 \A_{L_{-3+4/3}(A_2;\varpi_1) 
\otimes L_{-3+4/3}(A_2;\varpi_1)} && \\ 
+  m_2 \A_{L_{-3+4/3}(A_2;\varpi_2) 
\otimes L_{-3+4/3}(A_3;\varpi_2)}
= 
3 \A_{L_{-3+4/3}(A_2) 
\otimes L_{-3+4/3}(A_2)},&
\end{align*}
and $m_0=1$ the expected decomposition. 
Indeed, by Proposition \ref{Pro:asymptotic_data_H_DS}, 
$$\A_{L_{-3+4/3}(A_3;\varpi_1)}=\A_{L_{-3+4/3}(A_2;\varpi_2)} =\A_{L_{-3+4/3}(A_2)}.$$  
\end{proof} 
For $(p,q)=(12,7)$ and $f \in A_5$, similar 
arguments as before Conjecture \ref{Conj:dec_finite_extensions-sp_n} 
lead to the following conjecture.  
\begin{Conj}
\label{Conj:dec_finite_extensions-E6}
$$\W_{-12+12/7}(E_6,A_5)
\cong L_{-2+3/14}(A_1) \oplus L_{-2+3/14}(A_1;\varpi_1).$$
\end{Conj}

We conjecture the following isomorphisms at non admissible level.
\begin{Conj}
\label{Conj:main_E6}
\begin{align*}
& \W_{-9}(E_6,2A_2)\cong L_{-3}(G_2), \quad 
 \W_{-6}(E_6,2 A_1) \cong L_{-2}(B_3).  & 
\end{align*}
\end{Conj}

\begin{Th}
\label{Th:main_E7}
The following isomorphisms hold, 
providing collapsing levels for $\g=E_7$. 
{\footnotesize \begin{align*}
& \W_{-18+18/19}(E_7,E_7) \cong \C, && \W_{-18+19/18}(E_7,E_7) \cong \C, & \\
& \W_{-18+19/14}(E_7,E_7(a_1)) \cong \C, && \W_{-18+18/13}(E_7,E_6) \cong L_{-2+2/13}(A_1), & \\
& \W_{-18+19/12}(E_7,E_6) \cong L_{-2+3/4}(A_1), && \W_{-18+19/10}(E_7,D_6) \cong L_{-2+2/5}(A_1),&\\ 
& \W_{-18+18/7}(E_7,A_6) \cong \C, &&  \W_{-18+19/7}(E_7,A_6) \cong L_{1}(A_1), & \\
& \W_{-18+18/7}(E_7,(A_5)'') \cong L_{-4+4/7}(G_2), && 
 \W_{-18+18/7}(E_7,D_4) \cong L_{-4+4/7}(C_3),  \\
&  \W_{-18+19/6}(E_7,E_7(a_5)) \cong \C, && 
  \W_{-18+19/6}(E_7,E_6(a_3)) \cong L_{-2+3/2}(A_1),& \\
& \W_{-18+19/6}(E_7,D_6(a_2)) \cong L_{-2+2/3}(A_1), &&  
 \W_{-18+19/6}(E_7,(A_5)') \cong L_{-2+2/3}(A_1)\otimes L_{-2+3/2}(A_1), &\\
&\W_{-18+19/6}(E_7,(A_5)'') \cong L_{-4+7/6}(G_2), && 
 \W_{-18+19/6}(E_7,D_4) \cong L_{-4+7/6}(C_3), & \\
 & \W_{-18+19/5}(E_7,A_4+A_2) \cong L_{3}(A_1),  &&
 \W_{-18+18/5}(E_7,A_4) \cong L_{-3+3/5}(A_2), & \\
 & \W_{-18+19/4}(E_7,A_3+A_2+A_1) \cong L_{2}(A_1), && 
  \W_{-18+19/4}(E_7,D_4(a_1)) \cong L_{-2+3/4}(A_1)^{\otimes 3}, &\\
 &  \W_{-18+19/3}(E_7,2A_2+A_1) \cong L_{1}(A_1), && 
  \W_{-18+19/3}(E_7,A_2+3A_1) \cong L_{- 4+8/3}(G_2), & \\
& \W_{-18+19/3}(E_7,2 A_2) \cong L_{1}(A_1)\otimes L_{- 4+7/3}(G_2), &&
 \W_{-18+19/2}(E_7,4 A_1) \cong \C, &\\ 
 & \W_{-18+19/2}(E_7,(3 A_1)') \cong L_{-4+7/2}(C_3), && 
 \W_{-18+19/2}(E_7,(3 A_1)'') \cong L_{-9+13/2}(F_4). &  
\end{align*}}
Moreover, the following inclusions are finite extensions: 
{\footnotesize \begin{align*}
& L_{-2+3/22}(A_1) \longhookrightarrow \W_{-18+18/11}(E_7,D_6), & 
& L_{-2+3/8}(A_1)\otimes L_{-2+3/4}(A_1) \longhookrightarrow \W_{-18+19/8}(E_7,D_5),& \\
& L_{-2+3/4}(A_1) \otimes L_{-5+7/4}(B_3)\longhookrightarrow \W_{-18+19/4}(E_7,A_3) ,&
& L_{-6+7/3}(A_5) \longhookrightarrow \W_{-18+19/3}(E_7,A_2), & \\
& L_{-2+3/2}(A_1) \otimes L_{-7+9/2}(B_4) \longhookrightarrow \W_{-18+19/2}(E_7,2 A_1), &
& L_{-10+11/2}(D_6) \longhookrightarrow \W_{-18+19/2}(E_7,A_1). &
\end{align*}}
\end{Th}

\begin{Rem}
\label{Rem:birational_equivalent_E7}
In the comments below, we argue as in Remark \ref{Rem:birational_equivalent_E6} 
to conclude that the varieties are not isomorphic with the same dimension. 
\begin{enumerate}
\item The associated 
variety of $H_{DS,A_3}^0(L_{-18+19/4}(E_7))$ is not isomorphic to the associated variety 
of $L_{-2+3/4}(A_1) \otimes L_{-5+7/4}(B_3) $. 
The former is the nilpotent Slodowy 
slice $\Slo_{A_3+A_2+A_1,A_3}$ in $E_7$   
while the later is the closure of the nilpotent orbit of $a_1 \times \O_{1,(3^2,1)} \subset A_1 \times B_3$. 
\item The associated 
variety of $H_{DS,A_2}^0(L_{-18+19/3}(E_7))$ is not isomorphic to the associated variety 
of $L_{-6+7/3}(A_5)$. 
The former is the nilpotent Slodowy 
slice $\Slo_{2 A_2+A_1,A_2}$ in $E_7$   
while the later is the closure of the nilpotent orbit of $A_5=\sl_6$ attached to the partition 
$(3^2)$.  
\item The associated 
variety of $H_{DS,2 A_1}^0(L_{-18+19/2}(E_7))$ is not isomorphic to the associated variety 
of $L_{-2+3/2}(A_1) \otimes L_{-7+9/2}(B_4)$. 
The former is the nilpotent Slodowy 
slice $\Slo_{4 A_1,2 A_1}$ in $E_7$   
while the later is the closure of the nilpotent orbit of $a_1 \times \O_{1,(2^4,1)} \subset A_1 \times B_4$. 
Here to conclude that varieties are not isomorphic one needs to use the singularities 
of $\Slo_{4 A_1,2 A_1}$ as described in \cite{FuJutLev17}. 
\item The associated 
variety of $H_{DS,A_1}^0(L_{-18+19/2}(E_7))$ is not isomorphic to the associated variety 
of $L_{-10+11/2}(D_6)$. 
The former is the nilpotent Slodowy 
slice $\Slo_{4 A_1,A_1}$ in $E_7$   
while the later is the closure of the nilpotent orbit of $D_6=\so_{12}$ attached to the partition 
$(3,2^4,1)$. 
\end{enumerate}
\end{Rem}

{\tiny 
\begin{table}
\begin{tabular}{lllllllll}
\hline
&& &&& &&& \\[-0.5em]
$\O_k$ & $G.f$ & $\Slo_{\O_k,f}$ & $\frac{p}{q}=k+h_{\g}^\vee$ & $k^\natural+h_{\g^\natural}^\vee$ & $c_V$ & $\G_V$ & $\A_V$ & comments \\
&& &&& &&& \\[-0.5em]
\hline
&& &&& &&& \\[-0.5em]
$E_7$ & $E_7$ & $\{0\}$ & $18/19, 19/18$ &  & $0$ & $0$ & $1$  
& $\checkmark$  \\[1em]
$E_7(a_1)$ & $E_7(a_1)$ & $\{0\}$ & $19/14$ &  & $0$ & $0$ & $1$  
& $\checkmark$ \\[1em] 
$E_7(a_2)$ & $E_6$ & $a_1$ & $18/13$ & $2/13$ & $-36$ & $\dfrac{36}{13}$ & $\dfrac{1}{13\sqrt{13}}$  
& $\checkmark$ \\
&& & $19/12$ & $3/4$ & $-5$ & $\dfrac{5}{2}$ & $\dfrac{1}{8\sqrt{2}}$  
& $\checkmark$ \\[1em]
$E_7(a_3)$ & $D_6$ & $a_1$ & $18/11$ & $3/22$ & $-41$ & $\dfrac{32}{11}$ & 
$\dfrac{2}{44\sqrt{11}}\not=
\dfrac{1}{44\sqrt{11}}$    
& fin.~ext.  \\
&& & $19/10$ & $12/5$ & $-12$ & $\dfrac{12}{5}$ & 
$\dfrac{1}{5\sqrt{5}}$    
& $\checkmark$ \\[1em]
$E_7(a_4)$ & $D_5$ & $a_1\times a_1$ & $19/8$ & $3/8, 3/4$ & $-18$ & $\dfrac{21}{4}$ & $
\dfrac{2}{256\sqrt{2}}\not=
\dfrac{1}{256\sqrt{2}}$    
& fin.~ext.   \\[1em]
$A_6$ & $A_6$ & $\{0\} \subset A_1$ & $18/7$ & $0$ & $0$ & $0$ & $1$  
& $\checkmark$ \\
&& & $19/7$ & $3/1$ & $1$ & $1$ & $\dfrac{1}{\sqrt{2}}$  
& $\checkmark$ \\[1em] 
$A_6$ & $(A_5)''$ & $\mc{N}_{G_2}$ & $18/7$ & $4/7$ & $-84$  & $12$  & $\dfrac{1}{7^7}$   
& $\checkmark$ \\[1em]
 & $D_4$ & $\mc{N}_{C_3}$ & $18/7$ & $4/7$  & $-126$ & $18$ & $\dfrac{1}{7^{10}\sqrt{7}}$    
& $\checkmark$ \\[1em]
$E_7(a_5)$ & $E_7(a_5)$ & $\{0\}$ & $19/6$ &  & $0$ & $0$ & $1$   
& $\checkmark$ \\[1em]
 & $E_6(a_3)$ & $a_1$ & $19/6$ & $3/2$ & $-1$ & $2$ & $\dfrac{1}{4}$   
& $\checkmark$ \\[1em]
 & $D_6(a_2)$ & $a_1$ & $19/6$ & $2/3$ & $-6$ & $2$ & $\dfrac{1}{3\sqrt{3}}$   
& $\checkmark$ \\[1em]
 & $(A_5)'$ & {$a_1\times a_1$} & $19/6$ & $2/3, 3/2$ & $-7$ & $4$ & $\dfrac{1}{2^23\sqrt{3}}$   
& $\checkmark$  \\[1em]
 & $(A_5)''$ & $g_2^{sp}$ & $19/6$ & $7/6$  & $-34$ & $10$ & $\dfrac{1}{2^7 3^3 \sqrt{3}}$    
& $\checkmark$ \\[1em]
& $D_4$ & $\O_{(4,2)}^{C_3}$ & $19/6$ & $7/6$  & $-51$ & $16$ & $\dfrac{1}{6^4  \sqrt{3}}$    
& $\checkmark$ \\[1em]
$A_4+A_2$ & $A_4+A_2$ & $\{0\} \subset A_1$ & $19/5$ &$5/1$ & $\dfrac{9}{5}$ & $\dfrac{9}{5}$ &  
$\sqrt{\dfrac{2}{5}} \sin \dfrac{\pi}{5}$ & $\checkmark$ \\[1em]
 & $A_4$ & $a_2$ & $18/5$ &$3/5$ & $-32$ &$\dfrac{32}{5}$ &  $\dfrac{1}{5^4}$ 
& $\checkmark$ \\[1em]
$A_3+A_2+A_1$ & $A_3+A_2+A_1$ & $\{0\} \subset A_1$ &  $19/4$ & $4/1$ & $\dfrac{3}{2}$  &  $\dfrac{3}{2}$ &   
$\dfrac{1}{2}$ &  $\checkmark$ \\[1em]
 & $D_4(a_1)$ & $\mc{N}_{A_1} \times\mc{N}_{A_1}\times \mc{N}_{A_1} $ &  $19/4$ & $3/4$ & $-15$  &  $\dfrac{15}{2}$ &   
$\dfrac{1}{2^{10}\sqrt{2}}$ & $\checkmark$ \\[1em]
 & $A_3$ & {unknown in $A_1\times B_3$} &  $19/4$ & $3/4,7/4$ & $-44$  &  $\dfrac{35}{2}$ &   
$\dfrac{2}{2^{21}\sqrt{2}}\not=\dfrac{1}{2^{22}\sqrt{2}}$ &  fin.~ext.  \\[1em]
$2 A_2+A_1$ & $2 A_2+A_1$ & $\{0\} \subset A_1\times A_1$ & $19/3$ & $3/1, 2/1$  & $1$ & $1$ &   $\dfrac{1}{\sqrt{2}}$ 
& $\checkmark$   \\[1em]
 & $A_2+3 A_1$ & $g_2$ &$19/3$ & $8/3$  & $-7$ & $7$ &   $\dfrac{1}{3^3 \sqrt{6}}$ 
& $\checkmark$ \\
&& & $18/3$ & $6/3$ & $-14$ & $\dfrac{14}{3}$ &    $\dfrac{1}{3^3\sqrt{3}}$ 
&  not adm. \\[1em]
 & $2 A_2$ & $\{0\}\times g_2\subset A_1\times G_2$ &$19/3$ & $3/1,7/3$  & $-9$ & $7$ &   $\dfrac{1}{3^3 \sqrt{6}}$ 
& $\checkmark$  \\[1em]
 & $A_2$ & {unknown in $A_5$} &$19/3$ & $7/3$  & $-55$ & $25$ & 
 $\dfrac{3}{3^{18} \sqrt{2}}\not=\dfrac{1}{3^{18} \sqrt{2}}$ 
& fin.~ext.  \\[1em]
$4 A_1$ & $4 A_1$ & $\{0\} \subset C_3$ & $19/2$ &  & $0$ & $0$ &  $1$ 
& $\checkmark$ \\[1em]
& $ (3 A_1)'$ & $c_3$ &$19/2$ & $7/2$ & $-3$ & $6$ &  $\dfrac{1}{2^4}$ 
& $\checkmark$ \\[1em]
 & $ (3 A_1)''$ & $f_4$ &$19/2$ & $13/2$ & $-20$ & $16$ &  $\dfrac{1}{2^{13}}$ 
& $\checkmark$ \\[1em]
 & $2 A_1$ & unknown in $A_1\times B_4$ & 
 $19/2$ & $3/2,9/2$ & $-21$ & $18$ &  $\dfrac{2}{2^{16}} \not= \dfrac{1}{2^{16}}$ 
& fin.~ext. \\[1em]
 & $A_1$ & {unknown in $D_6$} & $19/2$ & $11/2$ & $-54$ & $36$ &  
 $\dfrac{2}{2^{34}} 
\not= \dfrac{1}{2^{34}}$ 
& fin.~ext.   \\
&& 
&   &  &   &   &    &   \\[1em] 
\hline
\end{tabular}\\[1em]
\caption{\footnotesize{Main asymptotic data in type $E_7$}}
\label{Tab:main_results-E7}
\end{table}
}

\begin{Pro}
\label{Pro:fin_dec-E7}
The following decompositions hold:  
{\footnotesize \begin{align*}
& \W_{-18+19/8}(E_7,D_5) 
\cong (L_{-2+3/8}(A_1)\otimes L_{-2+3/4}(A_1)) 
{ \oplus (L_{-2+3/8}(A_1;\varpi_1)\otimes L_{-2+3/4}(A_1;\varpi_1))}, & \\  
& \W_{-18+19/4}(E_7,A_3) 
\cong (L_{-2+3/4}(A_1) \otimes L_{-5+7/4}(B_3)) 
\oplus (L_{-2+3/4}(A_1;\varpi_1) \otimes L_{-5+7/4}(B_3;\varpi_3)) , & \\
& \W_{-18+19/3}(E_7,A_2) 
\cong L_{-6+7/3}(A_5) 
{\oplus L_{-6+7/3}(A_5;\varpi_2) \oplus L_{-6+7/3}(A_5;\varpi_4)}, & \\
& \W_{-18+19/2}(E_7,2 A_1)
\cong (L_{-2+3/2}(A_1) \otimes L_{-7+9/2}(B_4)) 
\oplus (L_{-2+3/2}(A_1;\varpi_1) \otimes L_{-7+9/2}(B_4;\varpi_1)) , & \\
& \W_{-18+19/2}(E_7,A_1) 
\cong 
L_{-10+11/2}(D_6) 
{\oplus L_{-10+11/2}(D_6;\varpi_1) }. & 
\end{align*}}
\end{Pro}

\begin{proof}
We argue as in the proof of Proposition \ref{Pro:fin_dec-E6}. 
More precisely, for $f$ in $D_5$ or $A_2$ we can use that fact that $f$ is even. 
For the other cases, we use Proposition \ref{Pro:fin_ext-assoc} 
and Remark~\ref{Rem:birational_equivalent_E7}. 
\end{proof}

For $(p,q)=(18,11)$ and $f \in D_6$, similar 
arguments as before Conjecture \ref{Conj:dec_finite_extensions-sp_n} 
lead to the following conjecture.  
\begin{Conj}
$$\W_{-18+18/11}(E_7,D_6)
\cong L_{-2+3/22}(A_1) \oplus  {L_{-2+3/22}(A_1;\varpi_1)}.$$
\end{Conj}

Based on the coincidence of central charges and asymptotic growths and dimensions at the non admissible levels $k$ and $k^\natural$ corresponding to $k+18=p/q=18/3$ and $k^\natural+4=6/3$, we conjecture:
\begin{Conj} 
\label{Conj:main_E7}
$$\W_{-12}(E_7,A_2+3A_1) \cong L_{-2}(G_2).$$
\end{Conj}

\begin{Th}
\label{Th:main_E8}
The following isomorphisms hold, 
providing collapsing levels for $\g=E_8$. 
{\footnotesize 
\begin{align*}
& \W_{-30+30/31}(E_8,E_8) \cong \C, 
&&  \W_{-30+31/30}(E_8,E_8) \cong \C, &\\ 
& \W_{-30+30/24}(E_8,E_8(a_1)) \cong \C, 
&& \W_{-30+31/20}(E_8,E_8(a_2)) \cong \C, & \\ 
& \W_{-30+31/18}(E_8,E_7) \cong L_{-2+2/9}(A_1), 
&& \W_{-30+31/15}(E_8,E_8(a_4)) \cong \C, &\\
& \W_{-30+31/12}(E_8,D_7) \cong L_{-2+2/3}(A_1), 
&& \W_{-30+31/12}(E_8,E_8(a_5)) \cong \C,  &\\
& \W_{-30+30/13}(E_8,E_6) \cong L_{-4+4/13}(G_2), 
&& \W_{-30+31/12}(E_8,E_6) \cong L_{-4+7/12}(G_2), &\\  
& \W_{-30+31/10}(E_8,E_8(a_6)) \cong \C, 
&& \W_{-30+31/10}(E_8,D_6) \cong L_{-3+3/5}(B_2), &\\
& \W_{-30+31/9}(E_8,E_6(a_1)) \cong L_{-3+4/9}(A_2), 
& & \W_{-30+31/8}(A_7,A_7) \cong \C,&\\
& \W_{-30+30/7}(E_8,A_6) \cong L_{-2+2/7}(A_1),  & 
& \W_{-30+31/6}(E_8,E_8(a_7)) \cong \C, &\\
& \W_{-30+31/6}(E_8,E_8(a_7)) \cong \C, & 
 & \W_{-30+31/6}(E_8,D_6(a_2)) \cong L_{-2+2/3}(A_1)\otimes L_{-2+2/3}(A_1), & \\
& \W_{-30+31/6}(E_8,E_6(a_3)) \cong L_{-4+7/6}(G_2),  & 
 & \W_{-30+31/6}(E_8,D_4) \cong L_{-9+13/6}(F_4),  &\\
& \W_{-30+30/7}(E_8,D_4) \cong L_{-9+9/7}(F_4), &   
& \W_{-30+31/5}(E_8,A_4+A_3) \cong \C,  & \\
& {\W_{-30+32/5}(E_8,A_4+A_3) \cong L_2(A_1)}, &
& \W_{-30+31/4}(E_8,2 A_3) \cong \C, & \\
& \W_{-30+31/4}(E_8,D_4(a_1)+A_2) \cong L_{-3+3/2}(A_2), &
& \W_{-30+31/3}(E_8,2A_2+2A_1) \cong \C, & \\
&  \W_{-30+32/3}(E_8,2A_2+2A_1) \cong L_1(B_2), &
&  \W_{-30+31/3}(E_8,2A_2) \cong L_{-4+7/3}(G_2) \otimes L_{-4+7/3}(G_2), &  \\
& \W_{-30+31/2}(E_8, 4 A_1) \cong \C, &
& \W_{-30+31/2}(E_8,3 A_1) \cong L_{-9+13/2}(F_4). &
\end{align*}}
Moreover, the following inclusions are finite extensions. 
{\footnotesize 
\begin{align*}
& L_{-2+3/38}(A_1) \longhookrightarrow \W_{-30+30/19}(E_8,E_7),&
& L_{-2+3/14}(A_1) \longhookrightarrow \W_{-30+31/14}(E_8,E_7(a_1)), & \\
& L_{-2+3/26}(A_1) \longhookrightarrow \W_{-30+30/13}(E_8,D_7),  &
& L_{-3+5/22}(B_2) \longhookrightarrow \W_{-30+30/11}(E_8,D_6), &\\
& L_{-3+7/8}(B_3) \longhookrightarrow \W_{-30+31/8}(E_8,D_5), &
& L_{-2+3/7}(A_1)\otimes L_{1}(A_1) \longhookrightarrow \W_{-30+31/7}(E_8,A_6), & \\
& L_{-6+7/4}(D_4) \longhookrightarrow \W_{-30+31/4}(E_8,D_4(a_1)), &
& L_{-9+11/4}(B_5) \longhookrightarrow \W_{-30+31/4}(E_8,A_3), & \\
& L_{-12+13/3}(E_6) \longhookrightarrow \W_{-30+31/3}(E_8,A_2), &
& L_{-11+13/2}(B_6) \longhookrightarrow \W_{-30+31/2}(E_8,2 A_1), &\\
& L_{-18+19/2}(E_7) \longhookrightarrow \W_{-30+31/2}(E_8,A_1).&&
\end{align*}}
\end{Th}

{\tiny 
\begin{table}
\begin{tabular}{lllllllll}
\hline
&& &&& &&& \\[-0.5em]
$\O_k$ & $G.f$ & $\Slo_{\O_k,f}$ & $\frac{p}{q}=k+h_{\g}^\vee$ & $k^\natural+h_{\g^\natural}^\vee$ & $c_V$ & $\G_V$ & $\A_V$ & comments \\
&&&&&& \\[-0.5em]
\hline
 && &&& &&& \\[-0.5em]
$E_8$ & $E_8$ & $\{0\}$ &30/31, 31/30&  & $0$ & $0$ & $1$ & $\checkmark$  \\
 $E_8(a_1)$ & $E_8(a_1)$ & $\{0\}$ &30/24&  & $0$ & $0$ & $1$ & $\checkmark$  \\[1em]
 $E_8(a_2)$ & $E_8(a_2)$ & $\{0\}$ &31/20&  & $0$ & $0$ & $1$ & $\checkmark$  \\[1em]
$E_8(a_3)$ & $E_7$ & $a_1$ &31/18 & $2/9$ & $-24$ & $\dfrac{8}{3}$ & $\dfrac{1}{27}$ & $\checkmark$\\
&&  & $30/19$ & $3/38$ & 
$-73$ & $\dfrac{56}{19}$ & $\dfrac{2}{76 \sqrt{19}} \not=\dfrac{1}{76 \sqrt{19}}$ &  fin.~ext. \\[1em]
$E_8(a_4)$ & $E_8(a_4)$ & $\{0\}$ &30/15&  & $0$ & $0$ & $1$ & $\checkmark$  \\[1em]
$E_8(b_4)$ & $E_7(a_1)$ & $a_1$ & $31/14$ & $3/14$ & $-25$ & 
$\dfrac{20}{7}$ & $\dfrac{2}{28\sqrt{7}}\not=\dfrac{1}{28\sqrt{7}}$ &  fin.~ext.  \\[1em]
$E_8(a_5)$ & $E_8(a_5)$ & $\{0\}$ &31/12&  & $0$ & $0$ & $1$ & $\checkmark$  \\[1em]
 & $D_7$ & $A_1$ & $30/13$ &  $3/26$ & $-49$ & $\dfrac{38}{13}$ & $\frac{2}{52\sqrt{13}}\not=
\frac{1}{52\sqrt{13}}$  & fin.~ext.   \\[1em]
  & &  & $31/12$  & $2/3$ & $-6$ & $2$ & $\frac{1}{3\sqrt{13}}$ & $\checkmark$  \\[1em]
 & $E_6$ & $\mc{N}_{G_2}$ & 
 $30/13$ & $4/13$ &$-168$ & $\dfrac{168}{13}$ &  
$\dfrac{1}{13^7}$ & $\checkmark$ \\[1em]
&& & $31/12$ & $7/12$ & $-82$ & $12$ & $\dfrac{1}{3^3 2^{14} \sqrt{3}}$ & 
$\checkmark$\\
 $E_8(a_6)$ & $E_8(a_6)$ & $\{0\}$ & $31/10$ &  & $0$ & $0$ & $1$ & $\checkmark$  \\[1em]
 & $D_6$ & $\mc{N}_{B_2}$  & $30/11$ & $5/22$ & $-122$ & $\dfrac{104}{11}$ 
& $\dfrac{2}{2^3 11^5} \not= 
 \dfrac{1}{2^3 11^5}$ &{fin.~ext.} \\[1em]
&& & $31/10$ & $3/5$ & $-40$ & $8$ & $\dfrac{1}{5^5}$ & 
$\checkmark$ \\[1em]
 $E_8(b_6)$ & $E_6(a_1)$ & $A_2$ & $31/9$ & $4/9$  & $-46$ & $\dfrac{22}{3}$ & $\dfrac{1}{9^4\sqrt{3}}$ & $\checkmark$  \\[1em]
$A_7$ & $A_7$ & $\{0\} \subset A_1$ & $31/8$ &  & $0$ & $0$ & $1$ & $\checkmark$  \\[1em]
 & $D_5$ & {unknown in $B_3$} & $31/8$ & $7/8$ & $-99$ & $18$ & 
 $\dfrac{2}{2^{28} \sqrt{2}}\not=\dfrac{1}{2^{28} \sqrt{32}} $ & fin.~ext.\\[1em]
$A_6+A_1$ & $A_6$ & $a_1$ & $30/7$ & $2/7$ & $-18$ & $\dfrac{18}{7}$ & 
$\dfrac{1}{7\sqrt{7}}$ & $\checkmark$ \\[1em]
 &  &  & $31/7$ & $3/7, 3/1$ & $-10$ & $\dfrac{26}{7}$ & 
$\dfrac{2}{14\sqrt{7}}\not=\dfrac{1}{14\sqrt{7}}$ &  fin.~ext. \\[1em]
$E_8(a_7)$ & $E_8(a_7)$ & $\{0\}$ & $31/6$ &  & $0$ & $0$ & $1$ & $\checkmark$  \\[1em]
 & $E_7(a_5)$ & $a_1$ & $31/6$ & $2/3$ & $-6$ & $2$  & $\dfrac{1}{3\sqrt{3}}$ & $\checkmark$\\[1em]
 & $D_6(a_2)$ & {$a_1\times a_1$} & $31/6$ & $2/3, 2/3$ & $-12$ & $4$  & $\dfrac{1}{3^3}$ & 
 $\checkmark$ \\[1em]
 & $E_6(a_3)$ & $g_2^{sp}$ & $31/6$ & $7/6$ & $-34$ & $10$ &$\dfrac{1}{2^7 3^3 \sqrt{3}}$ & 
$\checkmark$ \\[1em]
 & $A_5$ & $a_1\times g_2^{sp}$ & $31/6$ & $2/3,7/6$ & $-40$ & $12$ &$\dfrac{1}{2^7 3^5}$ & 
$\checkmark$ \\[1em]
 & $D_4$ & {$F_4(a_3)$} & $31/6$ & $13/6$ & $-164$ & $40$ & 
$\dfrac{1}{3^{26} 2^{13}}$ & $\checkmark$ \\[1em]
$A_6+A_1$ & $D_4$ & {$F_4(a_2)$} & $30/7$ & $9/7$ & $-312$ & $\dfrac{312}{7}$ & $\dfrac{1}{7^{26}}$ & 
$\checkmark$ \\[1em]
$A_4+A_3$ & $A_4+A_3$ & $\{0\}  \subset A_1$ & $31/5$ &  & $0$ & $0$ & $1$ & $\checkmark$  \\
&& & $32/5$ & $4/1$  & $\dfrac{3}{2}$ & $\dfrac{3}{2}$  & $\dfrac{1}{2}$ & $\checkmark$  \\[1em]
 $2 A_3$ & $2 A_3$ & $\{0\}  \subset A_1$ & $31/4$ &  & $0$ & $0$ & $1$ & $\checkmark$  \\[1em]
 & $D_4(a_1)+A_2$ & $a_2$ & $31/4$  & $3/2$ & $-8$ &4 &$\dfrac{1}{2^4}$ & $\checkmark$  \\[1em]
 & $D_4(a_1)$ & {unknown in $D_4$} & $31/4$  & $7/4$ & $-68$ & $22$ &$\dfrac{4}{2^{29}}\not=\dfrac{1}{2^{29}}$ & 
fin.~ext.  \\[1em]
 & $A_3$ & {unknown in $B_5$} & $31/4$ & $11/4$ & $-125$ & $40$ & $\dfrac{2}{2^{50}}
\not= \dfrac{1}{2^{50}}$ & fin.~ext.  \\[1em]
\hline
\end{tabular}\\[1em]
\caption{\footnotesize{Main asymptotic data in type $E_8$}}
\label{Tab:main_results-E8-a}
\end{table}
}

{\tiny 
\begin{table}
\begin{tabular}{lllllllll}
\hline
&& &&& &&& \\[-0.5em]
$\O_k$ & $G.f$ & $\Slo_{\O_k,f}$ & $\frac{p}{q}=k+h_{\g}^\vee$ & $k^\natural+h_{\g^\natural}^\vee$ & $c_V$ & $\G_V$ & $\A_V$ & comments \\
&&&&&& \\[-0.5em]
\hline
$2 A_2+2 A_1$ & $2 A_2+2 A_1$ & $\{0\}  \subset B_2$ & $31/3$ &  & $0$ & $0$ & $1$ & $\checkmark$  \\
 & & & $32/3$ & $4/1$  & $\dfrac{5}{2}$ & $\dfrac{5}{2}$  & $\dfrac{1}{2}$ & $\checkmark$  \\[1em]
 & $2 A_2$ & {$g_2 \times g_2$} & $31/3$ & $7/3, 7/3$ & $-20$ & $12$ &$\dfrac{1}{3^7}$ 
 & $\checkmark$ \\[1em]
   & $A_2$ & {unknown in $E_6$} & $31/3$ & $13/3$ & $-138$ & $54$ & $\dfrac{3}{3^{39} \sqrt{3}}
\not= \dfrac{1}{3^{39} \sqrt{3}}$  & fin.~ext.  \\[1em]
$4 A_1$ & $4 A_1$ & $\{0\}  \subset C_4$ & $31/2$ &  & $0$ & $0$ & $1$ & $\checkmark$  \\[1em]
 & $3 A_1$ & $f_4$ & $31/2$ & $13/2$ & $-20$ & $16$ & $\dfrac{1}{2^{13}}$ & 
$\checkmark$ \\[1em]
 & $2 A_1$ & {unknown in $B_6$} & $31/2$ & $13/2$ & $-54$ & $36$ & 
 $\dfrac{2}{2^{33}} \not= \dfrac{1}{2^{33}}$ &  fin.~ext.  \\[1em]
 & $A_1$ & {unknown in $E_7$} & $31/2$ & $19/2$ & $-119$ & $70$ & $\dfrac{2}{2^{67}} \not=\dfrac{1}{2^{67}}$ & 
fin.~ext.  \\[1em]
\hline
\end{tabular}\\[1em]
\caption{\footnotesize{Main asymptotic data in type $E_8$ (continued)}}
\label{Tab:main_results-E8}
\end{table}
}

\begin{Rem}
\label{Rem:birational_equivalent_E8}
In the below comments, we argue as in Remark \ref{Rem:birational_equivalent_E6} 
to conclude that the varieties are not isomorphic with the same dimension. 
\begin{enumerate}
\item The associated 
variety of $H_{DS,D_5}^0(L_{-30+31/8}(E_8))$ 
is not isomorphic to the associated variety 
of $L_{-5+7/8}(B_3)$. 
The former is the nilpotent Slodowy 
slice $\Slo_{A_7,D_5}$ in $E_8$   
while the later is the closure of the the nilpotent cone $\mc{N}_{B_3}$ of $B_3$. 
\item The associated 
variety of $H_{DS,D_4(a_1)}^0(L_{-30+31/4}(E_8))$ is not isomorphic to the associated variety 
of $L_{-4+7/4}(D_4)$. 
The former is the nilpotent Slodowy 
slice $\Slo_{2 A_3,D_4(a_1)}$ in $E_8$   
while the later is the closure of the orbit $\O_{1;(5,3)}$ in $D_4$. 
\item The associated 
variety of $H_{DS,A_3}^0(L_{-30+31/4}(E_8))$ is not isomorphic to the associated variety 
of $L_{-9+11/4}(B_5)$. 
The former is the nilpotent Slodowy 
slice $\Slo_{2 A_3,A_3}$ in $E_8$   
while the later is the closure of the orbit $\O_{1;(4^2,3)}$ in $B_5$. 
\item The associated 
variety of $H_{DS,A_2}^0(L_{-30+31/3}(E_8))$ is not isomorphic to the associated variety 
of $L_{-12+13/3}(E_6)$. 
The former is the nilpotent Slodowy 
slice $\Slo_{2 A_2+2 A_1,A_2}$ in $E_8$   
while the later is the closure of the orbit $2 A_2+A_1$ in $E_6$.  
\item The associated 
variety of $H_{DS,2A_1}^0(L_{-30+31/2}(E_8))$ is isomorphic to the associated variety 
of $L_{-11+13/2}(B_6)$. The former is the nilpotent Slodowy 
slice $\Slo_{4 A_1, 2A_1}$ in $E_8$   
while the later is the closure of the orbit $\O_{1;(2^6,1)}$ in $B_6$. 

\item The associated 
variety of $H_{DS,A_1}^0(L_{-30+31/2}(E_8))$ is not isomorphic to the associated variety 
of $L_{-18+19/2}(E_7)$. Indeed the former is the nilpotent Slodowy 
slice $\Slo_{4 A_1,A_1}$ in $E_8$   
while the later is the closure of the orbit $4 A_1$ in $E_7$. 
\end{enumerate}
\end{Rem}

\begin{Pro}
\label{Pro:fin_dec-E8}
The following decompositions hold:
{\footnotesize 
\begin{align*}
& \W_{-30+31/8}(E_8,D_5) 
\cong L_{-3+7/8}(B_3) {\oplus L_{-3+7/8}(B_3;\varpi_3)}, & \\
& \W_{-30+31/4}(E_8,D_4(a_1)) 
 \cong L_{-6+7/4}(D_4) \oplus 
\bigoplus_{i=1,3,4} L_{-6+7/4}(D_4;\varpi_i), & \\ 
& \W_{-30+31/4}(E_8,A_3) \cong L_{-9+11/4}(B_5) \oplus  L_{-9+11/4}(B_5;\varpi_1),& \\
& \W_{-30+31/3}(E_8,A_2) \cong L_{-12+13/3}(E_6) {\oplus  L_{-12+13/3}(E_6;\varpi_1)
\oplus L_{-12+13/3}(E_6;\varpi_6)}, & \\
& \W_{-30+31/2}(E_8,2 A_1) 
\cong L_{-11+13/2}(B_6) \oplus  L_{-11+13/2}(B_6;\varpi_6) , & \\
& \W_{-30+31/2}(E_8,A_1) 
\cong L_{-18+19/2}(E_7) {\oplus L_{-18+19/2}(E_7;\varpi_7)}. &    
\end{align*}}
\end{Pro}

\begin{proof}
We argue as in the proof of Proposition \ref{Pro:fin_dec-E6}. 
More precisely, for $f$ in $D_5$, $D_4(a_1)$ or $A_2$ we can use that fact that $f$ is even. 
For the other cases, we use Proposition \ref{Pro:fin_ext-assoc} 
and Remark~\ref{Rem:birational_equivalent_E8}. 
\end{proof}

For other cases, similar 
arguments as before Conjecture \ref{Conj:dec_finite_extensions-sp_n} 
lead to the following conjecture.  

\begin{Conj}
{\footnotesize 
\begin{align*}
& 
\W_{-30+30/19}(E_8,E_7) 
\cong L_{-2+3/38}(A_1) {\oplus L_{-2+3/38}(A_1;\varpi_1)}, & \\
& \W_{-30+31/14}(E_8,E_7(a_1)) 
\cong L_{-2+3/14}(A_1) \oplus L_{-2+3/14}(A_1;\varpi_1) &\\
& \W_{-30+30/13}(E_8,D_7) 
\cong L_{-2+3/26}(A_1) \oplus L_{-2+3/26}(A_1;\varpi_1) &\\
& \W_{-30+30/11}(E_8,D_6) 
\cong L_{-3+5/22}(B_2) {\oplus L_{-3+5/22}(B_2;\varpi_2)}, & \\
& \W_{-30+31/7}(E_8,A_6) \cong  \left(L_{-2+3/7}(A_1)\otimes L_{1}(A_1)\right) 
\oplus \left( L_{-2+3/7}(A_1;\varpi_1)\otimes L_{1}(A_1;\varpi_1)\right) .&
\end{align*}}
\end{Conj}

In type $E_8$ we make the following conjecture concerning non admissible levels.
\begin{Conj}
\label{Conj:main_E8}
The following isomorphisms hold. 
{\footnotesize \begin{align*}
& \W_{-30+32/12}(E_8,E_6) \cong L_{-4+2/3}(G_2) , & 
& \W_{-24}(E_8,E_6(a_3)) \cong L_{-2}(G_2) , & \\
& \W_{-45/2}(E_8,A_4+2A_1) \cong \C , & 
& \W_{-70/3}(E_8,A_4+A_2+A_1) \cong \C. & 
\end{align*}}
\end{Conj}
For example in the case of $f$ of Bala-Carter type $A_4+2A_1$ the conjecture is motivated by the fact that the central charge is $0$ and $k_0^\natural=k_1^\natural=0$. Similarly for $f$ of type $A_4+A_2+A_1$ we find $k_1^\natural=0$.

\begin{Th}
\label{Th:main_G2}
The following isomorphisms hold, providing collapsing levels 
for $G_2$. 
{\footnotesize \begin{align*}
& \W_{-4+7/12}(G_2,G_2) \cong \C, &&  \W_{-4+4/7}(G_2,G_2) \cong \C & \\
& \W_{-4+7/6}(G_2,G_2(a_1)) \cong \C , &&  
\W_{-4+7/6}(G_2,\tilde A_1) \cong L_{-2+2/3}(A_1), & \\
& \W_{-4+7/3}(G_2,A_1) \cong \C, &&  \W_{-4+8/3}(G_2,A_1) \cong L_{1}(A_1).& 
\end{align*}}
Moreover, the following isomorphism holds. 
{\footnotesize 
\begin{align*}
\W_{-4+5/2}(G_2,  A_1) 
\cong L_{-2+5/2}(A_1) \oplus L_{-2+5/2}(A_1;\varpi_1). 
\end{align*}}
\end{Th}

\begin{proof}
To prove the second assertion we argue as in the proof of Proposition \ref{Pro:fin_dec-E6} 
using Remark~\ref{Rem:birational_equivalent_G2} below. 
\end{proof}

{\tiny 
\begin{table}
\begin{tabular}{lllllllll}
\hline
&& &&& &&& \\[-0.5em] 
$\O_k$ & $G.f$ & $\Slo_{\O_k,f}$ & $\frac{p}{q}=k+h_{\g}^\vee$ & $k^\natural+h_{\g^\natural}^\vee$ & $c_V$ & $\G_V$ & $\A_V$ & comments \\
&&&&&& \\[-0.5em]
\hline
 && &&& &&& \\[-0.5em]
$G_2$ & $G_2$ & $\{0\}$ & $7/12$ &  & $0$ & 
$0$ & $1$ & $\checkmark$ \\
 & & &  $4/7$ &  & $0$ & $0$ & $1$ & $\checkmark$ \\[1em]
$G_2(a_1)$ & $G_2(a_1)$  & $\{0\}$ & $7/6$ &  & $0$ & 
$0$ & $1$ & $\checkmark$\\[1em]
 & $\tilde A_1$ & $a_1$ & $7/6$ & $2/3$ & $-6$ & 
$2$ & $\dfrac{1}{3\sqrt{3}}$ 
& $\checkmark$ \\[1em]
$\tilde A_1$ & $A_1$  & {$m$} & $5/2$ & $5/2$  & $3/5$ & 
$12/5$ & $\dfrac{2}{2\sqrt{5}} \sin \dfrac{\pi}{5} 
\not=\dfrac{1}{2\sqrt{5}} \sin \dfrac{\pi}{5}$ & fin.~ext.\\[1em]
$A_1$ & $A_1$  & $\{0\} \subset A_1$ & $7/3$ & $2/1$ & $0$ & 
$0$ & $1$ & $\checkmark$ \\
 &   &  & $8/3$ & $3/1$ & $1$ & 
$1$ & $\dfrac{1}{\sqrt{2}}$ & $\checkmark$ \\[1em]
\hline
\end{tabular}\\[1em]
\caption{\footnotesize{Main asymptotic data in type $G_2$}}
\label{Tab:main_results-G2}
\end{table}
}

\begin{Rem}
\label{Rem:birational_equivalent_G2}
The associated 
variety of $\W_{-4+5/2}(G_2,A_1)$ is not isomorphic to the associated variety 
of $L_{-2+5/2}(A_1)$. 
The former is the nilpotent Slodowy 
slice $\Slo_{\tilde A_1,A_1}$ in $G_2$  
while the later is the 
the nilpotent cone in $A_1$, 
and it is known that $\Slo_{\tilde A_1,A_1}$ is 
a 2-dimensional non-normal 
variety. 
\end{Rem}

\begin{Th} 
\label{Th:main_F4}
The following isomorphisms hold, providing collapsing levels for $\g=F_4$. 
{\footnotesize \begin{align*}
& \W_{-9+13/18}(F_4,F_4) \cong \C, &&  \W_{-9+9/13}(F_4,F_4) \cong \C & \\
& \W_{-9+13/12}(F_4,F_4(a_1)) \cong \C, &&  \W_{-9+9/7}(F_4,C_3) \cong L_{-2+2/7}(A_1) & \\
& \W_{-9+9/7}(F_4,B_3) \cong L_{-2+2/7}(A_1) , &&  \W_{-9+13/8}(F_4,B_3)  \cong L_{1}(A_1) & \\
& \W_{-9+13/6}(F_4,F_4(a_3)) \cong \C , &&  \W_{-9+13/6}(F_4,C_3(a_1))  \cong L_{-2+2/3}(A_1) & \\
& \W_{-9+13/6}(F_4,B_3) \cong L_{-2+2/3}(A_1) \otimes L_{-2+2/3}(A_1) , &&  
\W_{-9+13/6}(F_4,\tilde A_2))  \cong L_{-4+7/6}(G_2) & \\
& \W_{-9+13/4}(F_4,A_2 +\tilde A_1) \cong \C , &&  
\W_{-9+13/4}(F_4,A_2))  \cong L_{-3+3/2}(A_2), & \\
& \W_{-9+13/2}(F_4,A_1))  \cong \C. && & 
\end{align*}}
Moreover, the following inclusions are finite extensions. 
{\footnotesize 
\begin{align*}
& L_{-2+5/2}(A_1) \longhookrightarrow \W_{-9+10/3}(F_4,A_2+\tilde A_1), & 
& L_{-3+5/3}(A_2)   \longhookrightarrow \W_{-9+10/3}(F_4,A_2), & \\
& L_{-4+5/4}(A_3) \longhookrightarrow \W_{-9+13/4}(F_4,\tilde A_1). & 
\end{align*}}
\end{Th}

\begin{Rem}
\label{Rem:birational_equivalent_F4}
In the below comments, we argue as in Remark \ref{Rem:birational_equivalent_E6} 
to conclude that the varieties are not isomorphic. 
\begin{enumerate}
\item The associated 
variety of $H_{DS,A_2+\tilde A_1}^0(L_{-9+10/3}(F_4))$ is not isomorphic to the associated variety 
of $L_{-2+5/2}(A_1)$. 
The former is the nilpotent Slodowy 
slice $\Slo_{\tilde A_2+A_1, A_2+\tilde A_1}$ 
whose singularity at $f \in A_2+\tilde A_1$ is not normal 
(see \cite[\S1.8.4]{FuJutLev17}) 
while the later is the nilpotent cone of $A_1=\sl_2$. 
\item The associated 
variety of $H_{DS,A_2}^0(L_{-9+10/3}(F_4,A_2)$ is not isomorphic to the associated variety 
of $L_{-2+5/3}(A_2)$. 
The former is the nilpotent Slodowy 
slice $\Slo_{\tilde \tilde A_2+A_1, A_2}$ 
while the later is the nilpotent cone of $A_2$. 

\item The associated 
variety of $H_{DS,\tilde A_1}^0(L_{-9+13/4}(F_4))$ is not isomorphic to the associated variety 
of $L_{-4+5/4}(A_3)$. Indeed the former is the nilpotent Slodowy 
slice $\Slo_{A_2+\tilde A_1, \tilde A_1}$ in $F_4$   
while the later is the nilpotent cone of $A_3=\sl_4$. 
But the number 
of nilpotent $G^\natural$-orbits in $\mc{N}_{A_3}$ is 5 while the number 
of  $G^\natural$-orbits in $\Slo_{A_2+\tilde A_1, \tilde A_1}$ would be 4. 
\end{enumerate}
\end{Rem}

\begin{Pro}
\label{Pro:fin_dec-F4}
The following decompositions hold:
{\footnotesize 
\begin{align*}
& \W_{-9+10/3}(F_4,A_2+\tilde A_1) 
\cong L_{-2+5/2}(A_1) \oplus L_{-2+5/2}(A_1;\varpi_1), & \\
& \W_{-9+10/3}(F_4,A_2) 
\cong L_{-3+5/3}(A_2) \oplus L_{-3+5/3}(A_2;\varpi_1) \oplus L_{-3+5/3}(A_2;\varpi_2), & \\
& \W_{-9+13/4}(F_4,\tilde A_1) 
\cong L_{-4+5/4}(A_3) 
\oplus \bigoplus_{i=1,2,3}L_{-4+5/4}(A_3;\varpi_i). & 
\end{align*}}
\end{Pro}

\begin{proof}
We argue as in the proof of Proposition \ref{Pro:fin_dec-E6}. 
More precisely, for $f \in A_2$ we can use that fact that $f$ is even. 
For the other cases, we use Proposition \ref{Pro:fin_ext-assoc} 
and Remark~\ref{Rem:birational_equivalent_F4}. 
\end{proof}

For other cases, similar 
arguments as before Conjecture \ref{Conj:dec_finite_extensions-sp_n} 
lead to the following conjecture.  

\begin{Conj}
{\footnotesize 
\begin{align*}
& \W_{-9+13/10}(F_4,C_3) 
\cong L_{-2+3/10}(A_1) \oplus L_{-2+3/10}(A_1;\varpi_1), & \\
& \W_{-9+9/5}(F_4,B_2)  
\cong (L_{-2+3/10}(A_1) \otimes L_{-2+3/10}(A_1)) 
\oplus( L_{-2+3/10}(A_1;\varpi_1) \otimes L_{-2+3/10}(A_1;\varpi_1)).  & \\
\end{align*}}
\end{Conj}

{As in the classical cases, we have the following conjecture.
\begin{Conj} 
\label{Conj:exhaustive_exceptional}
Assume that $\g$ simple of exceptional type. 
The cases covered by Theorems~\ref{Th:main_E6}, 
\ref{Th:main_E7}, \ref{Th:main_E8}, \ref{Th:main_G2} and \ref{Th:main_F4} 
give the exhaustive list of pairs $(f,k)$  
where $f$ is a nilpotent element of $\g$ 
and $k$ is an admissible 
collapsing levels for $\g$. 
\end{Conj}
}

{\tiny 
\begin{table}
\begin{tabular}{lllllllll}
\hline
&& &&& &&& \\[-0.5em]
$\O_k$ & $G.f$ & $\Slo_{\O_k,f}$ & $\frac{p}{q}=k+h_{\g}^\vee$ & $k^\natural+h_{\g^\natural}^\vee$ & $c_V$ & $\G_V$ & $\A_V$ & comments \\
&& &&& &&& \\[-0.5em]
\hline
&& &&& &&& \\[-0.5em]
$F_4$ & $F_4$ & $\{0\}$ & $13/18, 9/13$ &  &$0$ & $0$ 
& $1$ & $\checkmark$ \\[1em]
$F_4(a_1)$ & $F_4(a_1)$ & $\{0\}$ & $13/12$ &  &$0$ & $0$ 
& $1$ & $\checkmark$\\[1em]
$F_4(a_2)$ & $C_3$ & $a_1$ & $9/7$ & $2/7$ &$-18$ & $\dfrac{18}{7}$ 
& $\dfrac{1}{7\sqrt{7}}$ & $\checkmark$ \\
&  & & $13/10$ &$3/10$ & $-17$ & $\dfrac{14}{5}$ 
& $\dfrac{2}{2^2 5\sqrt{5}}\not=\dfrac{1}{2^2 5 \sqrt{5}}$  &  fin.~ext. \\[1em]
 & $B_3$ & $a_1$ & $9/7$ & $2/7$ &$-18$ & $\dfrac{18}{7}$ 
& $\dfrac{1}{7\sqrt{7}}$ & $\checkmark$ \\[1em]
$B_3$ & $B_3$ & $\{0\} \subset A_1$ & $13/8$ & $3/1$ &$1$ & $1$ 
& $ \dfrac{1}{\sqrt{2}}$ & $\checkmark$ \\[1em]
 & $ \tilde A_2$ & $\mc{N}_{G_2}$ & $12/8$ & $4/8$ & $-98$ & $\dfrac{49}{4}$ & & 
not admissible \\[1em]
$F_4(a_3)$ & $F_4(a_3)$ & $\{0\}$ &$13/6$ &  & $0$ & $0$ & $1$ & $\checkmark$\\[1em] 
 & $C_3(a_1)$ & $a_1$ &$13/6$ & $2/3$ & $-6$ & $2$ & $\dfrac{1}{3\sqrt{3}}$ & 
$\checkmark$ \\[1em]
 & $B_2$ & ${a_1 \times a_1}$ & $13/6$ & $2/3, 2/3$ &$-12$ & $4$ 
& $\dfrac{1}{3 \sqrt{3}}$ & $\checkmark$ \\[1em]
 &  & & $9/5$ & $3/10, 3/10$ &$-34$ & $\dfrac{28}{5}$ 
& $\dfrac{4}{5^{3/2} 2^2} \not=\dfrac{2}{5^{3/2} 2^2}$ & fin.~ext.  \\[1em]
 & $\tilde A_2$ & $g_2^{sp}$ & $13/6$ & $7/6$ & $-34$ & $10$ & 
$\dfrac{1}{2^7 3^3 \sqrt{3}}$  & $\checkmark$ \\[1em] 
 &  &  & $9/6$ & $3/6$ & $-98$ &$\dfrac{14}{3}$  & $\dfrac{1}{3^3\sqrt{3}}$
 & 
 not admissible \\[1em]
$\tilde A_2+A_1$ & $\tilde A_2$ & $g_2$ & $9/3$ & $6/3$ & $-14$ &$\dfrac{14}{3}$  & $\dfrac{1}{3^3\sqrt{3}}$
 & 
 not admissible \\[1em]
& $A_2+\tilde A_1$ &{$m$}  & $10/3$ & $5/2$ & $\dfrac{3}{5}$ & 
$\dfrac{13}{5}$ & 
$\dfrac{2}{2\sqrt{5}} \sin \dfrac{\pi}{5}$ & fin.~ext.\\[1em]
&  &   & &   &   &  & $\not=\dfrac{1}{2\sqrt{5}} \sin \dfrac{\pi}{5}$ & \\[1em]
 & $A_2$ & {unknown in }$A_2$ &$10/3$ & $5/3$ & 
$-\dfrac{32}{5}$ & $\dfrac{32}{5}$ & $\dfrac{3 \times 2^3}{3^{7/2} 5} \left( \sin \dfrac{\pi}{5}\right)^2  
\sin \dfrac{2\pi}{5}$ & fin.~ext. \\
 &  &  & &  &   &  & 
 $\not= \dfrac{2^3}{3^{7/2} 5} \left( \sin \dfrac{\pi}{5}\right)^2  
\sin \dfrac{2\pi}{5} $ &  \\[1em]
 $A_2+\tilde A_1$ &  $A_2+\tilde A_1$ & $\{0\} \subset A_1$ &$13/4$ & $2/1$ & 
$0$ & $0$ & $1$ & 
$\checkmark$ \\[1em]
 & $A_2$ & $a_2^+$ &$13/4$ & $3/2$ & 
$-8$ & $4$ & $\dfrac{1}{2^4}$ & 
$\checkmark$ \\[1em]
 & $ \tilde A_1$ & {unknown in }$A_3$ &$13/4$ & $5/4$ & 
$-33$ & $12$ & $\dfrac{4}{2^{16}} \not=\dfrac{1}{2^{16}}$ & 
fin.~ext. \\[1em]
$A_1$ & $ A_1$ & $\{0\} \subset C_3$ &$13/2$ & $4/1$ & 
$0$ & $0$ & $1$ & 
$\checkmark$ \\[1em]
\hline
\end{tabular}\\[1em]
\caption{\footnotesize{Main asymptotic data in type $F_4$}}
\label{Tab:main_results-F4}
\end{table}
}

\begin{Conj} 
\label{Conj:main_F4} 
The following isomorphisms hold. 
{\footnotesize
\begin{align*}
& \W_{-15/2}(F_4,\tilde A_2) \cong L_{-7/2}(G_2) , && 
\W_{-6}(F_4,\tilde A_2) \cong L_{-2}(G_2). & \\
\end{align*}}
\end{Conj}
Writing $k=-9+12/8$ or $-9+9/6$ and $k^\natural = -4+/8$ or $-4+3/6$, 
we find that $\W_{-15/2}(F_4,\tilde A_2)$ and $L_{-7/2}(G_2)$ 
have the same central charge, and that the (conjectural) formula for asymptotic 
growth and asymptotic dimension coincide; see Table \ref{Tab:main_results-F4}. 
Similarly, writing $k=-9+9/3$ and $k^\natural = -4+6+3$ 
we find that $\W_{-6}(F_4,\tilde A_2)$ and $L_{-2}(G_2)$ 
have the same central charge, and that the (conjectural) formula for asymptotic 
growth and asymptotic dimension coincide; see Table \ref{Tab:main_results-F4}.

\newpage

\appendix

\section{Centralisers of $\sl_2$-triples  
in simple exceptional Lie algebras}
\label{App:Centralisers_exceptional}

In this appendix we collect the data relative to each $\sl_2$-triples 
in simple exceptional Lie algebras. 
Our results are obtained using the software \texttt{GAP4} and 
are summarised in Tables \ref{Tab:Data-G2}, \ref{Tab:Data-F4}, \ref{Tab:Data-E6}, 
\ref{Tab:Data-E7} and \ref{Tab:Data-E8}.

In the Tables, nilpotent orbits are given by the Bala-Carter classification (first column). 
We indicate in the second column whether the nilpotent orbit is even or not. 
The third column gives the type of $\g^\natural$, and the last column gives 
the values of the $k_i^\natural$'s. Obviously, the order follows the order 
of the simple factors of $\g^{\natural}$ as appearing in the third column. 

\bigskip

\bigskip

\begin{minipage}{.5\textwidth}
\centering
{\tiny 
\begin{tabular}{|llll|}
\hline
&&&\\[-0.5em] 
$G.f$ & even & $\g^\natural= \bigoplus_{i} \g_i^\natural$ &   $k_i^\natural$  \\[-0.5em] 
&&&\\
\hline
&&&\\[-0.5em] 
$G_2$  & yes  & $\{0\}$ &   \\[0.5em] 
$G_2(a_1)$ & yes  & $\{0\}$ &     \\[0.5em] 
$\tilde{A}_1$ & no  & $A_1$ & $k_1^\natural=k+\frac{3}{2}$ \\[0.5em] 
$A_1$  & no  & $A_1$ & $k_1^\natural=3 k+5$   \\[0.5em] 
\hline 
\end{tabular} \\[1em]  
\captionof{table}{\footnotesize{Centralisers of $\sl_2$-triples in type $G_2$}} 
\label{Tab:Data-G2}
\vspace{0.44cm}
}
{\tiny 
\begin{tabular}{|llll|}
\hline
&&&\\[-0.5em] 
$G.f$ & even & $\g^\natural= \bigoplus_{i} \g_i^\natural$ &  
 $k_i^\natural$  \\[-0.5em] 
&&&\\
\hline
&&&\\[-0.5em] 
$F_4$ & yes  &  $\{0\}$ &   \\[0.5em]
$F_4(a_1)$ & yes &  $\{0\}$ &   \\[0.5em]
${F_4(a_2)}$ & yes  & $\{0\}$  &   \\[0.5em]
${C_3}$ & yes  & $A_1$  & $k_1^\natural=k+6$   \\[0.5em]
${B_3}$ & yes  &  $A_1$ & $k_1^\natural=8k +60$   \\[0.5em]
${F_4(a_3)}$ & yes  & $\{0\}$ &   \\[0.5em]
${C_3(a_1)}$ & no  &  $A_1$ & $k_1^\natural=k+\frac{11}{2}$   \\[0.5em] 
${\tilde{A}_2 +A_1}$  & no & $A_1$ & $k_1^\natural=3k+17$  \\[0.5em]
$B_2$ & no &  $A_1 \times A_1$ & $k_1^\natural=k+\frac{11}{2}$   \\
&&& $k_2^\natural=k+\frac{11}{2}$   \\[0.5em] 
$A_2 +\tilde{A}_1$ & no &  $A_1$ & $k_1^\natural=6k+\frac{69}{2}$  \\[0.5em] 
$\tilde{A}_2$ & yes  & $G_2$  & $k_1^\natural=k+4$   \\[0.5em] 
$A_2$ & yes  &  $A_2$ & $k_1^\natural=2k+10$  \\[0.5em]
$A_1 + \tilde{A}_1$ &  no  & $A_1\times A_1$ & $k_1^\natural=k+4$  \\
&&& $k_2^\natural=8k+40$  \\[0.5em]
$\tilde{A}_1$ &  no  & $A_3$ &  $k_1^\natural=k+3$\\[0.5em]
$A_1$ & no &  $C_3$ & $k_1^\natural= k+ \frac{5}{2}$  \\[0.5em]
\hline 
\end{tabular} \\[1em] 
\captionof{table}{\footnotesize{Centralisers of $\sl_2$-triples in type $F_4$}} 
\label{Tab:Data-F4}
}
\end{minipage}\hspace{1cm}
\begin{minipage}{.5\textwidth} 
\centering
{\tiny
\begin{tabular}{|llll|}
\hline
&&&\\[-0.5em] 
$G.f$ & even & $\g^\natural= \bigoplus_{i} \g_i^\natural$ &  
 $k_i^\natural$  \\[-0.5em] 
&&&\\
\hline
&&&\\[-0.5em] 
$E_6$ & yes  &  $\{0\}$ &   \\[0.5em]
$E_6(a_1)$ & yes  &  $\{0\}$ &   \\[0.5em]
$D_5$ & yes & $\C$ &  $k_0^\natural= k+\frac{21}{2}$  \\[0.5em]
${E_6(a_3)}$ & yes & $\{0\}$ &    \\[0.5em]
$D_5(a_1)$ & no  & $\C$ &  $k_0^\natural= 12k+117$  \\[0.5em]
$A_5$ & no  & $A_1$ &  $k_1^\natural= k+\frac{17}{2}$  \\[0.5em]
$A_4+A_1$ & no  & $\C$ & $k_0^\natural= 15 k+144$ \\[0.5em]
$D_4$ & yes & $A_2$ &  $k_1^\natural=2k+18$  \\[0.5em]
$A_4$ & yes  & $\C \times A_1$ & $k_0^\natural=15 k+144$ \\ 
& & & $k_1^\natural=k+8$  \\[0.5em]
$D_4(a_1)$ & yes  & $\C^2$ & 
$\phi^\natural_0 \not= 0$ for admissible $k$\\[0.5em]
$A_3+A_1$ & no  & $\C\times A_1$ & $k_0^\natural=k+9$ \\ 
& & &  $k_1^\natural=k +\frac{15}{2}$ \\[0.5em]
$2 A_2+A_1$  & no &  $A_1$ & $k_1^\natural=3k+23$  \\[0.5em]
$A_3$ & no & $\C\times B_2$ & $k_0^\natural= k+9$ \\
& & & $k_1^\natural=k+7$ \\[0.5em] 
$A_2 +2 A_1$ & no  & $\C\times A_1$ &  $k_0^\natural=6k+45$   \\
& & & $k_1^\natural= 6k+45$  \\[0.5em]
$2 A_2$  & no & $G_2$ & $k_1^\natural=k+6$ \\[0.5em]
$A_2 +A_1$ & no & $\C\times A_2$ &  $k_0^\natural=6k+45$   \\
& & & $k_1^\natural=k+6 $  \\[0.5em]
$A_2$ & yes  & $A_2 \times A_2$   &  $k_1^\natural=k+6$   \\
& & & $k_2^\natural=k+6 $  \\[0.5em]
$3 A_1$ & no & $A_1\times A_2$ & $k_1^\natural=k+\frac{11}{2}$  \\
& & & $k_2^\natural= 2 k+12$  \\[0.5em]
$2 A_1$ & no  &  $\C\times B_3$ & $k_0^\natural=3 k+18$ \\
& & & $k_1^\natural=k+4$  \\[0.5em]
$A_1$ & no & $A_5$ & $k_1^\natural=k+3$\\[0.5em]
\hline
\end{tabular} \\[1em] 
\captionof{table}{\footnotesize{Centralisers of $\sl_2$-triples in type $E_6$}} 
\label{Tab:Data-E6}
}
\end{minipage}

\begin{minipage}{.5\textwidth}
\centering
\vspace{-0.05cm}
{\tiny 
\begin{tabular}{|llll|}
\hline
&&&\\[-0.5em] 
$G.f$ & even & $\g^\natural= \bigoplus_{i} \g_i^\natural$ &  
 $k_i^\natural$  \\[-0.5em] 
&&&\\
\hline
&&&\\[-0.5em] 
$E_7$ & yes &  $\{0\}$ &   \\[0.5em] 
$E_7(a_1)$ & yes &  $\{0\}$ &   \\[0.5em] 
$E_7(a_2)$ & yes &  $\{0\}$ &    \\[0.5em] 
$E_7(a_3)$ & yes&  $\{0\}$ &   \\[0.5em] 
$E_6$ & yes & $A_1$ &  $k_1^\natural=3k+48$ \\[0.5em] 
$E_6(a_1)$ & yes &  $\C$  &  $k_0^\natural=k+16$ \\[0.5em] 
$D_6$ & no & $A_1$ & $k_1^\natural=k+ \frac{29}{2}$  \\[0.5em] 
$E_7(a_4)$ & yes &  $\{0\}$ &    \\[0.5em] 
$D_6(a_1)$  & no &  $A_1$ &   $k_1^\natural=k+14$  \\[0.5em] 
$D_5+A_1$  & no &  $A_1$ &  $k_1^\natural=2k+30$  \\[0.5em] 
$A_6$  & yes &   $A_1$ &  $k_1^\natural=7k+108$  \\[0.5em]
$E_7(a_5)$ & yes &  $\{0\}$ &    \\[0.5em] 
$D_5$  & yes &  $A_1\times A_1$ &  $k_1^\natural=k+14$ \\
&&& {$k_2^\natural=2k+30$} \\[0.5em] 
$E_6(a_3)$ & yes & $A_1$ &  $k_1^\natural=3k+44$ \\[0.5em] 
$D_6(a_2)$ & yes & $A_1$ &  $k_1^\natural=k+\frac{27}{2}$ \\[0.5em]
$D_5(a_1)+A_1$ &  {yes} &  {$A_1$} &  {$k_1^\natural=8 k+116$} \\[0.5em]  
$A_5+A_1$ &  no &  {$A_1$} &  {$k_1^\natural=3 k+41$} \\[0.5em] 
$(A_5)'$ &  {no} &  $A_1 \times A_1$ &  {$k_1^\natural=k+\frac{27}{2}$} \\
&&&  {$k_2^\natural=3k+44$} \\[0.5em] 
$A_4+A_2$ & yes &  $A_1$ &  $k_1^\natural= 15 k+216$  \\[0.5em] 
$D_5(a_1)$ &  no &  $\C \times A_1$ & $k_0^\natural=2k+29$ \\[0.5em]
&&& $k_1^\natural=k+13$  \\[0.5em] 
$A_4+A_1$ &  no & $\C^2$ &  
$\phi^\natural_0 \not= 0$ for admissible $k$\\[0.5em]
$D_4+A_1$ &  no &  {$B_2$} &  {$k_1^\natural=k+\frac{25}{2}$} \\[0.5em]
$(A_5)''$ & yes & $G_2$  &  $k_1^\natural=k+12$  \\[0.5em]
$A_3+A_2+A_1$ & yes  &  $A_1$ &  $k_1^\natural=24 k+320$ \\[0.5em] 
$A_4$ & yes & $\C\times A_2$ &  $k_0^\natural=k+\frac{72}{5}$ \\
& && $k_1^\natural=k+12$  \\[0.5em] 
$A_3+A_2$ & no & $\C\times A_1$ &  $k_0^\natural=k+\frac{40}{3}$ \\
& && $k_1^\natural=k+12$  \\[0.5em] 
\hline
\end{tabular} \\[1em] 
\captionof{table}{\footnotesize{Centralisers of $\sl_2$-triples in type  $E_7$}} 
\label{Tab:Data-E7-a}
}
\end{minipage}\hspace{1cm} 
\begin{minipage}{.5\textwidth} 
\centering
{\tiny 
\begin{tabular}{|llll|}
\hline
&&&\\[-0.5em] 
$G.f$ & even & $\g^\natural= \bigoplus_{i} \g_i^\natural$ &  
 $k_i^\natural$  \\[-0.5em] 
&&&\\
\hline
&&&\\[-0.5em] 
$D_4(a_1)+A_1$ &  no &  $A_1 \times A_1$ &  $k_1^\natural=k+12$ \\
& && $k_2^\natural=k+12$\\[0.5em] 
$D_4$ &  yes &  $C_3$ &  $k_1^\natural=k+12$ \\[0.5em]
$A_3+2 A_1$ &  no &  $A_1 \times A_1$ &  $k_1^\natural=k+\frac{23}{2}$ \\
& && $k_2^\natural=2k+24$\\[0.5em] 
$D_4(a_1)$&  {yes} &  {$A_1 \times A_1 \times A_1$} &  $k_1^\natural=k+12$ \\
& && $k_2^\natural=k+12$  \\
& && $k_3^\natural=k+12$  \\[0.5em] 
$(A_3+A_1)^\prime$ &  no &  $A_1 \times A_1\times A_1$ &  $k_1^\natural=k+\frac{23}{2}$ \\
&&&  $k_2^\natural=k+12$ \\
&&&  $k_3^\natural=2k+24$ \\[0.5em] 
$2 A_2+A_1$ & no  &  $A_1 \times A_1$ & $k_1^\natural=3k+36$  \\ 
&&& {$k_2^\natural=3k+35$} \\[0.5em] 
$(A_3+A_1)^{\prime\prime}$ &  no &  $B_3$ &  $k_1^\natural=k+10$ \\[0.5em] 
$A_2+3 A_1$ & yes & $G_2$ & $k_1^\natural=2 k+22$ \\[0.5em] 
$2 A_2$ & yes & $A_1\times G_2$ & $k_1^\natural=3k+36$ \\
&&&  $k_2^\natural=k+10$ \\[0.5em]  
$A_3$ &  no & $A_1 \times B_3$ &  $k_1^\natural=k+12$ \\[0.5em] 
&&& $k_2^\natural=k+10$\\[0.5em]
$A_2+2A_1$ &  no &  $A_1 \times A_1\times A_1$ &  $k_1^\natural=k+10$ \\
&&&  $k_2^\natural=2k+22$ \\
&&&  $k_3^\natural=6k+66$ \\[0.5em]
$A_2+A_1$ &  no &  $\C\times A_3$ &  $k_0^\natural=k+11$ \\
&&&  $k_1^\natural=k+9$ \\[0.5em] 
$4 A_1$ & no &  $C_3$ &  $k_1^\natural=k+\frac{17}{2}$  \\[0.5em] 
$A_2$ & yes & $A_5$ & $k_1^\natural=k+8$\\[0.5em]
$(3 A_1)'$ & no & $A_1\times C_3$ & $k_1^\natural=k +\frac{17}{2}$ \\
&& & $k_2^\natural=k+8$  \\[0.5em] 
$(3 A_1)''$ & yes & $F_4$ &  $k_1^\natural=k+6$ \\[0.5em] 
$2 A_1$ & no  & $A_1 \times B_4$ &  $k_1^\natural=k+8$\\
&&&  {$k_2^\natural=k+6$} \\[0.5em] 
$A_1$ & no & $D_6$ & $k_1^\natural=k+4$ \\[0.5em] 
\hline
\end{tabular} \\[1em] 
\captionof{table}{\footnotesize{Centralisers of $\sl_2$-triples in type  $E_7$ (continued)}} 
\label{Tab:Data-E7}
}
\end{minipage}

\begin{minipage}{.5\textwidth}
\centering
\vspace{-0.5cm}
{\tiny 
\begin{tabular}{|llll|}
\hline
&&&\\[-0.5em] 
$G.f$ & even & $\g^\natural= \bigoplus_{i} \g_i^\natural$ &  
 $k_i^\natural$  \\[-0.5em] 
&&&\\
\hline
&&&\\[-0.5em] 
$E_8$  & yes &  $\{0\}$ &   \\[0.5em] 
$E_8(a_1)$ & yes &  $\{0\}$ &   \\[0.5em] 
$E_8(a_2)$ & yes  & $\{0\}$ &   \\[0.5em] 
$E_8(a_3)$  & yes  &  $\{0\}$ &  \\[0.5em] 
$E_8(a_4)$ & yes  &  $\{0\}$ &   \\[0.5em] 
$E_7$   & no &  $A_1$ & $k_1^\natural=k+\frac{53}{2}$  \\[0.5em] 
$E_8(b_4)$  & yes  &  $\{0\}$ &  \\[0.5em] 
$E_8(a_5)$  & yes &  $\{0\}$ & \\[0.5em] 
$E_7(a_1)$ & no & $A_1$ & $k_1^\natural=k+26$  \\[0.5em] 
$E_8(b_5)$ & yes & $\{0\}$ &  \\[0.5em] 
$D_7$& no & $A_1$ & $k_1^\natural=2 k+\frac{107}{2}$ \\[0.5em]
$E_8(a_6)$ & yes   &  $\{0\}$ & \\[0.5em] 
$E_7(a_2)$ & no  & $A_1$ & $k_1^\natural=k +\frac{51}{2}$ \\[0.5em] 
$E_6+A_1$ & no  & $A_1$ & $k_1^\natural=3k +77$ \\[0.5em]  
$D_7(a_1)$ & no & $\C$ & $k_0^\natural=4k +106$ \\[0.5em] 
$E_8(b_6)$  & yes  &  $\{0\}$ & \\[0.5em] 
$E_7(a_3)$ & no & $A_1$ & $k_1^\natural=k+25$ \\[0.5em] 
$E_6(a_1)+A_1$ & no  & $\C$ & $k_0^\natural=2k+51$ \\[0.5em] 
$A_7$& no &  $A_1$ & $k_1^\natural=4k+\frac{209}{2}$   \\[0.5em] 
$D_7(a_2)$ & no  & $\C$ & $k_0^\natural=k+26$ \\[0.5em] 
$E_6$& yes & $G_2$ & $k_1^\natural=k+24$ \\[0.5em] 
$D_6$ & no & $B_2$ & $k_1^\natural=k+\frac{49}{2}$  \\[0.5em] 
$D_5+A_2$ & yes   & $\C$ & $k_0^\natural=3k+76$ \\[0.5em] 
$E_6(a_1)$ & yes  & $A_2$ & $k_1^\natural=k+24$   \\[0.5em] 
$E_7(a_4)$ & no & $A_1$ & $k_1^\natural=k+24$ \\[0.5em] 
$A_6+A_1$ & no   &  $A_1$ & $k_1^\natural=7k+180$   \\[0.5em] 
$D_6(a_1)$ &no    & $A_1\times A_1$ & $k_1^\natural=k+24$ \\
&&& $k_2^\natural=k+24$ \\[0.5em] 
$A_6$ & yes   & $A_1 \times A_1$ & ${k_1^\natural=k+24}$\\
&&& $k_2^\natural=7 k+180$ \\[0.5em] 
$E_8(a_7)$ & yes   &  $\{0\}$ &   \\[0.5em] 
$D_5+A_1$ & no   & $A_1\times A_1$ & $k_1^\natural=k+\frac{47}{2}$ \\
&&& $k_2^\natural=2k+48$ \\[0.5em] 
$E_7(a_5)$ & no  & $A_1$ & ${k_1^\natural=k+\frac{47}{2}}$  \\[0.5em] 
$E_6(a_3)+ A_1$ & no  & $A_1$ & $k_1^\natural=3k +71$ \\[0.5em] 
$D_6(a_2)$ & no & $A_1\times A_1$ & $k_1^\natural=k+\frac{47}{2}$ \\
&&& $k_2^\natural=k+\frac{47}{2}$ \\[0.5em] 
$D_5(a_1)+A_2$ & no  & $A_1$ & $k_1^\natural=6k +\frac{285}{2}$ \\[0.5em] 
$A_5+A_1$ & no & $A_1\times A_1$ & $k_1^\natural=k+\frac{47}{2}$  \\
&&& $k_2^\natural=3k+71$ \\[0.5em] 
$A_4+A_3$ & no &  $A_1$ &  $k_1^\natural=10k+238$ \\[0.5em] 
$D_5$ & yes   & $B_3$ & $k_1^\natural=k+22$ \\[0.5em]  
$E_6(a_3)$ & yes   & $G_2$ & $k_1^\natural=k+22$  \\[0.5em] 
$D_4+A_2$ & yes & $A_2$ & $k_1^\natural=2k+46$ \\[0.5em] 
\hline
\end{tabular} \\[1em] 
\captionof{table}{\footnotesize{Centralisers of $\sl_2$-triples in type $E_8$}} 
\label{Tab:Data-E8-a}
}
\end{minipage}\hspace{1cm} 
\begin{minipage}{.5\textwidth} 
\centering
{\tiny 
\begin{tabular}{|llll|}
\hline
&&&\\[-0.5em] 
$G.f$ & even & $\g^\natural= \bigoplus_{i} \g_i^\natural$ &  
 $k_i^\natural$  \\[-0.5em] 
&&&\\
\hline
&&&\\[-0.5em] 
$A_4+ A_2+A_1$ & no & $A_1$ & $k_1^\natural=15k+350$ \\[0.5em]  
$D_5(a_1)+A_1$ & no & $A_1 \times A_1$ & $k_1^\natural=k+22$ \\
&&& $k_2^\natural=8k+184$ \\[0.5em] 
$A_5$ & no  & $A_1\times G_2$ & $k_1^\natural=k+\frac{47}{2}$  \\
&&& $k_2^\natural=k+22$ \\[0.5em] 
$A_4+ A_2$ & yes  & $A_1\times A_1$ & $k_1^\natural=15k+350$ \\
&&& $k_2^\natural=k+22$ \\[0.5em] 
$A_4+2 A_1$ & no  & $\C\times A_1$ & $k_0^\natural=k+\frac{45}{2}$ \\
&&& $k_1^\natural=2k+45$ \\[0.5em] 
$D_5(a_1)$ & no  & $A_3$ & $k_1^\natural=k+21$ \\[0.5em] 
$2 A_3$ & no &  $B_2$ & $k_1^\natural = 2 k + \frac{89}{2}$  \\[0.5em] 
$A_4+A_1$ & no & $\C\times A_2$ & $k_0^\natural=k+\frac{45}{2}$ \\
&&& $k_1^\natural=k+21$ \\[0.5em] 
$D_4(a_1)+A_2$ & yes  & $A_2$ & $k_1^\natural=6k+ 132$  \\[0.5em] 
$D_4+A_1$ & no & $C_3$ & $k_1^\natural=k+\frac{41}{2}$ \\[0.5em] 
$A_3+ A_2+A_1$ & no  & $A_1\times A_1$ & $k_1^\natural=k+\frac{41}{2}$ \\
&&& $k_2^\natural=24k+528$ \\[0.5em] 
$A_4$ & yes  & $A_4$ & $k_1^\natural=k+20$ \\[0.5em] 
$A_3+ A_2$ & no  &  $\C \times B_2$ & $k_0^\natural=k+22$ \\
&&& $k_1^\natural=k+20$ \\[0.5em] 
$D_4(a_1)+ A_1$ & no  & $A_1\times A_1\times A_1$ & $k_1^\natural=k+20$ \\
&&& {$k_2^\natural=k+20$} \\
&&& {$k_3^\natural=k+20$}  \\[0.5em] 
$A_3+2 A_1$ & no  & $A_1\times B_2$ & $k_1^\natural=2k+40$ \\
&&& $k_2^\natural=k+\frac{39}{2}$ \\[0.5em] 
$2 A_2+2 A_1$  & no  &  $B_2$ & $k_1^\natural = 3 k + 59$  \\[0.5em] 
$D_4$ &  yes  & $F_4$ & $k_1^\natural=k+18$\\[0.5em]   
$D_4(a_1)$ & no  & $D_4$ & $k_1^\natural=k+18$ \\[0.5em]
$A_3+A_1$ & no  & $A_1 \times B_3$ & $k_1^\natural=k+\frac{39}{2}$ \\ 
&&& $k_2^\natural=k+18$ \\[0.5em] 
$2 A_2+A_1$ & no  &  $A_1\times G_2$  & $k_1^\natural=3k +59$ \\ 
&&& {$k_2^\natural=k +8$} \\[0.5em] 
$2 A_2$ & yes  & $G_2 \times G_2$ & $k_1^\natural=k+18$ \\
&&& {$k_2^\natural=k+18$} \\[0.5em] 
$A_2+3 A_1$ & no  & $A_1 \times G_2$ & $k_1^\natural=k+\frac{35}{2}$ \\
&&& $k_2^\natural=2k+36$ \\[0.5em] 
$A_3$ & no  & $B_5$ & $k_1^\natural=k +16$ \\[0.5em] 
$A_2+2 A_1$ & no  & $A_1 \times B_3$ & $k_1^\natural=6k+108$ \\
&&& $k_2^\natural=k+16$ \\[0.5em] 
$A_2+A_1$ & no  & $A_5$ & $k_1^\natural=k +15$ \\[0.5em] 
$4 A_1$ & no &  $C_4$ & $k_1^\natural=k +\frac{29}{2}$  \\[0.5em] 
$A_2$ & yes  & $E_6$ & $k_1^\natural=k +12$ \\[0.5em] 
$3 A_1$ & no   & $A_1\times F_4$ & $k_1^\natural=k +\frac{29}{2}$ \\
&&& $k_2^\natural=k+12$ \\[0.5em] 
$2 A_1$ & no  & $B_6$ & $k_1^\natural=k +10$ \\[0.5em] 
$A_1$ & no  & $E_7$ & $k_1^\natural=k +6$ \\[0.5em]
\hline
\end{tabular} \\[1em] 
\captionof{table}{\footnotesize{Centralisers of $\sl_2$-triples in type $E_8$ (continued)}} 
\label{Tab:Data-E8}
}
\end{minipage}

\newcommand{\etalchar}[1]{$^{#1}$}


\begin{thebibliography}{BLL{\etalchar{+}}}

\bibitem{AbeBuhDon04}
Toshiyuki Abe, Geoffrey Buhl, and Chongying Dong.
\newblock Rationality, regularity, and {$C_2$}-cofiniteness.
\newblock {\em Trans. Amer. Math. Soc.}, 356(8):3391--3402 (electronic), 2004.

\bibitem{AdaKacMos17}
Dra\v{z}en Adamovi\'{c}, Victor~G. Kac, Pierluigi M\"{o}seneder~Frajria, Paolo
  Papi, and Ozren Per\v{s}e.
\newblock Conformal embeddings of affine vertex algebras in minimal
  {$W$}-algebras {II}: decompositions.
\newblock {\em Jpn. J. Math.}, 12(2):261--315, 2017.

\bibitem{AdaKacMos18}
Dra\v{z}en Adamovi\'{c}, Victor~G. Kac, Pierluigi M\"{o}seneder~Frajria, Paolo
  Papi, and Ozren Per\v{s}e.
\newblock Conformal embeddings of affine vertex algebras in minimal
  {$W$}-algebras {I}: structural results.
\newblock {\em J. Algebra}, 500:117--152, 2018.


\bibitem{Adamovic-et-al_collapsing} 
Dra{\v{z}}en Adamovi{\'c}, Victor~G. Kac, Pierluigi M\"{o}seneder~Frajria, Paolo Papi and 
Ozren Per{\v{s}}e. 
\newblock An application of collapsing levels to the representation theory of affine vertex algebras. 
\newblock {\em Int. Math. Res. Not.}, 13:4103--4143, 2020.

\bibitem{Adamovic-et-al_New-collapsing}  
Dra{\v{z}}en Adamovi{\'c}, Pierluigi M\"{o}seneder~Frajria and Paolo Papi. 
\newblock New approaches for studying conformal embeddings and collapsing levels for W--algebras
\newblock arXiv:2203.08497 [math.RT]. 

\bibitem{Allegra19} 
Francesco Alberto Allegra. 
\newblock $W$-algebras in type A and the Arakawa-Moreau conjecture
\newblock Thesis (Ph.D.). Sapienza--University of Rome, 2019. 

\bibitem{Ara12}
Tomoyuki Arakawa.
\newblock A remark on the {$C_2$} cofiniteness condition on vertex algebras.
\newblock {\em Math. Z.}, 270(1-2):559--575, 2012.


\bibitem{Arakawa2016}
Tomoyuki Arakawa.
\newblock Rationality of admissible affine vertex algebras in the category $\mc{O}$. 
\newblock {\em Duke Math. J}, 165(1), 67--93, 2016.  

\bibitem{A2012Dec}
Tomoyuki Arakawa.
\newblock Rationality of {W}-algebras: principal nilpotent cases.
\newblock {\em Ann. Math.}, 182(2):565--694, 2015.

\bibitem{Arakawa15a} Tomoyuki Arakawa.
\newblock Associated varieties of modules over {K}ac-{M}oody algebras and
  {$C_2$}-cofiniteness of {$W$}-algebras.
	       \newblock {{\em Int. Math. Res. Not.}
	       2015(22): 11605-11666, 2015.}
	       
\bibitem{A.Higgs}
Tomoyuki Arakawa.
\newblock Associated varieties and Higgs branches (a survey).
\newblock {\em Contemp. Math.} 2018(711), 37--44.

\bibitem{AraFutRam17}
Tomoyuki Arakawa, Vyacheslav Futorny, and Luis~Enrique Ramirez.
\newblock Weight representations of admissible affine vertex algebras.
\newblock {\em Comm. Math. Phys.}, 353(3):1151--1178, 2017.

\bibitem{AEkeren19admissible}
Tomoyuki Arakawa and Jethro van Ekeren.
\newblock
Modularity of relatively rational vertex algebras and fusion rules of principal affine W-algebras. 
\newblock {\em Comm. Math. Phys.} 370 (2019), no. 1, 205--247.


\bibitem{AEkeren19}
Tomoyuki Arakawa and Jethro van Ekeren.
\newblock Rationality and fusion rules of exceptional {W}-algebras.
\newblock  {\em J. Eur. Math. Soc. (JEMS)} published online.

\bibitem{AEM}
Tomoyuki Arakawa, Jethro van Ekeren and Anne Moreau.
\newblock  On the nilpotent orbits arising from admissible affine vertex algebras.
\newblock arXiv:2010.08429 [math.RT].

\bibitem{AraFie08}
Tomoyuki Arakawa and Peter Fiebig.
\newblock On the restricted {V}erma modules at the critical level.
\newblock {\em Trans. Amer. Math. Soc.}, 364(9):4683--4712, 2012.

\bibitem{AraKaw18}
Tomoyuki Arakawa and Kazuya Kawasetsu.
\newblock {\em Quasi-lisse vertex algebras and modular linear differential
  equations}, V. G. Kac, V. L. Popov (eds.), Lie Groups, Geometry, and Representation 
  Theory, A Tribute to the Life and Work of Bertram Kostant, Progr. Math., 
  326, Birkhauser, 2018.

\bibitem{ArakawaKuwabaraMalikov}
Tomoyuki Arakawa, Toshiro Kuwabara, and Fyodor Malikov.
\newblock Localization of {A}ffine {W}-{A}lgebras.
\newblock {\em Comm. Math. Phys.}, 335(1):143--182, 2015.

\bibitem{AM15}
Tomoyuki Arakawa and Anne Moreau.
\newblock {J}oseph ideals and lisse minimal {W}-algebras.
\newblock {\em J. Inst. Math. Jussieu}, {17(2):397--417, 2018.}


\bibitem{AraMor17}
Tomoyuki Arakawa and Anne Moreau.
\newblock Sheets and associated varieties of affine vertex algebras.
\newblock {\em Adv. Math.}, 320:157--209, 2017.

\bibitem{AraMor16b}
Tomoyuki Arakawa and Anne Moreau.
\newblock On the irreducibility of associated varieties of {W}-algebras.
\newblock {\em The special issue of J. Algebra in Honor of Efim Zelmanov on
  occasion of his 60th anniversary}, {500:542--568, 2018.}



\bibitem{ArgDou95}
Philip~C. Argyres and Michael~R. Douglas.
\newblock New phenomena in {${\rm SU}(3)$} supersymmetric gauge theory.
\newblock {\em Nuclear Phys. B}, 448(1-2):93--126, 1995.


\bibitem{ArgPleSei96}
Philip~C. Argyres, M.~Ronen Plesser, Nathan Seiberg, and Edward Witten.
\newblock New {${\mc{N}}=2$} superconformal field theories in four dimensions.
\newblock {\em Nuclear Phys. B}, 461(1-2):71--84, 1996.

\bibitem{Bea00}
Arnaud Beauville.
\newblock Symplectic singularities.
\newblock {\em Invent. Math.}, 139(3):541--549, 2000.

\bibitem{BeeLemLie15}
Christopher Beem, Madalena Lemos, Pedro Liendo, Wolfger Peelaers, Leonardo
  Rastelli, and Balt~C. van Rees.
\newblock Infinite chiral symmetry in four dimensions.
\newblock {\em Comm. Math. Phys.}, 336(3):1359--1433, 2015.

\bibitem{BeeRas}
Christopher Beem and Leonardo Rastelli.
\newblock Vertex operator algebras, {H}iggs branches, and modular differential
  equations.
\newblock {\em J. High Energy Phys.}, 2018(8):1--72, 2018.


\bibitem{dBT1993}
Jan de Boer and Tjark Tjin.
\newblock The relation between quantum {$W$}-algebras and {L}ie algebras.
\newblock {\em Comm. Math. Phys.}, 160(2):317--332, 1994.


\bibitem{Bou.Sch.review}
Peter Bouwknegt and Kareljan Schoutens.
\newblock {$W$}-symmetry in conformal field theory.
\newblock {\em Phys. Rept.}, 223:183--276, 1993.

\bibitem{Brieskorn}
Egbert Brieskorn, 
\newblock  Singular elements of semi-simple algebraic groups. 
\newblock  in: Actes du Congr\`{e}s International des Math\'{e}maticiens, Tome 2, 
Nice, 1970, Gauthier-Villars, Paris, 1971, 279--284.



\bibitem{Brundan-Goodwin}
Jonathan Brundan and Simon Goodwin. 
\newblock Good grading polytopes. 
\newblock {\em Proc. London Math. Soc.} 94(3):155--180, 2007. 



\bibitem{DeSole-Kac} Alberto De Sole and Victor Kac. 
\newblock Finite vs affine $W$-algebras.  
\newblock {\em Jpn. J. Math.} 1(1):137--261, 2006.

\bibitem{CMa} David Collingwood and William M.~McGovern. 
\newblock Nilpotent orbits in semisimple {L}ie algebras. 
\newblock Van Nostrand Reinhold Co. New York, 65, 1993. 



\bibitem{CreutzigLinshaw}
Thomas Creutzig and Andrew R. Linshaw.
\newblock  Trialities of orthosymplectic $\W$-algebras.
\newblock 	arXiv:2102.10224 [math.RT].

\bibitem{De-Kac06}
Alberto De~Sole and Victor~G. Kac.
\newblock Finite vs affine {$W$}-algebras.
\newblock {\em Japan. J. Math.}, 1(1):137--261, 2006.

\bibitem{DonJiaXu13}
Chongying Dong, Xiangyu Jiao, and Feng Xu.
\newblock Quantum dimensions and quantum {G}alois theory.
\newblock {\em Trans. Amer. Math. Soc.}, 365(12):6441--6469, 2013.


\bibitem{DonMas04}
Chongying Dong and Geoffrey Mason.
\newblock Rational vertex operator algebras and the effective central charge.
\newblock {\em Int. Math. Res. Not.}, (56):2989--3008, 2004.

\bibitem{DonMas06}
Chongying Dong and Geoffrey Mason.
\newblock Integrability of {$C\sb 2$}-cofinite vertex operator algebras.
\newblock {\em Int. Math. Res. Not.}, pages Art. ID 80468, 15, 2006.


\bibitem{ElashviliKac}
Alexander G. Elashvili and Victor G. Kac. 
\newblock Classification of good gradings of simple Lie algebras. 
\newblock Lie groupsand invariant theory (ed. E. B. Vinberg), American Mathematical Society Translations (2) 213 (AMS,
Providence, RI, 2005) 85--104. 

\bibitem{EtiGelNik15}
Pavel Etingof, Shlomo Gelaki, Dmitri Nikshych, and Victor Ostrik.
\newblock {\em Tensor categories}, volume 205 of {\em Mathematical Surveys and
  Monographs}.
\newblock American Mathematical Society, Providence, RI, 2015.

\bibitem{JustineFasquel}
Justine Fasquel.
\newblock Rationality of the exceptional $\W$-algebras $\W_k(\sp_4,f_{subreg})$  associated with subregular nilpotent elements of 
$\sp_4$. 
\newblock {\em Commun.\ Math.\ Phys.}, 390(1), 33--65, 2022. 


\bibitem{Fat.Zam.1987} Vladimir Fateev and Alexander Zamolodchikov.
\newblock Conformal quantum field theory models in two dimensions having {$\mathbb{Z}_3$} symmetry.
\newblock {\em Nucl. Phys. B}, 280:644--660, 1987.


\bibitem{FeiFre90} Boris Feigin and Edward Frenkel.
\newblock Quantization of the {D}rinfel\cprime d-{S}okolov reduction.
\newblock {\em Phys. Lett. B}, 246(1-2):75--81, 1990.

  
  \bibitem{FeuFuc84}
B.~L. Fe{\u \i}gin and D.~B. Fuchs.
\newblock Verma modules over the {V}irasoro algebra.
\newblock In {\em Topology ({L}eningrad, 1982)}, volume 1060 of {\em Lecture
  Notes in Math.}, pages 230--245. Springer, Berlin, 1984.


\bibitem{FBZ.book}
Edward Frenkel and David Ben-Zvi.
\newblock {\em Vertex algebras and algebraic curves: Second Edition}, vol. 88 of {\em Mathematical Surveys and
  Monographs}.
\newblock American Mathematical Society, Providence, RI, 2004.

\bibitem{FKW92}
Edward Frenkel, Victor Kac, and Minoru Wakimoto.
\newblock Characters and fusion rules for {$W$}-algebras via quantized
  {D}rinfel$'$d-{S}okolov reduction.
\newblock {\em Comm. Math. Phys.}, 147(2):295--328, 1992.

\bibitem{Frenkel:1993aa}
Igor~B. Frenkel, Yi-Zhi Huang, and James Lepowsky.
\newblock On axiomatic approaches to vertex operator algebras and modules.
\newblock {\em Mem. Amer. Math. Soc.}, 104(494):viii+64, 1993.


\bibitem{FuJutLev17}
Baohua Fu, Daniel Juteau, Paul Levy, and Eric Sommers.
\newblock Generic singularities of nilpotent orbit closures.
\newblock {\em Adv. Math.}, 305:1--77, 2017.


\bibitem{Gin08}
Victor Ginzburg.
\newblock Harish-{C}handra bimodules for quantized {S}lodowy slices.
\newblock {\em Represent. Theory}, 13:236--271, 2009.


\bibitem{HuaLep92}
Yi-Zhi Huang and James Lepowsky.
\newblock Toward a theory of tensor products for representations of a vertex
  operator algebra.
\newblock In {\em Proceedings of the {XX}th {I}nternational {C}onference on
  {D}ifferential {G}eometric {M}ethods in {T}heoretical {P}hysics, {V}ol. 1, 2
  ({N}ew {Y}ork, 1991)}, pages 344--354. World Sci. Publ., River Edge, NJ,
  1992.
  
  \bibitem{Hua08rigidity}
Yi-Zhi Huang.
\newblock Rigidity and modularity of vertex tensor categories.
\newblock {\em Commun. Contemp. Math.}, 10(suppl. 1):871--911, 2008.

\bibitem{HuaLep94}
Yi-Zhi Huang and James Lepowsky.
\newblock Tensor products of modules for a vertex operator algebra and vertex
  tensor categories.
\newblock In {\em Lie theory and geometry}, volume 123 of {\em Progr. Math.},
  pages 349--383. Birkh\"{a}user Boston, Boston, MA, 1994.

\bibitem{HuaLep95}
Yi-Zhi Huang and James Lepowsky.
\newblock A theory of tensor products for module categories for a vertex
  operator algebra. {I}, {II}.
\newblock {\em Selecta Math. (N.S.)}, 1(4):699--756, 757--786, 1995.

\bibitem{HuaLep95III}
Yi-Zhi Huang and James Lepowsky.
\newblock A theory of tensor products for module categories for a vertex
  operator algebra. {III}.
\newblock {\em J. Pure Appl. Algebra}, 100(1-3):141--171, 1995.


\bibitem{Kac90} Victor Kac. 
\newblock  Infinite-dimensional Lie algebras. 
\newblock Third edition. Cambridge University Press, Cambridge, 1990. 

\bibitem{KacKaz79}
Victor~G. Kac and David~A. Kazhdan.
\newblock Structure of representations with highest weight of
  infinite-dimensional {L}ie algebras.
\newblock {\em Adv. in Math.}, 34(1):97--108, 1979.


\bibitem{KacPeterson84} Victor Kac and Dale H. Peterson. 
\newblock  Infinite-dimensional Lie algebras theta functions and modular forms. 
\newblock Adv. in Math. 53:125--264, 1984. 

\bibitem{KacRoaWak03}
Victor Kac, Shi-Shyr Roan, and Minoru Wakimoto.
\newblock Quantum reduction for affine superalgebras.
\newblock {\em Comm. Math. Phys.}, 241(2-3):307--342, 2003.

\bibitem{KacWak88}
Victor~G. Kac and Minoru Wakimoto.
\newblock Modular invariant representations of infinite-dimensional {L}ie
  algebras and superalgebras.
\newblock {\em Proc. Nat. Acad. Sci. U.S.A.}, 85(14):4956--4960, 1988.


\bibitem{KacWak89} Victor Kac and Minoru Wakimoto.
\newblock Classification of modular invariant representations of affine
  algebras.
\newblock In {\em Infinite-dimensional Lie algebras and groups
  (Luminy-Marseille, 1988)}, volume~7 of {\em Adv. Ser. Math. Phys.}, pages
  138--177. World Sci. Publ., Teaneck, NJ, 1989.

\bibitem{KacWak03} Victor Kac and Minoru Wakimoto.
\newblock Quantum reduction and representation theory of superconformal algebras. 
\newblock {\em Adv. Math.} 185(2):400--458, 2004.  

\bibitem{KacWak08} Victor Kac and Minoru Wakimoto.
\newblock On rationality of {$W$}-algebras.
\newblock {\em Transform. Groups}, 13(3-4):671--713, 2008.

\bibitem{KacWak17}
Victor~G. Kac and Minoru Wakimoto.
\newblock A remark on boundary level admissible representations.
\newblock {\em C. R. Math. Acad. Sci. Paris}, 355(2):128--132, 2017.

\bibitem{Kaledin06} Dmitry Kaledin. 
\newblock Symplectic singularities from the Poisson point of view. 
\newblock {\em J. Reine Angew. Math.}, 600:135--156, 2006.


\bibitem{KraftProcesi79} Hanspeter Kraft and Claudio Procesi.  
\newblock Closures of conjugacy classes of matrices are normal. 
\newblock {\em Invent. Math.} 53(3):227--247, 1979.

\bibitem{KraftProcesi81} Hanspeter Kraft and Claudio Procesi. 
\newblock Minimal singularities in $GL_n$. 
\newblock {\em Invent. Math.} 62(3):503--515, 1981.
 
\bibitem{KraftProcesi82} Hanspeter Kraft and  Claudio Procesi. 
\newblock On the geometry of conjugacy classes in classical groups. 
\newblock {\em Comment. Math. Helv.} 57(4):539--602, 1982. 



\bibitem{LepLi04}
James Lepowsky and Haisheng Li.
\newblock {\em Introduction to vertex operator algebras and their
  representations}, volume 227 of {\em Progress in Mathematics}.
\newblock Birkh\"{a}user Boston, Inc., Boston, MA, 2004.




\bibitem{Li17}
Yiqiang Li.  
\newblock Quiver varieties and symmetric pairs. 
\newblock {\em Representation Theory}, 
An Electronic Journal of the American Mathematical Society, 
23:1--56, 2019. 


\bibitem{Los07}
Ivan Losev.
\newblock Quantized symplectic actions and {$W$}-algebras.
\newblock {\em J. Amer. Math. Soc.}, 23(1):35--59, 2010.

\bibitem{Los11}
Ivan Losev.
\newblock Finite-dimensional representations of {$W$}-algebras.
\newblock {\em Duke Math. J.}, 159(1):99--143, 2011.

\bibitem{McRae}
Robert McRae. 
\newblock On rationality for $C_2$-cofinite vertex operator algebras. 
\newblock Preprint arXiv:2108.01898 [math.QA].



\bibitem{MooPia95}
Robert~V. Moody and Arturo Pianzola.
\newblock {\em Lie algebras with triangular decompositions}.
\newblock Canadian Mathematical Society Series of Monographs and Advanced
  Texts. John Wiley \& Sons Inc., New York, 1995.
\newblock A Wiley-Interscience Publication.



\bibitem{Namikawa04} Yoshinori Namikawa. 
\newblock Birational geometry of symplectic resolutions of nilpotent orbits, 
Moduli spaces and arithmetic geometry, 75--116. 
\newblock {\em Adv. Stud. Pure Math.}, 45, Math. Soc. Japan, Tokyo, 2006


\bibitem{Panyushev} Dmitri Panyushev. 
\newblock Rationality of singularities and the Gorenstein property for nilpotent orbits. 
\newblock {\em Functional Analysis and Its Applications}, 25(3):225--226, 
1991. 

\bibitem{Pre02}
Alexander Premet.
\newblock Special transverse slices and their enveloping algebras.
\newblock {\em Adv. Math.}, 170(1):1--55, 2002.
\newblock With an appendix by Serge Skryabin.


\bibitem{SonXieYan17}
Jaewon Song, Dan Xie, and Wenbin Yan.
\newblock Vertex operator algebras of {A}rgyres-{D}ouglas theories from
  {M}5-branes.
\newblock {\em J. High Energy Phys.}, (12):123, front matter+35, 2017.

\bibitem{Slo80}
Peter Slodowy.
\newblock {Simple singularities and simple algebraic groups}, volume 815 of
  {\em Lecture Notes in Mathematics}.
\newblock Springer, Berlin, 1980.

\bibitem{Spaltenstein82} 
Nicolas Spaltenstein. 
\newblock Classes Unipotentes et Sous-Groupes de Borel, volume 946 of 
{\em Lecture Notes in Mathematics}. 
\newblock Springer-Verlag, New York, 1982.



\bibitem{Wan93}
Weiqiang Wang.
\newblock Rationality of {V}irasoro vertex operator algebras.
\newblock {\em Internat. Math. Res. Notices}, (7):197--211, 1993.

\bibitem{WanXie}
Yifan Wang and Dan Xie.
\newblock Codimension-two defects and {A}rgyres-{D}ouglas theories from
  outer-automorphism twist in 6d $(2,0)$ theories.
\newblock {\em Phys.~Rev.~D} 100, 025001 

\bibitem{XieYan2}
Dan Xie and Wenbin Yan.
\newblock W algebra, Cosets and VOA for 4d N = 2 SCFT from M5 branes.
\newblock 
{\em J. High~Energ.~Phys.} 2021, 76 (2021). 




\bibitem{Dan}
Dan Xie, Wenbin Yan, and Shing-Tung Yau.
\newblock Chiral algebra of {A}rgyres-{D}ouglas theory from {M}5 brane.
\newblock {\em preprint}.
\newblock arXiv:1604.02155[hep-th].


\bibitem{Zhu96}
Yongchang Zhu.
\newblock Modular invariance of characters of vertex operator algebras.
\newblock {\em J. Amer. Math. Soc.}, 9(1):237--302, 1996.

\end{thebibliography}
\end{document}